\documentclass{amsart}
\usepackage{amssymb}
\usepackage{amsfonts}
\usepackage{geometry}

\newtheorem{theorem}{Theorem}
\theoremstyle{plain}
\newtheorem{acknowledgement}[theorem]{Acknowledgement}

\newtheorem{corollary}[theorem]{Corollary}

\newtheorem{definition}[theorem]{Definition}

\newtheorem{lemma}[theorem]{Lemma}

\newtheorem{proposition}[theorem]{Proposition}
\newtheorem{remark}[theorem]{Remark}

\numberwithin{equation}{section}
\input{tcilatex}
\begin{document}
\title[$T1$ theorem for comparable doubling weights]{A $T1$ theorem for
general Calder\'{o}n-Zygmund operators with comparable doubling weights, and
optimal cancellation conditions}
\author[E.T. Sawyer]{Eric T. Sawyer}
\address{ Department of Mathematics \& Statistics, McMaster University, 1280
Main Street West, Hamilton, Ontario, Canada L8S 4K1 }

\begin{abstract}
We begin an investigation into extending the $T1\,$theorem of David and Journ%
\'{e}, and the corresponding optimal cancellation conditions of Stein, to
more general pairs of distinct doubling weights. For example, when $0<\alpha
<n$, and $\sigma $ and $\omega $ are $A_{\infty }$ weights satisfying the
one-tailed Muckenhoupt conditions, and $K^{\alpha }$ is a smooth fractional
CZ kernel, we show there exists a bounded operator $T^{\alpha }:L^{2}\left(
\sigma \right) \rightarrow L^{2}\left( \omega \right) $ associated with $%
K^{\alpha }$ \emph{if and only if} there is a positive constant $\mathfrak{A}%
_{K^{\alpha }}\left( \sigma ,\omega \right) $ so that%
\begin{eqnarray*}
&&\int_{\left\Vert x-x_{0}\right\Vert <N}\left\vert \int_{\varepsilon
<\left\Vert x-y\right\Vert <N}K^{\alpha }\left( x,y\right) d\sigma \left(
y\right) \right\vert ^{2}d\omega \left( x\right) \leq \mathfrak{A}%
_{K^{\alpha }}\left( \sigma ,\omega \right) \ \int_{\left\Vert
x_{0}-y\right\Vert <N}d\sigma \left( y\right) , \\
&&\ \ \ \ \ \ \ \ \ \ \ \ \ \ \ \ \ \ \ \ \ \ \ \ \ \text{for all }%
0<\varepsilon <N\text{ and }x_{0}\in \mathbb{R}^{n},
\end{eqnarray*}%
where $\left\Vert y\right\Vert \equiv \max_{1\leq k\leq n}\left\vert
y_{k}\right\vert $, along with a dual inequality. More generally this holds
for measures $\sigma $ and $\omega $ comparable in the sense of Coifman and
Fefferman that satisfy a fractional $A_{\infty }^{\alpha }$ condition.

These results are deduced from the following theorem of $T1$ type, namely
that if $\sigma $ and $\omega $ are doubling measures, comparable in the
sense of Coifman and Fefferman, and satisfying one-tailed Muckenhoupt
conditions, then $T^{\alpha }:L^{2}\left( \sigma \right) \rightarrow
L^{2}\left( \omega \right) $ \emph{if and only if} the dual pair of testing
conditions hold, as well as a strong form of the weak boundedness propery,%
\begin{eqnarray*}
&&\left\vert \int_{F}\left( T^{\alpha }\mathbf{1}_{E}\right) d\omega
\right\vert \leq \mathcal{BICT}_{T^{\alpha }}\left( \sigma ,\omega \right) 
\sqrt{\left\vert Q_{\sigma }\right\vert _{\sigma }\left\vert Q_{\sigma
}\right\vert _{\omega }}, \\
&&\ \ \ \ \ \ \ \ \ \ \ \ \ \ \ \ \ \ \ \ \text{for all cubes }Q\subset 
\mathbb{R}^{n},
\end{eqnarray*}%
where $\mathcal{BICT}_{T^{\alpha }}\left( \sigma ,\omega \right) $ is a
positive constant called the bilinear cube/indicator testing constant. The
comparability of measures and the bilinear cube/indicator testing condition
can both be dropped\ if the stronger indicator/cube testing conditions are
assumed.
\end{abstract}

\dedicatory{In memory of Professor Elias M. Stein.}
\email{sawyer@mcmaster.ca}
\maketitle
\tableofcontents

\section{Introduction}

This paper includes the content of \cite{Saw2}, \cite{Saw3} and \cite{Saw4}
from the \texttt{arXiv}.

Given a Calder\'{o}n-Zygmund kernel $K\left( x,y\right) $ in Euclidean space 
$\mathbb{R}^{n}$, a classical problem for many decades was to identify
optimal cancellation conditions on $K$ so that there would exist an
associated singular integral operator $Tf\left( x\right) \sim \int K\left(
x,y\right) f\left( y\right) dy$ bounded on $L^{2}\left( \mathbb{R}%
^{n}\right) $. After a long history, involving contributions by many authors%
\footnote{%
see e.g. \cite[page 53]{Ste} for references to the earlier work in this
direction}, this effort culminated in the decisive $T1$ theorem of David and
Journ\'{e} \cite{DaJo}, in which boundedness of an operator $T$ on $%
L^{2}\left( \mathbb{R}^{n}\right) $ associated to $K$, was characterized by 
\begin{equation*}
T\mathbf{1},T^{\ast }\mathbf{1}\in BMO,
\end{equation*}%
together with a weak boundedness property for some $\eta >0$,%
\begin{eqnarray}
&&\left\vert \int_{Q}T\varphi \left( x\right) \ \psi \left( x\right)
dx\right\vert \lesssim \sqrt{\left\Vert \varphi \right\Vert _{\infty
}\left\vert Q\right\vert +\left\Vert \varphi \right\Vert _{\limfunc{Lip}\eta
}\left\vert Q\right\vert ^{1+\frac{\eta }{n}}}\sqrt{\left\Vert \psi
\right\Vert _{\infty }\left\vert Q\right\vert +\left\Vert \psi \right\Vert _{%
\limfunc{Lip}\eta }\left\vert Q\right\vert ^{1+\frac{\eta }{n}}},
\label{WBPDJ} \\
&&\text{for all }\varphi ,\psi \in \limfunc{Lip}\eta \text{ with }\func{Supp}%
\varphi ,\func{Supp}\psi \subset Q,\text{ and all cubes }Q\subset \mathbb{R}%
^{n};  \notag
\end{eqnarray}%
\emph{equivalently} by two testing conditions taken over indicators of cubes,%
\begin{equation*}
\int_{Q}\left\vert T\mathbf{1}_{Q}\left( x\right) \right\vert ^{2}dx\lesssim
\left\vert Q\right\vert \text{ and }\int_{Q}\left\vert T^{\ast }\mathbf{1}%
_{Q}\left( x\right) \right\vert ^{2}dx\lesssim \left\vert Q\right\vert ,\ \
\ \ \ \text{all cubes }Q\subset \mathbb{R}^{n}.
\end{equation*}%
The optimal cancellation conditions, which in the words of Stein were `a
rather direct consequence of' the $T1$ theorem, were given in \cite[Theorem
4, page 306]{Ste2}, involving integrals of the kernel over shells - see
Theorem \ref{Stein extension} below for an extension to certain more general
pairs of doubling weights and cubical shells.

An obvious next step is to replace Lebsegue measure with a fixed $A_{2}$
weight $w$,%
\begin{equation*}
\left( \frac{1}{\left\vert Q\right\vert }\int_{Q}w\left( x\right) dx\right)
\left( \frac{1}{\left\vert Q\right\vert }\int_{Q}\frac{1}{w\left( x\right) }%
dx\right) \lesssim 1,\ \ \ \ \ \text{all cubes }Q\subset \mathbb{R}^{n}\ ,
\end{equation*}%
and ask when $T$ is bounded on $L^{2}\left( w\right) $, i.e. satisfies the
one weight norm inequality. For elliptic Calder\'{o}n-Zygmund operators $T$,
this question is easily reduced to the David Journ\'{e} theorem using two
results from decades ago, namely the 1956 Stein-Weiss interpolation with
change of measures theorem \cite{StWe}, and the 1974 Coifman and Fefferman
extension \cite{CoFe} of the one weight Hilbert transform inequality of
Hunt, Muckenhoupt and Wheeden \cite{HuMuWh}, to a large class of general
Calder\'{o}n-Zygmund operators $T$\footnote{%
Indeed, if $T$ is bounded on $L^{2}\left( w\right) $, then by duality it is
also bounded on $L^{2}\left( \frac{1}{w}\right) $, and the Stein-Weiss
interpolation theorem with change of measure shows that $T$ is bounded on
unweighted $L^{2}\left( \mathbb{R}^{n}\right) $. Conversely, if $T$ is
bounded on unweighted $L^{2}\left( \mathbb{R}^{n}\right) $, the proof in 
\cite{CoFe} shows that $T$ is bounded on $L^{2}\left( w\right) $ using $w\in
A_{2}$.}.

However, for a pair of \emph{different} measures $\left( \sigma ,\omega
\right) $, the question is wide open in general, and we now turn to a brief
discussion of the problem of boundedness of a general Calder\'{o}n-Zygmund
operator $T$ from one general $L^{2}\left( \sigma \right) $ space to another 
$L^{2}\left( \omega \right) $ space. In the case of the Hilbert transform $H$
in dimension one, the two weight inequality was completely solved in the two
part paper \cite{LaSaShUr3};\cite{Lac}, where it was shown that $H$ is
bounded from $L^{2}\left( \sigma \right) $ to $L^{2}\left( \omega \right) $
if and only if the testing and one-tailed Muckenhoupt conditions hold, i.e.%
\begin{eqnarray*}
&&\int_{I}\left\vert H\left( \mathbf{1}_{I}\sigma \right) \right\vert
^{2}d\omega \lesssim \int_{I}d\sigma \text{ and }\int_{I}\left\vert H\left( 
\mathbf{1}_{I}\omega \right) \right\vert ^{2}d\sigma \lesssim
\int_{I}d\omega ,\ \ \ \ \ \text{all intervals }I\subset \mathbb{R}^{n}, \\
&&\left( \int_{\mathbb{R}}\frac{\left\vert I\right\vert }{\left\vert
I\right\vert ^{2}+\left\vert x-c_{I}\right\vert ^{2}}d\sigma \left( x\right)
\right) \left( \frac{1}{\left\vert I\right\vert }\int_{I}d\omega \right)
\lesssim 1,\ \text{and its dual},\ \ \ \ \ \text{all intervals }I\subset 
\mathbb{R}^{n},
\end{eqnarray*}%
and for fractional Riesz transforms in higher dimensions, it is known that
the two weight norm inequality \emph{with doubling measures} is equivalent
to the fractional one-tailed Muckenhoupt and $T1$ cube testing conditions,
see \cite[Theorem 1.4]{LaWi} and \cite[Theorem 2.11]{SaShUr9}. Here a
positive measure $\mu $ is doubling if%
\begin{equation*}
\int_{2Q}d\mu \lesssim \int_{Q}d\mu ,\ \ \ \ \ \text{all cubes }Q\subset 
\mathbb{R}^{n}.
\end{equation*}%
However, these results rely on certain `positivity' properties of the
gradient of the kernel (which for the Hilbert transform kernel $\frac{1}{y-x}
$ is simply $\frac{d}{dx}\frac{1}{y-x}>0$ for $x\neq y$), something that is
not available for general elliptic, or even strongly elliptic, fractional
Calder\'{o}n-Zygmund operators.

Our point of departure in this paper is the fact, easily proven below, that
for doubling weights, certain weak analogues of the pivotal conditions of
Nazarov, Treil and Volberg (often referred to now as NTV) \cite{NTV4} are
necessary, and this provides the framework for moving forward\footnote{%
We do not know if the usual pivotal conditions hold\ for doubling measures
that satisfy the Muckenhoupt conditions, and we thank Ignacio Uriarte-Tuero
for bringing this to our attention by pointing to an error in a previous
version of this paper.}. So we will assume that our weight pair $\left(
\sigma ,\omega \right) $ consists of doubling measures, and satisfies at
least the classical $A_{2}^{\alpha }$ condition of Muckenhoupt, and often
the one-tailed versions $\mathcal{A}_{2}^{\alpha }$ and $\mathcal{A}%
_{2}^{\alpha ,\ast }$ in \cite{SaShUr7}. The former condition is a necessary
consequence of boundedness of any elliptic Calder\'{o}n-Zygmund operator $T$%
, and the latter condition is necessary if $T$ is strongly elliptic, see 
\cite{SaShUr7}.

Finally, we will at times also require that the doubling measures $\sigma $
and $\omega $ are \emph{comparable} in the sense of Coifman and Fefferman 
\cite{CoFe}, which means that the measures are mutually absolutely
continuous, uniformly at all scales - i.e. there exist $0<\beta ,\gamma <1$
such that%
\begin{equation*}
\frac{\left\vert E\right\vert _{\sigma }}{\left\vert Q\right\vert _{\sigma }}%
<\beta \Longrightarrow \frac{\left\vert E\right\vert _{\omega }}{\left\vert
Q\right\vert _{\omega }}<\gamma \text{ for all Borel subsets }E\text{ of a
cube }Q.
\end{equation*}%
This condition is needed to prove the two weight bilinear Carleson Embedding
Theorem \ref{2 wt bil CET}\ below, and conversely we show that our bilinear
theorem implies it. The point is that if $\sigma $ and $\omega $ are
doubling and \emph{comparable}, then a collection of dyadic cubes $\mathcal{F%
}$ is $\sigma $-Carleson if and only if $\mathcal{F}$ is $\omega $-Carleson%
\footnote{%
We thank Alex Tkachman for pointing to an error in the proof of an earlier
incorrect version of Theorem \ref{2 wt bil CET}, where comparability was
needed but not assumed.} - see the next section for these definitions.

\begin{remark}
We do \textbf{not} assume in this paper that the weight pair $\left( \sigma
,\omega \right) $ satisfies the very strong \emph{energy} conditions,
something that is only necessary for boundedness of the Hilbert transform
and its perturbations on the real line (see \cite{SaShUr11} and \cite{Saw}),
nor the $k$\emph{-energy dispersed }conditions introduced in \cite{SaShUr10}%
, which only hold for perturbations of Riesz transforms in higher dimensions.
\end{remark}

The purpose of this paper is to consider measures $\sigma $ and $\omega $ in 
$\mathbb{R}^{n}$ that are

\begin{itemize}
\item \emph{doubling,}

\item and satisfy the \emph{one-tailed Muckenhoupt} conditions,
\end{itemize}

and then:

\begin{enumerate}
\item to characterize the \emph{two weight} norm inequality, for the class
of elliptic $\alpha $-fractional singular integral operators in $\mathbb{R}%
^{n}$, in terms of the $\mathcal{A}_{2}^{\alpha }$ conditions and the
Indicator/Cube Testing conditions - and if in addition the measures are
comparable in the sense of Coifman and Ferfferman, then in terms of $%
\mathcal{A}_{2}^{\alpha }$, the Cube Testing conditions, and the Bilinear
Indicator/Cube Testing property (often referred to now as the $\mathcal{BICT}
$ property), which plays a role analogous to the weak boundedness property (%
\ref{WBPDJ}) with $\gamma =0$. See Theorem \ref{pivotal theorem} below.

\item to eliminate the $\mathcal{BICT}$ property when the measures are both
doubling and comparable, with each of them satisfying either the $A_{\infty
}^{\alpha }$ condition, or the $C_{q}$ condition of Muckenhoupt for some $%
q>2 $, see \cite{Muc} and \cite{Saw1}\footnote{%
In \cite{Ler} Lerner has introduced a strong $SC_{2}$ condition on a weight $%
w$ which characterizes the related inequality $\left\Vert T_{\flat
}f\right\Vert _{L^{2,\infty }\left( w\right) }\lesssim \left\Vert
Mf\right\Vert _{L^{2}\left( w\right) }$.},

\begin{enumerate}
\item and furthermore, in the case when the measures are $A_{\infty }$
weights\footnote{%
In the case that both $\sigma $ and $\omega $ are $A_{\infty }$ weights
satisfying $A_{2}\left( \sigma ,\omega \right) <\infty $, C. Grigoriadis has
shown in \texttt{arXiv:2009.12091} that the classical pivotal conditions
hold, resulting in a $T1$ theorem for nonsmooth kernels.}, or more generally 
$C_{2+\varepsilon }$ weights or fractional $A_{\infty }^{\alpha }$ measures,
to give optimal cancellation conditions on a smooth Calder\'{o}n-Zygmund
kernel in order that there is an associated bounded operator from $%
L^{2}\left( \sigma \right) $ to $L^{2}\left( \omega \right) $, extending the
smooth part of Theorem 4 in \cite[Section 3 of Chapter VII]{Ste2}, see
Theorem \ref{Stein extension} below,
\end{enumerate}

\item and to give a function theoretic consequence, namely that strong type
is equivalent to weak type and dual weak type for elliptic operators, see
Corollary \ref{elliptic corollary} below. A one weight version of this
result, with optimal $A_{2}$ dependence, was obtained by Perez, Treil and
Volberg \cite[Theorem 2.1]{PeTrVo}.
\end{enumerate}

\subsection{Discussion of methodology}

Since the weaker pivotal conditions, that can be derived from doubling
measures, involve Poisson integrals whose tails have higher powers, we are
led naturally to the use of the weighted Alpert wavelets in \cite{RaSaWi},
instead of the traditional Haar wavelets, having correspondingly higher
order vanishing moments. In order to handle the global form associated with
the operator, it suffices to use testing over polynomials times indicators
of cubes. However, as pointed out in \cite{RaSaWi}, the weighted Alpert
wavelets, unlike the weighted Haar wavelets, do not behave well with respect
to the famous Paraproduct / Neighbour / Stopping (often referred to now as
P/N/S) form decomposition of NTV (because the extension of a nonconstant
polynomial\ from one cube to another is uncontrolled), and so we must divert
to an alternate fork in the proof path using the Parallel Corona, in which 
\emph{independent} stopping times are used for each function in a bilinear
form, in order to handle the local form. In the absense of a P/N/S
decomposition, this alternate fork then permits testing over polynomials
times indicators of cubes, coupled with testing a bilinear indicator/cube
testing property $\mathcal{BICT}$, taken over indicators of \emph{subsets}
of cubes on the left, rather than the cubes themselves,%
\begin{equation*}
\left\vert \int_{E}T^{\alpha }\left( \mathbf{1}_{F}\sigma \right) d\omega
\right\vert \lesssim \sqrt{\left\vert Q\right\vert _{\sigma }\left\vert
Q\right\vert _{\omega }},\ \ \ \ \ \text{Borel subsets }E,F\text{ of a cube }%
Q.
\end{equation*}%
On the real line, a stronger conjecture is made in \cite{RaSaWi} that the
norm inequality holds if testing over these polynomials times indicators of
intervals holds, in the presence of energy conditions, and that conjecture
remains open at this time.

Moreover, in the proof of our theorem, we will need to bound the $L^{\infty
} $ norm of $L^{2}\left( \mu \right) $-projections onto the space of
restrictions to $Q$ of polynomials of degree less than $\kappa $ (which is
trivial when $\kappa =1$), and for this we use the nondegeneracy conditions 
\begin{equation}
\frac{1}{\left\vert Q\right\vert _{\mu }}\int_{Q}\left\vert P\left( \frac{%
x-c_{Q}}{\ell \left( Q\right) }\right) \right\vert ^{2}d\mu \left( x\right)
\geq c>0,  \label{nondeg}
\end{equation}%
for all cubes $Q$ and normalized polynomials $P$ of degree less than $\kappa 
$, and with $\mu $ equal to either measure $\sigma ,\omega $. Such
conditions permit control of off-diagonal terms by a Calder\'{o}n-Zygmund
stopping time and corona decomposition. We will see that (\ref{nondeg}) is
implied by the doubling property for $\mu $, and provided $\kappa $ is large
enough\footnote{%
This restriction is removed in Sawyer and Uriarte-Tuero \cite{SaUr}.},
doubling is implied by (\ref{nondeg}), providing yet another instance of
poor behaviour of weighted Alpert wavelets, as compared to that for weighted
Haar wavelets. Thus doubling conditions on the weights permit a proof of NTV
type as in \cite{NTV4}, that both avoids the difficult control of functional
energy in \cite{LaSaShUr3} and \cite{SaShUr7}, and Lacey's deep breakthrough
in controlling the stopping form \cite{Lac}, of course at the expense of
including bilinear indicator/cube testing.

On the other hand, we are able to replace polynomial testing by the usual
Cube Testing in the setting of doubling weights, and if we make one
additional assumption, namely that the measures satisfy the fractional $%
A_{\infty }^{\alpha }$ condition, then we can do away with the $\mathcal{BICT%
}$ property as well. Our approach to these results will follow the series of
papers \cite{Saw2}, \cite{Saw3} and \cite{Saw4}:

\begin{enumerate}
\item First we prove that the two weight norm inequality, for a general
elliptic $\alpha $-fractional Calderon-Zygmund singular integral with
comparable doubling weights, is controlled by the \emph{classical} $%
A_{2}^{\alpha }$ condition of Muckenhoupt, the two dual \emph{Polynomial}%
/Cube Testing conditions (referred to now as $P/CT$), the Bilinear
Indicator/Cube Testing property, and a certain weak boundedness property -
this latter property is then removed using the doubling properties of the
measures together with the $\mathcal{A}_{2}^{\alpha }$ conditions.

\item Second, we replace the $P/CT$ conditions in the previous theorem with
the usual Cube Testing condition over indicators of cubes, assuming only
that the one-tailed Muckenhoupt constants $\mathcal{A}_{2}^{\alpha }$ and $%
\mathcal{A}_{2}^{\alpha ,\ast }$ are both finite.

\item Third, we eliminate the $\mathcal{BICT}$ from each of the previous two
theorems if one of the measures satisfies the fractional $A_{\infty
}^{\alpha }$ condition.

\item Finally, we use the previous result to extend Stein's cancellation
theorem to certain pairs of doubling measures satisfying the one-tailed $%
\mathcal{A}_{2}^{\alpha }$ conditions.
\end{enumerate}

\begin{acknowledgement}
We thank the referee for a number of corrections and simplifications of
arguments that greatly contribute to the readability of the paper.
\end{acknowledgement}

\section{Two weight $T1$ theorems and cancellation conditions for Calder\'{o}%
n-Zygmund operators}

Denote by $\mathcal{P}^{n}$ the collection of cubes in $\mathbb{R}^{n}$
having sides parallel to the coordinate axes. A positive locally finite
Borel measure $\mu $ on $\mathbb{R}^{n}$ is said to satisfy the\emph{\
doubling condition} if there is a pair of constants $\left( \beta ,\gamma
\right) \in \left( 0,1\right) ^{2}$, called doubling parameters, such that%
\begin{equation}
\left\vert \beta Q\right\vert _{\mu }\geq \gamma \left\vert Q\right\vert
_{\mu }\ ,\ \ \ \ \ \text{for all cubes }Q\in \mathcal{P}^{n}.
\label{def rev doub}
\end{equation}%
A familiar equivalent reformulation of (\ref{def rev doub}) is that there is
a positive constant $C_{\limfunc{doub}}$, called the doubling constant, such
that $\left\vert 2Q\right\vert _{\mu }\leq C_{\limfunc{doub}}\left\vert
Q\right\vert _{\mu }$ for all cubes $Q\in \mathcal{P}^{n}$.

\subsection{Conditions on measures and kernels}

We begin with various conditions on measures and measure pairs, with the
fractional $A_{\infty }^{\alpha }$ condition being new, and the others
classical. Then\ we recall ellipticity conditions for Calder\'{o}n-Zygmund
kernels.

\subsubsection{The $A_{\infty }$ and $C_{p}$ conditions}

The absolutely continuous measure $d\sigma \left( x\right) =s\left( x\right)
dx$ is said to be an $A_{\infty }$ weight if there are constants $%
0<\varepsilon ,\eta <1$, called $A_{\infty }$ parameters, such that%
\begin{equation*}
\frac{\left\vert E\right\vert _{\sigma }}{\left\vert Q\right\vert _{\sigma }}%
<\eta \text{ whenever }E\text{ compact }\subset Q\text{ a cube with }\frac{%
\left\vert E\right\vert }{\left\vert Q\right\vert }<\varepsilon .
\end{equation*}%
A useful reformulation given in \cite[Theorem III on page 244]{CoFe} is that
there is $C>0$ and an $A_{\infty }$ exponent $\varepsilon >0$ such that%
\begin{equation}
\frac{\left\vert E\right\vert _{\sigma }}{\left\vert Q\right\vert _{\sigma }}%
\leq C\left( \frac{\left\vert E\right\vert }{\left\vert Q\right\vert }%
\right) ^{\varepsilon }\text{ whenever }E\text{ compact }\subset Q\text{ a
cube}.  \label{reform}
\end{equation}%
Recall that there are doubling measures that are mutually singular with
respect to Lebesgue measure, for a nice exposition see\ e.g. \cite{GaKiSc},
and references given there.

Finally an absolutely continuous measure $d\sigma \left( x\right) =s\left(
x\right) dx$ is said to be a $C_{p}$ weight for $1<p<\infty $ if%
\begin{equation*}
\frac{\left\vert E\right\vert _{\sigma }}{\int_{\mathbb{R}^{n}}\left\vert M%
\mathbf{1}_{Q}\right\vert ^{p}d\sigma }\leq C\left( \frac{\left\vert
E\right\vert }{\left\vert Q\right\vert }\right) ^{\varepsilon }\text{
whenever }E\text{ compact }\subset Q\text{ a cube}.
\end{equation*}

\subsubsection{Comparability of measures}

A generalization of the $A_{\infty }$ property to more general pairs $\left(
\sigma ,\omega \right) $ of doubling measures was also given by Coifman and
Fefferman in \cite{CoFe}: A pair $\left( \sigma ,\omega \right) $ of
doubling measures is \emph{comparable} if there are constants $0<\varepsilon
,\eta <1$, called comparability parameters, such that%
\begin{equation}
\frac{\left\vert E\right\vert _{\sigma }}{\left\vert Q\right\vert _{\sigma }}%
<\eta \text{ whenever }E\text{ compact }\subset Q\text{ a cube with }\frac{%
\left\vert E\right\vert _{\omega }}{\left\vert Q\right\vert _{\omega }}%
<\varepsilon .  \label{comp param}
\end{equation}%
This condition is easily seen to be symmetric and the reformulation proved
in \cite{CoFe} is that there is $C>0$ and a comparability exponent $%
\varepsilon >0$ such that%
\begin{equation}
\frac{\left\vert E\right\vert _{\sigma }}{\left\vert Q\right\vert _{\sigma }}%
\leq C\left( \frac{\left\vert E\right\vert _{\omega }}{\left\vert
Q\right\vert _{\omega }}\right) ^{\varepsilon }\text{ whenever }E\text{
compact }\subset Q\text{ a cube}.  \label{comparable}
\end{equation}

A further reformulation is given in \cite[Theorem 3 (ii)]{Saw0} in terms of
Carleson conditions, namely that a pair $\left( \sigma ,\omega \right) $ of
doubling measures is comparable if and only if%
\begin{equation}
\left\Vert \mathcal{F}\right\Vert _{\limfunc{Car}\left( \sigma \right) }\leq
C\left\Vert \mathcal{F}\right\Vert _{\limfunc{Car}\left( \omega \right) }\
,\ \ \ \ \ \text{for all grids }\mathcal{F}\subset \mathcal{D},
\label{sigma omega}
\end{equation}%
where for a positive locally finite Borel measure $\mu $, we say that a grid 
$\mathcal{F}\subset \mathcal{D}$, is $\mu $-Carleson if 
\begin{equation*}
\sum_{Q^{\prime }\in \mathcal{F}:\ Q^{\prime }\subset Q}\left\vert Q^{\prime
}\right\vert _{\mu }\leq C\left\vert Q\right\vert _{\mu }\ ,\ \ \ \ \ \text{%
for all }Q\in \mathcal{F},
\end{equation*}%
and we define\ the `Carleson norm' $\left\Vert \mathcal{F}\right\Vert _{%
\limfunc{Car}\left( \mu \right) }$ of $\mathcal{F}$\ to be\ the infimum of
such constants $C$. We repeat the simple proof here for the sake of
completeness. Suppose (\ref{comparable}) holds.\ Given $Q\in \mathcal{F}$,
let $G_{k}\left( Q\right) \equiv \dbigcup\limits_{Q^{\prime }\in \mathfrak{C}%
_{\mathcal{F}}^{\left( k\right) }}Q^{\prime }$ be the union of the $k^{th}$%
-grandchildren of $Q$. Since the\ grid $\mathcal{F}$ satisfies a Carleson
condition with respect to $\omega $, there are positive constants $C,\delta
>0$ such that $\left\vert G_{k}\left( Q\right) \right\vert _{\omega }\leq
C2^{-k\delta }\left\vert Q\right\vert _{\omega }$ (this is a well-known
consequence of the Carleson condition dating back to a paper of Carleson).
Then we have for $Q\in \mathcal{F}$,%
\begin{equation*}
\sum_{Q^{\prime }\in \mathcal{F}:\ Q^{\prime }\subset Q}\left\vert Q^{\prime
}\right\vert _{\sigma }=\sum_{k=0}^{\infty }\left\vert G_{k}\left( Q\right)
\right\vert _{\sigma }\leq \sum_{k=0}^{\infty }\left\vert Q\right\vert
_{\sigma }C\left( \frac{\left\vert G_{k}\left( Q\right) \right\vert _{\omega
}}{\left\vert Q\right\vert _{\omega }}\right) ^{\varepsilon }\leq
\sum_{k=0}^{\infty }\left\vert Q\right\vert _{\sigma }C\left( C2^{-k\delta
}\right) ^{\varepsilon }\leq C\left\vert Q\right\vert _{\sigma }\ .
\end{equation*}%
Conversely, (\ref{sigma omega}) implies (\ref{comp param}) is easy.

Since comparability of doubling measures is symmetric,\ (\ref{sigma omega})
is equivalent to its dual%
\begin{equation}
\left\Vert \mathcal{F}\right\Vert _{\limfunc{Car}\left( \omega \right) }\leq
C\left\Vert \mathcal{F}\right\Vert _{\limfunc{Car}\left( \sigma \right) }\
,\ \ \ \ \ \text{for all grids }\mathcal{F}\subset \mathcal{D}.
\label{omega sigma}
\end{equation}%
This suggests the following extension of comparability to more general pairs
of measures.

\begin{definition}
\label{def comparable}A pair $\left( \sigma ,\omega \right) $ of positive
locally finite Borel measures is said to be comparable if both (\ref{sigma
omega}) and (\ref{omega sigma}) hold.
\end{definition}

Note that the equivalence of (\ref{sigma omega}) and (\ref{omega sigma}) for
pairs of doubling measures, doesn't carry over to more general pairs of
measures, which explains why we incorporate both conditions (\ref{sigma
omega}) and (\ref{omega sigma}) in the definition of comparability.

\subsubsection{The fractional $A_{\infty }^{\protect\alpha }$ condition}

In order to\ introduce the larger class of measures satisfying the
fractional $A_{\infty }^{\alpha }$ condition, we define a relative $\alpha $%
-capacity $\mathbf{Cap}_{\alpha }\left( E;Q\right) $ of a compact subset $E$
of a cube $Q$ by%
\begin{eqnarray*}
\mathbf{Cap}_{\alpha }\left( E;Q\right) &\equiv &\inf \left\{ \int h\left(
x\right) dx:h\geq 0,\func{Supp}h\subset 2Q\text{ and }I_{\alpha }h\geq
\left( \func{diam}2Q\right) ^{\alpha -n}\text{ on }E\right\} \\
&\approx &\inf \left\{ \left\vert 2Q\right\vert ^{\frac{\alpha }{n}%
-1}\int_{2Q}h\left( x\right) dx:h\geq 0,\func{Supp}h\subset 2Q\text{ and }%
I_{\alpha }h\geq 1\text{ on }E\right\} .
\end{eqnarray*}%
This relative capacity is\ closely related to the $\left( \alpha ,1\right) $%
-capacity as defined e.g. in Adams and Hedberg \cite{AdHe}, where numerous
properties of capacities are developed. We now use this relative capacity to
define a fractional $A_{\infty }^{\alpha }$ condition (different than the $%
A_{\infty }^{\alpha }$ condition appearing in \cite[page 818]{SaWh}).

\begin{definition}
A locally finite positive Borel measure $\omega $ is said to be an $%
A_{\infty }^{\alpha }$ measure if%
\begin{eqnarray*}
&&\frac{\left\vert E\right\vert _{\omega }}{\left\vert 2Q\right\vert
_{\omega }}\leq \eta \left( \mathbf{Cap}_{\alpha }\left( E;Q\right) \right)
,\ \ \ \ \ \ \text{for all compact subsets }E\text{ of a cube }Q, \\
&&\ \ \ \ \ \ \ \ \ \ \text{for some function }\eta :\left[ 0,1\right]
\rightarrow \left[ 0,1\right] \text{ with }\lim_{t\searrow 0}\eta \left(
t\right) =0.
\end{eqnarray*}
\end{definition}

Note that omitting the factor $2$ in $\left\vert 2Q\right\vert _{\omega }$
above makes the condition more restrictive in general, but remains
equivalent for doubling measures. We let $A_{\infty }^{0}$ be the class of $%
A_{\infty }$ weights, so that statements involving both $A_{\infty }$ and $%
A_{\infty }^{\alpha }$ can be given together. Note the inequalities%
\begin{equation*}
c\left( \frac{\left\vert E\right\vert }{\left\vert Q\right\vert }\right) ^{1-%
\frac{\alpha }{n}}\leq \mathbf{Cap}_{\alpha }\left( E;Q\right) \mathbf{\leq
Cap}_{\beta }\left( E;Q\right) ,\ \ \ \ \ 0<\alpha <\beta <n,
\end{equation*}%
which hold since 
\begin{equation*}
\left( \func{diam}2Q\right) ^{\alpha -n}\left\vert E\right\vert ^{1-\frac{%
\alpha }{n}}\leq \left( \func{diam}2Q\right) ^{\alpha -n}\left\vert \left\{
I_{\alpha }h\geq \left( \func{diam}2Q\right) ^{\alpha -n}\right\}
\right\vert ^{1-\frac{\alpha }{n}}\leq \left\Vert I_{\alpha }\right\Vert
_{L^{1}\rightarrow L^{\frac{n}{n-\alpha },\infty }}\int h
\end{equation*}%
implies $\mathbf{Cap}_{\alpha }\left( E;Q\right) \geq \frac{\left( \func{diam%
}2Q\right) ^{\alpha -n}\left\vert E\right\vert ^{1-\frac{\alpha }{n}}}{%
\left\Vert I_{\alpha }\right\Vert _{L^{1}\rightarrow L^{\frac{n}{n-\alpha }%
,\infty }}}\int h=C\left( \frac{\left\vert E\right\vert }{\left\vert
Q\right\vert }\right) ^{1-\frac{\alpha }{n}}$, and since $\frac{I_{\alpha
}h\left( x\right) }{\left( \func{diam}2Q\right) ^{\alpha -n}}%
=\int_{2Q}\left( \frac{\left\vert x-y\right\vert }{\func{diam}2Q}\right)
^{\alpha -n}dy$ is decreasing in $\alpha $. It follows that%
\begin{equation*}
A_{\infty }\subset A_{\infty }^{\alpha }\subset A_{\infty }^{\beta }\text{\
\ \ \ \ for }0<\alpha <\beta <n.
\end{equation*}

\subsubsection{The Muckenhoupt conditions}

\begin{definition}
Let $\sigma $ and $\omega $ be locally finite positive Borel measures on $%
\mathbb{R}^{n}$, and denote by $\mathcal{P}^{n}$ the collection of all cubes
in $\mathbb{R}^{n}$ with sides parallel to the coordinate axes. For $0\leq
\alpha <n$, the classical $\alpha $-fractional Muckenhoupt condition for the
weight pair $\left( \sigma ,\omega \right) $ is given by%
\begin{equation}
A_{2}^{\alpha }\left( \sigma ,\omega \right) \equiv \sup_{Q\in \mathcal{P}%
^{n}}\frac{\left\vert Q\right\vert _{\sigma }}{\left\vert Q\right\vert ^{1-%
\frac{\alpha }{n}}}\frac{\left\vert Q\right\vert _{\omega }}{\left\vert
Q\right\vert ^{1-\frac{\alpha }{n}}}<\infty ,  \label{frac Muck}
\end{equation}%
and the corresponding one-tailed conditions by%
\begin{eqnarray}
\mathcal{A}_{2}^{\alpha }\left( \sigma ,\omega \right) &\equiv &\sup_{Q\in 
\mathcal{Q}^{n}}\mathcal{P}^{\alpha }\left( Q,\sigma \right) \frac{%
\left\vert Q\right\vert _{\omega }}{\left\vert Q\right\vert ^{1-\frac{\alpha 
}{n}}}<\infty ,  \label{one-sided} \\
\mathcal{A}_{2}^{\alpha ,\ast }\left( \sigma ,\omega \right) &\equiv
&\sup_{Q\in \mathcal{Q}^{n}}\frac{\left\vert Q\right\vert _{\sigma }}{%
\left\vert Q\right\vert ^{1-\frac{\alpha }{n}}}\mathcal{P}^{\alpha }\left(
Q,\omega \right) <\infty ,  \notag
\end{eqnarray}%
where the reproducing Poisson integral $\mathcal{P}^{\alpha }$ is given by%
\begin{equation*}
\mathcal{P}^{\alpha }\left( Q,\mu \right) \equiv \int_{\mathbb{R}^{n}}\left( 
\frac{\left\vert Q\right\vert ^{\frac{1}{n}}}{\left( \left\vert Q\right\vert
^{\frac{1}{n}}+\left\vert x-x_{Q}\right\vert \right) ^{2}}\right) ^{n-\alpha
}d\mu \left( x\right) .
\end{equation*}
\end{definition}

\subsubsection{Ellipticity of kernels}

Finally, as in \cite{SaShUr7}, an $\alpha $-fractional vector Calder\'{o}%
n-Zygmund kernel $K^{\alpha }=\left( K_{j}^{\alpha }\right) $ is said to be 
\emph{elliptic} if there is $c>0$ such that for each unit vector $\mathbf{u}%
\in \mathbb{R}^{n}$ there is $j$ satisfying%
\begin{equation*}
\left\vert K_{j}^{\alpha }\left( x,x+tu\right) \right\vert \geq ct^{\alpha
-n},\ \ \ \ \ \text{for all }t>0;
\end{equation*}%
and $K^{\alpha }=\left( K_{j}^{\alpha }\right) $ is said to be \emph{%
strongly elliptic} if for each $m\in \left\{ 1,-1\right\} ^{n}$, there is a
sequence of coefficients $\left\{ \lambda _{j}^{m}\right\} _{j=1}^{J}$ such
that%
\begin{equation}
\left\vert \sum_{j=1}^{J}\lambda _{j}^{m}K_{j}^{\alpha }\left( x,x+t\mathbf{u%
}\right) \right\vert \geq ct^{\alpha -n},\ \ \ \ \ t\in \mathbb{R}.
\label{Ktalpha strong}
\end{equation}%
holds for \emph{all} unit vectors $\mathbf{u}$ in the $n$-ant 
\begin{equation*}
V_{m}=\left\{ x\in \mathbb{R}^{n}:m_{i}x_{i}>0\text{ for }1\leq i\leq
n\right\} ,\ \ \ \ \ m\in \left\{ 1,-1\right\} ^{n}.
\end{equation*}%
For example, the vector Riesz transform kernel is strongly elliptic (\cite%
{SaShUr7}).

\subsection{Standard fractional singular integrals, the norm inequality and
testing conditions}

Let $0\leq \alpha <n$ and $\kappa _{1},\kappa _{2}\in \mathbb{N}$. We define
a standard $\left( \kappa _{1}+\delta ,\kappa _{2}+\delta \right) $-smooth $%
\alpha $-fractional Calder\'{o}n-Zygmund kernel $K^{\alpha }(x,y)$ to be a
function $K^{\alpha }:\mathbb{R}^{n}\times \mathbb{R}^{n}\rightarrow \mathbb{%
R}$ satisfying the following fractional size and smoothness conditions: For $%
x\neq y$, and with $\nabla _{1}$ denoting gradient in the first variable,
and $\nabla _{2}$ denoting gradient in the second variable, 
\begin{eqnarray}
&&\left\vert \nabla _{1}^{j}K^{\alpha }\left( x,y\right) \right\vert \leq
C_{CZ}\left\vert x-y\right\vert ^{\alpha -j-n-1},\ \ \ \ \ 0\leq j\leq
\kappa _{1},  \label{sizeandsmoothness'} \\
&&\left\vert \nabla _{1}^{\kappa }K^{\alpha }\left( x,y\right) -\nabla
_{1}^{\kappa }K^{\alpha }\left( x^{\prime },y\right) \right\vert \leq
C_{CZ}\left( \frac{\left\vert x-x^{\prime }\right\vert }{\left\vert
x-y\right\vert }\right) ^{\delta }\left\vert x-y\right\vert ^{\alpha -\kappa
_{1}-n-1},\ \ \ \ \ \frac{\left\vert x-x^{\prime }\right\vert }{\left\vert
x-y\right\vert }\leq \frac{1}{2},  \notag
\end{eqnarray}%
and where the same inequalities hold for the adjoint kernel $K^{\alpha ,\ast
}\left( x,y\right) \equiv K^{\alpha }\left( y,x\right) $, in which $x$ and $%
y $ are interchanged, and where $\kappa _{1}$ is replaced by $\kappa _{2}$,
and $\nabla _{1}$\ by $\nabla _{2}$.

\subsubsection{Defining the norm inequality\label{Subsubsection norm}}

We now turn to a precise definition of the weighted norm inequality%
\begin{equation}
\left\Vert T_{\sigma }^{\alpha }f\right\Vert _{L^{2}\left( \omega \right)
}\leq \mathfrak{N}_{T^{\alpha }}\left\Vert f\right\Vert _{L^{2}\left( \sigma
\right) },\ \ \ \ \ f\in L^{2}\left( \sigma \right) ,  \label{two weight'}
\end{equation}%
where of course $L^{2}\left( \sigma \right) $ is the Hilbert space
consisting of those functions $f:\mathbb{R}^{n}\rightarrow \mathbb{R}$ for
which 
\begin{equation*}
\left\Vert f\right\Vert _{L^{2}\left( \sigma \right) }\equiv \sqrt{\int_{%
\mathbb{R}^{n}}\left\vert f\left( x\right) \right\vert ^{2}d\sigma \left(
x\right) }<\infty ,
\end{equation*}%
and equipped with the usual inner product. A similar definition holds for $%
L^{2}\left( \omega \right) $. For a precise definition of (\ref{two weight'}%
), it is possible to proceed with the notion of associating operators and
kernels through the identity (\ref{identify}), and more simply by using the
notion of restricted boundedness introduced by Liaw and Treil in \cite[see
Theorem 3.4]{LiTr}. However, we choose to follow the approach in \cite[see
page 314]{SaShUr9}. So we suppose that $K^{\alpha }$ is a standard $\left(
\kappa _{1}+\delta ,\kappa _{2}+\delta \right) $-smooth $\alpha $-fractional
Calder\'{o}n-Zygmund kernel, and we introduce a family $\left\{ \eta
_{\delta ,R}^{\alpha }\right\} _{0<\delta <R<\infty }$ of nonnegative
functions on $\left[ 0,\infty \right) $ so that the truncated kernels $%
K_{\delta ,R}^{\alpha }\left( x,y\right) =\eta _{\delta ,R}^{\alpha }\left(
\left\vert x-y\right\vert \right) K^{\alpha }\left( x,y\right) $ are bounded
with compact support for fixed $x$ or $y$, and uniformly satisfy (\ref%
{sizeandsmoothness'}). Then the truncated operators 
\begin{equation*}
T_{\sigma ,\delta ,R}^{\alpha }f\left( x\right) \equiv \int_{\mathbb{R}%
^{n}}K_{\delta ,R}^{\alpha }\left( x,y\right) f\left( y\right) d\sigma
\left( y\right) ,\ \ \ \ \ x\in \mathbb{R}^{n},
\end{equation*}%
are pointwise well-defined, and we will refer to the pair $\left( K^{\alpha
},\left\{ \eta _{\delta ,R}^{\alpha }\right\} _{0<\delta <R<\infty }\right) $
as an $\alpha $-fractional singular integral operator, which we typically
denote by $T^{\alpha }$, suppressing the dependence on the truncations. We
also consider vector kernels $K^{\alpha }=\left( K_{j}^{\alpha }\right) $
where each $K_{j}^{\alpha }$ is as above, often without explicit mention.
This includes for example the vector Riesz transform in higher dimensions.

\begin{definition}
We say that an $\alpha $-fractional singular integral operator $T^{\alpha
}=\left( K^{\alpha },\left\{ \eta _{\delta ,R}^{\alpha }\right\} _{0<\delta
<R<\infty }\right) $ satisfies the norm inequality (\ref{two weight'})
provided%
\begin{equation*}
\left\Vert T_{\sigma ,\delta ,R}^{\alpha }f\right\Vert _{L^{2}\left( \omega
\right) }\leq \mathfrak{N}_{T^{\alpha }}\left( \sigma ,\omega \right)
\left\Vert f\right\Vert _{L^{2}\left( \sigma \right) },\ \ \ \ \ f\in
L^{2}\left( \sigma \right) ,0<\delta <R<\infty .
\end{equation*}
\end{definition}

\begin{description}
\item[Independence of Truncations] \label{independence}In the presence of
the classical Muckenhoupt condition $A_{2}^{\alpha }$, the norm inequality (%
\ref{two weight'}) is essentially independent of the choice of truncations
used, including \emph{nonsmooth} truncations as well - see \cite{LaSaShUr3}.
However, in dealing with the Monotonicity Lemma \ref{mono} below, where $%
\kappa ^{th}$ order Taylor approximations are made on the truncated kernels,
it is necessary to use sufficiently smooth truncations. Similar comments
apply to the Bilinear Indicator/Cube Testing property (\ref{def ind WBP})
and the Indicator/Cube Testing conditions (\ref{def kth order testing}), as
well as to the $\kappa $-cube testing conditions (\ref{kappa testing}) used
later in the proof.
\end{description}

The weak type norms $\mathfrak{N}_{\limfunc{weak}T^{\alpha }}\left( \sigma
,\omega \right) $ and $\mathfrak{N}_{\limfunc{weak}T^{\alpha ,\ast }}\left(
\omega ,\sigma \right) $ are the best constants in the inequalities%
\begin{equation}
\left\Vert T_{\sigma }^{\alpha }f\right\Vert _{L^{2,\infty }\left( \omega
\right) }\leq \mathfrak{N}_{\limfunc{weak}T^{\alpha }}\left( \sigma ,\omega
\right) \left\Vert f\right\Vert _{L^{2}\left( \sigma \right) }\text{ and }%
\left\Vert T_{\omega }^{\alpha ,\ast }f\right\Vert _{L^{2,\infty }\left(
\sigma \right) }\leq \mathfrak{N}_{\limfunc{weak}T^{\alpha ,\ast }}\left(
\omega ,\sigma \right) \left\Vert f\right\Vert _{L^{2}\left( \omega \right)
}\ .  \label{two weight weak'}
\end{equation}

\subsection{Cube Testing}

The $\kappa $\emph{-cube testing conditions} associated with an $\alpha $%
-fractional singular integral operator $T^{\alpha }$ introduced in \cite%
{RaSaWi} are given by%
\begin{eqnarray}
\left( \mathfrak{T}_{T^{\alpha }}^{\left( \kappa \right) }\left( \sigma
,\omega \right) \right) ^{2} &\equiv &\sup_{Q\in \mathcal{P}^{n}}\max_{0\leq
\left\vert \beta \right\vert <\kappa }\frac{1}{\left\vert Q\right\vert
_{\sigma }}\int_{Q}\left\vert T_{\sigma }^{\alpha }\left( \mathbf{1}%
_{Q}m_{Q}^{\beta }\right) \right\vert ^{2}\omega <\infty ,
\label{def Kappa polynomial} \\
\left( \mathfrak{T}_{\left( T^{\alpha }\right) ^{\ast }}^{\left( \kappa
\right) }\left( \omega ,\sigma \right) \right) ^{2} &\equiv &\sup_{Q\in 
\mathcal{P}^{n}}\max_{0\leq \left\vert \beta \right\vert <\kappa }\frac{1}{%
\left\vert Q\right\vert _{\omega }}\int_{Q}\left\vert \left( T_{\sigma
}^{\alpha }\right) ^{\ast }\left( \mathbf{1}_{Q}m_{Q}^{\beta }\right)
\right\vert ^{2}\sigma <\infty ,  \notag
\end{eqnarray}%
with $m_{Q}^{\beta }\left( x\right) \equiv \left( \frac{x-c_{Q}}{\ell \left(
Q\right) }\right) ^{\beta }$ for any cube $Q$ and multiindex $\beta $, where 
$c_{Q}$ is the center of the cube $Q$, and where we interpret the right hand
sides as holding uniformly over all sufficiently smooth truncations of $%
T^{\alpha }$. Equivalently, in the presence of $A_{2}^{\alpha }$, we can
take a single suitable truncation, see Independence of Truncations in
Subsubsection \ref{independence} above. We will also use the larger \emph{%
full }$\kappa $\emph{-cube testing conditions} in which the integrals over $%
Q $ are extended to the whole space $\mathbb{R}^{n}$:%
\begin{eqnarray*}
\left( \mathfrak{FT}_{T^{\alpha }}^{\left( \kappa \right) }\left( \sigma
,\omega \right) \right) ^{2} &\equiv &\sup_{Q\in \mathcal{P}^{n}}\max_{0\leq
\left\vert \beta \right\vert <\kappa }\frac{1}{\left\vert Q\right\vert
_{\sigma }}\int_{\mathbb{R}^{n}}\left\vert T_{\sigma }^{\alpha }\left( 
\mathbf{1}_{Q}m_{Q}^{\beta }\right) \right\vert ^{2}\omega <\infty , \\
\left( \mathfrak{FT}_{\left( T^{\alpha }\right) ^{\ast }}^{\left( \kappa
\right) }\left( \omega ,\sigma \right) \right) ^{2} &\equiv &\sup_{Q\in 
\mathcal{P}^{n}}\max_{0\leq \left\vert \beta \right\vert <\kappa }\frac{1}{%
\left\vert Q\right\vert _{\omega }}\int_{\mathbb{R}^{n}}\left\vert \left(
T^{\alpha }\right) _{\omega }^{\ast }\left( \mathbf{1}_{Q}m_{Q}^{\beta
}\right) \right\vert ^{2}\sigma <\infty .
\end{eqnarray*}
We only use the case $\kappa =1$\ in the statements of the four theorems in
the next section, and so we will drop the superscript $\left( \kappa \right) 
$ when $\kappa =1$, e.g. $\mathfrak{T}_{T^{\alpha }}=\mathfrak{T}_{T^{\alpha
}}^{\left( 1\right) }$ and $\mathfrak{T}_{\left( T^{\alpha }\right) ^{\ast
}}=\mathfrak{T}_{\left( T^{\alpha }\right) ^{\ast }}^{\left( 1\right) }$.

Finally, we define the \emph{Indicator/Cube Testing constants} by%
\begin{eqnarray}
\left( \mathfrak{T}_{T^{\alpha }}^{IC}\left( \sigma ,\omega \right) \right)
^{2} &\equiv &\sup_{E\subset Q\in \mathcal{P}^{n}}\frac{1}{\left\vert
Q\right\vert _{\sigma }}\int_{Q}\left\vert T^{\alpha }\left( \mathbf{1}%
_{E}\sigma \right) \right\vert ^{2}\omega <\infty ,
\label{def kth order testing} \\
\left( \mathfrak{T}_{\left( T^{\alpha }\right) ^{\ast }}^{IC}\left( \omega
,\sigma \right) \right) ^{2} &\equiv &\sup_{E\subset Q\in \mathcal{P}^{n}}%
\frac{1}{\left\vert Q\right\vert _{\omega }}\int_{Q}\left\vert \left(
T^{\alpha }\right) ^{\ast }\left( \mathbf{1}_{E}\omega \right) \right\vert
^{2}\sigma <\infty ,  \notag
\end{eqnarray}%
which are larger than the $\kappa $-cube testing conditions.

\subsection{Bilinear Indicator/Cube Testing}

Here we introduce a variant of the weak boundedness property of David and
Journe in (\ref{WBPDJ}), but stronger because we take $\eta =0$ in (\ref%
{WBPDJ}). The Bilinear Indicator/Cube Testing property is%
\begin{equation}
\mathcal{BICT}_{T^{\alpha }}\left( \sigma ,\omega \right) \equiv \sup_{Q\in 
\mathcal{P}^{n}}\sup_{E,F\subset Q}\frac{1}{\sqrt{\left\vert Q\right\vert
_{\sigma }\left\vert Q\right\vert _{\omega }}}\left\vert \int_{F}T_{\sigma
}^{\alpha }\left( \mathbf{1}_{E}\right) \omega \right\vert <\infty ,
\label{def ind WBP}
\end{equation}%
where the second supremum is taken over all compact sets $E$ and $F$
contained in a cube $Q$. Note in particular that the bilinear indicator/cube
testing property $\mathcal{BICT}_{T^{\alpha }}\left( \sigma ,\omega \right)
<\infty $ is restricted to considering the \textbf{same} cube $Q$ for each
measure $\sigma $ and $\omega $ - in contrast to the weak boundedness
property $\mathcal{WBP}_{T^{\alpha }}^{\left( \kappa _{1},\kappa _{2}\right)
}<\infty $ in (\ref{kappa WBP}) below, that takes the supremum of the inner
product over pairs of nearby disjoint cubes $Q,Q^{\prime }$. However, in the
setting of doubling measures, the latter constant $\mathcal{WBP}_{T^{\alpha
}}^{\left( \kappa _{1},\kappa _{2}\right) }$ can be controlled by $\kappa
^{th}$-order testing and the one-tailed Muckenhoupt condition $\mathcal{A}%
_{2}^{\alpha }$ since the cube pairs are disjoint, and hence $\mathcal{WBP}%
_{T^{\alpha }}^{\left( \kappa _{1},\kappa _{2}\right) }$ is removable. On
the other hand, the former constant $\mathcal{BICT}_{T^{\alpha }}$ cannot be
controlled in the same way since the cubes coincide, and we are only able to
remove $\mathcal{BICT}_{T^{\alpha }}$ if one of the measures is an $%
A_{\infty }^{\alpha }$ or $C_{2+\varepsilon }$ measure.

\section{The four main theorems}

Here is our general $T1$ theorem for doubling measures, first with
indicator/cube testing, then with cube testing and a bilinear indicator/cube
testing property when in the addition the measures are comparable. See Lemma %
\ref{doub exp} below for the definition of the \emph{doubling exponent},
whose only role here is to determine the degree of smoothness required of
the kernel $K^{\alpha }$ below.

\begin{theorem}
\label{pivotal theorem}Suppose $0\leq \alpha <n$, and $\kappa _{1},\kappa
_{2}\in \mathbb{N}$ and $0<\delta <1$. Let $T^{\alpha }$ be an $\alpha $%
-fractional Calder\'{o}n-Zygmund singular integral operator on $\mathbb{R}%
^{n}$ with a standard $\left( \kappa _{1}+\delta ,\kappa _{2}+\delta \right) 
$-smooth $\alpha $-fractional kernel $K^{\alpha }$. Assume that $\sigma $
and $\omega $ are doubling measures on $\mathbb{R}^{n}$ with doubling
exponents $\theta _{1}$ and $\theta _{2}$ respectively satisfying 
\begin{equation*}
\kappa _{1}>\theta _{1}+\alpha -n\text{ and }\kappa _{2}>\theta _{2}+\alpha
-n.
\end{equation*}%
Set 
\begin{equation*}
T_{\sigma }^{\alpha }f=T^{\alpha }\left( f\sigma \right)
\end{equation*}%
for any smooth truncation of $T^{\alpha }$.\newline
Then%
\begin{equation}
\mathfrak{N}_{T^{\alpha }}\left( \sigma ,\omega \right) \leq C_{\alpha
,n}\left( \sqrt{\mathcal{A}_{2}^{\alpha }\left( \sigma ,\omega \right) +%
\mathcal{A}_{2}^{\alpha }\left( \omega ,\sigma \right) }+\mathfrak{T}%
_{T^{\alpha }}^{IC}\left( \sigma ,\omega \right) +\mathfrak{T}_{\left(
T^{\alpha }\right) ^{\ast }}^{IC}\left( \omega ,\sigma \right) \right) ,
\label{NIC}
\end{equation}%
and if in addition $\sigma $ and $\omega $ are \emph{comparable} doubling
measures on $\mathbb{R}^{n}$ in the sense of Coifman and Fefferman, then%
\begin{equation}
\mathfrak{N}_{T^{\alpha }}\left( \sigma ,\omega \right) \leq C_{\alpha
,n}^{\prime }\left( \sqrt{\mathcal{A}_{2}^{\alpha }\left( \sigma ,\omega
\right) +\mathcal{A}_{2}^{\alpha }\left( \omega ,\sigma \right) }+\mathfrak{T%
}_{T^{\alpha }}\left( \sigma ,\omega \right) +\mathfrak{T}_{\left( T^{\alpha
}\right) ^{\ast }}\left( \omega ,\sigma \right) +\mathcal{BICT}_{T^{\alpha
}}\left( \sigma ,\omega \right) \right) ,  \label{NBICT}
\end{equation}%
where the constant $C_{\alpha ,n}$ depends also on $C_{CZ}$ in (\ref%
{sizeandsmoothness'}) and the doubling exponents $\theta _{1}$ and $\theta
_{2}$, and where $C_{\alpha ,n}^{\prime }$ also depends on the comparability
constants in (\ref{comp param}).
\end{theorem}

Now $\mathcal{A}_{2}^{\alpha }$ is necessary for boundedness of a strongly
elliptic operator as defined in \cite{SaShUr7}, see also Liaw and Triel \cite%
[Theorem 5.1]{LiTr}. Thus we obtain the following corollary.

\begin{corollary}
\label{elliptic corollary}If in addition to the hypotheses of Theorem \ref%
{pivotal theorem}, we assume the operator $T^{\alpha }$ is strongly
elliptic, then we can reverse the inequalities in both (\ref{NIC}) and (\ref%
{NBICT}), i.e.%
\begin{equation*}
\mathfrak{N}_{T^{\alpha }}\left( \sigma ,\omega \right) \approx \sqrt{%
\mathcal{A}_{2}^{\alpha }\left( \sigma ,\omega \right) +\mathcal{A}%
_{2}^{\alpha }\left( \omega ,\sigma \right) }+\mathfrak{T}_{T^{\alpha
}}^{IC}\left( \sigma ,\omega \right) +\mathfrak{T}_{\left( T^{\alpha
}\right) ^{\ast }}^{IC}\left( \omega ,\sigma \right) ,
\end{equation*}%
and if in addition $\sigma $ and $\omega $ are \emph{comparable} doubling
measures on $\mathbb{R}^{n}$ in the sense of Coifman and Fefferman,%
\begin{equation}
\mathfrak{N}_{T^{\alpha }}\left( \sigma ,\omega \right) \approx \sqrt{%
\mathcal{A}_{2}^{\alpha }\left( \sigma ,\omega \right) +\mathcal{A}%
_{2}^{\alpha }\left( \omega ,\sigma \right) }+\mathfrak{T}_{T^{\alpha
}}\left( \sigma ,\omega \right) +\mathfrak{T}_{\left( T^{\alpha }\right)
^{\ast }}\left( \omega ,\sigma \right) +\mathcal{BICT}_{T^{\alpha }}\left(
\sigma ,\omega \right) .  \label{character}
\end{equation}%
In particular, assuming just doubling without comparability of the measures,
we have the equivalence of the strong type inequality (\ref{two weight'})
with both weak type inequalities\ (\ref{two weight weak'}), i.e. 
\begin{equation*}
\mathfrak{N}_{T^{\alpha }}\left( \sigma ,\omega \right) \approx \mathfrak{N}%
_{\limfunc{weak}T^{\alpha }}\left( \sigma ,\omega \right) +\mathfrak{N}_{%
\limfunc{weak}T^{\alpha ,\ast }}\left( \omega ,\sigma \right) .
\end{equation*}
\end{corollary}

\begin{remark}
As mentioned earlier, for operators with a partial reversal of energy, it is
already known that, for doubling measures, the norm inequalities are
characterized by one-tailed Muckenhoupt conditions and the usual $T1$
testing conditions taken over indicators of cubes, see \cite{LaWi} and \cite%
{SaShUr9}. However, energy reversal fails spectacularly for elliptic
operators in general, see \cite{SaShUr4}, and even the weaker energy
condition itself fails to be necessary for boundedness of the fractional
Riesz transforms with respect to general measures \cite{Saw}.
\end{remark}

\subsection{Elimination of the $\mathcal{BICT}$}

The following $T1$ theorem provides a Cube Testing extension of the $T1$
theorem of David and Journ\'{e} \cite{DaJo} to a pair of comparable doubling
measures when one of them, hence each of them\footnote{%
since if $\omega $ is comparable to an $A_{\infty }^{\alpha }$ measure $%
\sigma $, then $\frac{\left\vert E\right\vert _{\omega }}{\left\vert
Q\right\vert _{\omega }}\leq C\left( \frac{\left\vert E\right\vert _{\sigma }%
}{\left\vert Q\right\vert _{\sigma }}\right) ^{\varepsilon }\leq C\eta
\left( \mathbf{Cap}_{\alpha }\left( E;Q\right) \right) ^{\varepsilon }$
shows that $\omega \in A_{\infty }^{\alpha }$.}, satisfies the $A_{\infty }$
or more generally the $A_{\infty }^{\alpha }$ condition (and provided the
operator is bounded on unweighted $L^{2}\left( \mathbb{R}^{n}\right) $ when $%
\alpha =0$).

\begin{theorem}
\label{A infinity theorem}Suppose $0\leq \alpha <n$, and $\kappa _{1},\kappa
_{2}\in \mathbb{N}$ and $0<\delta <1$. Let $T^{\alpha }$ be an $\alpha $%
-fractional Calder\'{o}n-Zygmund singular integral operator on $\mathbb{R}%
^{n}$ with a standard $\left( \kappa _{1}+\delta ,\kappa _{2}+\delta \right) 
$-smooth $\alpha $-fractional kernel $K^{\alpha }$, and when $\alpha =0$,
suppose that $T^{0}$ is bounded on unweighted $L^{2}\left( \mathbb{R}%
^{n}\right) $. Assume that $\sigma $ and $\omega $ are \emph{comparable
doubling} measures on $\mathbb{R}^{n}$ that satisfy the one-tailed
Muckenhoupt conditions, and with doubling exponents $\theta _{1}$ and $%
\theta _{2}$ respectively satisfying 
\begin{equation*}
\kappa _{1}>\theta _{1}+\alpha -n\text{ and }\kappa _{2}>\theta _{2}+\alpha
-n.
\end{equation*}%
Furthermore suppose that the measures are $A_{\infty }$ weights or more
generally,%
\begin{eqnarray*}
&&\text{\textbf{either} at least one, and hence each, of }\sigma ,\omega 
\text{ satisfies the }A_{\infty }^{\alpha }\text{ condition,} \\
&&\text{\textbf{or} at least one of }\sigma ,\omega \text{ is a }C_{q}\text{
weight, for some }q>2\text{.}
\end{eqnarray*}%
Set 
\begin{equation*}
T_{\sigma }^{\alpha }f=T^{\alpha }\left( f\sigma \right)
\end{equation*}%
for any smooth truncation of $T^{\alpha }$.\newline
Then%
\begin{equation}
\mathfrak{N}_{T^{\alpha }}\left( \sigma ,\omega \right) \leq C\left( \sqrt{%
\mathcal{A}_{2}^{\alpha }\left( \sigma ,\omega \right) +\mathcal{A}%
_{2}^{\alpha }\left( \omega ,\sigma \right) }+\mathfrak{T}_{T^{\alpha
}}\left( \sigma ,\omega \right) +\mathfrak{T}_{\left( T^{\alpha }\right)
^{\ast }}\left( \omega ,\sigma \right) \right) ,  \label{more'}
\end{equation}%
where the constant $C$ depends on $C_{CZ}$ in (\ref{sizeandsmoothness'}),
and the appropriate doubling, comparability, $A_{\infty }^{\alpha }$, and $%
C_{q}$ constants. If $T^{\alpha }$ is elliptic, and also strongly elliptic
if $\frac{n}{2}\leq \alpha <n$, the inequality can be reversed.
\end{theorem}

\subsection{Optimal cancellation conditions for Calder\'{o}n-Zygmund kernels}

In the \emph{two weight} setting\ of comparable doubling measures, we give
an `optimal cancellation' analogue of the $T1$ theorem for \emph{smooth}
kernels in the context of singular integrals as defined in \cite{DaJo} or 
\cite[Section 3 of Chapter VII]{Ste2}. We now briefly recall that setup.

For $0\leq \alpha <n$, let $T^{\alpha }$ be a continuous linear map from
rapidly decreasing smooth test functions $\mathcal{S}$ to tempered
distributions in $\mathcal{S}^{\prime }$, to which is associated a kernel $%
K^{\alpha }\left( x,y\right) $, defined when $x\neq y$, that satisfies the
inequalities (more restrictive than those in (\ref{sizeandsmoothness'})
above),%
\begin{equation}
\left\vert \partial _{x}^{\beta }\partial _{y}^{\gamma }K^{\alpha }\left(
x,y\right) \right\vert \leq A_{\alpha ,\beta ,\gamma ,n}\left\vert
x-y\right\vert ^{\alpha -n-\left\vert \beta \right\vert -\left\vert \gamma
\right\vert },\ \ \ \ \ \text{for all multiindices }\beta ,\gamma ;
\label{diff ineq}
\end{equation}%
such kernels are called \emph{smooth} $\alpha $-fractional Calder\'{o}%
n-Zygmund kernels on $\mathbb{R}^{n}$. Here we say that an operator $%
T^{\alpha }$ is \emph{associated} with a kernel $K^{\alpha }$ if, whenever $%
f\in \mathcal{S}$ has compact support, the tempered distribution $T^{\alpha
}f$ can be identified, in the complement of the support, with the function
obtained by integration with respect to the kernel, i.e.%
\begin{equation}
T^{\alpha }f\left( x\right) \equiv \int K^{\alpha }\left( x,y\right) f\left(
y\right) d\sigma \left( y\right) ,\ \ \ \ \ \text{for }x\in \mathbb{R}%
^{n}\setminus \func{Supp}f.  \label{identify}
\end{equation}

The characterization in terms of (\ref{can cond}) in the next theorem is
identical to that in Stein \cite[Theorem 4 on page 306]{Ste2}, except that
the doubling measures $\sigma $ and $\omega $ appear here in place of
Legesgue measure in \cite{Ste2}, and the Euclidean distance function is
replaced by the maximum distance function $\left\Vert y\right\Vert \equiv
\max_{1\leq k\leq n}\left\vert y_{k}\right\vert $, whose associated balls
are cubes.

\begin{theorem}
\label{Stein extension}Let $0\leq \alpha <n$. Suppose that $\sigma $ and $%
\omega $ are comparable doubling measures on $\mathbb{R}^{n}$ that satisfy
the one-tailed Muckenhoupt conditions. Suppose also that the measure pair $%
\left( \sigma ,\omega \right) $ satisfies the one-tailed conditions in (\ref%
{one-sided}). Furthermore suppose that the measures are $A_{\infty }$
weights or more generally,%
\begin{eqnarray*}
&&\text{\textbf{either} at least one, and hence each, of }\sigma ,\omega 
\text{ satisfies the }A_{\infty }^{\alpha }\text{ condition,} \\
&&\text{\textbf{or} at least one of }\sigma \text{, }\omega \text{ is a }%
C_{q}\text{ weight, for some }q>2\text{.}
\end{eqnarray*}%
Suppose finally that $K^{\alpha }\left( x,y\right) $ is a smooth $\alpha $%
-fractional Calder\'{o}n-Zygmund kernel on $\mathbb{R}^{n}$. In the case $%
\alpha =0$, we also assume there is $T^{0}$ associated with the kernel $K^{0}
$ that is bounded on unweighted $L^{2}\left( \mathbb{R}^{n}\right) $.\newline
Then there exists a bounded operator $T^{\alpha }:L^{2}\left( \sigma \right)
\rightarrow L^{2}\left( \omega \right) $, that is associated with the kernel 
$K^{\alpha }$ in the sense that (\ref{identify}) holds, \emph{if and only if}
there is a positive constant $\mathfrak{A}_{K^{\alpha }}\left( \sigma
,\omega \right) $ so that%
\begin{eqnarray}
&&\int_{\left\Vert x-x_{0}\right\Vert <N}\left\vert \int_{\varepsilon
<\left\Vert x-y\right\Vert <N}K^{\alpha }\left( x,y\right) d\sigma \left(
y\right) \right\vert ^{2}d\omega \left( x\right) \leq \mathfrak{A}%
_{K^{\alpha }}\left( \sigma ,\omega \right) \ \int_{\left\Vert
x_{0}-y\right\Vert <N}d\sigma \left( y\right) ,  \label{can cond} \\
&&\ \ \ \ \ \ \ \ \ \ \ \ \ \ \ \ \ \ \ \ \ \ \ \ \ \text{for all }%
0<\varepsilon <N\text{ and }x_{0}\in \mathbb{R}^{n},  \notag
\end{eqnarray}%
along with a similar inequality with constant $\mathfrak{A}_{K^{\alpha ,\ast
}}\left( \omega ,\sigma \right) $, in which the measures $\sigma $ and $%
\omega $ are interchanged and $K^{\alpha }\left( x,y\right) $ is replaced by 
$K^{\alpha ,\ast }\left( x,y\right) =K^{\alpha }\left( y,x\right) $.
Moreover, if such $T^{\alpha }$ has minimal norm, then 
\begin{equation}
\left\Vert T^{\alpha }\right\Vert _{L^{2}\left( \sigma \right) \rightarrow
L^{2}\left( \omega \right) }\lesssim \mathfrak{A}_{K^{\alpha }}\left( \sigma
,\omega \right) +\mathfrak{A}_{K^{\alpha ,\ast }}\left( \omega ,\sigma
\right) +\sqrt{\mathcal{A}_{2}^{\alpha }\left( \sigma ,\omega \right) +%
\mathcal{A}_{2}^{\alpha }\left( \omega ,\sigma \right) },  \label{can char}
\end{equation}%
with implied constant depending on $C_{CZ}$ in (\ref{sizeandsmoothness'}),
and the appropriate doubling, $A_{\infty }$, comparability, and $C_{q}$
constants. If $T^{\alpha }$ is strongly elliptic, the inequality can be
reversed.
\end{theorem}

It should be noted that (\ref{can cond}) is \textbf{not} simply the testing
condition for a truncation of $T$ over subsets of a cube, but instead has
the historical form of bounding in some average sense, integrals of the
kernel over annuli (of cubes here rather than balls). Nevertheless, this
theorem is still a rather direct consequence of Theorem \ref{pivotal theorem}%
, with both doubling and $A_{2}^{\alpha }$ playing key roles. The reader can
check that a more complicated form of Theorem \ref{Stein extension} holds
that involves bilinear indicator/cube testing if the $A_{\infty }^{\alpha }$
conditions on $\alpha $ and $\omega $ are dropped.

\subsection{The restricted weak type theorem with an $A_{\infty }^{\protect%
\alpha }$ measure}

Here we eliminate the $BICT$ from Theorem \ref{pivotal theorem} when one of
the measures satisfies either the fractional $A_{\infty }^{\alpha }$
condition or the $C_{2+\varepsilon }$ condition\footnote{%
Recently the $BICT$ has also been eliminated \ from Theorem \ref{pivotal
theorem} when the product measure $\sigma \times \omega $ has an appropriate
reverse doubling exponent. See \cite{SaUr}.}. Note that we do \textbf{not}
assume comparability of measures here, and so conditions imposed on one
measure no longer transfer automatically to the other measure. Let $%
T_{\sigma }^{\alpha }f=T^{\alpha }\left( f\sigma \right) $. We say that an $%
\alpha $-fractional singular integral operator $T^{\alpha }$ satisfies the
restricted weak type inequality relative to the measure pair $\left( \sigma
,\omega \right) $ provided $T^{\alpha }:L^{2,1}\left( \sigma \right)
\rightarrow L^{2,\infty }\left( \omega \right) $ where $L^{2,1}\left( \sigma
\right) $ and $L^{2,\infty }\left( \omega \right) $ are the Lorentz spaces
as defined e.g. in \cite[page 188]{StWe}. As shown in \cite[see Theorem 3.13]%
{StWe}, this is equivalent to 
\begin{eqnarray}
&&\mathfrak{N}_{T^{\alpha }}^{\limfunc{restricted}\limfunc{weak}}\left(
\sigma ,\omega \right) \equiv \sup_{Q\in \mathcal{P}^{n}}\sup_{E,F\subset Q}%
\frac{1}{\sqrt{\left\vert E\right\vert _{\sigma }\left\vert F\right\vert
_{\omega }}}\left\vert \int_{F}T_{\sigma }^{\alpha }\left( \mathbf{1}%
_{E}\right) \omega \right\vert <\infty ,  \label{RWT} \\
&&\text{where the second sup is taken over all compact subsets }E,F\ \text{%
of the cube }Q,  \notag \\
&&\text{and where }0<\delta <R<\infty .  \notag
\end{eqnarray}%
Thus we see that the $\mathcal{BICT}$ condition (\ref{def ind WBP}), having $%
\left\vert Q\right\vert _{\sigma }\left\vert Q\right\vert _{\omega }$ in the
denominator, is implied by the restricted weak type condition (\ref{RWT}),
having the smaller $\left\vert E\right\vert _{\sigma }\left\vert
F\right\vert _{\omega }$ in the denominator. In the presence of the
classical Muckenhoupt condition $A_{2}^{\alpha }$, the restricted weak type
inequality in (\ref{RWT}) is essentially independent of the choice of
truncations used - see \cite{LaSaShUr3}.

\begin{remark}
In the special case $\alpha =0$, we will make the additional assumption that 
$T^{0}$ is bounded on \emph{unweighted} $L^{2}\left( \mathbb{R}^{n}\right) $%
. This is done in order to be able to use the weak type $\left( 1,1\right) $
result on Lebesgue measure for maximal truncations of such operators, that
follows from standard Calder\'{o}n-Zygmund theory as in \cite[Corollary 2 on
page 36]{Ste2}.
\end{remark}

\begin{theorem}
\label{restricted weak type}Let $0\leq \alpha <n$. Suppose that $\sigma $
and $\omega $ are locally finite positive Borel measures on $\mathbb{R}^{n}$
such that%
\begin{eqnarray*}
&&\text{\textbf{either} at least one of }\sigma \text{, }\omega \ \text{%
satisfies the }A_{\infty }^{\alpha }\text{ condition,} \\
&&\text{\textbf{or} at least one of }\sigma \text{, }\omega \text{ is a }%
C_{q}\text{ weight, for some }q>2\text{.}
\end{eqnarray*}%
Suppose also that $T^{\alpha }$ is a standard $\alpha $-fractional Calder%
\'{o}n-Zygmund singular integral in $\mathbb{R}^{n}$, and that when $\alpha
=0$ the operator $T^{0}$ is bounded on unweighted $L^{2}\left( \mathbb{R}%
^{n}\right) $. Then the two weight restricted weak type inequality for $%
T^{\alpha }$ relative to the measure pair $\left( \sigma ,\omega \right) $
holds \emph{if }the classical fractional Muckenhoupt constant $A_{2}^{\alpha
}$ in (\ref{frac Muck}) is finite. Moreover,%
\begin{equation*}
\mathcal{BICT}_{T^{\alpha }}\left( \sigma ,\omega \right) \leq \mathfrak{N}%
_{T^{\alpha }}^{\limfunc{restricted}\limfunc{weak}}\left( \sigma ,\omega
\right) \lesssim \sqrt{A_{2}^{\alpha }},
\end{equation*}%
and provided $T^{\alpha }$ is elliptic, 
\begin{equation*}
\mathcal{BICT}_{T^{\alpha }}\left( \sigma ,\omega \right) \approx \mathfrak{N%
}_{T^{\alpha }}^{\limfunc{restricted}\limfunc{weak}}\left( \sigma ,\omega
\right) \approx \sqrt{A_{2}^{\alpha }\left( \sigma ,\omega \right) },
\end{equation*}%
where the implied constants depend on the Calder\'{o}n-Zygmund norm $C_{CZ}$
in (\ref{sizeandsmoothness'}) and the $A_{\infty }^{\alpha }$ or $%
C_{2+\varepsilon }$ parameters of one of the measures.
\end{theorem}

\begin{remark}
The proof of the theorem shows a bit more, namely that the restricted weak
type norms of $T^{\alpha }$ and its maximal trunction operator $T_{\flat
}^{\alpha }$ (see below) are equivalent under the hypotheses of the theorem,
and including the fractional integral $I^{\alpha }$ (see below) when $%
0<\alpha <n$.
\end{remark}

\section{Preliminaries}

Here we introduce the $\kappa ^{th}$-order pivotal conditions, recall the
weighted Alpert wavelets from \cite{RaSaWi}, and establish some connections
with doubling weights.

\subsection{Necessity of the $\protect\kappa ^{th}$ order Pivotal Condition
for doubling weights}

The smaller fractional Poisson integrals $\mathrm{P}_{\kappa }^{\alpha
}\left( Q,\mu \right) $ used here, in \cite{RaSaWi} and elsewhere, are given
by 
\begin{equation}
\mathrm{P}_{\kappa }^{\alpha }\left( Q,\mu \right) =\int_{\mathbb{R}^{n}}%
\frac{\ell \left( Q\right) ^{\kappa }}{\left( \ell \left( Q\right)
+\left\vert y-c_{Q}\right\vert \right) ^{n+\kappa -\alpha }}d\mu \left(
y\right) ,\ \ \ \ \ \kappa \geq 1,  \label{def kappa Poisson}
\end{equation}%
and the $\kappa ^{th}$-order fractional Pi\textbf{V}otal Conditions $%
\mathcal{V}_{2}^{\alpha ,\kappa },\mathcal{V}_{2}^{\alpha ,\kappa ,\ast
}<\infty $, $\kappa \geq 1$, are given by%
\begin{eqnarray}
\left( \mathcal{V}_{2}^{\alpha ,\kappa }\right) ^{2} &=&\sup_{Q\supset \dot{%
\cup}Q_{r}}\frac{1}{\left\vert Q\right\vert _{\sigma }}\sum_{r=1}^{\infty }%
\mathrm{P}_{\kappa }^{\alpha }\left( Q_{r},\mathbf{1}_{Q}\sigma \right)
^{2}\left\vert Q_{r}\right\vert _{\omega }\ ,  \label{both pivotal k} \\
\left( \mathcal{V}_{2}^{\alpha ,\kappa ,\ast }\right) ^{2} &=&\sup_{Q\supset 
\dot{\cup}Q_{r}}\frac{1}{\left\vert Q\right\vert _{\omega }}%
\sum_{r=1}^{\infty }\mathrm{P}_{\kappa }^{\alpha }\left( Q_{r},\mathbf{1}%
_{Q}\omega \right) ^{2}\left\vert Q_{r}\right\vert _{\sigma }\ ,  \notag
\end{eqnarray}%
where the supremum is taken over all subdecompositions of a cube $Q\in 
\mathcal{P}^{n}$ into pairwise disjoint subcubes $Q_{r}$.

We begin with the elementary derivation of $\kappa ^{th}$ order pivotal
conditions from doubling assumptions. From Lemma \ref{doubling} below, a
doubling measure $\omega $ with doubling parameters $0<\beta ,\gamma <1$ as
in (\ref{def rev doub}), has a `doubling exponent' $\theta >0$ and a
positive constant $c$ depending on $\beta ,\gamma $ that satisfy the
condition,%
\begin{equation*}
\left\vert 2^{-j}Q\right\vert _{\omega }\geq c2^{-j\theta }\left\vert
Q\right\vert _{\omega }\ ,\ \ \ \ \ \text{for all }j\in \mathbb{N}.
\end{equation*}%
We can then exploit the doubling exponents $\theta =\theta \left( \beta
,\gamma \right) $ of the doubling measures $\sigma $ and $\omega $ in order
to derive certain $\kappa ^{th}$ order pivotal conditions $\mathcal{V}%
_{2}^{\alpha ,\kappa },\mathcal{V}_{2}^{\alpha ,\kappa ,\ast }<\infty $.
Indeed, if $\omega $ has doubling exponent $\theta $ and $\kappa >\theta
+\alpha -n$, we have%
\begin{eqnarray}
&&\int_{\mathbb{R}^{n}\setminus I}\frac{\ell \left( I\right) ^{\kappa }}{%
\left( \ell \left( I\right) +\left\vert x-c_{I}\right\vert \right)
^{n+\kappa -\alpha }}d\omega \left( x\right) =\sum_{j=1}^{\infty }\ell
\left( I\right) ^{\alpha -n}\int_{2^{j}I\setminus 2^{j-1}I}\frac{1}{\left( 1+%
\frac{\left\vert x-c_{I}\right\vert }{\ell \left( I\right) }\right)
^{n+\kappa -\alpha }}d\omega \left( x\right)  \label{kappa large} \\
&\lesssim &\left\vert I\right\vert ^{\frac{\alpha }{n}-1}\sum_{j=1}^{\infty
}2^{-j\left( n+\kappa -\alpha \right) }\left\vert 2^{j}I\right\vert _{\omega
}\lesssim \left\vert I\right\vert ^{\frac{\alpha }{n}-1}\sum_{j=1}^{\infty
}2^{-j\left( n+\kappa -\alpha \right) }\frac{1}{c2^{-j\theta }}\left\vert
I\right\vert _{\omega }\leq C_{n,\kappa ,\alpha ,\left( \beta ,\gamma
\right) }\left\vert I\right\vert ^{\frac{\alpha }{n}-1}\left\vert
I\right\vert _{\omega }\ ,  \notag
\end{eqnarray}%
provided $n+\kappa -\alpha -\theta >0$, i.e. $\kappa >\theta +\alpha -n$. It
follows that if $I\supset \overset{\cdot }{\dbigcup }_{r=1}^{\infty }I_{r}$
is a subdecomposition of $I$ into pairwise disjoint cubes $I_{r}$, and $%
\kappa >\theta +\alpha -n$, then 
\begin{equation*}
\sum_{r=1}^{\infty }\mathrm{P}_{\kappa }^{\alpha }\left( I_{r},\omega
\right) ^{2}\left\vert I_{r}\right\vert _{\sigma }\lesssim
\sum_{r=1}^{\infty }\left( \left\vert I_{r}\right\vert ^{\frac{\alpha }{n}%
-1}\left\vert I_{r}\right\vert _{\omega }\right) ^{2}\left\vert
I_{r}\right\vert _{\sigma }=\sum_{r=1}^{\infty }\frac{\left\vert
I_{r}\right\vert _{\omega }\left\vert I_{r}\right\vert _{\sigma }}{%
\left\vert I\right\vert ^{2\left( 1-\frac{\alpha }{n}\right) }}\left\vert
I_{r}\right\vert _{\omega }\lesssim A_{2}^{\alpha }\sum_{r=1}^{\infty
}\left\vert I_{r}\right\vert _{\omega }=A_{2}^{\alpha }\left\vert
I\right\vert _{\omega }\ ,
\end{equation*}%
which gives 
\begin{equation}
\mathcal{V}_{2}^{\alpha ,\kappa ,\ast }\leq C_{\kappa ,\left( \beta ,\gamma
\right) }A_{2}^{\alpha }\ ,\ \ \ \ \ \kappa >\theta +\alpha -n,
\label{piv control}
\end{equation}%
where the constant $C_{\kappa ,\left( \beta ,\gamma \right) }$ depends on
the doubling parameters $\left( \beta ,\gamma \right) $ and on $\kappa $.
Thus the dual $\kappa ^{th}$ order pivotal condition is controlled by $%
A_{2}^{\alpha }$ provided $\kappa +n-\alpha $ exceeds the doubling exponent
of the measure $\omega $. A similar result holds for $\mathcal{V}%
_{2}^{\alpha ,\kappa }$ if $\kappa +n-\alpha $ exceeds the doubling exponent
of $\sigma $.

\begin{remark}
The integers $\kappa $ may have to be taken quite large depending on the
doubling exponent of the doubling measures. In fact, the proof of Lemma \ref%
{doubling} shows that we may take $\theta =\frac{\log _{2}\frac{1}{\gamma }}{%
\log _{2}\frac{1}{\beta }}$, and so we need $\kappa >\frac{\log _{2}\frac{1}{%
\gamma }}{\log _{2}\frac{1}{\beta }}+\alpha -n$, where $\beta $ and $\gamma $
are the doubling parameters for the measure. Since $C_{\limfunc{doub}}=\frac{%
1}{\gamma }$ when $\beta =\frac{1}{2}$, we can equivalently write $\kappa
>\log _{2}C_{\limfunc{doub}}+\alpha -n$, where $C_{\limfunc{doub}}$ can be
thought of as the `upper dimension' of the doubling measure. Indeed, in the
case $\alpha =0$ and $d\sigma \left( x\right) =d\omega \left( x\right) =dx$
on $\mathbb{R}^{n}$, we have $\left\vert \beta Q\right\vert =\beta
^{n}\left\vert Q\right\vert $ implies $\theta =\frac{n\log _{2}\frac{1}{%
\beta }}{\log _{2}\frac{1}{\beta }}=n$.
\end{remark}

\subsection{Weighted Alpert bases for $L^{2}\left( \protect\mu \right) $ and 
$L^{\infty }$ control of projections\label{Subsection Haar}}

The proof of Theorem \ref{pivotal theorem} will require weighted wavelets
with higher vanishing moments in order to accommodate the Poisson integrals
with smaller tails. We now briefly recall the construction of weighted
Alpert wavelets in \cite{RaSaWi}. Let $\mu $ be a locally finite positive
Borel measure on $\mathbb{R}^{n}$, and fix $\kappa \in \mathbb{N}$. For $%
Q\in \mathcal{P}^{n}$, the collection of cubes with sides parallel to the
coordinate axes, denote by $L_{Q;k}^{2}\left( \mu \right) $ the finite
dimensional subspace of $L^{2}\left( \mu \right) $ that consists of linear
combinations of the indicators of\ the children $\mathfrak{C}\left( Q\right) 
$ of $Q$ multiplied by polynomials of degree less than $\kappa $, and such
that the linear combinations have vanishing $\mu $-moments on the cube $Q$
up to order $\kappa -1$:%
\begin{equation*}
L_{Q;\kappa }^{2}\left( \mu \right) \equiv \left\{ f=\dsum\limits_{Q^{\prime
}\in \mathfrak{C}\left( Q\right) }\mathbf{1}_{Q^{\prime }}p_{Q^{\prime
};\kappa }\left( x\right) :\int_{Q}f\left( x\right) x^{\beta }d\mu \left(
x\right) =0,\ \ \ \text{for }0\leq \left\vert \beta \right\vert <\kappa
\right\} ,
\end{equation*}%
where $p_{Q^{\prime };\kappa }\left( x\right) =\sum_{\beta \in \mathbb{Z}%
_{+}^{n}:\left\vert \beta \right\vert \leq \kappa -1\ }a_{Q^{\prime };\alpha
}x^{\beta }$ is a polynomial in $\mathbb{R}^{n}$ of degree $\left\vert \beta
\right\vert =\beta _{1}+...+\beta _{n}$ less than $\kappa $. Here $x^{\beta
}=x_{1}^{\beta _{1}}x_{2}^{\beta _{2}}...x_{n}^{\beta _{n}}$. Let $%
d_{Q;\kappa }\equiv \dim L_{Q;\kappa }^{2}\left( \mu \right) $ be the
dimension of the finite dimensional linear space $L_{Q;\kappa }^{2}\left(
\mu \right) $. Now define%
\begin{eqnarray*}
&&\mathcal{F}_{\infty }^{\kappa }\left( \mu \right) \equiv \left\{ \beta \in 
\mathbb{Z}_{+}^{n}:\left\vert \beta \right\vert \leq \kappa -1:x^{\beta }\in
L^{2}\left( \mu \right) \right\} \ , \\
&&\ \ \ \ \ \ \ \ \ \ \ \ \ \ \ \text{and }\mathcal{P}_{\mathbb{R}%
^{n};\kappa }^{n}\left( \mu \right) \equiv \limfunc{Span}\left\{ x^{\beta
}\right\} _{\beta \in \mathcal{F}_{\infty }^{\kappa }}\ .
\end{eqnarray*}

Let $\bigtriangleup _{Q;\kappa }^{\mu }$ denote orthogonal projection onto
the finite dimensional subspace $L_{Q;\kappa }^{2}\left( \mu \right) $, let $%
\mathbb{E}_{Q;\kappa }^{\mu }$ denote orthogonal projection onto the finite
dimensional subspace%
\begin{equation*}
\mathcal{P}_{Q;\kappa }^{n}\left( \sigma \right) \equiv \mathnormal{\limfunc{%
Span}}\{\mathbf{1}_{Q}x^{\beta }:0\leq \left\vert \beta \right\vert <\kappa
\},
\end{equation*}%
and let $\bigtriangleup _{\mathbb{R}^{n};\kappa }^{\mu }$ denote orthogonal
projection onto $\mathcal{P}_{\mathbb{R}^{n};\kappa }^{n}\left( \mu \right) $%
.

The following theorem was proved in \cite{RaSaWi}, which establishes the
existence of Alpert wavelets, for $L^{2}\left( \mu \right) $ in all
dimensions, having the three important properties of orthogonality,
telescoping and moment vanishing.

\begin{theorem}[Weighted Alpert Bases]
\label{main1}Let $\mu $ be a locally finite positive Borel measure on $%
\mathbb{R}^{n}$, fix $\kappa \in \mathbb{N}$, and fix a dyadic grid $%
\mathcal{D}$ in $\mathbb{R}^{n}$.

\begin{enumerate}
\item Then $\left\{ \bigtriangleup _{\mathbb{R}^{n};\kappa }^{\mu }\right\}
\cup \left\{ \bigtriangleup _{Q;\kappa }^{\mu }\right\} _{Q\in \mathcal{D}}$
is a complete set of orthogonal projections in $L_{\mathbb{R}^{n}}^{2}\left(
\mu \right) $ and%
\begin{eqnarray*}
f &=&\bigtriangleup _{\mathbb{R}^{n};\kappa }^{\mu }f+\sum_{Q\in \mathcal{D}%
}\bigtriangleup _{Q;\kappa }^{\mu }f,\ \ \ \ \ f\in L_{\mathbb{R}%
^{n}}^{2}\left( \mu \right) , \\
&&\left\langle \bigtriangleup _{\mathbb{R}^{n};\kappa }^{\mu
}f,\bigtriangleup _{Q;\kappa }^{\mu }f\right\rangle =\left\langle
\bigtriangleup _{P;\kappa }^{\mu }f,\bigtriangleup _{Q;\kappa }^{\mu
}f\right\rangle =0\text{ for }P\neq Q,
\end{eqnarray*}%
where convergence in the first line holds both in $L_{\mathbb{R}%
^{n}}^{2}\left( \mu \right) $ norm and pointwise $\mu $-almost everywhere.

\item Moreover we have the telescoping identities%
\begin{equation}
\mathbf{1}_{Q}\sum_{Q\subsetneqq I\subset P}\bigtriangleup _{I;\kappa }^{\mu
}=\mathbb{E}_{Q;\kappa }^{\mu }-\mathbb{E}_{P;\kappa }^{\mu }\ \text{ \ for }%
P,Q\in \mathcal{D}\text{ with }Q\subsetneqq P,  \label{telescoping}
\end{equation}

\item and the moment vanishing conditions%
\begin{equation}
\int_{\mathbb{R}^{n}}\bigtriangleup _{Q;\kappa }^{\mu }f\left( x\right) \
x^{\beta }d\mu \left( x\right) =0,\ \ \ \text{for }Q\in \mathcal{D},\text{ }%
\beta \in \mathbb{Z}_{+}^{n},\ 0\leq \left\vert \beta \right\vert <\kappa \ .
\label{mom con}
\end{equation}
\end{enumerate}
\end{theorem}

We can fix\ an orthonormal basis $\left\{ h_{Q;\kappa }^{\mu ,a}\right\}
_{a\in \Gamma _{Q,n,\kappa }}$ of $L_{Q;\kappa }^{2}\left( \mu \right) $
where $\Gamma _{Q,n,\kappa }$ is a convenient finite index set. Then 
\begin{equation*}
\left\{ h_{Q;\kappa }^{\mu ,a}\right\} _{a\in \Gamma _{Q,n,\kappa }\text{
and }Q\in \mathcal{D}}
\end{equation*}%
is an orthonormal basis for $L^{2}\left( \mu \right) $, with the
understanding that we add an orthonormal basis of $\mathcal{P}_{\mathbb{R}%
^{n}}^{\kappa }\left( \mu \right) $ if it is nontrivial. In particular we
have from the theorem above that (at least when $\mathcal{P}_{\mathbb{R}%
^{n}}^{\kappa }\left( \mu \right) =\left\{ 0\right\} $), 
\begin{eqnarray*}
\left\Vert f\right\Vert _{L^{2}\left( \mu \right) }^{2} &=&\sum_{Q\in 
\mathcal{D}}\left\Vert \bigtriangleup _{Q}^{\mu }f\right\Vert _{L^{2}\left(
\mu \right) }^{2}=\sum_{Q\in \mathcal{D}}\left\vert \widehat{f}\left(
Q\right) \right\vert ^{2}, \\
\left\vert \widehat{f}\left( Q\right) \right\vert ^{2} &\equiv &\sum_{a\in
\Gamma _{Q,n,\kappa }\text{ }}\left\vert \left\langle f,h_{Q}^{\mu
,a}\right\rangle _{\mu }\right\vert ^{2}.
\end{eqnarray*}%
In the case $\kappa =1$, this construction reduces to the familiar Haar
wavelets, where with $\mathbb{E}_{I}^{\mu }=\mathbb{E}_{I}^{\mu ,1}$ we have
the following useful bound,%
\begin{equation*}
\left\Vert \mathbb{E}_{I}^{\mu }f\right\Vert _{L_{I}^{\infty }\left( \mu
\right) }=\left\Vert \left\langle f,\frac{1}{\sqrt{\left\vert I\right\vert
_{\mu }}}\mathbf{1}_{I}\right\rangle \frac{1}{\sqrt{\left\vert I\right\vert
_{\mu }}}\mathbf{1}_{I}\right\Vert _{L_{I}^{\infty }\left( \mu \right)
}=\left\vert E_{I}^{\mu }f\right\vert \leq E_{I}^{\mu }\left\vert
f\right\vert .
\end{equation*}

We will consider below an analogous bound for the Alpert projections $%
\mathbb{E}_{I;\kappa }^{\mu }$ when $\kappa >1$, that is of the form 
\begin{equation}
\left\Vert \mathbb{E}_{I}^{\mu ,\kappa }f\right\Vert _{L_{I}^{\infty }\left(
\mu \right) }\lesssim E_{I}^{\mu }\left\vert f\right\vert ,\ \ \ \ \ \text{%
for all }f\in L_{\limfunc{loc}}^{1}\left( \mu \right) .  \label{analogue}
\end{equation}%
This will require certain energy nondegeneracy conditions to be imposed on $%
\mu $, which turn out to be essentially equivalent to doubling conditions
(thus limiting our application of Alpert wavelets to doubling measures in
this paper).

\subsubsection{Doubling and energy nondegeneracy conditions}

We will need the following relation between energy nondegeneracy and
doubling conditions. We say that a polynomial $P\left( y\right) =\sum_{0\leq
\left\vert \beta \right\vert <\kappa }c_{\beta }y^{\beta }$ of degree less
than $\kappa $ is \emph{normalized} if 
\begin{equation*}
\sup_{y\in Q_{0}}\left\vert P\left( y\right) \right\vert =1,\ \ \ \ \ \text{%
where }Q_{0}\equiv \dprod\limits_{i=1}^{n}\left[ -\frac{1}{2},\frac{1}{2}%
\right) .
\end{equation*}

\begin{remark}
\label{normalized}Since all norms on a finite dimensional vector space are
equivalent, we have 
\begin{equation}
\left\Vert P\right\Vert _{L^{\infty }\left( Q_{0}\right) }\approx \left\vert
P\left( 0\right) \right\vert +\left\Vert \nabla P\right\Vert _{L^{\infty
}\left( Q_{0}\right) },\ \ \ \ \ \deg P<\kappa ,  \label{fin equiv}
\end{equation}%
with implicit constants depending only on $n$ and $\kappa $, and so a
compactness argument shows there is $\varepsilon _{\kappa }>0$ such that for
every \emph{normalized} polynomial $P$ of degree less than $\kappa $, there
is a ball $B\left( y,\varepsilon _{\kappa }\right) \subset Q_{0}\ $on which $%
P$ is nonvanishing.
\end{remark}

\begin{definition}
\label{def Q norm}Denote by $c_{Q}$ the center of the cube $Q$, and by $\ell
\left( Q\right) $ its side length, and for any polynomial $P$ set 
\begin{equation*}
P^{Q}\left( y\right) \equiv P\left( c_{Q}+\ell \left( Q\right) y\right) .
\end{equation*}%
We say that $P\left( x\right) $ is $Q$\emph{-normalized} if $P^{Q}$ is
normalized. Denote by $\left( \mathcal{P}_{\kappa }^{Q}\right) _{\limfunc{%
norm}}$ the set of $Q$-normalized polynomials of degree less than $\kappa $.
\end{definition}

Thus a $Q$-normalized polynomial has its supremum norm on $Q$ equal to $1$.
Recall from (\ref{def rev doub}) that a locally finite positive Borel
measure $\mu $ on $\mathbb{R}^{n}$ is \emph{doubling} if there exist
constants $0<\beta ,\gamma <1$ such that%
\begin{equation}
\left\vert \beta Q\right\vert _{\mu }\geq \gamma \left\vert Q\right\vert
_{\mu },\ \ \ \ \ \text{for all cubes }Q\text{ in }\mathbb{R}^{n}.
\label{doub}
\end{equation}%
Note that $\sup_{y\in Q_{0}}\left\vert P\left( y\right) \right\vert
=\left\Vert \mathbf{1}_{Q_{0}}P\right\Vert _{L^{\infty }\left( \mu \right) }$
for any cube $Q_{0}$, polynomial $P$, and doubling measure $\mu $. The
following lemma on doubling measures is well known.

\begin{lemma}
\label{doub exp}Let $\mu $ be a locally finite positive Borel measure on $%
\mathbb{R}^{n}$. Then $\mu $ is doubling if and only if there exists a
positive constant $\theta $, called the doubling exponent, such that%
\begin{equation*}
\left\vert 2^{-k}Q\right\vert _{\mu }\geq 2^{-\theta k}\left\vert
Q\right\vert _{\mu },\ \ \ \ \ \text{for all cubes }Q\text{ in }\mathbb{R}%
^{n}\text{ and }k\in \mathbb{N}.
\end{equation*}
\end{lemma}

\begin{proof}
Suppose there are $0<\beta ,\gamma <1$ such that $\left\vert \beta
Q\right\vert _{\mu }\geq \gamma \left\vert Q\right\vert _{\mu }\ $for all
cubes $Q$ in $\mathbb{R}^{n}$. Iteration of this inequality leads to $%
\left\vert \beta ^{j}Q\right\vert _{\mu }\geq \gamma ^{j}\left\vert
Q\right\vert _{\mu }$. Now choose $t>0$ so that $\beta \leq 2^{-t}<2\beta $,
which then gives 
\begin{eqnarray*}
\left\vert 2^{-k}Q\right\vert _{\sigma } &=&\left\vert \left( 2^{-t}\right)
^{\frac{k}{t}}Q\right\vert _{\sigma }\geq \left\vert \beta ^{\frac{k}{t}%
}Q\right\vert _{\sigma }\geq \left\vert \beta ^{\left[ \frac{k}{t}\right]
}Q\right\vert _{\sigma } \\
&\geq &\gamma ^{\left[ \frac{k}{t}\right] }\left\vert Q\right\vert _{\sigma
}=2^{-\left[ \frac{k}{t}\right] \log _{2}\frac{1}{\gamma }}\left\vert
Q\right\vert _{\sigma }\geq 2^{-\frac{k}{t}\log _{2}\frac{1}{\gamma }%
}\left\vert Q\right\vert _{\sigma }=2^{-\theta k}\left\vert Q\right\vert
_{\sigma }
\end{eqnarray*}%
with $\theta =\frac{\log _{2}\frac{1}{\gamma }}{t}\geq \frac{\log _{2}\frac{1%
}{\gamma }}{\log _{2}\frac{1}{\beta }}>0$. The converse statement is trivial
with $\beta =\frac{1}{2}$ and $\gamma =2^{-\theta }=\frac{1}{C_{\limfunc{doub%
}}}$.
\end{proof}

The doubling exponent $\theta =\log _{2}C_{\limfunc{doub}}$ can be thought
of as the upper dimension of $\mu $. Here now is the connection between
doubling measures and energy degeneracy. We thank Ignacio Uriarte-Tuero for
pointing to a gap in the proof of part (2) in the first version of this
paper.

\begin{lemma}
\label{doubling}Let $\mu $ be a locally finite positive Borel measure on $%
\mathbb{R}^{n}$.

\begin{enumerate}
\item If $\mu $ is doubling on $\mathbb{R}^{n}$, then for every $\kappa \in 
\mathbb{N}$ there exists a positive constant $C_{\kappa }$ such that%
\begin{eqnarray}
\left\vert Q\right\vert _{\mu } &\leq &C_{\kappa }\int_{Q}\left\vert P\left(
x\right) \right\vert ^{2}d\mu \left( x\right) ,\ \ \ \ \ \text{for all cubes 
}Q\text{ in }\mathbb{R}^{n}\text{,}  \label{energy nondeg} \\
&&\text{and for all }Q\text{-normalized polynomials }P\text{ of degree less
than }\kappa .  \notag
\end{eqnarray}

\item Conversely, if (\ref{energy nondeg}) holds for some positive integer $%
\kappa >2n$, then $\mu $ is doubling.
\end{enumerate}
\end{lemma}

\begin{proof}
Fix a cube $Q$ and a positive integer $\kappa \in \mathbb{N}$. By Remark \ref%
{normalized}, there is a positive integer $L=L\left( \kappa \right) \in 
\mathbb{N}$ with the property that for every $Q$-normalized polynomial $P$
of degree less than $\kappa $ on $\mathbb{R}^{n}$, at least one of the
dyadic children $K\in \mathfrak{C}^{\left( L\right) }\left( Q\right) $ at
level $L$ beneath $Q$ satisfies $3K\subset Q\setminus Z_{P}$, where $Z_{P}$
is the zero set of\ the polynomial $P$. Furthermore, if $P$ is a $Q$%
-normalized polynomial of degree less than $\kappa $, then $P^{Q}\left(
y\right) \equiv P\left( c_{Q}+\ell \left( Q\right) y\right) $ is normalized
and $P\left( x\right) =P^{Q}\left( \frac{x-c_{Q}}{\ell \left( Q\right) }%
\right) $, and so we have from (\ref{fin equiv}) the inequality%
\begin{equation*}
\left\vert P\left( x\right) \right\vert =\left\vert P^{Q}\left( \frac{x-c_{Q}%
}{\ell \left( Q\right) }\right) \right\vert \geq c\left( \limfunc{dist}%
\left( \frac{x-c_{Q}}{\ell \left( Q\right) },Z_{P^{Q}}\right) \right)
^{\kappa }=c\left( \frac{\limfunc{dist}\left( x,Z_{P}\right) }{\ell \left(
Q\right) }\right) ^{\kappa },\ \ \ \ \ x\in Q.
\end{equation*}%
Moreover, $Q\subset 2^{L+1}K$, and hence we have the lower bound%
\begin{eqnarray*}
\int_{Q}\left\vert P\left( x\right) \right\vert ^{2}d\sigma \left( x\right)
&\geq &c^{2}\int_{K}\left( \frac{\limfunc{dist}\left( x,Z_{P}\right) }{\ell
\left( Q\right) }\right) ^{2\kappa }d\sigma \left( x\right) \geq
c^{2}\int_{K}\left( \frac{\ell \left( K\right) }{\ell \left( Q\right) }%
\right) ^{2\kappa }d\sigma \left( x\right) \\
&=&c^{2}2^{-2\kappa L}\left\vert K\right\vert _{\sigma }\geq
c^{2}2^{-2\kappa L}2^{-\left( L+1\right) \theta }\left\vert
2^{L+1}K\right\vert _{\sigma }\geq c_{\kappa }\left\vert Q\right\vert
_{\sigma }\ ,
\end{eqnarray*}%
where $c_{\kappa }=c^{2}2^{-2\kappa L}2^{-\left( L+1\right) \theta }$. Thus (%
\ref{energy nondeg}) holds with $C_{\kappa }=\frac{1}{c_{\kappa }}$.

Conversely, assume that (\ref{energy nondeg}) holds for some $\kappa >2n$.
Momentarily fix a cube $Q$. Then the polynomial 
\begin{equation*}
P\left( x\right) \equiv \dprod\limits_{i=1}^{n}\left[ 1-\left( \frac{%
x_{i}-\left( c_{Q}\right) _{i}}{\ell \left( Q\right) }\right) ^{2}\right]
\end{equation*}%
is $Q$-normalized of degree less than $\kappa $, vanishes on the boundary of 
$Q$, and is $1$ at the center $c_{Q}$ of $Q$. Now using that $2n<\kappa $ in
(\ref{energy nondeg}), there is $\beta <1$, sufficiently close to $1$, and 
\emph{independent} of the cube $Q$, so that%
\begin{eqnarray*}
\left\vert Q\right\vert _{\mu } &\leq &C_{\kappa }\int_{Q}\left\vert
P\right\vert ^{2}d\mu =C_{\kappa }\left\{ \int_{Q\setminus \beta
Q}\left\vert P\right\vert ^{2}d\mu +\int_{\beta Q}\left\vert P\right\vert
^{2}d\mu \right\} \\
&\leq &\frac{1}{2}\left\vert Q\setminus \beta Q\right\vert _{\mu }+C_{\kappa
}\left\vert \beta Q\right\vert _{\mu }\leq \frac{1}{2}\left\vert
Q\right\vert _{\mu }+C_{\kappa }\left\vert \beta Q\right\vert _{\mu }\ .
\end{eqnarray*}%
Thus we have%
\begin{equation*}
\left\vert Q\right\vert _{\mu }\leq 2C_{\kappa }\left\vert \beta
Q\right\vert _{\mu }\ ,
\end{equation*}%
which is (\ref{doub}) with $\gamma =\frac{1}{2C_{\kappa }}$.
\end{proof}

\subsubsection{Control of Alpert projections}

For $n,\kappa \in \mathbb{N}$, let $\mathcal{P}_{\kappa }^{n}$ denote the
finite dimensional vector space of real polynomials $P\left( x\right) $ on $%
\mathbb{R}^{n}$ with degree less than $\kappa $, i.e. $P\left( x\right)
=\sum_{0\leq \left\vert \beta \right\vert <\kappa }c_{\beta }x^{\beta }$
where $\beta =\left( \beta _{i}\right) _{i=}^{n}\in \mathbb{Z}_{+}^{n}$ and $%
\left\vert \beta \right\vert =\sum_{i=1}^{n}\beta _{i}$. Then denote by $%
\mathcal{P}_{I;\kappa }^{n}$ the space of restrictions of polynomials in $%
\mathcal{P}_{\kappa }^{n}$ to the interval $I$, also denoted $\mathcal{P}%
_{I;\kappa }^{n}\left( \mu \right) $ when we wish to emphasize the
underlying measure. Now let $\left\{ b_{I;\kappa }^{j}\right\} _{j=1}^{N}$
be an orthonormal basis for $\mathcal{P}_{I;\kappa }^{n}$ with the inner
product of $L^{2}\left( \mu \right) $. If we assume that $\mu $ is doubling,
and define the polynomial $P_{j}$ by $P_{j}\left( x\right) =\frac{1}{%
\left\Vert b_{I;\kappa }^{j}\right\Vert _{L_{I}^{\infty }\left( \mu \right) }%
}b_{I;\kappa }^{j}\left( x\right) $, then $P_{j}\in \left( \mathcal{P}%
_{\kappa }^{I}\right) _{\limfunc{norm}}$ is $I$-normalized, and so part (1)
of Lemma \ref{doubling} shows that 
\begin{equation*}
\frac{1}{\left\Vert b_{I;\kappa }^{j}\right\Vert _{L_{I}^{\infty }\left( \mu
\right) }^{2}}=\int_{I}\left\vert \frac{1}{\left\Vert b_{I;\kappa
}^{j}\right\Vert _{L_{I}^{\infty }\left( \mu \right) }}b_{I;\kappa
}^{j}\left( x\right) \right\vert ^{2}d\mu \left( x\right)
=\int_{I}\left\vert P_{j}\left( x\right) \right\vert ^{2}d\mu \left(
x\right) \approx \left\vert I\right\vert _{\mu }.
\end{equation*}%
This then gives (\ref{analogue}):%
\begin{eqnarray*}
\left\Vert \mathbb{E}_{I}^{\mu ,\kappa }f\right\Vert _{L_{I}^{\infty }\left(
\mu \right) } &=&\left\Vert \sum_{j=1}^{N}\left\langle f,b_{I;\kappa
}^{j}\right\rangle b_{I;\kappa }^{j}\right\Vert _{L_{I}^{\infty }\left( \mu
\right) }\leq \sum_{j=1}^{N}\left\vert \left\langle f,P_{j}\right\rangle
\right\vert \left\Vert b_{I;\kappa }^{j}\right\Vert _{L_{I}^{\infty }\left(
\mu \right) }\left\Vert b_{I;\kappa }^{j}\right\Vert _{L_{I}^{\infty }\left(
\mu \right) } \\
&\leq &\sum_{j=1}^{N}\left( \int_{I}\left\vert f\right\vert d\mu \right)
\left\Vert b_{I;\kappa }^{j}\right\Vert _{L_{I}^{\infty }\left( \mu \right)
}^{2}\lesssim \sum_{j=1}^{N}\frac{1}{\left\vert I\right\vert _{\mu }}%
\int_{I}\left\vert f\right\vert d\mu =N\ E_{I}^{\mu }\left\vert f\right\vert
.
\end{eqnarray*}

We also record the following additional consequence of (\ref{energy nondeg}):%
\begin{equation}
\left\Vert \mathbb{E}_{I}^{\mu ,\kappa }f\right\Vert _{L_{I}^{\infty }\left(
\mu \right) }^{2}\left\vert I\right\vert _{\mu }\lesssim \left\Vert \mathbb{E%
}_{I}^{\mu ,\kappa }f\right\Vert _{L_{I}^{2}\left( \mu \right) }^{2}\ ,
\label{add con}
\end{equation}%
which follows from%
\begin{eqnarray*}
\left\Vert \mathbb{E}_{I}^{\mu ,\kappa }f\right\Vert _{L_{I}^{\infty }\left(
\mu \right) }^{2}\left\vert I\right\vert _{\mu } &\lesssim &\left(
\sum_{j=1}^{N}\left\vert \left\langle f,b_{I;\kappa }^{j}\right\rangle
\right\vert \right) ^{2}\left( \max_{1\leq j\leq N}\left\Vert b_{I;\kappa
}^{j}\right\Vert _{L_{I}^{\infty }\left( \mu \right) }\right) ^{2}\left\vert
I\right\vert _{\mu } \\
&\lesssim &N\ \sum_{j=1}^{N}\left\vert \left\langle f,b_{I;\kappa
}^{j}\right\rangle \right\vert ^{2}=N\ \left\Vert \mathbb{E}_{I}^{\mu
,\kappa }f\right\Vert _{L_{I}^{2}\left( \mu \right) }^{2}\ .
\end{eqnarray*}

\subsection{A two weight bilinear Carleson Embedding Theorem}

The classical Carleson Embedding Theorem \cite{NTV4} states that for any
dyadic grid $\mathcal{D}$, and any sequence $\left\{ c_{I}\right\} _{I\in 
\mathcal{D}}$ of nonnegative numbers indexed by $\mathcal{D}$, 
\begin{equation}
\sum_{I\in \mathcal{D}}c_{I}\left( \frac{1}{\left\vert I\right\vert _{\sigma
}}\int_{I}fd\sigma \right) ^{2}\leq C\left\Vert f\right\Vert _{L^{2}\left(
\sigma \right) }^{2}  \label{class Car}
\end{equation}%
for all nonnegative $f\in L^{2}\left( \sigma \right) $,\emph{\ if and only if%
} the sequence $\left\{ c_{I}\right\} _{I\in \mathcal{D}}$ satisfies a
Carleson condition%
\begin{equation}
\sum_{I\in \mathcal{D}:\ I\subset J}c_{I}\leq C^{\prime }\left\vert
J\right\vert _{\sigma },\ \ \ \ \ \text{for all }J\in \mathcal{D}.
\label{class Car cond}
\end{equation}%
Moreover, $C^{\prime }\leq C\leq 4C^{\prime }$. The two weight bilinear
analogue of (\ref{class Car}) is the inequality%
\begin{equation}
\sum_{I\in \mathcal{D}}a_{I}\left( \frac{1}{\left\vert I\right\vert _{\sigma
}}\int_{I}fd\sigma \right) \left( \frac{1}{\left\vert I\right\vert _{\omega }%
}\int_{I}gd\omega \right) \leq C\left\Vert f\right\Vert _{L^{2}\left( \sigma
\right) }\left\Vert g\right\Vert _{L^{2}\left( \omega \right) }\ ,
\label{bil Car}
\end{equation}%
which is equivalent to the pair of Carleson-type conditions,%
\begin{equation}
\sum_{I^{\prime },I\in \mathcal{D}:\ I^{\prime }\subset I\subset K}\frac{%
a_{I^{\prime }}a_{I}}{\left\vert I\right\vert _{\sigma }}\leq C^{\prime
}\left\vert K\right\vert _{\omega }\text{ and }\sum_{I^{\prime },I\in 
\mathcal{D}:\ I^{\prime }\subset I\subset K}\frac{a_{I^{\prime }}a_{I}}{%
\left\vert I\right\vert _{\omega }}\leq C^{\prime }\left\vert K\right\vert
_{\sigma }\ ,\ \ \ \ \ \text{for all cubes }K\in \mathcal{D}.  \label{pair}
\end{equation}%
Indeed, (\ref{bil Car}) is equivalent to 
\begin{equation*}
\int \left\vert \sum_{I\in \mathcal{D}}\frac{a_{I}}{\left\vert I\right\vert
_{\sigma }}\left( \frac{1}{\left\vert I\right\vert _{\omega }}%
\int_{I}gd\omega \right) \mathbf{1}_{I}\left( y\right) \right\vert
^{2}d\sigma \left( y\right) \leq C^{2}\left\Vert g\right\Vert _{L^{2}\left(
\omega \right) }^{2},
\end{equation*}%
which by \cite{NTV} and \cite{LaSaUr2} is equivalent to the pair of testing
conditions%
\begin{eqnarray}
\int_{K}\left\vert \sum_{I\in \mathcal{D}:\ I\subset K}\frac{a_{I}}{%
\left\vert I\right\vert _{\sigma }}\mathbf{1}_{I}\left( y\right) \right\vert
^{2}d\sigma \left( y\right) &\leq &C^{2}\left\vert K\right\vert _{\omega }\
,\ \ \ \ \ \text{for all cubes }K\in \mathcal{D},  \label{pair test} \\
\int_{K}\left\vert \sum_{I\in \mathcal{D}:\ I\subset K}\frac{a_{I}}{%
\left\vert I\right\vert _{\omega }}\mathbf{1}_{I}\left( y\right) \right\vert
^{2}d\omega \left( y\right) &\leq &C^{2}\left\vert K\right\vert _{\sigma }\
,\ \ \ \ \ \text{for all cubes }K\in \mathcal{D}.  \notag
\end{eqnarray}

However, the Carleson-type conditions in (\ref{pair}) are too strong for our
purposes in this paper, and instead, we prove a bilinear extension of the
Carleson Embedding Theorem (related to the Bilinear Imbedding Theorem of
Nazarov, Treil and Volberg in \cite[page 915]{NTV}) which uses the more
familiar bilinear Carleson condition in (\ref{bil Car cond}) below - at the
expense of assuming comparability of the measure pair as in Definition \ref%
{def comparable} above.

Given any subset $\mathcal{A}$ of the dyadic grid $\mathcal{D}$, we view $%
\mathcal{A}$ as a subtree of $\mathcal{D}$, and denote by $\mathfrak{C}_{%
\mathcal{A}}\left( A\right) $ the set of $\mathcal{A}$-children of $A$ in
the tree $\mathcal{A}$, and\ by $\mathcal{C}_{\mathcal{A}}\left( A\right) $
the $\mathcal{A}$-corona of $A$ in the tree $\mathcal{A}$, so that%
\begin{equation*}
\mathcal{C}_{\mathcal{A}}\left( A\right) =\dbigcup\limits_{A^{\prime }\in 
\mathfrak{C}_{\mathcal{A}}\left( A\right) }\left\{ I\in \mathcal{D}%
:A^{\prime }\subsetneqq I\subset A\right\} .
\end{equation*}

\begin{theorem}[Two weight bilinear Carleson Embedding Theorem, c.f. 
\protect\cite{NTV}]
\label{2 wt bil CET}Suppose $\sigma $ and $\omega $ are locally finite
positive Borel measures on $\mathbb{R}^{n}$, and that $\mathcal{D}$ is a
dyadic grid.

\begin{enumerate}
\item Suppose further that $\left\{ a_{I}\right\} _{I\in \mathcal{D}}$ is a
sequence of nonnegative real numbers indexed by $\mathcal{D}$. If in
addition $\sigma $ and $\omega $ are comparable in the sense of Definition %
\ref{def comparable},\ then%
\begin{equation}
\sum_{I\in \mathcal{D}}a_{I}\left( \sup_{K\in \mathcal{D}:\ K\supset I}\frac{%
1}{\left\vert K\right\vert _{\sigma }}\int_{K}fd\sigma \right) \left(
\sup_{L\in \mathcal{D}:\ L\supset I}\frac{1}{\left\vert L\right\vert
_{\omega }}\int_{L}gd\omega \right) \leq C\left\Vert f\right\Vert
_{L^{2}\left( \sigma \right) }\left\Vert g\right\Vert _{L^{2}\left( \omega
\right) }  \label{bil Car Embed}
\end{equation}%
for all nonnegative $f\in L^{2}\left( \sigma \right) $ and nonnegative $g\in
L^{2}\left( \omega \right) $,\emph{\ if and only if} the sequence $\left\{
a_{I}\right\} _{I\in \mathcal{D}}$ satisfies the bilinear Carleson condition,%
\begin{equation}
\sum_{I\in \mathcal{D}:\ I\subset J}a_{I}\leq C^{\prime }\sqrt{\left\vert
J\right\vert _{\sigma }\left\vert J\right\vert _{\omega }},\ \ \ \ \ \text{%
for all }J\in \mathcal{D}\ ,  \label{bil Car cond}
\end{equation}%
where $C^{\prime }\leq C\lesssim C^{\prime }$.

\item The inequality%
\begin{equation}
\sum_{I\in \mathcal{D}}a_{I}\left( \frac{1}{\left\vert I\right\vert _{\sigma
}}\int_{I}fd\sigma \right) \left( \frac{1}{\left\vert I\right\vert _{\omega }%
}\int_{I}gd\omega \right) \leq C\left\{ \frac{1}{\sqrt{\left\vert
J\right\vert _{\sigma }\left\vert J\right\vert _{\omega }}}\sum_{I\in 
\mathcal{D}:\ I\subset J}a_{I}\right\} \left\Vert f\right\Vert _{L^{2}\left(
\sigma \right) }\left\Vert g\right\Vert _{L^{2}\left( \omega \right) }
\label{big}
\end{equation}%
holds \emph{if and only if} $\sigma $ and $\omega $ are comparable in the
sense of Definition \ref{def comparable}.
\end{enumerate}
\end{theorem}

\begin{proof}
\textbf{Part (1)}: The necessity of the bilinear Carleson condition follows
upon setting $f=g=\mathbf{1}_{J}$ in the bilinear inequality, since then for 
$I\subset J$ we have%
\begin{equation*}
\sup_{K\in \mathcal{D}:\ K\supset I}\frac{1}{\left\vert K\right\vert
_{\sigma }}\int_{K}fd\sigma \geq \frac{1}{\left\vert I\right\vert _{\sigma }}%
\int_{I}\mathbf{1}_{J}d\sigma =1\text{ and similarly }\sup_{L\in \mathcal{D}%
:\ L\supset I}\frac{1}{\left\vert L\right\vert _{\omega }}\int_{L}gd\omega
\geq 1,
\end{equation*}%
which gives%
\begin{equation*}
\sum_{I\in \mathcal{D}:\ I\subset J}a_{I}\leq C\left\Vert f\right\Vert
_{L^{2}\left( \sigma \right) }\left\Vert g\right\Vert _{L^{2}\left( \omega
\right) }=C\sqrt{\left\vert J\right\vert _{\sigma }\left\vert J\right\vert
_{\omega }}.
\end{equation*}

For the converse assertion, fix $\Gamma \geq 4$, and let $\mathcal{A}$ be a
collection of $\Gamma $-Calder\'{o}n-Zygmund stopping cubes for $f\in
L^{2}\left( \sigma \right) $, and let $\mathcal{B}$ be a collection of $%
\Gamma $-Calder\'{o}n-Zygmund stopping cubes for $g\in L^{2}\left( \omega
\right) $. Then we have%
\begin{eqnarray}
\frac{1}{\left\vert A^{\prime }\right\vert _{\sigma }}\int_{A^{\prime
}}fd\sigma &>&\Gamma \frac{1}{\left\vert A\right\vert _{\sigma }}%
\int_{A}fd\sigma ,\ \ \ \ \ A^{\prime }\in \mathfrak{C}_{\mathcal{A}}\left(
A\right) \ ,  \label{sigma CZ} \\
\frac{1}{\left\vert I\right\vert _{\sigma }}\int_{I}fd\sigma &\leq &\Gamma 
\frac{1}{\left\vert A\right\vert _{\sigma }}\int_{A}fd\sigma ,\ \ \ \ \ I\in 
\mathcal{C}_{\mathcal{A}}\left( A\right) \ ,  \notag \\
\sum_{A^{\prime }\in \mathcal{A}:\ A^{\prime }\subset A}\left\vert A^{\prime
}\right\vert _{\sigma } &\leq &C_{\Gamma }\left\vert A\right\vert _{\sigma
}\ ,  \notag
\end{eqnarray}%
and similarly%
\begin{eqnarray*}
\frac{1}{\left\vert B^{\prime }\right\vert _{\omega }}\int_{B^{\prime
}}gd\omega &>&\Gamma \frac{1}{\left\vert B\right\vert _{\omega }}%
\int_{B}gd\omega \ ,\ \ \ \ \ B^{\prime }\in \mathfrak{C}_{\mathcal{B}%
}\left( B\right) , \\
\frac{1}{\left\vert J\right\vert _{\omega }}\int_{J}gd\omega &\leq &\Gamma 
\frac{1}{\left\vert B\right\vert _{\omega }}\int_{B}gd\omega \ ,\ \ \ \ \
J\in \mathcal{C}_{\mathcal{B}}\left( B\right) , \\
\sum_{B^{\prime }\in \mathcal{B}:\ B^{\prime }\subset B}\left\vert B^{\prime
}\right\vert _{\omega } &\leq &C_{\Gamma }\left\vert B\right\vert _{\omega
}\ .
\end{eqnarray*}%
Now we estimate the left hand side of (\ref{bil Car Embed}),%
\begin{eqnarray*}
&&\sum_{I\in \mathcal{D}}a_{I}\left( \sup_{K\in \mathcal{D}:\ K\supset I}%
\frac{1}{\left\vert K\right\vert _{\sigma }}\int_{K}fd\sigma \right) \left(
\sup_{L\in \mathcal{D}:\ L\supset I}\frac{1}{\left\vert L\right\vert
_{\omega }}\int_{L}gd\omega \right) \\
&=&\sum_{A\in \mathcal{A}}\sum_{B\in \mathcal{B}}\sum_{I\in \mathcal{D}:\
I\in \mathcal{C}_{\mathcal{A}}\left( A\right) \cap \mathcal{C}_{\mathcal{B}%
}\left( B\right) }a_{I}\left( \sup_{K\in \mathcal{D}:\ K\supset I}\frac{1}{%
\left\vert K\right\vert _{\sigma }}\int_{K}fd\sigma \right) \left(
\sup_{L\in \mathcal{D}:\ L\supset I}\frac{1}{\left\vert L\right\vert
_{\omega }}\int_{L}gd\omega \right) \\
&\leq &\Gamma ^{2}\sum_{A\in \mathcal{A}}\sum_{B\in \mathcal{B}}\left\{
\sum_{I\in \mathcal{D}:\ I\in \mathcal{C}_{\mathcal{A}}\left( A\right) \cap 
\mathcal{C}_{\mathcal{B}}\left( B\right) }a_{I}\right\} \left( \frac{1}{%
\left\vert A\right\vert _{\sigma }}\int_{A}fd\sigma \right) \left( \frac{1}{%
\left\vert B\right\vert _{\omega }}\int_{B}gd\omega \right) .
\end{eqnarray*}%
Since (\ref{bil Car cond}) implies%
\begin{equation*}
\sum_{I\in \mathcal{D}:\ I\in \mathcal{C}_{\mathcal{A}}\left( A\right) \cap 
\mathcal{C}_{\mathcal{B}}\left( B\right) }a_{I}\leq \left\{ 
\begin{array}{ccc}
C^{\prime }\min \left\{ \sqrt{\left\vert A\right\vert _{\sigma }\left\vert
A\right\vert _{\omega }},\sqrt{\left\vert B\right\vert _{\sigma }\left\vert
B\right\vert _{\omega }}\right\} & \text{ if } & A\cap B\neq \emptyset \\ 
0 & \text{ if } & A\cap B=\emptyset%
\end{array}%
\right. ,
\end{equation*}%
we conclude that the left hand side of (\ref{bil Car Embed}) is at most%
\begin{eqnarray*}
&&C^{\prime }\Gamma ^{2}\sum_{A\in \mathcal{A}}\sum_{B\in \mathcal{B}:\ B\in 
\mathcal{C}_{\mathcal{A}}\left( A\right) }\sqrt{\left\vert B\right\vert
_{\sigma }\left\vert B\right\vert _{\omega }}\left( \frac{1}{\left\vert
A\right\vert _{\sigma }}\int_{A}fd\sigma \right) \left( \frac{1}{\left\vert
B\right\vert _{\omega }}\int_{B}gd\omega \right) \\
&&+C^{\prime }\Gamma ^{2}\sum_{B\in \mathcal{B}}\sum_{A\in \mathcal{A}:\
A\in \mathcal{C}_{\mathcal{B}}\left( B\right) }\sqrt{\left\vert A\right\vert
_{\sigma }\left\vert A\right\vert _{\omega }}\left( \frac{1}{\left\vert
A\right\vert _{\sigma }}\int_{A}fd\sigma \right) \left( \frac{1}{\left\vert
B\right\vert _{\omega }}\int_{B}gd\omega \right) \\
&\equiv &S_{1}+S_{2}.
\end{eqnarray*}

By symmetry it suffices to bound the first sum $S_{1}$. By Cauchy-Schwarz,
we have%
\begin{equation*}
\sum_{B\in \mathcal{B}:\ B\in \mathcal{C}_{\mathcal{A}}\left( A\right) }%
\sqrt{\left\vert B\right\vert _{\sigma }\left\vert B\right\vert _{\omega }}%
\left( \frac{1}{\left\vert B\right\vert _{\omega }}\int_{B}gd\omega \right)
\leq \sqrt{\sum_{B\in \mathcal{B}:\ B\in \mathcal{C}_{\mathcal{A}}\left(
A\right) }\left\vert B\right\vert _{\sigma }}\sqrt{\sum_{B\in \mathcal{B}:\
B\in \mathcal{C}_{\mathcal{A}}\left( A\right) }\left\vert B\right\vert
_{\omega }\left( \frac{1}{\left\vert B\right\vert _{\omega }}%
\int_{B}gd\omega \right) ^{2}}.
\end{equation*}%
We now invoke the comparability assumption on the measures $\sigma $ and $%
\omega $, which implies that the grid $\mathcal{B}$ is also $\sigma $%
-Carleson, hence $\sum_{B\in \mathcal{B}:\ B\in \mathcal{C}_{\mathcal{A}%
}\left( A\right) }\left\vert B\right\vert _{\sigma }\leq C\left\vert
A\right\vert _{\sigma }$. Thus we conclude%
\begin{eqnarray*}
S_{1} &\leq &C^{\prime }\Gamma ^{2}\sum_{A\in \mathcal{A}}\left( \frac{1}{%
\left\vert A\right\vert _{\sigma }}\int_{A}fd\sigma \right) \sqrt{\left\vert
A\right\vert _{\sigma }}\sqrt{\sum_{B\in \mathcal{B}:\ B\in \mathcal{C}_{%
\mathcal{A}}\left( A\right) }\left\vert B\right\vert _{\omega }\left( \frac{1%
}{\left\vert B\right\vert _{\omega }}\int_{B}gd\omega \right) ^{2}} \\
&\leq &C^{\prime }\Gamma ^{2}\sqrt{\sum_{A\in \mathcal{A}}\left\vert
A\right\vert _{\sigma }\left( \frac{1}{\left\vert A\right\vert _{\sigma }}%
\int_{A}fd\sigma \right) ^{2}}\sqrt{\sum_{A\in \mathcal{A}}\sum_{B\in 
\mathcal{B}:\ B\in \mathcal{C}_{\mathcal{A}}\left( A\right) }\left\vert
B\right\vert _{\omega }\left( \frac{1}{\left\vert B\right\vert _{\omega }}%
\int_{B}gd\omega \right) ^{2}} \\
&\leq &C\left\Vert f\right\Vert _{L^{2}\left( \sigma \right) }\left\Vert
g\right\Vert _{L^{2}\left( \omega \right) }\ ,
\end{eqnarray*}%
with $C$ depending on $C^{\prime }$ and $\Gamma $, upon applying the usual
Carleson Embedding Theorem to both stopping collections $\mathcal{A}$ and $%
\mathcal{B}$. Indeed, we take $c_{I}\equiv \left\{ 
\begin{array}{ccc}
\left\vert I\right\vert _{\sigma } & \text{ if } & I\in \mathcal{A} \\ 
0 & \text{ if } & I\not\in \mathcal{A}%
\end{array}%
\right. $ in (\ref{class Car}), note that $\left\{ c_{I}\right\} _{I\in 
\mathcal{D}}$ satisfies the Carleson condition (\ref{class Car cond}) with $%
C^{\prime }=C_{\Gamma }$ by the third line in (\ref{sigma CZ}), and it then
follows from (\ref{class Car}) that 
\begin{equation*}
\sum_{A\in \mathcal{A}}\left\vert A\right\vert _{\sigma }\left( \frac{1}{%
\left\vert A\right\vert _{\sigma }}\int_{A}fd\sigma \right) ^{2}\leq
C_{\Gamma }\left\Vert f\right\Vert _{L^{2}\left( \sigma \right) }^{2}.
\end{equation*}%
Similarly, we obtain 
\begin{equation*}
\sum_{A\in \mathcal{A}}\sum_{B\in \mathcal{B}:\ B\in \mathcal{C}_{\mathcal{A}%
}\left( A\right) }\left\vert B\right\vert _{\omega }\left( \frac{1}{%
\left\vert B\right\vert _{\omega }}\int_{B}gd\omega \right) ^{2}\leq
C\left\Vert g\right\Vert _{L^{2}\left( \omega \right) }^{2}\ .
\end{equation*}

\textbf{Part (2)}: It remains to show that if (\ref{big}) holds, then $%
\sigma $ and $\omega $ are comparable in the sense of Definition \ref{def
comparable}. So let the dyadic grid $\mathcal{F}$ be $\omega $-Carleson, and
let $C,\delta >0$ be such that $\left\vert G_{k}\left( Q\right) \right\vert
_{\omega }\leq C2^{-k\delta }\left\vert Q\right\vert _{\omega }$ for all $%
Q\in \mathcal{F}$. Define%
\begin{equation*}
a_{I}\equiv \left\{ 
\begin{array}{ccc}
\sqrt{\left\vert I\right\vert _{\sigma }\left\vert I\right\vert _{\omega }}
& \text{ if } & I\in \mathcal{F} \\ 
0 & \text{ if } & I\not\in \mathcal{F}%
\end{array}%
\right. ,
\end{equation*}%
and note that if $\left\{ M_{i}\right\} _{i=1}^{\infty }$ are the maximal
cubes in $\mathcal{F}$ that are contained in $J$, then 
\begin{eqnarray*}
\sum_{I\in \mathcal{D}:\ I\subset J}a_{I} &=&\sum_{i=1}^{\infty }\sum_{I\in 
\mathcal{F}:\ I\subset M_{i}}\sqrt{\left\vert I\right\vert _{\sigma
}\left\vert I\right\vert _{\omega }}=\sum_{i=1}^{\infty }\sum_{k=0}^{\infty
}\sum_{I\in \mathfrak{C}_{\mathcal{F}}^{\left( k\right) }\left( M_{i}\right)
}\sqrt{\left\vert I\right\vert _{\sigma }\left\vert I\right\vert _{\omega }}
\\
&\leq &\sum_{i=1}^{\infty }\sqrt{\sum_{k=0}^{\infty }2^{-k\varepsilon
}\sum_{I\in \mathfrak{C}_{\mathcal{F}}^{\left( k\right) }\left( M_{i}\right)
}\left\vert I\right\vert _{\sigma }}\sqrt{\sum_{k=0}^{\infty
}2^{k\varepsilon }\sum_{I\in \mathfrak{C}_{\mathcal{F}}^{\left( k\right)
}\left( M_{i}\right) }\left\vert I\right\vert _{\omega }} \\
&\leq &\sum_{i=1}^{\infty }\sqrt{\sum_{k=0}^{\infty }2^{-k\varepsilon
}\left\vert M_{i}\right\vert _{\sigma }}\sqrt{\sum_{k=0}^{\infty
}2^{k\varepsilon }C2^{-k\delta }\left\vert M_{i}\right\vert _{\omega }} \\
&\leq &C\sum_{i=1}^{\infty }\sqrt{\left\vert M_{i}\right\vert _{\sigma
}\left\vert M_{i}\right\vert _{\omega }}\leq C\sqrt{\sum_{i=1}^{\infty
}\left\vert M_{i}\right\vert _{\sigma }}\sqrt{\sum_{i=1}^{\infty }\left\vert
M_{i}\right\vert _{\omega }}\leq C\sqrt{\left\vert J\right\vert _{\sigma }}%
\sqrt{\left\vert J\right\vert _{\omega }}\ .
\end{eqnarray*}%
Thus from (\ref{big}) we obtain%
\begin{equation*}
\sum_{I\in \mathcal{F}}\sqrt{\left\vert I\right\vert _{\sigma }\left\vert
I\right\vert _{\omega }}\left( \frac{1}{\left\vert I\right\vert _{\sigma }}%
\int_{I}fd\sigma \right) \left( \frac{1}{\left\vert I\right\vert _{\omega }}%
\int_{I}gd\omega \right) \leq C\left\Vert f\right\Vert _{L^{2}\left( \sigma
\right) }\left\Vert g\right\Vert _{L^{2}\left( \omega \right) }\ ,
\end{equation*}%
and then from (\ref{pair test}) we conclude%
\begin{eqnarray*}
\sum_{I\in \mathcal{F}:\ I\subset K}\left\vert I\right\vert _{\sigma }
&=&\sum_{I\in \mathcal{F}:\ I\subset K}\frac{\left\vert I\right\vert
_{\omega }\left\vert I\right\vert _{\sigma }}{\left\vert I\right\vert
_{\omega }^{2}}\left\vert I\right\vert _{\omega }=\int_{K}\sum_{I\in 
\mathcal{F}:\ I\subset K}\left\vert \frac{\sqrt{\left\vert I\right\vert
_{\omega }\left\vert I\right\vert _{\sigma }}}{\left\vert I\right\vert
_{\omega }}\mathbf{1}_{I}\left( y\right) \right\vert ^{2}d\omega \left(
y\right) \\
&\leq &\int_{K}\left\vert \sum_{I\in \mathcal{F}:\ I\subset K}\frac{\sqrt{%
\left\vert I\right\vert _{\omega }\left\vert I\right\vert _{\sigma }}}{%
\left\vert I\right\vert _{\omega }}\mathbf{1}_{I}\left( y\right) \right\vert
^{2}d\omega \left( y\right) \leq C^{2}\left\vert K\right\vert _{\sigma }\ ,\
\ \ \ \ \text{for all cubes }K\in \mathcal{F},
\end{eqnarray*}%
which is $\left\Vert \mathcal{F}\right\Vert _{\limfunc{Car}\left( \sigma
\right) }\leq C\left\Vert \mathcal{F}\right\Vert _{\limfunc{Car}\left(
\omega \right) }$. A dual argument gives $\left\Vert \mathcal{F}\right\Vert
_{\limfunc{Car}\left( \omega \right) }\leq C\left\Vert \mathcal{F}%
\right\Vert _{\limfunc{Car}\left( \sigma \right) }$, and so $\sigma $ and $%
\omega $ are comparable in the sense of Definition \ref{def comparable}.
\end{proof}

\section{Controlling polynomial testing conditions by $T1$ and $\mathcal{A}%
_{2}$}

Here we show that the familiar $T1$ testing conditions over indicators of
cubes, imply the $Tp$ testing conditions over polynomials times indicators
of cubes. To highlight the main idea, we begin with the simpler case of
dimension $n=1$. We start with the elementary formula for recovering a
linear function, restricted to an interval, from indicators of intervals:%
\begin{equation}
\mathbf{1}_{\left[ a,b\right) }\left( y\right) \left( \frac{y-a}{b-a}\right)
=\int_{a}^{b}\mathbf{1}_{\left[ r,b\right) }\left( y\right) \frac{dr}{b-a},\
\ \ \ \ \text{for all }y\in \mathbb{R}\ .  \label{elem form}
\end{equation}%
We conclude that for any locally finite positive Borel measure $\sigma $,
and any operator $T$ bounded from $L^{2}\left( \sigma \right) $ to $%
L^{2}\left( \omega \right) $,%
\begin{equation*}
T_{\sigma }\left( \mathbf{1}_{\left[ a,b\right) }\left( y\right) \left( 
\frac{y-a}{b-a}\right) \right) \left( x\right) =T_{\sigma }\left(
\int_{a}^{b}\mathbf{1}_{\left[ r,b\right) }\left( y\right) \frac{dr}{b-a}%
\right) \left( x\right) =\int_{a}^{b}\left( T_{\sigma }\mathbf{1}_{\left[
r,b\right) }\right) \left( x\right) \frac{dr}{b-a},
\end{equation*}
where $T_{\sigma }$ has moved inside the integral since truncations of
fractional Calder\'{o}n-Zygmund operators have bounded compactly supported
kernels. We then use the testing estimate $\left\Vert T_{\sigma }\mathbf{1}_{%
\left[ r,b\right) }\right\Vert _{L^{2}\left( \omega \right) }^{2}\leq \left( 
\mathfrak{FT}_{T}\right) ^{2}\left\vert \left[ r,b\right) \right\vert
_{\sigma }$, together with Minkowski's inequality $\left\Vert \int
f\right\Vert \leq \int \left\Vert f\right\Vert $, to obtain 
\begin{eqnarray*}
&&\left\Vert T_{\sigma }\left[ \mathbf{1}_{\left[ a,b\right) }\left(
y\right) \left( \frac{y-a}{b-a}\right) \right] \right\Vert _{L^{2}\left(
\omega \right) }=\left\Vert T_{\sigma }\left[ \int_{a}^{b}\mathbf{1}_{\left[
r,b\right) }\left( y\right) \frac{dr}{b-a}\right] \right\Vert _{L^{2}\left(
\omega \right) } \\
&\leq &\int_{a}^{b}\left\Vert T_{\sigma }\left[ \mathbf{1}_{\left[
r,b\right) }\left( y\right) \right] \right\Vert _{L^{2}\left( \omega \right)
}\frac{dr}{b-a}\leq \int_{a}^{b}\mathfrak{FT}_{T}\sqrt{\left\vert \left[
r,b\right) \right\vert _{\sigma }}\frac{dr}{b-a} \\
&\leq &\mathfrak{FT}_{T}\sqrt{\int_{a}^{b}\left\vert \left[ r,b\right)
\right\vert _{\sigma }\frac{dr}{b-a}}=\mathfrak{FT}_{T}\sqrt{%
\int_{a}^{b}\left( \int_{\left[ r,b\right) }d\sigma \left( y\right) \right) 
\frac{dr}{b-a}} \\
&=&\mathfrak{FT}_{T}\sqrt{\int_{\left[ a,b\right) }\left( \int_{a}^{y}\frac{%
dr}{b-a}\right) d\sigma \left( y\right) }=\mathfrak{FT}_{T}\sqrt{\int_{\left[
a,b\right) }\frac{y-a}{b-a}d\sigma \left( y\right) }\leq \mathfrak{FT}_{T}%
\sqrt{\left\vert \left[ a,b\right) \right\vert _{\sigma }}\ ,
\end{eqnarray*}%
and hence $\mathfrak{FT}_{T}^{\left( 1\right) }\leq \mathfrak{FT}%
_{T}^{\left( 0\right) }\equiv \mathfrak{FT}_{T}$. Similarly, the identity%
\begin{equation*}
\mathbf{1}_{\left[ a,b\right) }\left( y\right) \left( \frac{y-a}{b-a}\right)
^{2}=\int_{a}^{b}1_{\left[ r,b\right) }\left( y\right) 2\left( \frac{y-r}{b-a%
}\right) \frac{dr}{b-a},\ \ \ \ \ \text{for all }y\in \mathbb{R},
\end{equation*}%
shows that%
\begin{eqnarray*}
&&\left\Vert T\left[ \mathbf{1}_{\left[ a,b\right) }\left( y\right) \left( 
\frac{y-a}{b-a}\right) ^{2}\right] \right\Vert _{L^{2}\left( \omega \right)
}=\left\Vert T\left[ \int_{a}^{b}\mathbf{1}_{\left[ r,b\right) }\left(
y\right) 2\left( \frac{y-r}{b-a}\right) dr\right] \right\Vert _{L^{2}\left(
\omega \right) } \\
&\leq &2\int_{a}^{b}\left\Vert T\left[ \mathbf{1}_{\left[ r,b\right) }\left(
y\right) \left( \frac{y-r}{b-a}\right) \right] \right\Vert _{L^{2}\left(
\omega \right) }\frac{dr}{b-a}\leq 2\mathfrak{FT}_{T}^{\left( 1\right) }%
\sqrt{\left\vert \left[ a,b\right) \right\vert _{\sigma }},
\end{eqnarray*}%
and hence $\mathfrak{FT}_{T}^{\left( 2\right) }\leq 2\mathfrak{FT}%
_{T}^{\left( 1\right) }$. Continuing in this manner we obtain%
\begin{equation*}
\mathfrak{FT}_{T}^{\left( \kappa \right) }\leq \kappa \ \mathfrak{FT}%
_{T}^{\left( \kappa -1\right) },\ \ \ \ \ \text{for all }\kappa \geq 1,
\end{equation*}%
which when iterated gives%
\begin{equation*}
\mathfrak{FT}_{T}^{\left( \kappa \right) }\left( \sigma ,\omega \right) \leq
\kappa !\mathfrak{FT}_{T}\left( \sigma ,\omega \right) .
\end{equation*}

By a result of Hyt\"{o}nen \cite{Hyt2}, see also \cite{SaShUr12} for the
straightforward extension to fractional singular integrals, the full testing
constant $\mathfrak{FT}_{T}\left( \sigma ,\omega \right) $ in dimension $n=1$%
, is controlled by the usual testing constant $\mathfrak{T}_{T}\left( \sigma
,\omega \right) $ and the one-tailed Muckenhoupt condition $\mathcal{A}%
_{2}^{\alpha }$ . Thus we have proved the following lemma for the case when $%
T=T^{\alpha }$ is a fractional Calder\'{o}n-Zygmund operator in dimension $%
n=1$.

\begin{lemma}
Suppose that $\sigma $ and $\omega $ are locally finite positive Borel
measures on $\mathbb{R}$ and $\kappa \in \mathbb{N}$. If $T^{\alpha }$ is a
bounded $\alpha $-fractional Calder\'{o}n-Zygmund operator from $L^{2}\left(
\sigma \right) $ to $L^{2}\left( \omega \right) $, then we have 
\begin{equation*}
\mathfrak{T}_{T^{\alpha }}^{\left( \kappa \right) }\left( \sigma ,\omega
\right) \leq \kappa !\mathfrak{T}_{T^{\alpha }}\left( \sigma ,\omega \right)
+C_{\kappa }\mathcal{A}_{2}^{\alpha }\left( \sigma ,\omega \right) \ ,\ \ \
\ \ \kappa \geq 1,
\end{equation*}%
where the constant $C_{\kappa }$ depends on the kernel constant $C_{CZ}$ in (%
\ref{sizeandsmoothness'}), but is independent of the operator norm $%
\mathfrak{N}_{T^{\alpha }}\left( \sigma ,\omega \right) $.
\end{lemma}

\subsection{The higher dimensional case}

The higher dimensional version of this lemma will include a small multiple
of the operator norm $\mathfrak{N}_{T}\left( \sigma ,\omega \right) $ in
place of the one-tailed Muckenhoupt constant $\mathcal{A}_{2}^{\alpha
}\left( \sigma ,\omega \right) $ on the right hand side, since we no longer
have available an analogue of Hyt\"{o}nen's result. Nevertheless, we show
below that for doubling measures, the two testing conditions are equivalent
in the presence of one-tailed Muckenhoupt conditions (\ref{one-sided}) in
all dimensions, and so we will be able to prove a $T1$ theorem in higher
dimensions in certain cases.

\begin{theorem}
\label{Tp control by T1}Suppose that $\sigma $ and $\omega $ are locally
finite positive Borel measures on $\mathbb{R}^{n}$, and let $\kappa \in 
\mathbb{N}$. If $T$ is a bounded operator from $L^{2}\left( \sigma \right) $
to $L^{2}\left( \omega \right) $, then\ for every $0<\varepsilon <1$, there
is a positive constant $C\left( \kappa ,\varepsilon \right) $ such that 
\begin{equation*}
\mathfrak{FT}_{T}^{\left( \kappa \right) }\left( \sigma ,\omega \right) \leq
C\left( \kappa ,\varepsilon \right) \mathfrak{FT}_{T}\left( \sigma ,\omega
\right) +\varepsilon \mathfrak{N}_{T}\left( \sigma ,\omega \right) \ ,\ \ \
\ \ \kappa \geq 1,
\end{equation*}%
and where the constants $C\left( \kappa ,\varepsilon \right) $ depend only
on $\kappa $ and $\varepsilon $, and not on the operator norm $\mathfrak{N}%
_{T}\left( \sigma ,\omega \right) $.
\end{theorem}

\begin{proof}
We begin with the following geometric observation, similar to a construction
used in the recursive control of the nearby form in \cite{SaShUr12}. Let $R=%
\left[ 0,1\right) ^{n-1}\times \left[ 0,t\right) $ be a rectangle in $%
\mathbb{R}^{n}$ with $0<t<1$. Then given $0<\varepsilon <1$, there is a
positive integer $m\in \mathbb{N}$ and a dyadic number $t^{\ast }\equiv 
\frac{b}{2^{m}}$ with $0\leq b<2^{m}$, so that%
\begin{eqnarray}
R &=&E\overset{\cdot }{\cup }\left\{ \overset{\cdot }{\dbigcup }%
_{i=1}^{B}K_{i}\right\} ;  \label{decomp} \\
E &=&\left[ 0,1\right) ^{n-1}\times \left[ t^{\ast },t\right) \text{ with }%
\left\vert t-t^{\ast }\right\vert <\varepsilon ,  \notag \\
B &\leq &2^{nm-n-m+2},  \notag
\end{eqnarray}%
and where the $K_{i}$ are pairwise disjoint cubes inside $R$. To see (\ref%
{decomp}) we choose $m\in \mathbb{N}$ so that $\frac{1}{2^{m}}<\varepsilon $
and then let $b\in \mathbb{N}$ satisfy $2^{m}t-1\leq b<2^{m}t$. Then with $%
t^{\ast }=\frac{b}{2^{m}}$ we have $\left\vert t-t^{\ast }\right\vert <\frac{%
1}{2^{m}}<\varepsilon $. Now expand $t^{\ast }$ in binary form,%
\begin{equation*}
t^{\ast }=b_{1}\frac{1}{2}+b_{2}\frac{1}{4}+...+b_{m-1}\frac{1}{2^{m-1}},\ \
\ \ \ b_{k}\in \left\{ 0,1\right\} .
\end{equation*}%
Then for each $k$ with $b_{k}=1$ we decompose the rectangle 
\begin{equation*}
R_{k}\equiv \left[ 0,1\right) ^{n-1}\times \left[ b_{1}\frac{1}{2}+b_{2}%
\frac{1}{4}+...+b_{k-1}\frac{1}{2^{k-1}},b_{1}\frac{1}{2}+b_{2}\frac{1}{4}%
+...+b_{k-1}\frac{1}{2^{k-1}}+\frac{1}{2^{k}}\right)
\end{equation*}%
into $2^{\left( n-1\right) k}$ pairwise disjoint dyadic cubes of side length 
$\frac{1}{2^{k}}$. Then we take the collection of all such cubes, noting
that the number $B$ of such cubes is at most 
\begin{equation*}
\sum_{k=1}^{m-1}2^{\left( n-1\right) k}\leq 2\cdot 2^{\left( n-1\right)
\left( m-1\right) }=2^{nm-n-m+2},
\end{equation*}%
and label them as $\left\{ K_{i}\right\} _{i=1}^{B}$ with $B\leq
2^{nm-n-m+2} $. Finally we note that 
\begin{equation*}
\overset{\cdot }{\dbigcup }_{i=1}^{B}K_{i}=\overset{\cdot }{\dbigcup }_{k:\
b_{k}=1}R_{k}=\left[ 0,1\right) ^{n-1}\times \left[ 0,t^{\ast }\right) .
\end{equation*}%
This completes the proof of (\ref{decomp}). Note that we may arrange to have 
$m\approx \ln \frac{1}{\varepsilon }$.

We also have the same result for the complementary rectangle $R=\left[
0,1\right) ^{n-1}\times \left[ r,1\right) $ by simply reflecting about the
plane $y_{n}=\frac{1}{2}$ and taking $r=1-t$. It is in this complementary
form that we will use (\ref{decomp}).

Again we start by considering the full testing condition $\mathfrak{FT}%
_{T}^{1}$ over linear functions, and we begin by estimating%
\begin{equation*}
\left\Vert T_{\sigma }\left[ \mathbf{1}_{Q}\left( y\right) y_{j}\right]
\right\Vert _{L^{2}\left( \omega \right) }^{2},\ \ \ \ \ Q\in \mathcal{P}%
^{n},1\leq j\leq n.
\end{equation*}%
In order to reduce notational clutter in appealing to the complementary form
of the geometric observation above, we will suppose - without loss of
generality - that $Q=\left[ 0,1\right) ^{n}$ is the unit cube in $\mathbb{R}%
^{n}$, and that $j=n$. Then we have%
\begin{equation*}
\mathbf{1}_{\left[ 0,1\right) ^{n}}\left( y\right) y_{n}=\int_{0}^{1}\mathbf{%
1}_{\left[ 0,1\right) ^{n-1}\times \left[ r,1\right) }\left( y\right) dr,\ \
\ \ \ \text{for all }y\in \mathbb{R}^{n},
\end{equation*}%
and%
\begin{equation*}
T_{\sigma }\left( \mathbf{1}_{\left[ 0,1\right) ^{n}}\left( y\right)
y_{n}\right) \left( x\right) =T_{\sigma }\left( \int_{0}^{1}\mathbf{1}_{%
\left[ 0,1\right) ^{n-1}\times \left[ r,1\right) }\left( y\right) dr\right)
\left( x\right) =\int_{0}^{1}\left( T_{\sigma }\mathbf{1}_{\left[ 0,1\right)
^{n-1}\times \left[ r,1\right) }\right) \left( x\right) dr.
\end{equation*}%
The norm estimate is complicated by the lack of Hyt\"{o}nen's result in
higher dimensions, and we compensate by using the complementary form of the
geometric observation (\ref{decomp}), together with a simple probability
argument. Let $\left[ r,1\right) =\left[ r,r^{\ast }\right) \overset{\cdot }{%
\cup }\left[ r^{\ast },1\right) $ and write%
\begin{eqnarray*}
&&\left\Vert T_{\sigma }\mathbf{1}_{\left[ 0,1\right) ^{n-1}\times \left[
r,1\right) }\right\Vert _{L^{2}\left( \omega \right) }^{2}=\int \left\vert
T_{\sigma }\left\{ \mathbf{1}_{\left[ 0,1\right) ^{n-1}\times \left[
r,r^{\ast }\right) }+\mathbf{1}_{\left[ 0,1\right) ^{n-1}\times \left[
r^{\ast },1\right) }\right\} \left( x\right) \right\vert ^{2}d\omega \left(
x\right) \\
&=&\int \left\vert T_{\sigma }\left\{ \mathbf{1}_{\left[ 0,1\right)
^{n-1}\times \left[ r,r^{\ast }\right) }+\sum_{i=1}^{B}\mathbf{1}%
_{K_{i}}\right\} \left( x\right) \right\vert ^{2}d\omega \left( x\right) \\
&\lesssim &\int \left\vert T_{\sigma }\mathbf{1}_{\left[ 0,1\right)
^{n-1}\times \left[ r,r^{\ast }\right) }\left( x\right) \right\vert
^{2}d\omega \left( x\right) +\sum_{i=1}^{B}\int \left\vert T_{\sigma }%
\mathbf{1}_{K_{i}}\left( x\right) \right\vert ^{2}d\omega \left( x\right) \\
&\leq &\int \left\vert T_{\sigma }\mathbf{1}_{\left[ 0,1\right) ^{n-1}\times %
\left[ r,r^{\ast }\right) }\left( x\right) \right\vert ^{2}d\omega \left(
x\right) +\left( \mathfrak{FT}_{T}\right) ^{2}\sum_{i=1}^{B}\left\vert
K_{i}\right\vert _{\sigma }.
\end{eqnarray*}%
First, we apply a simple probability argument to the integral over $r$ of
the last integral above by pigeonholing the values taken by $r^{\ast }\in
\left\{ \frac{b}{2^{m}}\right\} _{0\leq b<2^{m}}$: 
\begin{eqnarray*}
&&\int_{0}^{1}\int \left\vert T_{\sigma }\mathbf{1}_{\left[ 0,1\right)
^{n-1}\times \left[ r,r^{\ast }\right) }\left( x\right) \right\vert
^{2}d\omega \left( x\right) dr\leq \mathfrak{N}_{T}\left( \sigma ,\omega
\right) ^{2}\int_{0}^{1}\left\{ \int_{\left[ 0,1\right) ^{n-1}\times \left[
r,r^{\ast }\right) }d\sigma \right\} dr \\
&=&\mathfrak{N}_{T}\left( \sigma ,\omega \right) ^{2}\sum_{0<b\leq
2^{m}}\int_{\frac{b-1}{2^{m}}}^{\frac{b}{2^{m}}}\left\{ \int_{\left[
0,1\right) ^{n-1}\times \left[ r,\frac{b}{2^{m}}\right) }d\sigma \right\}
dr\leq \mathfrak{N}_{T}\left( \sigma ,\omega \right) ^{2}\int_{\left[
0,1\right) ^{n}}\left\{ \int_{y_{n}-\varepsilon }^{y_{n}}dr\right\} d\sigma
\left( y_{1},...,y_{n}\right) \\
&\leq &\mathfrak{N}_{T}\left( \sigma ,\omega \right) ^{2}\int_{\left[
0,1\right) ^{n}}\varepsilon d\sigma \left( y_{1},...,y_{n}\right)
=\varepsilon \mathfrak{N}_{T}\left( \sigma ,\omega \right) ^{2}\left\vert %
\left[ 0,1\right) ^{n}\right\vert _{\sigma }\ ,
\end{eqnarray*}%
since $\frac{b-1}{2^{m}}\leq r\leq y_{n}<\frac{b}{2^{m}}$ implies $%
y_{n}-\varepsilon <y_{n}-\frac{1}{2^{m}}\leq r\leq y_{n}$.

Combining estimates, and setting $R_{r}\equiv \left[ 0,1\right) ^{n-1}\times %
\left[ r,1\right) $ for convenience, we obtain%
\begin{eqnarray*}
&&\left\Vert T_{\sigma }\left[ \mathbf{1}_{R_{r}}\left( y\right) y_{n}\right]
\right\Vert _{L^{2}\left( \omega \right) }=\left\Vert T_{\sigma }\left[
\int_{0}^{1}\mathbf{1}_{R_{r}}\left( y\right) dr\right] \right\Vert
_{L^{2}\left( \omega \right) } \\
&\leq &\int_{0}^{1}\left\Vert T_{\sigma }\left[ \mathbf{1}_{R_{r}}\left(
y\right) \right] \right\Vert _{L^{2}\left( \omega \right) }dr\leq \mathfrak{%
FT}_{T}\left( \sigma ,\omega \right) \int_{0}^{1}\sqrt{\left\vert
R_{r}\right\vert _{\sigma }}dr+\varepsilon \mathfrak{N}_{T}\left( \sigma
,\omega \right) \left\vert \left[ 0,1\right) ^{n}\right\vert _{\sigma }\ ,
\end{eqnarray*}%
where%
\begin{eqnarray*}
\int_{0}^{1}\sqrt{\left\vert R_{r}\right\vert _{\sigma }}dr &\leq &\sqrt{%
\int_{0}^{1}\left\vert \left[ r,b\right) \right\vert _{\sigma }\frac{dr}{b-a}%
}=\sqrt{\int_{0}^{1}\int_{\left[ 0,1\right) ^{n-1}\times \left[ r,1\right)
}d\sigma \left( y\right) dr} \\
&=&\sqrt{\int_{\left[ 0,1\right) ^{n}}\int_{\left[ 0,y_{n}\right) }drd\sigma
\left( y\right) }=\sqrt{\int_{\left[ 0,1\right) ^{n}}y_{n}d\sigma \left(
y\right) }\ .
\end{eqnarray*}%
Noting that $\sqrt{\int_{\left[ 0,1\right) ^{n}}y_{n}d\sigma \left( y\right) 
}\leq \left\vert \left[ 0,1\right) ^{n}\right\vert _{\sigma }$, that the
same estimates hold for $y_{j}$ in place of $y_{n}$, and finally that there
are appropriate analogues of these estimates for all cubes $Q\in \mathcal{P}%
^{n}$ in place of $\left[ 0,1\right) ^{n}$, we see that 
\begin{equation*}
\mathfrak{FT}_{T}^{\left( 1\right) }\left( \sigma ,\omega \right) \leq
C_{m,0}\mathfrak{FT}\left( \sigma ,\omega \right) +\varepsilon \mathfrak{N}%
_{T}\left( \sigma ,\omega \right) \mathfrak{.}
\end{equation*}

Similarly, for each $i<n$ we can consider the monomial $y_{i}y_{n}$, and
obtain from the above argument with $y_{i}$ included in the integrand, that%
\begin{equation}
\left\Vert T_{\sigma }\left[ \mathbf{1}_{R_{r}}\left( y\right) y_{i}y_{n}%
\right] \right\Vert _{L^{2}\left( \omega \right) }\lesssim \sqrt{\int
\left\vert T_{\sigma }\left( \mathbf{1}_{\left[ 0,1\right) ^{n-1}\times %
\left[ r,r^{\ast }\right) }\left( y\right) y_{i}\right) \left( x\right)
\right\vert ^{2}d\omega \left( x\right) }+\mathfrak{FT}_{T}^{\left( 1\right)
}\ \left\vert \left[ 0,1\right) ^{n}\right\vert _{\sigma }\ .
\label{i less than n}
\end{equation}%
For the monomial $y_{n}^{2}$ we use the identity 
\begin{equation*}
\mathbf{1}_{\left[ 0,1\right) ^{n}}\left( y\right) y_{n}^{2}=\int_{0}^{1}%
\mathbf{1}_{\left[ 0,1\right) ^{n-1}\times \left[ r,1\right) }\left(
y\right) 2\left( y_{n}-r\right) dr,\ \ \ \ \ \text{for all }y\in \mathbb{R}%
^{n},
\end{equation*}%
to obtain%
\begin{equation*}
\left\Vert T_{\sigma }\left[ \mathbf{1}_{R_{r}}\left( y\right) y_{n}^{2}%
\right] \right\Vert _{L^{2}\left( \omega \right) }\lesssim \sqrt{\int
\left\vert T_{\sigma }\left( \mathbf{1}_{\left[ 0,1\right) ^{n-1}\times %
\left[ r,r^{\ast }\right) }\left( y\right) \left( y_{n}-r\right) \right)
\left( x\right) \right\vert ^{2}d\omega \left( x\right) }+\mathfrak{FT}%
_{T}^{\left( 1\right) }\ \left\vert \left[ 0,1\right) ^{n}\right\vert
_{\sigma }\ .
\end{equation*}%
Then in either case, integrating in $r$, using the simple probability
argument above, and finally using the appropriate analogues of these
estimates for all cubes $Q\in \mathcal{P}^{n}$ in place of $\left[
0,1\right) ^{n}$, we obtain 
\begin{equation*}
\mathfrak{FT}_{T}^{\left( 2\right) }\left( \sigma ,\omega \right) \leq
C_{m,1}\mathfrak{FT}_{T}^{\left( 1\right) }\left( \sigma ,\omega \right)
+\varepsilon \mathfrak{N}_{T}\left( \sigma ,\omega \right) \mathfrak{.}
\end{equation*}

Continuing in this way, using the identity%
\begin{equation*}
\mathbf{1}_{\left[ 0,1\right) ^{n}}\left( y\right) y^{\beta }=\int_{0}^{1}%
\mathbf{1}_{\left[ 0,1\right) ^{n-1}\times \left[ r,1\right) }\left(
y\right) \left( y_{1}^{\beta _{1}}...y_{n-1}^{\beta _{n-1}}\right) \left(
2\beta _{n}\left( y_{n}-r\right) ^{\beta _{n}-1}\right) dr,\ \ \ \ \ \text{%
for all }y\in \mathbb{R}^{n},
\end{equation*}%
yields the inequality%
\begin{equation*}
\mathfrak{FT}_{T}^{\left( \kappa \right) }\left( \sigma ,\omega \right) \leq
C_{m,\kappa -1}\mathfrak{FT}_{T}^{\left( \kappa -1\right) }\left( \sigma
,\omega \right) +\varepsilon \mathfrak{N}_{T}\left( \sigma ,\omega \right)
,\ \ \ \ \ \kappa \in \mathbb{N}.
\end{equation*}%
Iteration then gives%
\begin{eqnarray*}
\mathfrak{FT}_{T}^{\left( \kappa \right) }\left( \sigma ,\omega \right)
&\leq &\varepsilon \mathfrak{N}_{T}\left( \sigma ,\omega \right)
+C_{m,\kappa -1}\mathfrak{FT}_{T}^{\left( \kappa -1\right) }\left( \sigma
,\omega \right) \\
&\leq &\varepsilon \mathfrak{N}_{T}\left( \sigma ,\omega \right)
+C_{m,\kappa -1}\left\{ \varepsilon \mathfrak{N}_{T}\left( \sigma ,\omega
\right) +C_{m,\kappa -2}\mathfrak{FT}_{T}^{\left( \kappa -2\right) }\left(
\sigma ,\omega \right) \right\} \\
&&\vdots \\
&\leq &\varepsilon \left\{ 1+C_{m,\kappa -1}+C_{m,\kappa -1}C_{m,\kappa
-2}+...+C_{m,\kappa -1}C_{m,\kappa -2}...C_{m,0}\right\} \mathfrak{N}%
_{T}\left( \sigma ,\omega \right) \\
&&+\left\{ C_{m,\kappa -1}C_{m,\kappa -2}+...+C_{m,\kappa -1}C_{m,\kappa
-2}...C_{0m,}\right\} \mathfrak{FT}_{T}\left( \sigma ,\omega \right) \\
&=&\varepsilon A\left( \kappa ,\varepsilon \right) \mathfrak{N}_{T}\left(
\sigma ,\omega \right) +B\left( \kappa ,\varepsilon \right) \mathfrak{FT}%
_{T}\left( \sigma ,\omega \right) ,
\end{eqnarray*}%
where the constants $A\left( \kappa ,\varepsilon \right) $ and $B\left(
\kappa ,\varepsilon \right) $ are independent of the operator norm $%
\mathfrak{N}_{T}\left( \sigma ,\omega \right) $. Here we have taken $%
m\approx \log _{2}\frac{1}{\varepsilon }$. This completes the proof of
Theorem \ref{Tp control by T1}.
\end{proof}

We have already pointed out in dimension $n=1$, the equivalence of full
testing with the usual $1$-testing in the presence of one-tailed Muckenhoupt
conditions. In higher dimensions the same is true for at least doubling
measures. For this we use a quantitative expression of the fact that
doubling measures don't charge the boundaries of cubes \cite[see e.g. 8.6
(b) on page 40]{Ste2}.

\begin{lemma}
\label{boundary doubling}Suppose $\sigma $ is a doubling measure on $\mathbb{%
R}^{n}$ and that $Q\in \mathcal{P}^{n}$. Then for $0<\delta <1$ we have%
\begin{equation*}
\left\vert Q\setminus \left( 1-\delta \right) Q\right\vert _{\sigma }\leq 
\frac{C}{\ln \frac{1}{\delta }}\left\vert Q\right\vert _{\sigma }\ .
\end{equation*}
\end{lemma}

\begin{proof}
Let $\delta =2^{-m}$. Denote by $\mathfrak{C}^{\left( m\right) }\left(
Q\right) $ the set of $m^{th}$ generation dyadic children of $Q$ , so that
each $I\in \mathfrak{C}^{\left( m\right) }\left( Q\right) $ has side length $%
\ell \left( I\right) =2^{-m}\ell \left( Q\right) $, and define the
collections%
\begin{eqnarray*}
\mathfrak{G}^{\left( m\right) }\left( Q\right) &\equiv &\left\{ I\in 
\mathfrak{C}^{\left( m\right) }\left( Q\right) :I\subset Q\text{ and }%
\partial I\cap \partial Q\neq \emptyset \right\} , \\
\mathfrak{H}^{\left( m\right) }\left( Q\right) &\equiv &\left\{ I\in 
\mathfrak{C}^{\left( m\right) }\left( Q\right) :3I\subset Q\text{ and }%
\partial \left( 3I\right) \cap \partial Q\neq \emptyset \right\} .
\end{eqnarray*}%
Then 
\begin{equation*}
Q\setminus \left( 1-\delta \right) Q=\mathfrak{G}^{\left( m\right) }\left(
Q\right) \text{ and }\left( 1-\delta \right) Q=\overset{\cdot }{\dbigcup }%
_{k=2}^{m}\mathfrak{H}^{\left( k\right) }\left( Q\right) .
\end{equation*}%
From the doubling condition we have $\left\vert 3I\right\vert _{\sigma }\leq
D\left\vert I\right\vert _{\sigma }$ for all cubes $I$, and so 
\begin{eqnarray*}
\left\vert \mathfrak{H}^{\left( k\right) }\left( Q\right) \right\vert
_{\sigma } &=&\sum_{I\in \mathfrak{H}^{\left( k\right) }\left( Q\right)
}\left\vert I\right\vert _{\sigma }\geq \sum_{I\in \mathfrak{H}^{\left(
k\right) }\left( Q\right) }\frac{1}{D}\left\vert 3I\right\vert _{\sigma }=%
\frac{1}{D}\int \left( \sum_{I\in \mathfrak{H}^{\left( k\right) }\left(
Q\right) }\mathbf{1}_{3I}\right) d\sigma \\
&\geq &\frac{1}{D}\int \left( \sum_{I\in \mathfrak{G}^{\left( k\right)
}\left( Q\right) }\mathbf{1}_{I}\right) d\sigma =\frac{1}{D}\left\vert 
\mathfrak{G}^{\left( k\right) }\left( Q\right) \right\vert _{\sigma }\geq 
\frac{1}{D}\left\vert \mathfrak{G}^{\left( m\right) }\left( Q\right)
\right\vert _{\sigma }=\frac{1}{D}\left\vert Q\setminus \left( 1-\delta
\right) Q\right\vert _{\sigma }\ .
\end{eqnarray*}%
Thus we have%
\begin{equation*}
\left\vert Q\right\vert _{\sigma }\geq \sum_{k=2}^{m}\left\vert \mathfrak{H}%
^{\left( k\right) }\left( Q\right) \right\vert _{\sigma }\geq \frac{m-1}{D}%
\left\vert Q\setminus \left( 1-\delta \right) Q\right\vert _{\sigma }\ ,
\end{equation*}%
which proves the lemma.
\end{proof}

\begin{proposition}
\label{full control for doub}Suppose that $\sigma $ and $\omega $ are
locally finite positive Borel measures on $\mathbb{R}^{n}$, and that $\sigma 
$ is doubling. Then for $0<\varepsilon <1$ there is a positive constant $%
C\left( \varepsilon \right) $ such that%
\begin{equation*}
\mathfrak{FT}_{T}\left( \sigma ,\omega \right) \leq \mathfrak{T}_{T}\left(
\sigma ,\omega \right) +C\left( \varepsilon \right) \mathcal{A}_{2}^{\alpha
}\left( \sigma ,\omega \right) +\varepsilon \mathfrak{N}_{T}\left( \sigma
,\omega \right) \ .
\end{equation*}
\end{proposition}

\begin{proof}
Let $\delta >0$ be defined by the equation $\varepsilon =\frac{C}{\ln \frac{1%
}{\delta }}$, i.e. $\delta =e^{-\frac{C}{\varepsilon }}$. Then we write%
\begin{eqnarray*}
\int_{\mathbb{R}^{n}}\left\vert T_{\sigma }\mathbf{1}_{Q}\right\vert
^{2}d\omega &=&\int_{Q}\left\vert T_{\sigma }\mathbf{1}_{Q}\right\vert
^{2}d\omega +\int_{\mathbb{R}^{n}\setminus Q}\left\vert T_{\sigma }\mathbf{1}%
_{\left( 1-\delta \right) Q}+T_{\sigma }\mathbf{1}_{Q\setminus \left(
1-\delta \right) Q}\right\vert ^{2}d\omega \\
&\leq &\mathfrak{T}_{T}\left( \sigma ,\omega \right) ^{2}\left\vert
Q\right\vert _{\sigma }+2\int_{\mathbb{R}^{n}\setminus Q}\left\vert
T_{\sigma }\mathbf{1}_{\left( 1-\delta \right) Q}\right\vert ^{2}d\omega
+2\int_{\mathbb{R}^{n}\setminus Q}\left\vert T_{\sigma }\mathbf{1}%
_{Q\setminus \left( 1-\delta \right) Q}\right\vert ^{2}d\omega \\
&\leq &\mathfrak{T}_{T}\left( \sigma ,\omega \right) ^{2}\left\vert
Q\right\vert _{\sigma }+C\frac{1}{\delta }\mathcal{A}_{2}^{\alpha }\left(
\sigma ,\omega \right) \left\vert Q\right\vert _{\sigma }+2\mathfrak{N}%
_{T}^{2}\left( \sigma ,\omega \right) \left\vert Q\setminus \left( 1-\delta
\right) Q\right\vert _{\sigma }\ .
\end{eqnarray*}%
Now invoke Lemma \ref{boundary doubling} to obtain%
\begin{equation*}
\int_{\mathbb{R}^{n}}\left\vert T_{\sigma }\mathbf{1}_{Q}\right\vert
^{2}d\omega \leq \mathfrak{T}_{T}\left( \sigma ,\omega \right)
^{2}\left\vert Q\right\vert _{\sigma }+C\frac{1}{\delta }\mathcal{A}%
_{2}^{\alpha }\left( \sigma ,\omega \right) \left\vert Q\right\vert _{\sigma
}+\varepsilon \mathfrak{N}_{T}^{2}\left( \sigma ,\omega \right) \left\vert
Q\right\vert _{\sigma }\ ,
\end{equation*}%
with $\varepsilon =\frac{2C}{\ln \frac{1}{\delta }}$.
\end{proof}

In the sequel we will want to combine Theorem \ref{Tp control by T1} and
Proposition \ref{full control for doub} into the following single estimate.

\begin{corollary}
\label{full Tp control}Suppose that $\sigma $ and $\omega $ are locally
finite positive Borel measures on $\mathbb{R}^{n}$, and that $\sigma $ is
doubling. Then for $\kappa \in \mathbb{N}$ and $0<\varepsilon <1$, there is
a positive constant $C_{\kappa ,\varepsilon }$ such that%
\begin{equation*}
\mathfrak{FT}_{T}^{\left( \kappa \right) }\left( \sigma ,\omega \right) \leq
C_{\kappa ,\varepsilon }\left[ \mathfrak{T}_{T}\left( \sigma ,\omega \right)
+\mathcal{A}_{2}^{\alpha }\left( \sigma ,\omega \right) \right] +\varepsilon 
\mathfrak{N}_{T}\left( \sigma ,\omega \right) \ .
\end{equation*}
\end{corollary}

\section{Proof of the $Tp$ theorem with $BICT$ and doubling weights}

We will prove Theorem \ref{pivotal theorem} by adapting the beautiful \emph{%
pivotal argument} of Nazarov, Treil and Volberg in \cite{NTV4}, that uses
weighted Haar wavelets and random grids, to a weaker $\kappa ^{th}$-order
pivotal condition with Alpert wavelets and the Parallel Corona
decomposition, the latter being used to circumvent difficulties in
establishing the paraproduct decomposition using weighted Alpert wavelets.
More precisely, we will work in the one-grid world, where the Alpert wavelet
expansions for $f$ and $g$ in $L^{2}\left( \sigma \right) $ and $L^{2}\left(
\omega \right) $ respectively are taken with respect to a common grid $%
\mathcal{D}$, and follow the standard NTV argument for $T1$-type theorems
already in the literature (see e.g. \cite{NTV4}, the two part paper \cite%
{LaSaShUr3}\cite{Lac}, \cite{Hyt2} and \cite{SaShUr7}), i.e. using NTV
random grids $\mathcal{D}$ and goodness, but using pivotal conditions when
possible to avoid functional energy, and using the Parallel Corona and $%
\kappa $-Cube Testing and Bilinear Indicator/Cube Testing to avoid
paraproduct terms, which as observed earlier behave poorly with respect to
weighted Alpert wavelets of order greater than $1$. But first we extend the
scope of the Indicator/Cube Testing condition and the Bilinear
Indicator/Cube Testing property.

\subsection{Extending indicators to bounded functions}

It was observed in \cite{LaSaUr1} that the supremum over $\mathbf{1}_{E}$ in
the Indicator/Cube testing condition (\ref{def kth order testing}) can be
replaced with the logically larger supremum over an arbitrary function $h$
with $\left\vert h\right\vert \leq 1$. Here we extend the analogue of this
observation to hold for the Bilinear Indicator/Cube Testing constant $%
\mathcal{BICT}_{T^{\alpha }}\left( \sigma ,\omega \right) $.

\begin{lemma}
Let $\sigma $ and $\omega $ be positive locally finite Borel measures on $%
\mathbb{R}^{n}$, and let $T^{\alpha }$ be a standard $\alpha $-fractional
singular integral operator on $\mathbb{R}^{n}$. Then%
\begin{equation*}
\mathcal{BICT}_{T^{\alpha }}\left( \sigma ,\omega \right) \leq \sup_{Q\in 
\mathcal{P}^{n}}\sup_{\substack{ \left\Vert f\right\Vert _{L^{\infty }\left(
\sigma \right) }\leq 1  \\ \left\Vert g\right\Vert _{L^{\infty }\left(
\omega \right) }\leq 1}}\frac{1}{\sqrt{\left\vert Q\right\vert _{\sigma
}\left\vert Q\right\vert _{\omega }}}\left\vert \int_{Q}T_{\sigma }^{\alpha
}\left( \mathbf{1}_{Q}f\right) g\omega \right\vert \leq 4\ \mathcal{BICT}%
_{T^{\alpha }}\left( \sigma ,\omega \right) .
\end{equation*}
\end{lemma}

\begin{proof}
Given a cube $Q$ and a bounded function $f\in L^{\infty }\left( \sigma
\right) $, define%
\begin{eqnarray*}
h_{Q}\left[ f\right] \left( x\right) &\equiv &\left\{ 
\begin{array}{ccc}
\frac{\left\vert T_{\sigma }^{\alpha }\left( \mathbf{1}_{Q}f\right) \left(
x\right) \right\vert }{T_{\sigma }^{\alpha }\left( \mathbf{1}_{Q}f\right)
\left( x\right) } & \text{ if } & T_{\sigma }^{\alpha }\left( \mathbf{1}%
_{Q}f\right) \left( x\right) \neq 0 \\ 
0 & \text{ if } & T_{\sigma }^{\alpha }\left( \mathbf{1}_{Q}f\right) \left(
x\right) =0%
\end{array}%
\right. \\
&=&\mathbf{1}_{F_{+}\left[ f\right] }\left( x\right) -\mathbf{1}_{F_{-}\left[
f\right] }\left( x\right) ,
\end{eqnarray*}%
where the sets%
\begin{eqnarray*}
F_{+}\left[ f\right] &\equiv &\left\{ x\in Q:T_{\sigma }^{\alpha }\left( 
\mathbf{1}_{Q}f\right) \left( x\right) >0\right\} , \\
F_{-}\left[ f\right] &\equiv &\left\{ x\in Q:T_{\sigma }^{\alpha }\left( 
\mathbf{1}_{Q}f\right) \left( x\right) <0\right\} ,
\end{eqnarray*}%
both depend on $f$. Then we have%
\begin{eqnarray*}
&&\sup_{Q\in \mathcal{P}^{n}}\sup_{\substack{ \left\Vert f\right\Vert
_{L^{\infty }\left( \sigma \right) }\leq 1  \\ \left\Vert g\right\Vert
_{L^{\infty }\left( \omega \right) }\leq 1}}\frac{1}{\sqrt{\left\vert
Q\right\vert _{\sigma }\left\vert Q\right\vert _{\omega }}}\left\vert
\int_{Q}T_{\sigma }^{\alpha }\left( \mathbf{1}_{Q}f\right) gd\omega
\right\vert =\sup_{Q\in \mathcal{P}^{n}}\sup_{\left\Vert f\right\Vert
_{L^{\infty }\left( \sigma \right) }\leq 1}\frac{1}{\sqrt{\left\vert
Q\right\vert _{\sigma }\left\vert Q\right\vert _{\omega }}}%
\int_{Q}\left\vert T_{\sigma }^{\alpha }\left( \mathbf{1}_{Q}f\right)
\right\vert d\omega \\
&=&\sup_{Q\in \mathcal{P}^{n}}\sup_{\left\Vert f\right\Vert _{L^{\infty
}\left( \sigma \right) }\leq 1}\frac{1}{\sqrt{\left\vert Q\right\vert
_{\sigma }\left\vert Q\right\vert _{\omega }}}\int_{Q}T_{\sigma }^{\alpha
}\left( \mathbf{1}_{Q}f\right) h_{Q}\left[ f\right] d\omega \\
&=&\sup_{Q\in \mathcal{P}^{n}}\sup_{\left\Vert f\right\Vert _{L^{\infty
}\left( \sigma \right) }\leq 1}\frac{1}{\sqrt{\left\vert Q\right\vert
_{\sigma }\left\vert Q\right\vert _{\omega }}}\int_{Q}f\left( T_{\omega
}^{\alpha ,\ast }\left( \mathbf{1}_{F_{+}\left[ f\right] }-\mathbf{1}_{F_{-}%
\left[ f\right] }\right) \right) d\sigma \\
&\leq &\sup_{Q\in \mathcal{P}^{n}}\frac{1}{\sqrt{\left\vert Q\right\vert
_{\sigma }\left\vert Q\right\vert _{\omega }}}\sup_{\left\Vert f\right\Vert
_{L^{\infty }\left( \sigma \right) }\leq 1}\int_{Q}\left\vert T_{\omega
}^{\alpha ,\ast }\left( \mathbf{1}_{F_{+}\left[ f\right] }-\mathbf{1}_{F_{-}%
\left[ f\right] }\right) \right\vert d\sigma .
\end{eqnarray*}%
But now%
\begin{equation*}
\sup_{\left\Vert f\right\Vert _{L^{\infty }\left( \sigma \right) }\leq
1}\int_{Q}\left\vert T_{\omega }^{\alpha ,\ast }\left( \mathbf{1}_{F_{+}%
\left[ f\right] }-\mathbf{1}_{F_{-}\left[ f\right] }\right) \right\vert
d\sigma =\int_{Q}T_{\omega }^{\alpha ,\ast }\left( \mathbf{1}_{F_{+}\left[ f%
\right] }-\mathbf{1}_{F_{-}\left[ f\right] }\right) k_{Q}\left[ f\right]
d\sigma
\end{equation*}%
where%
\begin{eqnarray*}
k_{Q}\left[ f\right] \left( y\right) &\equiv &\left\{ 
\begin{array}{ccc}
\frac{\left\vert T_{\omega }^{\alpha ,\ast }\left( \mathbf{1}_{F_{+}\left[ f%
\right] }-\mathbf{1}_{F_{-}\left[ f\right] }\right) \left( y\right)
\right\vert }{T_{\omega }^{\alpha ,\ast }\left( \mathbf{1}_{F_{+}\left[ f%
\right] }-\mathbf{1}_{F_{-}\left[ f\right] }\right) \left( y\right) } & 
\text{ if } & T_{\omega }^{\alpha ,\ast }\left( \mathbf{1}_{F_{+}\left[ f%
\right] }-\mathbf{1}_{F_{-}\left[ f\right] }\right) \left( y\right) \neq 0
\\ 
0 & \text{ if } & T_{\omega }^{\alpha ,\ast }\left( \mathbf{1}_{F_{+}\left[ f%
\right] }-\mathbf{1}_{F_{-}\left[ f\right] }\right) \left( y\right) =0%
\end{array}%
\right. \\
&=&\mathbf{1}_{E_{+}\left[ f\right] }\left( y\right) -\mathbf{1}_{E_{-}\left[
f\right] }\left( y\right) ,
\end{eqnarray*}%
where the sets 
\begin{eqnarray*}
E_{+}\left[ f\right] &\equiv &\left\{ y\in Q:T_{\omega }^{\alpha ,\ast
}\left( \mathbf{1}_{F_{+}\left[ f\right] }-\mathbf{1}_{F_{-}\left[ f\right]
}\right) \left( y\right) >0\right\} , \\
E_{-}\left[ f\right] &\equiv &\left\{ y\in Q:T_{\omega }^{\alpha ,\ast
}\left( \mathbf{1}_{F_{+}\left[ f\right] }-\mathbf{1}_{F_{-}\left[ f\right]
}\right) \left( y\right) <0\right\} ,
\end{eqnarray*}%
also both depend on $f$. Thus we have shown that%
\begin{eqnarray*}
&&\sup_{Q\in \mathcal{P}^{n}}\sup_{\substack{ \left\Vert f\right\Vert
_{L^{\infty }\left( \sigma \right) }\leq 1  \\ \left\Vert g\right\Vert
_{L^{\infty }\left( \omega \right) }\leq 1}}\frac{1}{\sqrt{\left\vert
Q\right\vert _{\sigma }\left\vert Q\right\vert _{\omega }}}\left\vert
\int_{Q}T_{\sigma }^{\alpha }\left( \mathbf{1}_{Q}f\right) gd\omega
\right\vert \\
&\leq &\sup_{Q\in \mathcal{P}^{n}}\sup_{\left\Vert f\right\Vert _{L^{\infty
}\left( \sigma \right) }\leq 1}\frac{1}{\sqrt{\left\vert Q\right\vert
_{\sigma }\left\vert Q\right\vert _{\omega }}}\int_{Q}\left( \mathbf{1}%
_{E_{+}\left[ f\right] }-\mathbf{1}_{E_{-}\left[ f\right] }\right) T_{\omega
}^{\alpha ,\ast }\left( \mathbf{1}_{F_{+}\left[ f\right] }-\mathbf{1}_{F_{-}%
\left[ f\right] }\right) d\sigma \\
&\leq &4\sup_{Q\in \mathcal{P}^{n}}\sup_{E,F\subset Q}\frac{1}{\sqrt{%
\left\vert Q\right\vert _{\sigma }\left\vert Q\right\vert _{\omega }}}%
\left\vert \int_{E}T_{\omega }^{\alpha ,\ast }\left( \mathbf{1}_{F}\right)
d\sigma \right\vert =4\mathcal{BICT}_{T^{\alpha }}\left( \sigma ,\omega
\right) .
\end{eqnarray*}%
The converse inequality%
\begin{equation*}
\mathcal{BICT}_{T^{\alpha }}\left( \sigma ,\omega \right) \leq \sup_{Q\in 
\mathcal{P}^{n}}\sup_{\substack{ \left\Vert f\right\Vert _{L^{\infty }\left(
\sigma \right) }\leq 1  \\ \left\Vert g\right\Vert _{L^{\infty }\left(
\omega \right) }\leq 1}}\frac{1}{\sqrt{\left\vert Q\right\vert _{\sigma
}\left\vert Q\right\vert _{\omega }}}\left\vert \int_{Q}T_{\sigma }^{\alpha
}\left( \mathbf{1}_{Q}f\right) gd\omega \right\vert
\end{equation*}%
is trivial.
\end{proof}

\subsection{Initial steps}

The first step in the proof of Theorem \ref{pivotal theorem} is to expand an
inner product $\left\langle T_{\sigma }^{\alpha }f,g\right\rangle
_{L^{2}\left( \omega \right) }$ in weighted Alpert projections $%
\bigtriangleup _{I;\kappa _{1}}^{\sigma }f$ and $\bigtriangleup _{J;\kappa
_{2}}^{\omega }g$ asssociated with a fixed dyadic grid $\mathcal{D}$:%
\begin{equation}
\left\langle T_{\sigma }^{\alpha }f,g\right\rangle _{L^{2}\left( \omega
\right) }=\sum_{I,J\in \mathcal{D}}\left\langle T_{\sigma }^{\alpha
}\bigtriangleup _{I;\kappa _{1}}^{\sigma }f,\bigtriangleup _{J;\kappa
_{2}}^{\omega }g\right\rangle _{L^{2}\left( \omega \right) }\ .
\label{inner pdt sum}
\end{equation}%
We next wish to reduce the above sum to $\left( I,J\right) \in \mathcal{D}%
\times \mathcal{D}$ such that $I\subset I_{0}$ and $J\subset J_{0}$ where $%
I_{0}$ and $J_{0}$ are large cubes in $\mathcal{D}$, and for this we will
use, in a standard way, the testing conditions over polynomials of degree
less than $\kappa $. This reduced sum is then decomposed into many separate
sums according to the relative sizes of Calder\'{o}n-Zygmund stopping cubes,
i.e. first into the Parallel Corona decomposition, then further into $\func{%
Near}$, $\func{Disjoint}$ and $\func{Far}$ forms, and then finally according
to the locations and goodness of the intervals $I$ and $J$. Each of the
resulting forms are then controlled using widely different techniques.

A crucial tool from \cite{RaSaWi} is the estimate for $L^{2}\left( \omega
\right) $ norms of Alpert projections $\left\Vert \bigtriangleup _{J;\kappa
}^{\omega }T^{\alpha }\mu \right\Vert _{L^{2}\left( \omega \right) }^{2}$,
called the Monotonicity Lemma below (see \cite{LaWi} and also \cite{SaShUr7}%
), and which is improved by the extra vanishing moments of Alpert wavelets
to the following NTV type estimate, 
\begin{equation*}
\left\Vert \bigtriangleup _{J;\kappa }^{\omega }T^{\alpha }\mu \right\Vert
_{L^{2}\left( \omega \right) }^{2}\lesssim \left( \frac{\mathrm{P}_{\kappa
}^{\alpha }\left( J,\mu \right) }{\ell \left( J\right) ^{\kappa }}\right)
^{2}\sum_{\left\vert \beta \right\vert =\kappa -1}\left\Vert \left(
x-m_{J}^{\kappa }\right) ^{\beta }\right\Vert _{L^{2}\left( \mathbf{1}%
_{J}\omega \right) }^{2}\ ,
\end{equation*}%
which in turn can then be controlled by a $\kappa ^{th}$-order pivotal
condition, \textbf{weaker} than the usual pivotal condition with $\kappa =1$%
. The telescoping identities (\ref{telescoping}) reduce sums of consecutive
Alpert projections $\bigtriangleup _{I;\kappa }^{\mu }$ to\ differences of
projections $\mathbb{E}_{Q;\kappa }^{\mu }$ onto spaces of polynomials of
degree at most $\kappa -1$. Since by (\ref{analogue}), the sup norms of
these latter projections are controlled by Calder\'{o}n-Zygmund averages, we
are able to obtain an analogue of the Intertwining Proposition in \cite%
{SaShUr7}, which controls the $\func{Far}$ forms. The $\func{Near}$ forms
are controlled by the $\kappa $-Cube Testing conditions and Bilinear
Indicator/Cube Testing property.

Underlying all of this analysis however, is the powerful tool of Nazarov,
Treil and Volberg introduced in \cite{NTV1}, that restricts wavelet
expansions to $\limfunc{good}$ cubes, thus permitting the geometric decay
necessary to control off-diagonal terms in the presence of some appropriate
side condition - such as a pivotal or energy condition, which can be thought
of as a proof catalyst.

Before proceeding with the Parallel Corona decomposition and the subsequent
elements of the proof of Theorem \ref{pivotal theorem} in Subsection \ref%
{details} below, we give detailed analogues of the Monotonicity Lemma and
Intertwining Proposition in the setting of Alpert wavelets.

\subsection{The Monotonicity Lemma}

For $0\leq \alpha <n$ and $m\in \mathbb{R}_{+}$, we recall from (\ref{def
kappa Poisson}) the $m^{th}$-order fractional Poisson integral%
\begin{equation*}
\mathrm{P}_{m}^{\alpha }\left( J,\mu \right) \equiv \int_{\mathbb{R}^{n}}%
\frac{\left\vert J\right\vert ^{m}}{\left( \left\vert J\right\vert
+\left\vert y-c_{J}\right\vert \right) ^{m+n-\alpha }}d\mu \left( y\right) ,
\end{equation*}%
where $\mathrm{P}_{1}^{\alpha }\left( J,\mu \right) =\mathrm{P}^{\alpha
}\left( J,\mu \right) $ is the standard Poisson integral. The following
extension of the Lacey-Wick formulation \cite{LaWi} of the Monotonicity
Lemma to weighted Alpert wavelets is due to Rahm, Sawyer and Wick \cite%
{RaSaWi}. Since the proof in \cite{RaSaWi} is given only for dimension $n=1$%
, we include the straightforward extension to the higher dimensional
operators considered here.

\begin{lemma}[Monotonicity \protect\cite{RaSaWi}]
\label{mono}Let $0\leq \alpha <n$, and $\kappa _{1},\kappa _{2}\in \mathbb{N}
$ and $0<\delta <1$. Suppose that$\ I$ and $J$ are cubes in $\mathbb{R}^{n}$
such that $J\subset 2J\subset I$, and that $\mu $ is a signed measure on $%
\mathbb{R}^{n}$ supported outside $I$. Finally suppose that $T^{\alpha }$ is
a standard $\left( \kappa _{1}+\delta ,\kappa _{2}+\delta \right) $-smooth
fractional singular integral on $\mathbb{R}^{n}$ with kernel $K^{\alpha
}\left( x,y\right) =K_{y}^{\alpha }\left( x\right) $. Then%
\begin{equation}
\left\Vert \bigtriangleup _{J;\kappa }^{\omega }T^{\alpha }\mu \right\Vert
_{L^{2}\left( \omega \right) }^{2}\lesssim \Phi _{\kappa }^{\alpha }\left(
J,\mu \right) ^{2}+\Psi _{\kappa }^{\alpha }\left( J,\left\vert \mu
\right\vert \right) ^{2},  \label{estimate}
\end{equation}%
where for a measure $\nu $,%
\begin{eqnarray*}
&&\Phi _{\kappa }^{\alpha }\left( J,\nu \right) ^{2}\equiv \sum_{\left\vert
\beta \right\vert =\kappa }\left\vert \int \left( K_{y}^{\alpha }\right)
^{\left( \kappa \right) }\left( m_{J}^{\kappa }\right) d\nu \left( y\right)
\right\vert ^{2}\left\Vert \bigtriangleup _{J;\kappa }^{\omega }x^{\beta
}\right\Vert _{L^{2}\left( \omega \right) }^{2}\ , \\
&&\Psi _{\kappa }^{\alpha }\left( J,\left\vert \nu \right\vert \right)
^{2}\equiv \left( \frac{\mathrm{P}_{\kappa +\delta }^{\alpha }\left(
J,\left\vert \nu \right\vert \right) }{\left\vert J\right\vert ^{\frac{%
\kappa }{n}}}\right) ^{2}\left\Vert \left\vert x-m_{J}^{\kappa }\right\vert
^{\kappa }\right\Vert _{L^{2}\left( \mathbf{1}_{J}\omega \right) }^{2}\ , \\
&&\text{where }m_{J}^{\kappa }\in J\text{ satisfies }\left\Vert \left\vert
x-m_{J}^{\kappa }\right\vert ^{\kappa }\right\Vert _{L^{2}\left( \mathbf{1}%
_{J}\omega \right) }^{2}=\inf_{m\in J}\left\Vert \left\vert x-m\right\vert
^{\kappa }\right\Vert _{L^{2}\left( \mathbf{1}_{J}\omega \right) }^{2}.
\end{eqnarray*}
\end{lemma}

\begin{proof}[Proof of Lemma \protect\ref{mono}]
The proof is an easy adaptation of the one-dimensional proof in \cite{RaSaWi}%
, which was in turn adapted from the proofs in \cite{LaWi} and \cite{SaShUr7}%
, but using a $\kappa ^{th}$ order Taylor expansion instead of a first order
expansion on the kernel $\left( K_{y}^{\alpha }\right) \left( x\right)
=K^{\alpha }\left( x,y\right) $. Due to the importance of this lemma, as
explained above, we repeat the short argument.

Let $\left\{ h_{J;\kappa }^{\mu ,a}\right\} _{a\in \Gamma _{J,n,\kappa }}$
be an orthonormal basis of $L_{J;\kappa }^{2}\left( \mu \right) $ consisting
of Alpert functions as above. Now we use the $\left( \kappa +\delta \right) $%
-smooth Calder\'{o}n-Zygmund smoothness estimate (\ref{sizeandsmoothness'}),
together with Taylor's formula 
\begin{eqnarray*}
K_{y}^{\alpha }\left( x\right) &=&\limfunc{Tay}\left( K_{y}^{\alpha }\right)
\left( x,c\right) +\frac{1}{\kappa !}\sum_{\left\vert \beta \right\vert
=\kappa }\left( K_{y}^{\alpha }\right) ^{\left( \beta \right) }\left( \theta
\left( x,c\right) \right) \left( x-c\right) ^{\beta }; \\
\limfunc{Tay}\left( K_{y}^{\alpha }\right) \left( x,c\right) &\equiv
&K_{y}^{\alpha }\left( c\right) +\left[ \left( x-c\right) \cdot \nabla %
\right] K_{y}^{\alpha }\left( c\right) +...+\frac{1}{\left( \kappa -1\right)
!}\left[ \left( x-c\right) \cdot \nabla \right] ^{\kappa -1}K_{y}^{\alpha
}\left( c\right) ,
\end{eqnarray*}%
and the vanishing means of the vector of Alpert functions $h_{J;\kappa
}^{\omega }=\left\{ h_{J;\kappa }^{\omega ,a}\right\} _{a\in \Gamma
_{J,n,\kappa }}$, to obtain 
\begin{eqnarray*}
&&\left\langle T^{\alpha }\mu ,h_{J;\kappa }^{\omega }\right\rangle
_{L^{2}\left( \omega \right) }=\int \left\{ \int K^{\alpha }\left(
x,y\right) h_{J;\kappa }^{\omega }\left( x\right) d\omega \left( x\right)
\right\} d\mu \left( y\right) =\int \left\langle K_{y}^{\alpha },h_{J;\kappa
}^{\omega }\right\rangle _{L^{2}\left( \omega \right) }d\mu \left( y\right)
\\
&=&\int \left\langle K_{y}^{\alpha }\left( x\right) -\limfunc{Tay}\left(
K_{y}^{\alpha }\right) \left( x,m_{J}^{\kappa }\right) ,h_{J;\kappa
}^{\omega }\left( x\right) \right\rangle _{L^{2}\left( \omega \right) }d\mu
\left( y\right) \\
&=&\int \left\langle \frac{1}{\kappa !}\sum_{\left\vert \beta \right\vert
=\kappa }\left( K_{y}^{\alpha }\right) ^{\left( \beta \right) }\left( \theta
\left( x,m_{J}^{\kappa }\right) \right) \left( x-m_{J}^{\kappa }\right)
^{\beta },h_{J;\kappa }^{\omega }\left( x\right) \right\rangle _{L^{2}\left(
\omega \right) }d\mu \left( y\right) \ \ \ \ \ \text{(some }\theta \left(
x,m_{J}^{\kappa }\right) \in J\text{) } \\
&=&\sum_{\left\vert \beta \right\vert =\kappa }\left\langle \left[ \int 
\frac{1}{\kappa !}\sum_{\left\vert \beta \right\vert =\kappa }\left(
K_{y}^{\alpha }\right) ^{\left( \beta \right) }\left( m_{J}^{\kappa }\right)
d\mu \left( y\right) \right] \left( x-m_{J}^{\kappa }\right) ^{\beta
},h_{J;\kappa }^{\omega }\right\rangle _{L^{2}\left( \omega \right) } \\
&&+\sum_{\left\vert \beta \right\vert =\kappa }\left\langle \left[ \int 
\frac{1}{\kappa !}\left[ \sum_{\left\vert \beta \right\vert =\kappa }\left(
K_{y}^{\alpha }\right) ^{\left( \beta \right) }\left( \theta \left(
x,m_{J}^{\kappa }\right) \right) -\sum_{\left\vert \beta \right\vert =\kappa
}\left( K_{y}^{\alpha }\right) ^{\left( \beta \right) }\left( m_{J}^{\kappa
}\right) \right] d\mu \left( y\right) \right] \left( x-m_{J}^{\kappa
}\right) ^{\beta },h_{J;\kappa }^{\omega }\right\rangle _{L^{2}\left( \omega
\right) }\ .
\end{eqnarray*}%
Then using that $\int \left( K_{y}^{\alpha }\right) ^{\left( \beta \right)
}\left( m_{J}^{\kappa }\right) d\mu \left( y\right) $ is independent of $%
x\in J$, and that $\left\langle \left( x-m_{J}^{\kappa }\right) ^{\beta
},h_{J;\kappa }^{\omega }\right\rangle _{L^{2}\left( \omega \right)
}=\left\langle x^{\beta },h_{J;\kappa }^{\omega }\right\rangle _{L^{2}\left(
\omega \right) }$ by moment vanishing of the Alpert wavelets, we can
continue with 
\begin{eqnarray*}
&&\left\langle T^{\alpha }\mu ,h_{J;\kappa }^{\omega }\right\rangle
_{L^{2}\left( \omega \right) }=\left[ \int \frac{1}{\kappa !}%
\sum_{\left\vert \beta \right\vert =\kappa }\left( K_{y}^{\alpha }\right)
^{\left( \beta \right) }\left( m_{J}^{\kappa }\right) d\mu \left( y\right) %
\right] \cdot \left\langle x^{\beta },h_{J;\kappa }^{\omega }\right\rangle
_{L^{2}\left( \omega \right) } \\
&&\ \ \ \ \ +\frac{1}{\kappa !}\sum_{\left\vert \beta \right\vert =\kappa
}\left\langle \left[ \int \left[ \sum_{\left\vert \beta \right\vert =\kappa
}\left( K_{y}^{\alpha }\right) ^{\left( \beta \right) }\left( \theta \left(
x,m_{J}^{\kappa }\right) \right) -\sum_{\left\vert \beta \right\vert =\kappa
}\left( K_{y}^{\alpha }\right) ^{\left( \beta \right) }\left( m_{J}^{\kappa
}\right) \right] d\mu \left( y\right) \right] \left( x-m_{J}^{\kappa
}\right) ^{\beta },h_{J;\kappa }^{\omega }\right\rangle _{L^{2}\left( \omega
\right) }\ .
\end{eqnarray*}%
Hence%
\begin{eqnarray*}
&&\left\vert \left\langle T^{\alpha }\mu ,h_{J;\kappa }^{\omega
}\right\rangle _{L^{2}\left( \omega \right) }-\left[ \int \frac{1}{\kappa !}%
\sum_{\left\vert \beta \right\vert =\kappa }\left( K_{y}^{\alpha }\right)
^{\left( \beta \right) }\left( m_{J}^{\kappa }\right) d\mu \left( y\right) %
\right] \cdot \left\langle x^{\beta },h_{J;\kappa }^{\omega }\right\rangle
_{L^{2}\left( \omega \right) }\right\vert \\
&\leq &\frac{1}{\kappa !}\sum_{\left\vert \beta \right\vert =\kappa
}\left\vert \left\langle \left[ \int \sup_{\theta \in J}\left\vert \left(
K_{y}^{\alpha }\right) ^{\left( \beta \right) }\left( \theta \right) -\left(
K_{y}^{\alpha }\right) ^{\left( \beta \right) }\left( m_{J}^{\kappa }\right)
\right\vert d\left\vert \mu \right\vert \left( y\right) \right] \left\vert
x-m_{J}^{\kappa }\right\vert ^{\kappa },\left\vert h_{J;\kappa }^{\omega
}\right\vert \right\rangle _{L^{2}\left( \omega \right) }\right\vert \\
&\lesssim &C_{CZ}\frac{\mathrm{P}_{\kappa +\delta }^{\alpha }\left(
J,\left\vert \mu \right\vert \right) }{\left\vert J\right\vert ^{\kappa }}%
\left\Vert \left\vert x-m_{J}^{\kappa }\right\vert ^{\kappa }\right\Vert
_{L^{2}\left( \mathbf{1}_{J}\omega \right) }
\end{eqnarray*}%
where in the last line we have used%
\begin{eqnarray*}
&&\int \sup_{\theta \in J}\left\vert \left( K_{y}^{\alpha }\right) ^{\left(
\beta \right) }\left( \theta \right) -\left( K_{y}^{\alpha }\right) ^{\left(
\beta \right) }\left( m_{J}^{\kappa }\right) \right\vert d\left\vert \mu
\right\vert \left( y\right) \\
&\lesssim &C_{CZ}\int \left( \frac{\left\vert J\right\vert }{\left\vert
y-c_{J}\right\vert }\right) ^{\delta }\frac{d\left\vert \mu \right\vert
\left( y\right) }{\left\vert y-c_{J}\right\vert ^{\kappa +1-\alpha }}=C_{CZ}%
\frac{\mathrm{P}_{\kappa +\delta }^{\alpha }\left( J,\left\vert \mu
\right\vert \right) }{\left\vert J\right\vert ^{\kappa }}.
\end{eqnarray*}

Thus with $\mathbf{v}_{J}^{\beta }=\frac{1}{\kappa !}\int \left(
K_{y}^{\alpha }\right) ^{\left( \beta \right) }\left( m_{J}^{\kappa }\right)
d\mu \left( y\right) $, and noting that the functions $\left\{ \mathbf{v}%
_{J}^{\beta }\cdot h_{J;\kappa }^{\omega ,a}\right\} _{a\in \Gamma
_{J,n,\kappa }}$ are orthonormal in $a\in \Gamma _{J,n,\kappa }$ for each $%
\beta $ and $J$, we have%
\begin{equation*}
\left\vert \mathbf{v}_{J}^{\beta }\cdot \left\langle x^{\beta },h_{J;\kappa
}^{\omega }\right\rangle _{L^{2}\left( \omega \right) }\right\vert
^{2}=\sum_{a\in \Gamma _{J,n,\kappa }}\left\vert \left\langle x^{\beta },%
\mathbf{v}_{J}^{\beta }\cdot h_{J;\kappa }^{\omega ,a}\right\rangle
_{L^{2}\left( \omega \right) }\right\vert ^{2}=\left\Vert \bigtriangleup
_{J;\kappa }^{\omega }\mathbf{v}_{J}^{\beta }x^{\beta }\right\Vert
_{L^{2}\left( \omega \right) }^{2}=\left\vert \mathbf{v}_{J}^{\beta
}\right\vert ^{2}\left\Vert \bigtriangleup _{J;\kappa }^{\omega }x^{\beta
}\right\Vert _{L^{2}\left( \omega \right) }^{2}\ ,
\end{equation*}%
and hence%
\begin{eqnarray*}
\left\Vert \bigtriangleup _{J;\kappa }^{\omega }T^{\alpha }\mu \right\Vert
_{L^{2}\left( \omega \right) }^{2} &=&\left\vert \left\langle T^{\alpha }\mu
,h_{J;\kappa }^{\omega }\right\rangle _{L^{2}\left( \omega \right)
}\right\vert ^{2} \\
&=&\sum_{\left\vert \beta \right\vert =\kappa }\left\vert \mathbf{v}%
_{J}^{\beta }\right\vert ^{2}\left\Vert \bigtriangleup _{J;\kappa }^{\omega
}x^{\kappa }\right\Vert _{L^{2}\left( \omega \right) }^{2}+O\left( \frac{%
\mathrm{P}_{\kappa +\delta }^{\alpha }\left( J,\left\vert \mu \right\vert
\right) }{\left\vert J\right\vert ^{\frac{\kappa }{n}}}\right)
^{2}\left\Vert \left\vert x-m_{J}^{\kappa }\right\vert ^{\kappa }\right\Vert
_{L^{2}\left( \mathbf{1}_{J}\omega \right) }.
\end{eqnarray*}

Thus we conclude that%
\begin{eqnarray*}
\left\Vert \bigtriangleup _{J;\kappa }^{\omega }T^{\alpha }\mu \right\Vert
_{L^{2}\left( \omega \right) }^{2} &\leq &C_{1}\sum_{\left\vert \beta
\right\vert =\kappa }\left\vert \left\vert \frac{1}{\kappa !}\int \left(
K_{y}^{\alpha }\right) ^{\left( \beta \right) }\left( m_{J}\right) d\mu
\left( y\right) \right\vert ^{2}\right\vert ^{2}\left\Vert \bigtriangleup
_{J;\kappa }^{\omega }x^{\kappa }\right\Vert _{L^{2}\left( \omega \right)
}^{2} \\
&&+C_{2}\left( \frac{\mathrm{P}_{\kappa +\delta }^{\alpha }\left(
J,\left\vert \mu \right\vert \right) }{\left\vert J\right\vert ^{\frac{%
\kappa }{n}}}\right) ^{2}\left\Vert \left\vert x-m_{J}^{\kappa }\right\vert
^{\kappa }\right\Vert _{L^{2}\left( \mathbf{1}_{J}\omega \right) }^{2}\ ,
\end{eqnarray*}%
where%
\begin{equation*}
\sum_{\left\vert \beta \right\vert =\kappa }\left\vert \frac{1}{\kappa !}%
\int \left( K_{y}^{\alpha }\right) ^{\left( \beta \right) }\left(
m_{J}\right) d\mu \left( y\right) \right\vert ^{2}\lesssim \left( \frac{%
\mathrm{P}_{\kappa }^{\alpha }\left( J,\left\vert \mu \right\vert \right) }{%
\left\vert J\right\vert ^{\frac{\kappa }{n}}}\right) ^{2}.
\end{equation*}
\end{proof}

The following Energy Lemma follows from the above Monotonicity Lemma in a
standard way, see e.g. \cite{SaShUr7}. Given a subset $\mathcal{J}\subset 
\mathcal{D}$, define the projection $\mathsf{P}_{\mathcal{J}}^{\omega
}\equiv \sum_{J^{\prime }\in \mathcal{J}}\bigtriangleup _{J^{\prime };\kappa
}^{\omega }$, and given a cube $J\in \mathcal{D}$, define the projection $%
\mathsf{P}_{J}^{\omega }\equiv \sum_{J^{\prime }\in \mathcal{D}:\ J^{\prime
}\subset J}\bigtriangleup _{J^{\prime };\kappa }^{\omega }$.

\begin{lemma}[\textbf{Energy Lemma}]
\label{ener}Fix $\kappa \geq 1$. Let $J\ $be a cube in $\mathcal{D}$. Let $%
\Psi _{J}$ be an $L^{2}\left( \omega \right) $ function supported in $J$
with vanishing $\omega $-means up to order less than $\kappa $, and let $%
\mathcal{J}\subset \mathcal{D}$ be such that $J^{\prime }\subset J$ for
every $J^{\prime }\in \mathcal{J}$. Let $\nu $ be a positive measure
supported in $\mathbb{R}^{n}\setminus \gamma J$ with $\gamma >1$, and for
each $J^{\prime }\in \mathcal{J}$, let $d\nu _{J^{\prime }}=\varphi
_{J^{\prime }}d\nu $ with $\left\vert \varphi _{J^{\prime }}\right\vert \leq
1$. Let $T^{\alpha }$ be a standard $\alpha $-fractional singular integral
operator with $0\leq \alpha <n$. Then we have%
\begin{eqnarray*}
&&\left\vert \sum_{J^{\prime }\in \mathcal{J}}\left\langle T^{\alpha }\left(
\nu _{J^{\prime }}\right) ,\bigtriangleup _{J^{\prime };\kappa }^{\omega
}\Psi _{J}\right\rangle _{\omega }\right\vert \lesssim C_{\gamma
}\sum_{J^{\prime }\in \mathcal{J}}\Phi _{\kappa }^{\alpha }\left( J^{\prime
},\nu \right) \left\Vert \bigtriangleup _{J^{\prime };\kappa }^{\omega }\Psi
_{J}\right\Vert _{L^{2}\left( \mu \right) } \\
&\lesssim &C_{\gamma }\sqrt{\sum_{J^{\prime }\in \mathcal{J}}\Phi _{\kappa
}^{\alpha }\left( J^{\prime },\nu \right) ^{2}}\sqrt{\sum_{J^{\prime }\in 
\mathcal{J}}\left\Vert \bigtriangleup _{J^{\prime };\kappa }^{\omega }\Psi
_{J}\right\Vert _{L^{2}\left( \mu \right) }^{2}} \\
&\leq &C_{\gamma }\left( \frac{\mathrm{P}_{\kappa }^{\alpha }\left( J,\nu
\right) }{\left\vert J\right\vert ^{\frac{\kappa }{n}}}\left\Vert \mathsf{P}%
_{\mathcal{J}}^{\omega }x\right\Vert _{L^{2}\left( \omega \right) }+\frac{%
\mathrm{P}_{\kappa +\delta }^{\alpha }\left( J,\nu \right) }{\left\vert
J\right\vert ^{\frac{\kappa }{n}}}\left\Vert \left\vert x-m_{J}^{\kappa
}\right\vert ^{\kappa }\right\Vert _{L^{2}\left( \mathbf{1}_{J}\omega
\right) }\right) \left\Vert \mathsf{P}_{\mathcal{J}}^{\omega }\Psi
_{J}\right\Vert _{L^{2}\left( \mu \right) },
\end{eqnarray*}%
and in particular the `energy' estimate%
\begin{equation*}
\left\vert \left\langle T^{\alpha }\varphi \nu ,\Psi _{J}\right\rangle
_{\omega }\right\vert \lesssim C_{\gamma }\left( \frac{\mathrm{P}_{\kappa
}^{\alpha }\left( J,\nu \right) }{\left\vert J\right\vert ^{\frac{\kappa }{n}%
}}\left\Vert \mathsf{P}_{J}^{\omega }x\right\Vert _{L^{2}\left( \omega
\right) }+\frac{\mathrm{P}_{\kappa +\delta }^{\alpha }\left( J,\nu \right) }{%
\left\vert J\right\vert ^{\frac{\kappa }{n}}}\left\Vert \left\vert
x-m_{J}^{\kappa }\right\vert ^{\kappa }\right\Vert _{L^{2}\left( \mathbf{1}%
_{J}\omega \right) }\right) \left\Vert \sum_{J^{\prime }\subset
J}\bigtriangleup _{J^{\prime };\kappa }^{\omega }\Psi _{J}\right\Vert
_{L^{2}\left( \omega \right) }\ ,
\end{equation*}%
where $\left\Vert \sum_{J^{\prime }\subset J}\bigtriangleup _{J^{\prime
};\kappa }^{\omega }\Psi _{J}\right\Vert _{L^{2}\left( \omega \right)
}\lesssim \left\Vert \Psi _{J}\right\Vert _{L^{2}\left( \omega \right) }$,
and the `pivotal' bound%
\begin{equation*}
\left\vert \left\langle T^{\alpha }\left( \varphi \nu \right) ,\Psi
_{J}\right\rangle _{L^{2}\left( \omega \right) }\right\vert \lesssim
C_{\gamma }\mathrm{P}_{k}^{\alpha }\left( J,\nu \right) \sqrt{\left\vert
J\right\vert _{\omega }}\left\Vert \Psi _{J}\right\Vert _{L^{2}\left( \omega
\right) }\ ,
\end{equation*}%
for any function $\varphi $ with $\left\vert \varphi \right\vert \leq 1$.
\end{lemma}

\subsubsection{Comparison of the $k^{th}$-order pivotal constant and the
usual pivotal constant}

As in \cite{RaSaWi}, where the corresponding estimate for $k^{th}$-order
energy constants was obtained, we clearly we have the inequality%
\begin{eqnarray*}
\mathrm{P}_{k}^{\alpha }\left( J,\mathbf{1}_{I}\sigma \right) &=&\int_{%
\mathbb{R}^{n}}\frac{\left\vert J\right\vert ^{k}}{\left( \ell \left(
J\right) +\left\vert y-c_{J}\right\vert \right) ^{k+n-\alpha }}d\sigma
\left( y\right) \\
&=&\int_{\mathbb{R}}\left( \frac{\left\vert J\right\vert }{\ell \left(
J\right) +\left\vert y-c_{J}\right\vert }\right) ^{k-\ell }\frac{\left\vert
J\right\vert ^{\ell }}{\left( \ell \left( J\right) +\left\vert
y-c_{J}\right\vert \right) ^{\ell +n-\alpha }}d\sigma \left( y\right) \\
&\leq &\int_{\mathbb{R}}\frac{\left\vert J\right\vert ^{\ell }}{\left( \ell
\left( J\right) +\left\vert y-c_{J}\right\vert \right) ^{\ell +n-\alpha }}%
d\sigma \left( y\right) =\mathrm{P}_{\ell }^{\alpha }\left( J,\mathbf{1}%
_{I}\sigma \right) ,
\end{eqnarray*}%
for $1\leq \ell \leq k$, and as a consequence, we obtain the decrease of the
pivotal constants $\mathcal{V}_{2}^{\alpha ,k}$ in $k$: 
\begin{equation*}
\mathcal{V}_{2}^{\alpha ,k}\leq \mathcal{V}_{2}^{\alpha ,\ell },\text{\ \ \
\ \ for }1\leq \ell \leq k.
\end{equation*}

\subsection{The Intertwining Proposition}

Here we prove the Intertwining Proposition of \cite[Proposition 9.4 on page
123]{SaShUr7} by appealing to the $\kappa ^{th}$\emph{-order pivotal
condition} rather than functional energy, and by using instead of the
Indicator/Cube Testing conditions (\ref{def kth order testing}), the weaker $%
\kappa $\emph{-Cube Testing }conditions (\ref{def Kappa polynomial}) similar
to those introduced in \cite{RaSaWi}:%
\begin{eqnarray}
\left( \mathfrak{T}_{T^{\alpha }}^{\left( \kappa _{1}\right) }\left( \sigma
,\omega \right) \right) ^{2} &\equiv &\sup_{Q\in \mathcal{P}^{n}}\max_{0\leq
\left\vert \beta \right\vert <\kappa _{1}}\frac{1}{\left\vert Q\right\vert
_{\sigma }}\int_{Q}\left\vert T_{\sigma }^{\alpha }\left( \mathbf{1}%
_{Q}m_{Q}^{\beta }\right) \right\vert ^{2}\omega <\infty ,
\label{kappa testing} \\
\left( \mathfrak{T}_{\left( T^{\alpha }\right) ^{\ast }}^{\left( \kappa
_{2}\right) }\left( \omega ,\sigma \right) \right) ^{2} &\equiv &\sup_{Q\in 
\mathcal{P}^{n}}\max_{0\leq \left\vert \beta \right\vert <\kappa _{2}}\frac{1%
}{\left\vert Q\right\vert _{\omega }}\int_{Q}\left\vert \left( T_{\sigma
}^{\alpha }\right) ^{\ast }\left( \mathbf{1}_{Q}m_{Q}^{\beta }\right)
\right\vert ^{2}\sigma <\infty ,  \notag
\end{eqnarray}%
with $m_{Q}^{\beta }\left( x\right) \equiv \left( \frac{x-c_{Q}}{\frac{\sqrt{%
n}}{2}\ell \left( Q\right) }\right) ^{\beta }$ for any cube $Q$ and
multiindex $\beta $, where $c_{Q}$ is the center of the cube $Q$. (The
factor $\frac{\sqrt{n}}{2}$in the denominator ensures that $m_{Q}^{\beta
}\in \left( \mathcal{P}_{\kappa }^{Q}\right) _{\limfunc{norm}}$ has supremum
norm $1$ on $Q$.) In this way we will avoid using the one-tailed Muckenhoupt
conditions, relying instead on only the simpler classical condition $%
A_{2}^{\alpha }$, while requiring the $\kappa $-Cube Testing condition and a
certain weak boundedness property. Later on we will use the one-tailed
Muckenhoupt conditions to both eliminate the weak boundedness property and
reduce $\kappa $-Cube testing to the usual testing over indicators.

\subsubsection{Three NTV estimates}

But first, we recall three estimates of Nazarov, Treil and Volberg \cite%
{NTV4}, in a form taken from \cite[Lemmas 7.1 and 7.2 on page 101]{SaShUr7},
where the `one-tailed' Muckenhoupt constants are not needed, only the
classical Muckenhoupt constant $A_{2}^{\alpha }$. The weak boundedness
constant $\mathcal{WBP}_{T^{\alpha }}^{\left( \kappa _{1},\kappa _{2}\right)
}\left( \sigma ,\omega \right) $ appearing in estimate (\ref{delta near})
below is,%
\begin{equation}
\mathcal{WBP}_{T^{\alpha }}^{\left( \kappa _{1},\kappa _{2}\right) }\left(
\sigma ,\omega \right) =\sup_{\mathcal{D}\in \Omega }\sup_{\substack{ %
Q,Q^{\prime }\in \mathcal{D}  \\ Q\subset 3Q^{\prime }\setminus Q^{\prime }%
\text{ or }Q^{\prime }\subset 3Q\setminus Q}}\frac{1}{\sqrt{\left\vert
Q\right\vert _{\sigma }\left\vert Q^{\prime }\right\vert _{\omega }}}\sup 
_{\substack{ f\in \left( \mathcal{P}_{Q}^{\kappa _{1}}\right) _{\limfunc{norm%
}}  \\ g\in \left( \mathcal{P}_{Q}^{\kappa _{2}}\right) _{\limfunc{norm}}}}%
\left\vert \int_{Q^{\prime }}T_{\sigma }^{\alpha }\left( \mathbf{1}%
_{Q}f\right) \ gd\omega \right\vert <\infty ,  \label{kappa WBP}
\end{equation}%
where the space $\left( \mathcal{P}_{\kappa }^{Q}\right) _{\limfunc{norm}}$
of $Q$-normalized polynomials of degree less than $\kappa $, is defined in
Definition \ref{def Q norm} above. Note that this notion of weak
boundedness, which unlike the Bilinear Indicator/Cube Testing property,
involves only pairs of \emph{disjoint} cubes. Finally, we need the concept
of $\left( \mathbf{r},\varepsilon \right) $-goodness introduced first in 
\cite{NTV1}, and used later in \cite{NTV3} and \cite{NTV4}, and then in
virtually every paper on the subject thereafter.

\begin{definition}
Let $\mathcal{D}$ be a dyadic grid. Given $\mathbf{r}\in \mathbb{N}$ and $%
0<\varepsilon <1$, called goodness parameters, a cube $Q\in \mathcal{D}$ is
said to be $\left( \mathbf{r},\varepsilon \right) $-$\limfunc{bad}$ if there
is a supercube $I\supset Q$ with $\ell \left( I\right) \geq 2^{\mathbf{r}%
}\ell \left( Q\right) $ that satisfies%
\begin{equation*}
\limfunc{dist}\left( Q,\partial I\right) <2\sqrt{n}\left\vert Q\right\vert
^{\varepsilon }\left\vert I\right\vert ^{1-\varepsilon }.
\end{equation*}%
Otherwise $Q$ is said to be $\left( \mathbf{r},\varepsilon \right) $-$%
\limfunc{good}$. The collection of $\left( \mathbf{r},\varepsilon \right) $-$%
\limfunc{good}$ cubes in $\mathcal{D}$ is denoted $\mathcal{D}^{\limfunc{good%
}}$. Finally, a function $f\in L^{2}\left( \mu \right) $ is said to be $%
\limfunc{good}$ if $f=\sum_{I\in \mathcal{D}^{\limfunc{good}}}\bigtriangleup
_{I;\kappa }^{\mu }f$.
\end{definition}

It is shown in \cite{NTV1}, \cite{NTV3} and \cite{NTV4} for the two-grid
world, and in \cite[Section 4]{HyPeTrVo} for the one-grid world, that in
order to prove a two weight testing theorem, it suffices to obtain estimates
for good functions, uniformly over all dyadic grids, provided $\mathbf{r}\in 
\mathbb{N}$ is chosen large enough depending on the choice of $\varepsilon $
satisfying $0<\varepsilon <1$. We assume this reduction is in force for an
appropriate $\varepsilon >0$ from now on.

\begin{lemma}
\label{standard delta}Suppose $T^{\alpha }$ is a standard fractional
singular integral with $0\leq \alpha <n$, and that all of the cubes $I,J\in 
\mathcal{D}$ below are $\left( \mathbf{r},\varepsilon \right) $-$\limfunc{%
good}$ with goodness parameters $\varepsilon $ and $\mathbf{r}$. Fix $\kappa
_{1},\kappa _{2}\geq 1$ and a positive integer $\mathbf{\rho }>\mathbf{r}$.
For $f\in L^{2}\left( \sigma \right) $ and $g\in L^{2}\left( \omega \right) $
we have%
\begin{eqnarray}
&&\sum_{\substack{ I,J\in \mathcal{D}  \\ 2^{-\mathbf{\rho }}\ell \left(
I\right) \leq \ell \left( J\right) \leq \ell \left( I\right) }}\left\vert
\left\langle T_{\sigma }^{\alpha }\left( \bigtriangleup _{I;\kappa
_{1}}^{\sigma }f\right) ,\bigtriangleup _{J;\kappa _{2}}^{\omega
}g\right\rangle _{\omega }\right\vert  \label{delta near} \\
&\lesssim &\left( \mathfrak{T}_{T^{\alpha }}^{\left( \kappa _{1}\right)
}\left( \sigma ,\omega \right) +\mathfrak{T}_{\left( T^{\alpha }\right)
^{\ast }}^{\left( \kappa _{2}\right) }\left( \omega ,\sigma \right) +%
\mathcal{WBP}_{T^{\alpha }}^{\left( \kappa _{1},\kappa _{2}\right) }\left(
\sigma ,\omega \right) +\sqrt{A_{2}^{\alpha }\left( \sigma ,\omega \right) }%
\right) \left\Vert f\right\Vert _{L^{2}\left( \sigma \right) }\left\Vert
g\right\Vert _{L^{2}\left( \omega \right) }  \notag
\end{eqnarray}%
and 
\begin{equation}
\sum_{\substack{ \left( I,J\right) \in \mathcal{D}^{\sigma }\times \mathcal{D%
}^{\omega }  \\ I\cap J=\emptyset \text{ and }\frac{\ell \left( J\right) }{%
\ell \left( I\right) }\notin \left[ 2^{-\mathbf{\rho }},2^{\mathbf{\rho }}%
\right] }}\left\vert \left\langle T_{\sigma }^{\alpha }\left( \bigtriangleup
_{I;\kappa _{1}}^{\sigma }f\right) ,\bigtriangleup _{J;\kappa _{2}}^{\omega
}g\right\rangle _{\omega }\right\vert \lesssim \sqrt{A_{2}^{\alpha }\left(
\sigma ,\omega \right) }\left\Vert f\right\Vert _{L^{2}\left( \sigma \right)
}\left\Vert g\right\Vert _{L^{2}\left( \omega \right) }.  \label{delta far}
\end{equation}
\end{lemma}

\begin{description}
\item[Justification] Using weak boundedness together with the $L^{\infty }$
control $\left\Vert \mathbb{E}_{I}^{\sigma ,\kappa }f\right\Vert
_{L_{I}^{\infty }\left( \sigma \right) }\lesssim E_{I}^{\sigma }\left\vert
f\right\vert $ of Alpert expectations given by (\ref{analogue}), the proof
in \cite{SaShUr7} \textbf{adapts readily} to obtain (\ref{delta near}), as
we sketch below. The proof of (\ref{delta far}) is \textbf{virtually
identical} to the corresponding proofs in \cite{SaShUr7}, and we will not
repeat those details here.
\end{description}

The inequality (\ref{delta near}) is the only place in the proof where the
weak boundedness constant $\mathcal{WBP}_{T^{\alpha }}^{\left( \kappa
_{1},\kappa _{2}\right) }$\ is used, and this constant $\mathcal{WBP}%
_{T^{\alpha }}^{\left( \kappa _{1},\kappa _{2}\right) }$\ will be eliminated
by exploiting the doubling properties of the measures in the final
subsubsection of the proof. This avoids the more difficult surgery argument
that was used to eliminate a weak boundedness property in Lacey and Wick in 
\cite{LaWi}. Moreover, surgery requires the use of two independent families
of grids, something we do not have in this proof.

\begin{proof}[Sketch of Proof of (\protect\ref{delta near})]
First, following \cite{SaShUr7}, which in turn followed \cite{NTV4}, we
reduce matters to the case when $J\subset I$. Then we break up the Alpert
projections $\bigtriangleup _{I;\kappa _{1}}^{\sigma }f$ and $\bigtriangleup
_{J;\kappa _{2}}^{\omega }g$ according to expectations over their respective
children,%
\begin{eqnarray*}
\bigtriangleup _{I;\kappa _{1}}^{\sigma }f &=&\sum_{I^{\prime }\in \mathfrak{%
C}\left( I\right) \ }\left( \bigtriangleup _{I;\kappa _{1}}^{\sigma
}f\right) \mathbf{1}_{I^{\prime }}=\sum_{I^{\prime }\in \mathfrak{C}\left(
I\right) \ }\left\Vert \left( \bigtriangleup _{I;\kappa _{1}}^{\sigma
}f\right) \mathbf{1}_{I^{\prime }}\right\Vert _{\infty }\ P_{I^{\prime
};\kappa _{1}}^{\sigma }f, \\
\bigtriangleup _{J;\kappa _{2}}^{\omega }g &=&\sum_{J^{\prime }\in \mathfrak{%
C}\left( J\right) \ }\left( \bigtriangleup _{J;\kappa _{2}}^{\omega
}g\right) \mathbf{1}_{J^{\prime }}=\sum_{J^{\prime }\in \mathfrak{C}\left(
J\right) \ }\left\Vert \left( \bigtriangleup _{J;\kappa _{2}}^{\omega
}g\right) \mathbf{1}_{J^{\prime }}\right\Vert _{\infty }\ Q_{J^{\prime
};\kappa _{2}}^{\omega }g,
\end{eqnarray*}%
where $P_{I^{\prime };\kappa _{1}}^{\sigma }f=\frac{\left( \bigtriangleup
_{I;\kappa _{1}}^{\sigma }f\right) \mathbf{1}_{I^{\prime }}}{\left\Vert
\left( \bigtriangleup _{I;\kappa _{1}}^{\sigma }f\right) \mathbf{1}%
_{I^{\prime }}\right\Vert _{\infty }}$ and $Q_{J^{\prime };\kappa
_{2}}^{\omega }g=\frac{\left( \bigtriangleup _{J;\kappa _{2}}^{\omega
}g\right) \mathbf{1}_{J^{\prime }}}{\left\Vert \left( \bigtriangleup
_{J;\kappa _{2}}^{\omega }g\right) \mathbf{1}_{J^{\prime }}\right\Vert
_{\infty }}$, to further reduce matters to proving that%
\begin{equation*}
\sum_{\substack{ I,J\in \mathcal{D}:\ J\subset I  \\ 2^{-\mathbf{\rho }}\ell
\left( I\right) \leq \ell \left( J\right) \leq \ell \left( I\right) }}%
\sum_{I^{\prime }\in \mathfrak{C}\left( I\right) ,J^{\prime }\in \mathfrak{C}%
\left( J\right) \ }\left\Vert \left( \bigtriangleup _{I;\kappa _{1}}^{\sigma
}f\right) \mathbf{1}_{I^{\prime }}\right\Vert _{\infty }\left\Vert \left(
\bigtriangleup _{J;\kappa _{2}}^{\omega }g\right) \mathbf{1}_{J^{\prime
}}\right\Vert _{\infty }\left\vert \left\langle T_{\sigma }^{\alpha }\left(
P_{I^{\prime };\kappa _{1}}^{\sigma }f\right) ,Q_{J^{\prime };\kappa
_{2}}^{\omega }g\right\rangle _{\omega }\right\vert
\end{equation*}%
is dominated by the right hand side of (\ref{delta near}). Note that $%
P_{I^{\prime };\kappa _{1}}^{\sigma }f\in \left( \mathcal{P}_{\kappa
_{1}}^{I^{\prime }}\right) _{\limfunc{norm}}$ and $Q_{J^{\prime };\kappa
_{2}}^{\omega }g\in \left( \mathcal{P}_{\kappa _{2}}^{J^{\prime }}\right) _{%
\limfunc{norm}}$ are $L^{\infty }$ normalized. Then with $\mathcal{NTV}%
_{T^{\alpha }}^{\left( \kappa _{1},\kappa _{2}\right) }\left( \sigma ,\omega
\right) $ denoting the constant in parentheses on the right hand side of (%
\ref{delta near}), we continue with 
\begin{eqnarray*}
&&\sum_{\substack{ I,J\in \mathcal{D}:\ J\subset I  \\ 2^{-\mathbf{\rho }%
}\ell \left( I\right) \leq \ell \left( J\right) \leq \ell \left( I\right) }}%
\left\vert \left\langle T_{\sigma }^{\alpha }\left( \bigtriangleup
_{I;\kappa _{1}}^{\sigma }f\right) ,\bigtriangleup _{J;\kappa _{2}}^{\omega
}g\right\rangle _{\omega }\right\vert \\
&\lesssim &\sum_{\substack{ I,J\in \mathcal{D}:\ J\subset I  \\ 2^{-\mathbf{%
\rho }}\ell \left( I\right) \leq \ell \left( J\right) \leq \ell \left(
I\right) }}\sum_{I^{\prime }\in \mathfrak{C}\left( I\right) ,J^{\prime }\in 
\mathfrak{C}\left( J\right) \ }\left\Vert \left( \bigtriangleup _{I;\kappa
_{1}}^{\sigma }f\right) \mathbf{1}_{I^{\prime }}\right\Vert _{\infty
}\left\Vert \left( \bigtriangleup _{J;\kappa _{2}}^{\omega }g\right) \mathbf{%
1}_{J^{\prime }}\right\Vert _{\infty }\left\vert \left\langle T_{\sigma
}^{\alpha }\left( P_{I^{\prime };\kappa _{1}}^{\sigma }f\right)
,Q_{J^{\prime };\kappa _{2}}^{\omega }g\right\rangle _{\omega }\right\vert \\
&\lesssim &\sum_{\substack{ I,J\in \mathcal{D}:\ J\subset I  \\ 2^{-\mathbf{%
\rho }}\ell \left( I\right) \leq \ell \left( J\right) \leq \ell \left(
I\right) }}\sum_{I^{\prime }\in \mathfrak{C}\left( I\right) ,J^{\prime }\in 
\mathfrak{C}\left( J\right) \ }\left\Vert \left( \bigtriangleup _{I;\kappa
_{1}}^{\sigma }f\right) \mathbf{1}_{I^{\prime }}\right\Vert _{\infty
}\left\Vert \left( \bigtriangleup _{J;\kappa _{2}}^{\omega }g\right) \mathbf{%
1}_{J^{\prime }}\right\Vert _{\infty }\mathcal{NTV}_{T^{\alpha }}^{\left(
\kappa _{1},\kappa _{2}\right) }\left( \sigma ,\omega \right) \sqrt{%
\left\vert I^{\prime }\right\vert _{\sigma }\left\vert J^{\prime
}\right\vert _{\omega }} \\
&\lesssim &\mathcal{NTV}_{T^{\alpha }}^{\left( \kappa _{1},\kappa
_{2}\right) }\left( \sigma ,\omega \right) \left( \sigma ,\omega \right) 
\sqrt{\sum_{I\in \mathcal{D}}\left\Vert \bigtriangleup _{I;\kappa
_{1}}^{\sigma }f\right\Vert _{L^{2}\left( \sigma \right) }^{2}}\sqrt{%
\sum_{J\in \mathcal{D}}\left\Vert \bigtriangleup _{J;\kappa _{2}}^{\omega
}g\right\Vert _{L^{2}\left( \omega \right) }^{2}},
\end{eqnarray*}%
since (\ref{add con}) yields both 
\begin{equation*}
\sum_{I^{\prime }\in \mathfrak{C}\left( I\right) }\left\Vert \left(
\bigtriangleup _{I;\kappa _{1}}^{\sigma }f\right) \mathbf{1}_{I^{\prime
}}\right\Vert _{\infty }^{2}\left\vert I^{\prime }\right\vert _{\sigma
}\lesssim \left\Vert \bigtriangleup _{I;\kappa _{1}}^{\sigma }f\right\Vert
_{L^{2}\left( \sigma \right) }^{2}\text{ and }\sum_{J^{\prime }\in \mathfrak{%
C}\left( J\right) \ }\left\Vert \left( \bigtriangleup _{J;\kappa
_{2}}^{\omega }g\right) \mathbf{1}_{J^{\prime }}\right\Vert _{\infty
}^{2}\left\vert J^{\prime }\right\vert _{\omega }\lesssim \left\Vert
\bigtriangleup _{J;\kappa _{2}}^{\omega }g\right\Vert _{L^{2}\left( \omega
\right) }^{2},
\end{equation*}%
and since the restriction $2^{-\mathbf{\rho }}\ell \left( I\right) \leq \ell
\left( J\right) \leq \ell \left( I\right) $ gives bounded overlap in the sum
over $I,J\in \mathcal{D}$ with$\ J\subset I$. Now we finish by applying the
orthonormality of Alpert projections, namely $\left\Vert f\right\Vert
_{L^{2}\left( \sigma \right) }^{2}=\sum_{I\in \mathcal{D}}\left\Vert
\bigtriangleup _{I;\kappa _{1}}^{\sigma }f\right\Vert _{L^{2}\left( \sigma
\right) }^{2}$ and $\left\Vert g\right\Vert _{L^{2}\left( \omega \right)
}^{2}=\sum_{J\in \mathcal{D}}\left\Vert \bigtriangleup _{J;\kappa
_{2}}^{\omega }g\right\Vert _{L^{2}\left( \omega \right) }^{2}$.
\end{proof}

\begin{lemma}
\label{standard indicator}Suppose $T^{\alpha }$ is a standard fractional
singular integral with $0\leq \alpha <n$, that all of the cubes $I,J\in 
\mathcal{D}$ below are $\left( \mathbf{r},\varepsilon \right) $-$\limfunc{%
good}$ with goodness parameters $\varepsilon $ and $\mathbf{r}$, that $%
\mathbf{\rho }>\mathbf{r}$, that $f\in L^{2}\left( \sigma \right) $ and $%
g\in L^{2}\left( \omega \right) $, that $\mathcal{F}\subset \mathcal{D}%
^{\sigma }$ is $\sigma $-Carleson, i.e.,%
\begin{equation*}
\sum_{F^{\prime }\in \mathcal{F}:\ F^{\prime }\subset F}\left\vert F^{\prime
}\right\vert _{\sigma }\lesssim \left\vert F\right\vert _{\sigma },\ \ \ \ \
F\in \mathcal{F},
\end{equation*}%
that there is a numerical sequence $\left\{ \alpha _{\mathcal{F}}\left(
F\right) \right\} _{F\in \mathcal{F}}$ such that%
\begin{equation}
\sum_{F\in \mathcal{F}}\alpha _{\mathcal{F}}\left( F\right) ^{2}\left\vert
F\right\vert _{\sigma }\leq \left\Vert f\right\Vert _{L^{2}\left( \sigma
\right) }^{2}\ ,  \label{qo}
\end{equation}%
and finally that for each pair of cubes $\left( I,J\right) \in \mathcal{D}%
^{\sigma }\times \mathcal{D}^{\omega }$, there is a bounded function $\beta
_{I,J}$ supported in $I\setminus 2J$ satisfying%
\begin{equation*}
\left\Vert \beta _{I,J}\right\Vert _{\infty }\leq 1.
\end{equation*}%
Then with $\kappa \geq 1$ we have%
\begin{equation}
\sum_{\substack{ \left( F,J\right) \in \mathcal{F}\times \mathcal{D}^{\omega
}  \\ F\cap J=\emptyset \text{ and }\ell \left( J\right) \leq 2^{-\mathbf{%
\rho }}\ell \left( F\right) }}\left\vert \left\langle T_{\sigma }^{\alpha
}\left( \beta _{F,J}\mathbf{1}_{F}\alpha _{\mathcal{F}}\left( F\right)
\right) ,\bigtriangleup _{J;\kappa }^{\omega }g\right\rangle _{\omega
}\right\vert \lesssim \sqrt{A_{2}^{\alpha }}\left\Vert f\right\Vert
_{L^{2}\left( \sigma \right) }\left\Vert g\right\Vert _{L^{2}\left( \omega
\right) }.  \label{indicator far}
\end{equation}
\end{lemma}

The proof of (\ref{indicator far}) is again \textbf{virtually identical} to
the corresponding proof in \cite{SaShUr7}, and we will not repeat the
details here.

We will also need the following Poisson estimate, that is a straightforward
extension of the case $m=1$ due to NTV in \cite{NTV4}.

\begin{lemma}
\label{Poisson inequality}Fix $m\geq 1$. Suppose that $J\subset I\subset K$
and that $\limfunc{dist}\left( J,\partial I\right) >2\sqrt{n}\ell \left(
J\right) ^{\varepsilon }\ell \left( I\right) ^{1-\varepsilon }$. Then%
\begin{equation}
\mathrm{P}_{m}^{\alpha }(J,\sigma \mathbf{1}_{K\setminus I})\lesssim \left( 
\frac{\ell \left( J\right) }{\ell \left( I\right) }\right) ^{m-\varepsilon
\left( n+m-\alpha \right) }\mathrm{P}_{m}^{\alpha }(I,\sigma \mathbf{1}%
_{K\setminus I}).  \label{e.Jsimeq}
\end{equation}
\end{lemma}

\begin{proof}
We have%
\begin{equation*}
\mathrm{P}_{m}^{\alpha }\left( J,\sigma \chi _{K\setminus I}\right) \approx
\sum_{k=0}^{\infty }2^{-km}\frac{1}{\left\vert 2^{k}J\right\vert ^{1-\frac{%
\alpha }{n}}}\int_{\left( 2^{k}J\right) \cap \left( K\setminus I\right)
}d\sigma ,
\end{equation*}%
and $\left( 2^{k}J\right) \cap \left( K\setminus I\right) \neq \emptyset $
requires%
\begin{equation*}
\limfunc{dist}\left( J,K\setminus I\right) \leq c2^{k}\ell \left( J\right) ,
\end{equation*}%
for some dimensional constant $c>0$. Let $k_{0}$ be the smallest such $k$.
By our distance assumption we must then have%
\begin{equation*}
2\sqrt{n}\ell \left( J\right) ^{\varepsilon }\ell \left( I\right)
^{1-\varepsilon }\leq \limfunc{dist}\left( J,\partial I\right) \leq
c2^{k_{0}}\ell \left( J\right) ,
\end{equation*}%
or%
\begin{equation*}
2^{-k_{0}-1}\leq c\left( \frac{\ell \left( J\right) }{\ell \left( I\right) }%
\right) ^{1-\varepsilon }.
\end{equation*}%
Now let $k_{1}$ be defined by $2^{k_{1}}\equiv \frac{\ell \left( I\right) }{%
\ell \left( J\right) }$. Then assuming $k_{1}>k_{0}$ (the case $k_{1}\leq
k_{0}$ is similar) we have%
\begin{eqnarray*}
\mathrm{P}_{m}^{\alpha }\left( J,\sigma \chi _{K\setminus I}\right) &\approx
&\left\{ \sum_{k=k_{0}}^{k_{1}}+\sum_{k=k_{1}}^{\infty }\right\} 2^{-km}%
\frac{1}{\left\vert 2^{k}J\right\vert ^{1-\frac{\alpha }{n}}}\int_{\left(
2^{k}J\right) \cap \left( K\setminus I\right) }d\sigma \\
&\lesssim &2^{-k_{0}m}\frac{\left\vert I\right\vert ^{1-\frac{\alpha }{n}}}{%
\left\vert 2^{k_{0}}J\right\vert ^{1-\frac{\alpha }{n}}}\left( \frac{1}{%
\left\vert I\right\vert ^{1-\frac{\alpha }{n}}}\int_{\left(
2^{k_{1}}J\right) \cap \left( K\setminus I\right) }d\sigma \right)
+2^{-k_{1}m}\mathrm{P}_{m}^{\alpha }\left( I,\sigma \chi _{\setminus
I}\right) \\
&\lesssim &\left( \frac{\ell \left( J\right) }{\ell \left( I\right) }\right)
^{\left( 1-\varepsilon \right) \left( n+m-\alpha \right) }\left( \frac{\ell
\left( I\right) }{\ell \left( J\right) }\right) ^{n-\alpha }\mathrm{P}%
_{m}^{\alpha }\left( I,\sigma \chi _{K\setminus I}\right) +\left( \frac{\ell
\left( J\right) }{\ell \left( I\right) }\right) ^{m}\mathrm{P}_{m}^{\alpha
}\left( I,\sigma \chi _{K\setminus I}\right) ,
\end{eqnarray*}%
which is the inequality (\ref{e.Jsimeq}).
\end{proof}

\subsubsection{Stopping data}

Next we review the notion of stopping data from \cite{LaSaShUr3}.

\begin{definition}
\label{general stopping data}Suppose we are given a positive constant $%
C_{0}\geq 4$, a subset $\mathcal{F}$ of the dyadic quasigrid $\mathcal{D}$
(called the stopping times), and a corresponding sequence $\alpha _{\mathcal{%
F}}\equiv \left\{ \alpha _{\mathcal{F}}\left( F\right) \right\} _{F\in 
\mathcal{F}}$ of nonnegative numbers $\alpha _{\mathcal{F}}\left( F\right)
\geq 0$ (called the stopping data). Let $\left( \mathcal{F},\prec ,\pi _{%
\mathcal{F}}\right) $ be the tree structure on $\mathcal{F}$ inherited from $%
\mathcal{D}$, and for each $F\in \mathcal{F}$ denote by $\mathcal{C}_{%
\mathcal{F}}\left( F\right) =\left\{ I\in \mathcal{D}:\pi _{\mathcal{F}%
}I=F\right\} $ the corona associated with $F$: 
\begin{equation*}
\mathcal{C}_{\mathcal{F}}\left( F\right) =\left\{ I\in \mathcal{D}:I\subset F%
\text{ and }I\not\subset F^{\prime }\text{ for any }F^{\prime }\prec
F\right\} .
\end{equation*}%
We say the triple $\left( C_{0},\mathcal{F},\alpha _{\mathcal{F}}\right) $
constitutes \emph{stopping data} for a function $f\in L_{loc}^{1}\left( \mu
\right) $ if

\begin{enumerate}
\item $\mathbb{E}_{I}^{\mu }\left\vert f\right\vert \leq \alpha _{\mathcal{F}%
}\left( F\right) $ for all $I\in \mathcal{C}_{F}$ and $F\in \mathcal{F}$,

\item $\sum_{F^{\prime }\preceq F}\left\vert F^{\prime }\right\vert _{\mu
}\leq C_{0}\left\vert F\right\vert _{\mu }$ for all $F\in \mathcal{F}$,

\item $\sum_{F\in \mathcal{F}}\alpha _{\mathcal{F}}\left( F\right)
^{2}\left\vert F\right\vert _{\mu }\mathbf{\leq }C_{0}^{2}\left\Vert
f\right\Vert _{L^{2}\left( \mu \right) }^{2}$,

\item $\alpha _{\mathcal{F}}\left( F\right) \leq \alpha _{\mathcal{F}}\left(
F^{\prime }\right) $ whenever $F^{\prime },F\in \mathcal{F}$ with $F^{\prime
}\subset F$,

\item $\left\Vert \sum_{F\in \mathcal{F}}\alpha _{\mathcal{F}}\left(
F\right) \mathbf{1}_{F}\right\Vert _{L^{2}\left( \mu \right) }^{2}\leq
C_{0}^{\prime }\left\Vert f\right\Vert _{L^{2}\left( \mu \right) }^{2}$.
\end{enumerate}
\end{definition}

\begin{definition}
If $\left( C_{0},\mathcal{F},\alpha _{\mathcal{F}}\right) $ constitutes
stopping data for a function $f\in L_{loc}^{1}\left( \mu \right) $, we refer
to the othogonal weighted Alpert decomposition 
\begin{equation*}
f=\sum_{F\in \mathcal{F}}\mathsf{P}_{\mathcal{C}_{\mathcal{F}}\left(
F\right) }^{\mu }f;\ \ \ \ \ \mathsf{P}_{\mathcal{C}_{\mathcal{F}}\left(
F\right) }^{\mu }f\equiv \sum_{I\in \mathcal{C}_{\mathcal{F}}\left( F\right)
}\bigtriangleup _{I;\kappa }^{\mu }f,
\end{equation*}%
as the \emph{corona decomposition} of $f$ associated with the stopping times 
$\mathcal{F}$.
\end{definition}

It is often convenient to extend the definition of $\alpha _{\mathcal{F}}$
from $\mathcal{F}$ to the entire grid $\mathcal{D}$ by setting%
\begin{equation*}
\alpha _{\mathcal{F}}\left( I\right) \equiv \sup_{F\in \mathcal{F}:\
F\supset I}\alpha _{\mathcal{F}}\left( F\right) .
\end{equation*}%
When we wish to emphasize the dependence of $\alpha _{\mathcal{F}}$ on $f$
we will write $\alpha _{\mathcal{F};f}$.

\begin{description}
\item[Comments on stopping data] Property (1) says that $\alpha _{\mathcal{F}%
}\left( F\right) $ bounds the averages of $f$ in the corona $\mathcal{C}_{F}$%
, and property (2) says that the cubes at the tops of the coronas satisfy a
Carleson condition relative to the weight $\mu $. Note that a standard
`maximal cube' argument extends the Carleson condition in property (2) to
the inequality%
\begin{equation*}
\sum_{F^{\prime }\in \mathcal{F}:\ F^{\prime }\subset A}\left\vert F^{\prime
}\right\vert _{\mu }\leq C_{0}\left\vert A\right\vert _{\mu }\text{ for all
open sets }A\subset \mathbb{R}^{n}.
\end{equation*}%
Property (3) is the quasiorthogonality condition that says the sequence of
functions $\left\{ \alpha _{\mathcal{F}}\left( F\right) \mathbf{1}%
_{F}\right\} _{F\in \mathcal{F}}$ is in the vector-valued space $L^{2}\left(
\ell ^{2};\mu \right) $, and property (4) says that the control on stopping
data is nondecreasing on the stopping tree $\mathcal{F}$ . (For the Calder%
\'{o}n-Zgumund stopping times above, we have the stronger property that $%
\alpha _{\mathcal{F}}\left( F^{\prime }\right) >C_{0}\alpha _{\mathcal{F}%
}\left( F\right) $ when $F^{\prime }$ is an $\mathcal{F}$-child of $F$, and
this stronger property implies both (2) and (3).) Finally, property (5) is a
consequence of (2) and (3) that says the sequence $\left\{ \alpha _{\mathcal{%
F}}\left( F\right) \mathbf{1}_{F}\right\} _{F\in \mathcal{F}}$ has a \emph{%
quasiorthogonal} property relative to $f$ with a constant $C_{0}^{\prime }$
depending only on $C_{0}$. Indeed, the Carleson condition (2) implies a
geometric decay in levels of the tree $\mathcal{F}$, namely that there are
positive constants $C_{1}$ and $\varepsilon $, depending on $C_{0}$, such
that if $\mathfrak{C}_{\mathcal{F}}^{\left( m\right) }\left( F\right) $
denotes the set of $m^{th}$ generation children of $F$ in $\mathcal{F}$,%
\begin{equation*}
\sum_{F^{\prime }\in \mathfrak{C}_{\mathcal{F}}^{\left( m\right) }\left(
F\right) :\ }\left\vert F^{\prime }\right\vert _{\mu }\leq \left(
C_{1}2^{-\varepsilon m}\right) ^{2}\left\vert F\right\vert _{\mu },\ \ \ \ \ 
\text{for all }m\geq 0\text{ and }F\in \mathcal{F},
\end{equation*}%
and the proof of Property (5) follows from this in a standard way, see e.g. 
\cite{SaShUr7}.
\end{description}

Define Alpert corona projections%
\begin{equation*}
\mathsf{P}_{\mathcal{C}_{\mathcal{F}}\left( F\right) }^{\sigma }\equiv
\sum_{I\in \mathcal{C}_{\mathcal{F}}\left( F\right) }\bigtriangleup
_{I;\kappa _{1}}^{\sigma }\text{ and }\mathsf{P}_{\mathcal{C}_{\mathcal{F}}^{%
\mathbf{\tau }-\limfunc{shift}}\left( F\right) }^{\omega }\equiv \sum_{J\in 
\mathcal{C}_{\mathcal{F}}^{\mathbf{\tau }-\limfunc{shift}}\left( F\right)
}\bigtriangleup _{J;\kappa _{2}}^{\omega }\ ,
\end{equation*}%
where%
\begin{eqnarray*}
\mathcal{C}_{\mathcal{F}}^{\mathbf{\tau }-\limfunc{shift}}\left( F\right)
&\equiv &\left[ \mathcal{C}_{\mathcal{F}}\left( F\right) \setminus \mathcal{N%
}_{\mathcal{D}}^{\mathbf{\tau }}\left( F\right) \right] \cup
\dbigcup\limits_{F^{\prime }\in \mathcal{F}}\mathcal{N}_{\mathcal{D}}^{%
\mathbf{\tau }}\left( F^{\prime }\right) ; \\
\text{where }\mathcal{N}_{\mathcal{D}}^{\mathbf{\tau }}\left( E\right)
&\equiv &\left\{ J\in \mathcal{D}:J\subset E\text{ and }\ell \left( J\right)
\geq 2^{\mathbf{\tau }}\ell \left( E\right) \right\} .
\end{eqnarray*}%
Thus the \emph{shifted} corona $\mathcal{C}_{\mathcal{F}}^{\mathbf{\tau }-%
\limfunc{shift}}\left( F\right) $ has the top $\mathbf{\tau }$ levels from $%
\mathcal{C}_{\mathcal{F}}\left( F\right) $ removed, and includes the first $%
\mathbf{\tau }$ levels from each of its $\mathcal{F}$-children, even if some
of them were initially removed. Keep in mind that we are restricting the
Alpert supports of $f$ and $g$ to $\limfunc{good}$ functions so that%
\begin{equation*}
\mathsf{P}_{\mathcal{C}_{\mathcal{F}}\left( F\right) }^{\sigma }f=\sum_{I\in 
\mathcal{C}_{\mathcal{F}}^{\limfunc{good}}\left( F\right) }\bigtriangleup
_{I;\kappa _{1}}^{\sigma }\text{ and }\mathsf{P}_{\mathcal{C}_{\mathcal{F}}^{%
\mathbf{\tau }-\limfunc{shift}}\left( F\right) }^{\omega }g=\sum_{J\in 
\mathcal{C}_{\mathcal{F}}^{\limfunc{good},\mathbf{\tau }-\limfunc{shift}%
}\left( F\right) }\bigtriangleup _{J;\kappa _{2}}^{\omega }\ ,
\end{equation*}%
where $\mathcal{C}_{\mathcal{F}}^{\limfunc{good}}\left( F\right) \equiv 
\mathcal{C}_{\mathcal{F}}\left( F\right) \cap \mathcal{D}^{\limfunc{good}}$
and $\mathcal{C}_{\mathcal{F}}^{\limfunc{good},\mathbf{\tau }-\limfunc{shift}%
}\left( F\right) \equiv \mathcal{C}_{\mathcal{F}}^{\mathbf{\tau }-\limfunc{%
shift}}\left( F\right) \cap \mathcal{D}^{\limfunc{good}}$. Note also that we
suppress the integers $\kappa _{1}$ and $\kappa _{2}$ from the notation for
the corona projections $\mathsf{P}_{\mathcal{C}_{\mathcal{F}}\left( F\right)
}^{\sigma }$ and $\mathsf{P}_{\mathcal{C}_{\mathcal{F}}^{\mathbf{\tau }-%
\limfunc{shift}}\left( F\right) }^{\omega }$. Finally note that we do not
assume that $\sigma $ is doubling for the next proposition, although the
assumptions come close to forcing this.

\subsubsection{The main Intertwining Proposition}

Here now is the Intertwining Proposition with a proof obtained by adapting
the argument in Nazarov, Treil and Volberg \cite{NTV4} to the argument in 
\cite{SaShUr7}, and using weaker pivotal conditions with Alpert wavelets.
Recall that $0<\varepsilon <1$ and $\mathbf{r}$ is chosen sufficiently large
depending on $\varepsilon $. Later, in using the Intertwining Proposition to
control the Far form in Subsubsection \ref{Subsubsection 6.6.1}\ below, we
will need to resolve the difference between the shifted coronas used here
and the parallel coronas used there.

\begin{proposition}[The Intertwining Proposition]
\label{strongly adapted}Suppose that $\mathcal{F}$ is $\sigma $-Carleson,
that $\left( C_{0},\mathcal{F},\alpha _{\mathcal{F};f}\right) $ constitutes
stopping data for $f$ for all $f\in L^{2}\left( \sigma \right) $, and that%
\begin{equation*}
\left\Vert \bigtriangleup _{I;\kappa _{1}}^{\sigma }f\right\Vert _{L^{\infty
}\left( \sigma \right) }\leq C\alpha _{\mathcal{F};f}\left( I\right) ,\ \ \
\ \ f\in L^{2}\left( \sigma \right) ,I\in \mathcal{D}.
\end{equation*}%
Then for $\limfunc{good}$ functions $f\in L^{2}\left( \sigma \right) $ and $%
g\in L^{2}\left( \omega \right) $, and with $\kappa _{1},\kappa _{2}\geq 1$,
we have%
\begin{equation*}
\left\vert \sum_{F\in \mathcal{F}}\ \sum_{I:\ I\supsetneqq F}\ \left\langle
T_{\sigma }^{\alpha }\bigtriangleup _{I;\kappa _{1}}^{\sigma }f,\mathsf{P}_{%
\mathcal{C}_{\mathcal{F}}^{\mathbf{\tau }-\limfunc{shift}}\left( F\right)
}^{\omega }g\right\rangle _{\omega }\right\vert \lesssim \left( \mathcal{V}%
_{2}^{\alpha ,\kappa _{1}}+\sqrt{A_{2}^{\alpha }}+\mathfrak{T}_{T^{\alpha
}}^{\left( \kappa _{1}\right) }\right) \ \left\Vert f\right\Vert
_{L^{2}\left( \sigma \right) }\left\Vert g\right\Vert _{L^{2}\left( \omega
\right) }.
\end{equation*}
\end{proposition}

\begin{proof}
We write the left hand side of the display above as%
\begin{equation*}
\sum_{F\in \mathcal{F}}\ \sum_{I:\ I\supsetneqq F}\ \left\langle T_{\sigma
}^{\alpha }\bigtriangleup _{I;\kappa _{1}}^{\sigma }f,g_{F}\right\rangle
_{\omega }=\sum_{F\in \mathcal{F}}\ \left\langle T_{\sigma }^{\alpha }\left(
\sum_{I:\ I\supsetneqq F}\bigtriangleup _{I;\kappa _{1}}^{\sigma }f\right)
,g_{F}\right\rangle _{\omega }\equiv \sum_{F\in \mathcal{F}}\ \left\langle
T_{\sigma }^{\alpha }f_{F},g_{F}\right\rangle _{\omega }\ ,
\end{equation*}%
where%
\begin{equation*}
g_{F}=\mathsf{P}_{\mathcal{C}_{\mathcal{F}}^{\mathbf{\tau }-\limfunc{shift}%
}\left( F\right) }^{\omega }g\text{ and }f_{F}\equiv \sum_{I:\ I\supsetneqq
F}\bigtriangleup _{I;\kappa _{1}}^{\sigma }f\ .
\end{equation*}%
Note that $g_{F}$ is supported in $F$, and that $f_{F}$ is the restriction
of a polynomial of degree less than $\kappa $ to $F$. We next observe that
the cubes $I$ occurring in this sum are linearly and consecutively ordered
by inclusion, along with the cubes $F^{\prime }\in \mathcal{F}$ that contain 
$F$. More precisely, we can write%
\begin{equation*}
F\equiv F_{0}\subsetneqq F_{1}\subsetneqq F_{2}\subsetneqq ...\subsetneqq
F_{n}\subsetneqq F_{n+1}\subsetneqq ...F_{N}
\end{equation*}%
where $F_{m}=\pi _{\mathcal{F}}^{m}F$ is the $m^{th}$ ancestor of $F$ in the
tree $\mathcal{F}$ for all $m\geq 1$. We can also write%
\begin{equation*}
F=F_{0}\subsetneqq I_{1}\subsetneqq I_{2}\subsetneqq ...\subsetneqq
I_{k}\subsetneqq I_{k+1}\subsetneqq ...\subsetneqq I_{K}=F_{N}
\end{equation*}%
where $I_{k}=\pi _{\mathcal{D}}^{k}F$ is the $k^{th}$ ancestor of $F$ in the
tree $\mathcal{D}$ for all $k\geq 1$. There is a (unique) subsequence $%
\left\{ k_{m}\right\} _{m=1}^{N}$ such that%
\begin{equation*}
F_{m}=I_{k_{m}},\ \ \ \ \ 1\leq m\leq N.
\end{equation*}%
Then we have%
\begin{equation*}
f_{F}\left( x\right) =\sum_{\ell =1}^{\infty }\bigtriangleup _{I_{\ell
};\kappa _{1}}^{\sigma }f\left( x\right) .
\end{equation*}

Assume now that $k_{m}\leq k<k_{m+1}$. We denote the $2^{n}-1$ siblings of $%
I $ by $\theta \left( I\right) $, $\theta \in \Theta $, i.e. $\left\{ \theta
\left( I\right) \right\} _{\theta \in \Theta }=\mathfrak{C}_{\mathcal{D}%
}\left( \pi _{\mathcal{D}}I\right) \setminus \left\{ I\right\} $. There are
two cases to consider here:%
\begin{equation*}
\theta \left( I_{k}\right) \notin \mathcal{F}\text{ and }\theta \left(
I_{k}\right) \in \mathcal{F}.
\end{equation*}%
Suppose first that $\theta \left( I_{k}\right) \notin \mathcal{F}$. Then $%
\theta \left( I_{k}\right) \in \mathcal{C}_{F_{m+1}}^{\sigma }$ and using a
telescoping sum, we compute that for 
\begin{equation*}
x\in \theta \left( I_{k}\right) \subset I_{k+1}\setminus I_{k}\subset
F_{m+1}\setminus F_{m},
\end{equation*}%
we have 
\begin{equation*}
\left\vert f_{F}\left( x\right) \right\vert =\left\vert \sum_{\ell
=k}^{\infty }\bigtriangleup _{I_{\ell };\kappa _{1}}^{\sigma }f\left(
x\right) \right\vert =\left\vert \mathbb{E}_{\theta \left( I_{k}\right)
}^{\sigma }f\left( x\right) -\mathbb{E}_{I_{K}}^{\sigma }f\left( x\right)
\right\vert \lesssim E_{F_{m+1}}^{\sigma }\left\vert f\right\vert \ ,
\end{equation*}%
by (\ref{analogue}).

On the other hand, if $\theta \left( I_{k}\right) \in \mathcal{F}$, then $%
I_{k+1}\in \mathcal{C}_{F_{m+1}}^{\sigma }$ and we have for $x\in \theta
\left( I_{k}\right) $ that%
\begin{equation*}
\left\vert f_{F}\left( x\right) -\bigtriangleup _{\theta \left( I_{k}\right)
;\kappa _{1}}^{\sigma }f\left( x\right) \right\vert =\left\vert \sum_{\ell
=k+1}^{\infty }\bigtriangleup _{I_{\ell };\kappa _{1}}^{\sigma }f\left(
x\right) \right\vert =\left\vert \mathbb{E}_{I_{k+1};\kappa _{1}}^{\sigma
}f\left( x\right) -\mathbb{E}_{I_{K};\kappa _{1}}^{\sigma }f\left( x\right)
\right\vert \lesssim E_{F_{m+1}}^{\sigma }\left\vert f\right\vert \ ,
\end{equation*}%
by (\ref{analogue}) again. Now we write%
\begin{eqnarray*}
f_{F} &=&\varphi _{F}+\psi _{F}, \\
\text{where }\varphi _{F} &\equiv &\sum_{1\leq k<\infty ,\theta :\ \theta
\left( I_{k}\right) \in \mathcal{F}}\mathbf{1}_{\theta \left( I_{k}\right)
}\bigtriangleup _{I_{k};\kappa _{1}}^{\sigma }f\text{ and }\psi
_{F}=f_{F}-\varphi _{F}\ ; \\
\sum_{F\in \mathcal{F}}\ \left\langle T_{\sigma }^{\alpha
}f_{F},g_{F}\right\rangle _{\omega } &=&\sum_{F\in \mathcal{F}}\
\left\langle T_{\sigma }^{\alpha }\varphi _{F},g_{F}\right\rangle _{\omega
}+\sum_{F\in \mathcal{F}}\ \left\langle T_{\sigma }^{\alpha }\psi
_{F},g_{F}\right\rangle _{\omega }\ .
\end{eqnarray*}%
We can apply (\ref{indicator far}) using $\theta \left( I_{k}\right) \in 
\mathcal{F}$ to the first sum here to obtain%
\begin{eqnarray*}
\left\vert \sum_{F\in \mathcal{F}}\ \left\langle T_{\sigma }^{\alpha
}\varphi _{F},g_{F}\right\rangle _{\omega }\right\vert &\lesssim &\sqrt{%
A_{2}^{\alpha }}\ \left\Vert \sum_{F\in \mathcal{F}}\varphi _{F}\right\Vert
_{L^{2}\left( \sigma \right) }\left\Vert \sum_{F\in \mathcal{F}%
}g_{F}\right\Vert _{L^{2}\left( \omega \right) }^{2} \\
&\lesssim &\sqrt{A_{2}^{\alpha }}\ \left\Vert f\right\Vert _{L^{2}\left(
\sigma \right) }\left[ \sum_{F\in \mathcal{F}}\left\Vert g_{F}\right\Vert
_{L^{2}\left( \omega \right) }^{2}\right] ^{\frac{1}{2}}.
\end{eqnarray*}

Turning to the second sum we note that%
\begin{eqnarray*}
\psi _{F}\left( x\right) &=&f_{F}\left( x\right) -\varphi _{F}\left(
x\right) =\sum_{\ell =1}^{\infty }\left[ 1-\mathbf{1}_{\mathcal{F}}\left(
\theta \left( I_{\ell }\right) \right) \mathbf{1}_{\theta \left( I_{\ell
}\right) }\left( x\right) \right] \bigtriangleup _{I_{\ell };\kappa
_{1}}^{\sigma }f\left( x\right) \\
&=&\sum_{m=1}^{\infty }\sum_{\ell =k_{m-1}}^{k_{m}}\left[ 1-\mathbf{1}_{%
\mathcal{F}}\left( \theta \left( I_{\ell }\right) \right) \mathbf{1}_{\theta
\left( I_{\ell }\right) }\left( x\right) \right] \bigtriangleup _{I_{\ell
};\kappa _{1}}^{\sigma }f\left( x\right) \equiv \sum_{m=1}^{\infty }\psi
_{F}^{\left( m\right) }\left( x\right) ,
\end{eqnarray*}%
where 
\begin{eqnarray*}
\psi _{F}^{m}\left( x\right) &=&\sum_{\ell =k_{m-1}}^{k_{m}}\left[ 1-\mathbf{%
1}_{\mathcal{F}}\left( \theta \left( I_{\ell }\right) \right) \mathbf{1}%
_{\theta \left( I_{\ell }\right) }\left( x\right) \right] \bigtriangleup
_{I_{\ell };\kappa }^{\sigma }f\left( x\right) \\
&=&\left\{ 
\begin{array}{ccc}
\mathbb{E}_{I_{\ell +1};\kappa _{1}}f-\mathbb{E}_{\pi _{\mathcal{F}}^{\left(
m\right) }F;\kappa _{1}}f & \text{ if } & x\in \theta \left( I_{\ell
}\right) \text{ and }\theta \left( I_{\ell }\right) \in \mathcal{F},\
k_{m}\leq \ell \leq k_{m+1} \\ 
\mathbb{E}_{\theta \left( I_{\ell }\right) ;\kappa }f-\mathbb{E}_{\pi _{%
\mathcal{F}}^{\left( m\right) }F;\kappa _{1}}f & \text{ if } & x\not\in
\theta \left( I_{\ell }\right) \text{ or }\theta \left( I_{\ell }\right)
\not\in \mathcal{F},\ k_{m}\leq \ell \leq k_{m+1}%
\end{array}%
\right. .
\end{eqnarray*}%
Now we write%
\begin{equation*}
\sum_{F\in \mathcal{F}}\ \left\langle T_{\sigma }^{\alpha }\psi
_{F},g_{F}\right\rangle _{\omega }=\sum_{m=1}^{\infty }\sum_{F\in \mathcal{F}%
}\ \left\langle T_{\sigma }^{\alpha }\psi _{F}^{m},g_{F}\right\rangle
_{\omega }=\sum_{m=1}^{\infty }\sum_{F\in \mathcal{F}}\ \left\langle
T_{\sigma }^{\alpha }\mathbf{1}_{\pi _{\mathcal{F}}^{m+1}F\setminus \pi _{%
\mathcal{F}}^{m}F}\psi _{F}^{m},g_{F}\right\rangle _{\omega }\equiv
\sum_{m=1}^{\infty }\sum_{F\in \mathcal{F}}\mathcal{I}_{m}\left( F\right) \ ,
\end{equation*}%
where 
\begin{equation}
\mathcal{I}_{m}\left( F\right) =\left\langle T_{\sigma }^{\alpha }\left( 
\mathbf{1}_{\pi _{\mathcal{F}}^{m+1}F\setminus \pi _{\mathcal{F}}^{m}F}\psi
_{F}^{m}\right) ,g_{F}\right\rangle _{\omega }\ .  \label{ImF}
\end{equation}

We then note that (\ref{analogue}) once more gives 
\begin{equation*}
\left\vert \psi _{F}^{m}\right\vert \lesssim E_{F_{m+1}}^{\sigma }\left\vert
f\right\vert \lesssim \alpha _{\mathcal{F}}\left( \pi _{\mathcal{F}%
}^{m+1}F\right) \ \mathbf{1}_{\pi _{\mathcal{F}}^{m+1}F\setminus \pi _{%
\mathcal{F}}^{m}F}\ ,
\end{equation*}%
and so 
\begin{eqnarray*}
\left\vert \psi _{F}\right\vert &\leq &\sum_{m=0}^{N}\left(
E_{F_{m+1}}^{\sigma }\left\vert f\right\vert \right) \ \mathbf{1}%
_{F_{m+1}\setminus F_{m}}=\left( E_{F}^{\sigma }\left\vert f\right\vert
\right) \ \mathbf{1}_{F}+\sum_{m=0}^{N}\left( E_{\pi _{\mathcal{F}%
}^{m+1}F}^{\sigma }\left\vert f\right\vert \right) \ \mathbf{1}_{\pi _{%
\mathcal{F}}^{m+1}F\setminus \pi _{\mathcal{F}}^{m}F} \\
&=&\left( E_{F}^{\sigma }\left\vert f\right\vert \right) \ \mathbf{1}%
_{F}+\sum_{F^{\prime }\in \mathcal{F}:\ F\subset F^{\prime }}\left( E_{\pi _{%
\mathcal{F}}F^{\prime }}^{\sigma }\left\vert f\right\vert \right) \ \mathbf{1%
}_{\pi _{\mathcal{F}}F^{\prime }\setminus F^{\prime }} \\
&\leq &\alpha _{\mathcal{F}}\left( F\right) \ \mathbf{1}_{F}+\sum_{F^{\prime
}\in \mathcal{F}:\ F\subset F^{\prime }}\alpha _{\mathcal{F}}\left( \pi _{%
\mathcal{F}}F^{\prime }\right) \ \mathbf{1}_{\pi _{\mathcal{F}}F^{\prime
}\setminus F^{\prime }} \\
&\leq &\alpha _{\mathcal{F}}\left( F\right) \ \mathbf{1}_{F}+\sum_{F^{\prime
}\in \mathcal{F}:\ F\subset F^{\prime }}\alpha _{\mathcal{F}}\left( \pi _{%
\mathcal{F}}F^{\prime }\right) \ \mathbf{1}_{\pi _{\mathcal{F}}F^{\prime }}\ 
\mathbf{1}_{F^{c}},\ \ \ \ \ \text{\ for all }F\in \mathcal{F}.
\end{eqnarray*}

Now we write%
\begin{eqnarray*}
&&\sum_{F\in \mathcal{F}}\ \left\langle T_{\sigma }^{\alpha }\psi
_{F},g_{F}\right\rangle _{\omega }=I+II\ ; \\
&&\text{where }I\equiv \sum_{F\in \mathcal{F}}\ \left\langle T_{\sigma
}^{\alpha }\left( \mathbf{1}_{F}\psi _{F}\right) ,g_{F}\right\rangle
_{\omega }\text{ and }II\equiv \sum_{F\in \mathcal{F}}\ \left\langle
T_{\sigma }^{\alpha }\left( \mathbf{1}_{F^{c}}\psi _{F}\right)
,g_{F}\right\rangle _{\omega }\ .
\end{eqnarray*}%
Then by $\kappa $-Cube Testing (\ref{kappa testing}), and the fact that $%
\psi _{F}\mathbf{1}_{F}$ is a polynomial on $F$ bounded in modulus by $%
\alpha _{\mathcal{F}}\left( F\right) $, we have 
\begin{equation*}
\left\vert \left\langle T_{\sigma }^{\alpha }\left( \psi _{F}\right)
,g_{F}\right\rangle _{\omega }\right\vert \leq \left\Vert T_{\sigma
}^{\alpha }\left( \psi _{F}\mathbf{1}_{F}\right) \right\Vert _{L^{2}\left(
\omega \right) }\left\Vert g_{F}\right\Vert _{L^{2}\left( \omega \right)
}\leq \mathfrak{T}_{T^{\alpha }}^{\left( \kappa \right) }\alpha _{\mathcal{F}%
}\left( F\right) \sqrt{\left\vert F\right\vert _{\sigma }}\left\Vert
g_{F}\right\Vert _{L^{2}\left( \omega \right) }\ ,
\end{equation*}%
and then quasi-orthogonality yields%
\begin{equation*}
\left\vert I\right\vert \leq \sum_{F\in \mathcal{F}}\ \left\vert
\left\langle T_{\sigma }^{\alpha }\mathbf{1}_{F}\psi _{F},g_{F}\right\rangle
_{\omega }\right\vert \lesssim \mathfrak{T}_{T^{\alpha }}^{\left( \kappa
\right) }\sum_{F\in \mathcal{F}}\ \alpha _{\mathcal{F}}\left( F\right) \sqrt{%
\left\vert F\right\vert _{\sigma }}\left\Vert g_{F}\right\Vert _{L^{2}\left(
\omega \right) }\lesssim \mathfrak{T}_{T^{\alpha }}^{\left( \kappa \right)
}\left\Vert f\right\Vert _{L^{2}\left( \sigma \right) }\left[ \sum_{F\in 
\mathcal{F}}\left\Vert g_{F}\right\Vert _{L^{2}\left( \omega \right) }^{2}%
\right] ^{\frac{1}{2}}.
\end{equation*}

On the other hand, $\mathbf{1}_{F^{c}}\psi _{F}$ is supported outside $F$,
and each $J$ in the Alpert support of $g_{F}$ is $\left( \mathbf{r}%
,\varepsilon \right) $-deeply embedded in $F$, which we write as $J\Subset _{%
\mathbf{r},\varepsilon }F$. So if we denote by 
\begin{equation*}
\mathcal{M}_{\left( \mathbf{r},\varepsilon \right) -\limfunc{deep}}^{%
\limfunc{good}}\left( F\right) \equiv \left\{ \text{maximal good }J\Subset _{%
\mathbf{r},\varepsilon }F\right\}
\end{equation*}%
the set of \emph{maximal} intervals\emph{\ }that are \emph{both} $\limfunc{%
good}$ and $\left( \mathbf{r},\varepsilon \right) $-deeply embedded in $F$,
then 
\begin{eqnarray*}
F &=&\overset{\cdot }{\dbigcup\limits_{K\in \mathcal{M}_{\left( \mathbf{r}%
,\varepsilon \right) -\limfunc{deep}}\left( F\right) }}K=\overset{\cdot }{%
\dbigcup\limits_{G\in \mathcal{M}_{\left( \mathbf{r},\varepsilon \right) -%
\limfunc{deep}}^{\limfunc{good}}\left( F\right) }}G. \\
\text{where each }G &\in &\mathcal{M}_{\left( \mathbf{r},\varepsilon \right)
-\limfunc{deep}}^{\limfunc{good}}\left( F\right) \text{ is contained in some 
}K\in \mathcal{M}_{\left( \mathbf{r},\varepsilon \right) -\limfunc{deep}%
}\left( F\right) .
\end{eqnarray*}%
Thus we can apply the Energy Lemma \ref{ener} to obtain from (\ref{ImF}) that%
\begin{eqnarray*}
&&\left\vert II\right\vert =\left\vert \sum_{F\in \mathcal{F}}\left\langle
T_{\sigma }^{\alpha }\left( \mathbf{1}_{F^{c}}\psi _{F}\right)
,g_{F}\right\rangle _{\omega }\right\vert =\left\vert \sum_{m=1}^{\infty
}\sum_{F\in \mathcal{F}}\mathcal{I}_{m}\left( F\right) \right\vert \leq
\sum_{m=1}^{\infty }\sum_{F\in \mathcal{F}}\left\vert \mathcal{I}_{m}\left(
F\right) \right\vert \\
&\lesssim &\sum_{m=1}^{\infty }\sum_{F\in \mathcal{F}}\sum_{J\in \mathcal{M}%
_{\left( \mathbf{r},\varepsilon \right) -\limfunc{deep}}^{\limfunc{good}%
}\left( F\right) }\frac{\mathrm{P}_{\kappa _{1}}^{\alpha }\left( J,\alpha _{%
\mathcal{F}}\left( \pi _{\mathcal{F}}^{m+1}F\right) \mathbf{1}_{\pi _{%
\mathcal{F}}^{m+1}F\setminus \pi _{\mathcal{F}}^{m}F}\sigma \right) }{%
\left\vert J\right\vert ^{\frac{\kappa }{n}}}\sqrt{\sum_{\left\vert \beta
\right\vert =\kappa }\left\Vert \mathsf{Q}_{\mathcal{C}_{F}^{\mathbf{\tau }-%
\limfunc{shift}};J}^{\omega }x^{\beta }\right\Vert _{L^{2}\left( \omega
\right) }^{2}}\left\Vert \mathsf{P}_{J}^{\omega }g_{F}\right\Vert
_{L^{2}\left( \omega \right) } \\
&&+\sum_{m=1}^{\infty }\sum_{F\in \mathcal{F}}\sum_{J\in \mathcal{M}_{\left( 
\mathbf{r},\varepsilon \right) -\limfunc{deep}}^{\limfunc{good}}\left(
F\right) }\frac{\mathrm{P}_{\kappa _{1}+\delta ^{\prime }}^{\alpha }\left(
J,\alpha _{\mathcal{F}}\left( \pi _{\mathcal{F}}^{m+1}F\right) \mathbf{1}%
_{\pi _{\mathcal{F}}^{m+1}F\setminus \pi _{\mathcal{F}}^{m}F}\sigma \right) 
}{\left\vert J\right\vert ^{\frac{\kappa }{n}}}\left\Vert \left\vert
x-m_{J}^{\kappa }\right\vert ^{\kappa }\right\Vert _{L^{2}\left( \omega
\right) }\left\Vert \mathsf{P}_{J}^{\omega }g_{F}\right\Vert _{L^{2}\left(
\omega \right) } \\
&\equiv &II_{G}+II_{B}\ .
\end{eqnarray*}

Then we have that $\left\vert II_{G}\right\vert $ is bounded by%
\begin{eqnarray}
&&  \label{II_G bound} \\
&&\sum_{m=1}^{\infty }\sum_{F\in \mathcal{F}}\alpha _{\mathcal{F}}\left( \pi
_{\mathcal{F}}^{m+1}F\right) \left\{ \sum_{J\in \mathcal{M}_{\left( \mathbf{r%
},\varepsilon \right) -\limfunc{deep}}^{\limfunc{good}}\left( F\right) }%
\mathrm{P}_{\kappa _{1}}^{\alpha }\left( J,\mathbf{1}_{\pi _{\mathcal{F}%
}^{m+1}F\setminus \pi _{\mathcal{F}}^{m}F}\sigma \right) \sqrt{%
\sum_{\left\vert \beta \right\vert =\kappa }\left\Vert \mathsf{Q}_{\mathcal{C%
}_{F}^{\mathbf{\tau }-\limfunc{shift}};J}^{\omega }x^{\beta }\right\Vert
_{L^{2}\left( \omega \right) }^{2}}\left\Vert \mathsf{P}_{J}^{\omega
}g_{F}\right\Vert _{L^{2}\left( \omega \right) }\right\}  \notag \\
&\lesssim &\sum_{m=1}^{\infty }\sum_{F\in \mathcal{F}}\alpha _{\mathcal{F}%
}\left( \pi _{\mathcal{F}}^{m+1}F\right) \left\{ \sum_{J\in \mathcal{M}%
_{\left( \mathbf{r},\varepsilon \right) -\limfunc{deep}}^{\limfunc{good}%
}\left( F\right) }\mathrm{P}_{\kappa _{1}}^{\alpha }\left( J,\mathbf{1}_{\pi
_{\mathcal{F}}^{m+1}F\setminus \pi _{\mathcal{F}}^{m}F}\sigma \right)
^{2}\sum_{\left\vert \beta \right\vert =\kappa }\left\Vert \mathsf{Q}_{%
\mathcal{C}_{F}^{\mathbf{\tau }-\limfunc{shift}};J}^{\omega }x^{\beta
}\right\Vert _{L^{2}\left( \omega \right) }^{2}\right\} ^{\frac{1}{2}} 
\notag \\
&&\ \ \ \ \ \ \ \ \ \ \ \ \ \ \ \ \ \ \ \ \ \ \ \ \ \ \ \ \ \ \ \ \ \ \ \ \
\ \ \ \ \ \ \ \ \ \ \ \ \ \ \ \ \ \ \ \ \ \ \ \ \ \ \ \ \ \ \ \ \ \ \ \
\times \left\{ \sum_{J\in \mathcal{M}_{\left( \mathbf{r},\varepsilon \right)
-\limfunc{deep}}^{\limfunc{good}}\left( F\right) }\left\Vert \mathsf{P}%
_{J}^{\omega }g_{F}\right\Vert _{L^{2}\left( \omega \right) }^{2}\right\} ^{%
\frac{1}{2}}.  \notag
\end{eqnarray}%
We now reindex the last sum in (\ref{II_G bound}) above sum using $F^{\ast
}=\pi _{\mathcal{F}}^{m+1}F$ to rewrite it as%
\begin{eqnarray}
&&  \label{at most} \\
&&\sum_{m=1}^{\infty }\sum_{F^{\ast }\in \mathcal{F}}\alpha _{\mathcal{F}%
}\left( F^{\ast }\right) \sum_{F^{\prime }\in \mathfrak{C}_{\mathcal{F}%
}\left( F^{\ast }\right) }\sum_{F\in \mathfrak{C}_{\mathcal{F}}^{\left(
m\right) }\left( F^{\prime }\right) }\left\{ \sum_{J\in \mathcal{M}_{\left( 
\mathbf{r},\varepsilon \right) -\limfunc{deep}}^{\limfunc{good}}\left(
F\right) }\mathrm{P}_{\kappa }^{\alpha }\left( J,\mathbf{1}_{F^{\ast
}\setminus F^{\prime }}\sigma \right) ^{2}\sum_{\left\vert \beta \right\vert
=\kappa }\left\Vert \mathsf{Q}_{\mathcal{C}_{F}^{\limfunc{good},\mathbf{\tau 
}-\limfunc{shift}};J}^{\omega }\left( \frac{x}{\ell \left( J\right) }\right)
^{\beta }\right\Vert _{L^{2}\left( \omega \right) }^{2}\right\} ^{\frac{1}{2}%
}  \notag \\
&&\ \ \ \ \ \ \ \ \ \ \ \ \ \ \ \ \ \ \ \ \ \ \ \ \ \ \ \ \ \ \ \ \ \ \ \ \
\ \ \ \ \ \ \ \ \ \ \ \ \ \ \ \ \ \ \ \ \ \ \ \ \ \ \ \ \ \ \ \ \ \ \ \
\times \left\{ \sum_{J\in \mathcal{M}_{\left( \mathbf{r},\varepsilon \right)
-\limfunc{deep}}^{\limfunc{good}}\left( F\right) }\left\Vert \mathsf{P}%
_{J}^{\omega }g_{F}\right\Vert _{L^{2}\left( \omega \right) }^{2}\right\} ^{%
\frac{1}{2}}.  \notag
\end{eqnarray}

Using (\ref{e.Jsimeq}) with $m=\kappa $ and $\mu =\sigma $, we obtain that
for $J\in \mathcal{M}_{\left( \mathbf{r},\varepsilon \right) -\limfunc{deep}%
}^{\limfunc{good}}\left( F\right) $ and $I=\pi _{\mathcal{F}}^{m-1}F$, we
have $\frac{\ell \left( J\right) }{\ell \left( \pi _{\mathcal{F}%
}^{m-1}F\right) }=2^{-k}$ for some $k\geq m-1$, and hence%
\begin{eqnarray*}
\mathrm{P}_{\kappa _{1}}^{\alpha }\left( J,\mathbf{1}_{F^{\ast }\setminus
F^{\prime }}\sigma \right) ^{2} &\leq &\left( \frac{\ell \left( J\right) }{%
\ell \left( \pi _{\mathcal{F}}^{m-1}F\right) }\right) ^{2\kappa
-2\varepsilon \left( n+\kappa -\alpha \right) }\mathrm{P}_{\kappa
_{1}}^{\alpha }\left( \pi _{\mathcal{F}}^{m-1}F,\mathbf{1}_{\pi _{\mathcal{F}%
}^{m+1}F\setminus \pi _{\mathcal{F}}^{m}F}\sigma \right) ^{2} \\
&=&\left( 2^{-k}\right) ^{2\kappa -2\varepsilon \left( n+\kappa -\alpha
\right) }\mathrm{P}_{\kappa _{1}}^{\alpha }\left( \pi _{\mathcal{F}}^{m-1}F,%
\mathbf{1}_{\pi _{\mathcal{F}}^{m+1}F\setminus \pi _{\mathcal{F}%
}^{m}F}\sigma \right) ^{2}.
\end{eqnarray*}

Now we pigeonhole the intervals $J$ by side length in the sum over $J\in 
\mathcal{M}_{\left( \mathbf{r},\varepsilon \right) -\limfunc{deep}}^{%
\limfunc{good}}\left( F\right) $ in the first factor in braces in (\ref{at
most}) to obtain that it satisfies, under the assumptions $F^{\prime }\in 
\mathfrak{C}_{\mathcal{F}}\left( F^{\ast }\right) $ and $F\in \mathfrak{C}_{%
\mathcal{F}}^{\left( m\right) }\left( F^{\prime }\right) $, 
\begin{eqnarray*}
&&\sum_{J\in \mathcal{M}_{\left( \mathbf{r},\varepsilon \right) -\limfunc{%
deep}}^{\limfunc{good}}\left( F\right) }\mathrm{P}_{\kappa _{1}}^{\alpha
}\left( J,\mathbf{1}_{F^{\ast }\setminus F^{\prime }}\sigma \right)
^{2}\sum_{\left\vert \beta \right\vert =\kappa }\left\Vert \mathsf{Q}_{%
\mathcal{C}_{F}^{\mathbf{\tau }-\limfunc{shift}};J}^{\omega }\left( \frac{x}{%
\ell \left( J\right) }\right) ^{\beta }\right\Vert _{L^{2}\left( \omega
\right) }^{2} \\
&\lesssim &\sum_{k=m-1}^{\infty }\left( 2^{-k}\right) ^{2\kappa
-2\varepsilon \left( n+\kappa -\alpha \right) }\mathrm{P}_{\kappa
_{1}}^{\alpha }\left( \pi _{\mathcal{F}}^{m-1}F,\mathbf{1}_{\pi _{\mathcal{F}%
}^{m+1}F\setminus \pi _{\mathcal{F}}^{m}F}\sigma \right) ^{2}\sum_{\substack{
J\in \mathcal{M}_{\left( \mathbf{r},\varepsilon \right) -\limfunc{deep}}^{%
\limfunc{good}}\left( F\right)  \\ \ell \left( J\right) =2^{-k}\ell \left(
\pi _{\mathcal{F}}^{m-1}F\right) }}\sum_{\left\vert \beta \right\vert
=\kappa }\left\Vert \mathsf{Q}_{\mathcal{C}_{F}^{\mathbf{\tau }-\limfunc{%
shift}};J}^{\omega }\left( \frac{x}{\ell \left( J\right) }\right) ^{\beta
}\right\Vert _{L^{2}\left( \omega \right) }^{2} \\
&&\ \ \ \lesssim \sum_{k=m-1}^{\infty }\left( 2^{-k}\right) ^{2\kappa
-2\varepsilon \left( n+\kappa -\alpha \right) }\mathrm{P}_{\kappa }^{\alpha
}\left( \pi _{\mathcal{F}}^{m-1}F,\mathbf{1}_{\pi _{\mathcal{F}%
}^{m+1}F\setminus \pi _{\mathcal{F}}^{m}F}\sigma \right) ^{2}\sum_{\substack{
J\in \mathcal{M}_{\left( \mathbf{r},\varepsilon \right) -\limfunc{deep}}^{%
\limfunc{good}}\left( F\right)  \\ \ell \left( J\right) =2^{-k}\ell \left(
\pi _{\mathcal{F}}^{m-1}F\right) }}\left\vert J\right\vert _{\omega } \\
&&\ \ \ \lesssim \left( 2^{-m}\right) ^{2\kappa -2\varepsilon \left(
n+\kappa -\alpha \right) }\mathrm{P}_{\kappa _{1}}^{\alpha }\left( \pi _{%
\mathcal{F}}^{m-1}F,\mathbf{1}_{\pi _{\mathcal{F}}^{m+1}F\setminus \pi _{%
\mathcal{F}}^{m}F}\sigma \right) ^{2}\sum_{J\in \mathcal{M}_{\left( \mathbf{r%
},\varepsilon \right) -\limfunc{deep}}^{\limfunc{good}}\left( F\right)
}\left\vert J\right\vert _{\omega }
\end{eqnarray*}%
so that altogether we have%
\begin{eqnarray*}
\left\vert II_{G}\right\vert &\lesssim &\sum_{F\in \mathcal{F}}\left\{
\sum_{m=1}^{\infty }\alpha _{\mathcal{F}}\left( \pi _{\mathcal{F}%
}^{m+1}F\right) ^{2}\left( 2^{-m}\right) ^{\kappa -\varepsilon \left(
n+\kappa -\alpha \right) }\mathrm{P}_{\kappa _{1}}^{\alpha }\left( \pi _{%
\mathcal{F}}^{m-1}F,\mathbf{1}_{\pi _{\mathcal{F}}^{m+1}F\setminus \pi _{%
\mathcal{F}}^{m}F}\sigma \right) ^{2}\left\vert F\right\vert _{\omega
}\right\} ^{\frac{1}{2}}C_{\varepsilon ,\alpha }\left\Vert g_{F}\right\Vert
_{L^{2}\left( \omega \right) } \\
&\lesssim &\left\{ \sum_{m=1}^{\infty }\left( 2^{-m}\right) ^{\kappa
-\varepsilon \left( n+\kappa -\alpha \right) }\sum_{F\in \mathcal{F}}\alpha
_{\mathcal{F}}\left( \pi _{\mathcal{F}}^{m+1}F\right) ^{2}\mathrm{P}_{\kappa
_{1}}^{\alpha }\left( \pi _{\mathcal{F}}^{m-1}F,\mathbf{1}_{\pi _{\mathcal{F}%
}^{m+1}F\setminus \pi _{\mathcal{F}}^{m}F}\sigma \right) ^{2}\left\vert
F\right\vert _{\omega }\right\} ^{\frac{1}{2}}\sqrt{\sum_{F\in \mathcal{F}%
}\left\Vert g_{F}\right\Vert _{L^{2}\left( \omega \right) }^{2}} \\
&\lesssim &\left( \mathcal{V}_{2}^{\alpha ,\kappa _{1}}\right)
^{2}\left\Vert f\right\Vert _{L^{2}\left( \sigma \right) }\left\Vert
g\right\Vert _{L^{2}\left( \omega \right) }\ ,
\end{eqnarray*}%
since $\kappa -\varepsilon \left( n+\kappa -\alpha \right) >0$ implies $%
C_{\varepsilon ,\alpha }=\sqrt{\sum_{m=1}^{\infty }\left( 2^{-m}\right)
^{\kappa -\varepsilon \left( n+\kappa -\alpha \right) }}<\infty $, and since
for each fixed $m\geq 1$ we have%
\begin{eqnarray*}
&&\sum_{F\in \mathcal{F}}\alpha _{\mathcal{F}}\left( \pi _{\mathcal{F}%
}^{m+1}F\right) ^{2}\mathrm{P}_{\kappa _{1}}^{\alpha }\left( \pi _{\mathcal{F%
}}^{m-1}F,\mathbf{1}_{\pi _{\mathcal{F}}^{m+1}F\setminus \pi _{\mathcal{F}%
}^{m}F}\sigma \right) ^{2}\left\vert F\right\vert _{\omega } \\
&=&\sum_{F^{\prime }\in \mathcal{F}}\alpha _{\mathcal{F}}\left( \pi _{%
\mathcal{F}}F^{\prime }\right) ^{2}\sum_{F^{\prime \prime }\in \mathfrak{C}_{%
\mathcal{F}}\left( F^{\prime }\right) }\mathrm{P}_{\kappa _{1}}^{\alpha
}\left( F^{\prime \prime },\mathbf{1}_{\pi _{\mathcal{F}}F^{\prime
}\setminus F^{\prime }}\sigma \right) ^{2}\sum_{F\in \mathfrak{C}_{\mathcal{F%
}}^{m-1}\left( F^{\prime \prime }\right) }\left\vert F\right\vert _{\omega }
\\
&\leq &\sum_{F^{\prime }\in \mathcal{F}}\alpha _{\mathcal{F}}\left( \pi _{%
\mathcal{F}}F^{\prime }\right) ^{2}\sum_{F^{\prime \prime }\in \mathfrak{C}_{%
\mathcal{F}}\left( F^{\prime }\right) }\mathrm{P}_{\kappa _{1}}^{\alpha
}\left( F^{\prime \prime },\mathbf{1}_{\pi _{\mathcal{F}}F^{\prime
}\setminus F^{\prime }}\sigma \right) ^{2}\left\vert F^{\prime \prime
}\right\vert _{\omega } \\
&=&\sum_{F^{\ast }\in \mathcal{F}}\alpha _{\mathcal{F}}\left( F^{\ast
}\right) ^{2}\sum_{F^{\prime \prime }\in \mathfrak{C}_{\mathcal{F}}^{\left(
2\right) }\left( F^{\ast }\right) }\mathrm{P}_{\kappa _{1}}^{\alpha }\left(
F^{\prime \prime },\mathbf{1}_{F^{\ast }\setminus \pi _{\mathcal{F}%
}F^{\prime \prime }}\sigma \right) ^{2}\left\vert F^{\prime \prime
}\right\vert _{\omega } \\
&\leq &\sum_{F^{\prime }\in \mathcal{F}}\alpha _{\mathcal{F}}\left( \pi _{%
\mathcal{F}}F^{\prime }\right) ^{2}\left( \mathcal{V}_{2}^{\alpha ,\kappa
_{1}}\right) ^{2}\left\vert \pi _{\mathcal{F}}F^{\prime }\right\vert
_{\sigma }\leq \left( \mathcal{V}_{2}^{\alpha ,\kappa _{1}}\right)
^{2}\left\Vert f\right\Vert _{L^{2}\left( \sigma \right) }^{2}\ .
\end{eqnarray*}

In term $II_{B}$ the expressions $\left\Vert \left\vert x-m_{J}^{\kappa
_{1}}\right\vert ^{\kappa _{1}}\right\Vert _{L^{2}\left( \omega \right)
}^{2} $ are no longer `almost orthogonal' in $J$, and we must instead
exploit the extra decay in the Poisson integral $\mathrm{P}_{\kappa +\delta
^{\prime }}^{\alpha }$ due to the addition of $\delta ^{\prime }>0$, along
with goodness of the cubes $J$. This idea was already used by M. Lacey and
B. Wick in \cite{LaWi} in a similar situation, and subsequently exploited in 
\cite{SaShUr7}. As a consequence of this decay we will be able to bound $%
II_{B}$ \emph{directly} by the $\kappa ^{th}$ order pivotal condition,
without having to invoke the more difficult functional energy condition of 
\cite{LaSaShUr3} and \cite{SaShUr7}. For the decay, we follow \cite{SaShUr7}
and use the `large' function 
\begin{equation*}
\Phi \equiv \sum_{F^{\prime \prime }\in \mathcal{F}}\alpha _{\mathcal{F}%
}\left( F^{\prime \prime }\right) \mathbf{1}_{F^{\prime \prime }}
\end{equation*}%
that dominates $\left\vert \psi _{F}\right\vert $ for all $F\in \mathcal{F}$%
, and compute that%
\begin{eqnarray*}
\frac{\mathrm{P}_{\kappa _{1}+\delta ^{\prime }}^{\alpha }\left( J,\Phi
\sigma \right) }{\left\vert J\right\vert ^{\frac{\kappa }{n}}}
&=&\int_{F^{c}}\frac{\left\vert J\right\vert ^{\frac{\delta ^{\prime }}{n}}}{%
\left\vert y-c_{J}\right\vert ^{n+\kappa +\delta -\alpha }}\Phi \left(
y\right) d\sigma \left( y\right) \\
&\leq &\sum_{t=0}^{\infty }\int_{\pi _{\mathcal{F}}^{t+1}F\setminus \pi _{%
\mathcal{F}}^{t}F}\left( \frac{\left\vert J\right\vert ^{\frac{1}{n}}}{%
\limfunc{dist}\left( c_{J},\left( \pi _{\mathcal{F}}^{t}F\right) ^{c}\right) 
}\right) ^{\delta ^{\prime }}\frac{1}{\left\vert y-c_{J}\right\vert
^{n+\kappa _{1}-\alpha }}\Phi \left( y\right) d\sigma \left( y\right) \\
&\leq &\sum_{t=0}^{\infty }\left( \frac{\left\vert J\right\vert ^{\frac{1}{n}%
}}{\limfunc{dist}\left( c_{J},\left( \pi _{\mathcal{F}}^{t}F\right)
^{c}\right) }\right) ^{\delta ^{\prime }}\frac{\mathrm{P}_{\kappa
_{1}}^{\alpha }\left( J,\mathbf{1}_{\pi _{\mathcal{F}}^{t+1}F\setminus \pi _{%
\mathcal{F}}^{t}F}\Phi \sigma \right) }{\left\vert J\right\vert ^{\frac{%
\kappa }{n}}},
\end{eqnarray*}%
and then use the goodness inequality%
\begin{equation*}
\limfunc{dist}\left( c_{J},\left( \pi _{\mathcal{F}}^{t}F\right) ^{c}\right)
\geq \frac{1}{2}\ell \left( \pi _{\mathcal{F}}^{t}F\right) ^{1-\varepsilon
}\ell \left( J\right) ^{\varepsilon }\geq \frac{1}{2}2^{t\left(
1-\varepsilon \right) }\ell \left( F\right) ^{1-\varepsilon }\ell \left(
J\right) ^{\varepsilon }\geq 2^{t\left( 1-\varepsilon \right) -1}\ell \left(
J\right) ,
\end{equation*}%
to conclude that%
\begin{eqnarray}
\left( \frac{\mathrm{P}_{\kappa _{1}+\delta ^{\prime }}^{\alpha }\left( J,%
\mathbf{1}_{F^{c}}\Phi \sigma \right) }{\left\vert J\right\vert ^{\frac{%
\kappa }{n}}}\right) ^{2} &\lesssim &\left( \sum_{t=0}^{\infty }2^{-t\delta
^{\prime }\left( 1-\varepsilon \right) }\frac{\mathrm{P}_{\kappa
_{1}}^{\alpha }\left( J,\mathbf{1}_{\pi _{\mathcal{F}}^{t+1}F\setminus \pi _{%
\mathcal{F}}^{t}F}\Phi \sigma \right) }{\left\vert J\right\vert ^{\frac{%
\kappa }{n}}}\right) ^{2}  \label{decay in t} \\
&\lesssim &\sum_{t=0}^{\infty }2^{-t\delta ^{\prime }\left( 1-\varepsilon
\right) }\left( \frac{\mathrm{P}_{\kappa _{1}}^{\alpha }\left( J,\mathbf{1}%
_{\pi _{\mathcal{F}}^{t+1}F\setminus \pi _{\mathcal{F}}^{t}F}\Phi \sigma
\right) }{\left\vert J\right\vert ^{\frac{\kappa }{n}}}\right) ^{2}.  \notag
\end{eqnarray}%
Now we apply Cauchy-Schwarz to obtain%
\begin{eqnarray*}
II_{B} &=&\sum_{F\in \mathcal{F}}\sum_{J\in \mathcal{M}_{\left( \mathbf{r}%
,\varepsilon \right) -\limfunc{deep}}\left( F\right) }\frac{\mathrm{P}%
_{\kappa _{1}+\delta ^{\prime }}^{\alpha }\left( J,\mathbf{1}_{F^{c}}\Phi
\sigma \right) }{\left\vert J\right\vert ^{\frac{\kappa }{n}}}\left\Vert
\left\vert x-m_{J}^{\kappa _{1}}\right\vert ^{\kappa _{1}}\right\Vert
_{L^{2}\left( \omega \right) }\left\Vert \mathsf{P}_{J}^{\omega
}g_{F}\right\Vert _{L_{\mathcal{H}_{2}}^{2}\left( \omega \right) } \\
&\leq &\left( \sum_{F\in \mathcal{F}}\sum_{J\in \mathcal{M}_{\left( \mathbf{r%
},\varepsilon \right) -\limfunc{deep}}\left( F\right) }\left( \frac{\mathrm{P%
}_{\kappa _{1}+\delta ^{\prime }}^{\alpha }\left( J,\mathbf{1}_{F^{c}}\Phi
\sigma \right) }{\left\vert J\right\vert ^{\frac{\kappa }{n}}}\right)
^{2}\left\Vert \left\vert x-m_{J}^{\kappa _{1}}\right\vert ^{\kappa
_{1}}\right\Vert _{L^{2}\left( \omega \right) }^{2}\right) ^{\frac{1}{2}}%
\left[ \sum_{F}\left\Vert g_{F}\right\Vert _{L_{\mathcal{H}_{2}}^{2}\left(
\omega \right) }^{2}\right] ^{\frac{1}{2}} \\
&\equiv &\sqrt{II_{\limfunc{energy}}}\left[ \sum_{F}\left\Vert
g_{F}\right\Vert _{L_{\mathcal{H}_{2}}^{2}\left( \omega \right) }^{2}\right]
^{\frac{1}{2}},
\end{eqnarray*}%
and it remains to estimate $II_{\limfunc{energy}}$. From (\ref{decay in t})
and the $\kappa ^{th}$ order pivotal condition we have%
\begin{eqnarray*}
II_{\limfunc{energy}} &\leq &\sum_{F\in \mathcal{F}}\sum_{J\in \mathcal{M}%
_{\left( \mathbf{r},\varepsilon \right) -\limfunc{deep}}\left( F\right)
}\sum_{t=0}^{\infty }2^{-t\delta ^{\prime }\left( 1-\varepsilon \right)
}\left( \frac{\mathrm{P}_{\kappa _{1}}^{\alpha }\left( J,\mathbf{1}_{\pi _{%
\mathcal{F}}^{t+1}F\setminus \pi _{\mathcal{F}}^{t}F}\Phi \sigma \right) }{%
\left\vert J\right\vert ^{\frac{\kappa }{n}}}\right) ^{2}\left\Vert
\left\vert x-m_{J}^{\kappa _{1}}\right\vert ^{\kappa _{1}}\right\Vert
_{L^{2}\left( \omega \right) }^{2} \\
&=&\sum_{t=0}^{\infty }2^{-t\delta ^{\prime }\left( 1-\varepsilon \right)
}\sum_{G\in \mathcal{F}}\sum_{F\in \mathfrak{C}_{\mathcal{F}}^{\left(
t+1\right) }\left( G\right) }\sum_{J\in \mathcal{M}_{\left( \mathbf{r}%
,\varepsilon \right) -\limfunc{deep}}\left( F\right) }\left( \frac{\mathrm{P}%
_{\kappa _{1}}^{\alpha }\left( J,\mathbf{1}_{G\setminus \pi _{\mathcal{F}%
}^{t}F}\Phi \sigma \right) }{\left\vert J\right\vert ^{\frac{\kappa }{n}}}%
\right) ^{2}\left\vert J\right\vert ^{\frac{k}{n}}\left\vert J\right\vert
_{\omega } \\
&\lesssim &\sum_{t=0}^{\infty }2^{-t\delta ^{\prime }\left( 1-\varepsilon
\right) }\sum_{G\in \mathcal{F}}\alpha _{\mathcal{F}}\left( G\right)
^{2}\sum_{F\in \mathfrak{C}_{\mathcal{F}}^{\left( t+1\right) }\left(
G\right) }\sum_{J\in \mathcal{M}_{\left( \mathbf{r},\varepsilon \right) -%
\limfunc{deep}}\left( F\right) }\mathrm{P}_{\kappa _{1}}^{\alpha }\left( J,%
\mathbf{1}_{G\setminus \pi _{\mathcal{F}}^{t}F}\sigma \right) ^{2}\left\vert
J\right\vert _{\omega } \\
&\lesssim &\sum_{t=0}^{\infty }2^{-t\delta ^{\prime }\left( 1-\varepsilon
\right) }\sum_{G\in \mathcal{F}}\alpha _{\mathcal{F}}\left( G\right)
^{2}\left( \left( \mathcal{V}_{\alpha }^{\alpha ,\kappa _{1}}\right)
^{2}+A_{2}^{\alpha }\right) \left\vert G\right\vert _{\sigma }\lesssim
\left( \left( \mathcal{V}_{\alpha }^{\alpha ,\kappa _{1}}\right)
^{2}+A_{2}^{\alpha }\right) \left\Vert f\right\Vert _{L^{2}\left( \sigma
\right) }^{2}.
\end{eqnarray*}

This completes the proof of the Intertwining Proposition \ref{strongly
adapted}.
\end{proof}

\subsubsection{An alternate Intertwining Corollary}

We will also need an alternate version of the Intertwining Proposition \ref%
{strongly adapted} in which $J$ and $I$ are at least $\mathbf{\tau }$ levels
apart, but where the proximity of $J$ and $I$ to $F$ is reversed, namely the
cubes $J$ are close to $F$ but the cubes $I$ are not. We exploit the
doubling property of $\sigma $ to obtain this alternate version as a
relatively simple corollary of the Intertwining Proposition \ref{strongly
adapted}.

\begin{corollary}[The Alternate Intertwining Corollary]
\label{alt int prop}Suppose that $\sigma $ is a \emph{doubling} measure,
that $\mathcal{F}$ is $\sigma $-Carleson, that $\left( C_{0},\mathcal{F}%
,\alpha _{\mathcal{F};f}\right) $ constitutes stopping data for $f$ for all $%
f\in L^{2}\left( \sigma \right) $, and that%
\begin{equation*}
\left\Vert \bigtriangleup _{I;\kappa _{1}}^{\sigma }f\right\Vert _{L^{\infty
}\left( \sigma \right) }\leq C\alpha _{\mathcal{F};f}\left( I\right) ,\ \ \
\ \ f\in L^{2}\left( \sigma \right) ,I\in \mathcal{D}.
\end{equation*}%
Let $\mathcal{W}_{F}$ be any subset of $\mathcal{C}_{\mathcal{F}}\left(
F\right) $. Then for $\limfunc{good}$ functions $f\in L^{2}\left( \sigma
\right) $ and $g\in L^{2}\left( \omega \right) $, and with $\kappa
_{1},\kappa _{2}\geq 1$, we have%
\begin{equation*}
\left\vert \sum_{F\in \mathcal{F}}\ \sum_{I:\ I\supsetneqq \pi _{\mathcal{D}%
}^{\left( \mathbf{\tau }\right) }F}\ \left\langle T_{\sigma }^{\alpha
}\bigtriangleup _{I;\kappa _{1}}^{\sigma }f,\mathsf{P}_{\mathcal{W}%
_{F}}^{\omega }g\right\rangle _{\omega }\right\vert \lesssim \left( \mathcal{%
V}_{2}^{\alpha ,\kappa _{1}}+\sqrt{A_{2}^{\alpha }}+\mathfrak{T}_{T^{\alpha
}}^{\left( \kappa _{1}\right) }\right) \ \left\Vert f\right\Vert
_{L^{2}\left( \sigma \right) }\left\Vert g\right\Vert _{L^{2}\left( \omega
\right) }.
\end{equation*}
\end{corollary}

Note that the cubes $J$ in $\mathcal{W}_{F}$ can be close to $F$, but that
the cubes $I$ with $I\supsetneqq \pi _{\mathcal{D}}^{\left( \mathbf{\tau }%
\right) }F$ are far from $F$.

\begin{proof}
We will apply the Intertwining Proposition \ref{strongly adapted} to
stopping data $\left( C_{0},\mathcal{H},\alpha _{\mathcal{H};f}\right) $
derived from the $\mathbf{\tau }$-grandparents of cubes in $\mathcal{F}$,
where%
\begin{equation*}
\mathcal{H}\equiv \left\{ \pi _{\mathcal{D}}^{\left( \mathbf{\tau }\right)
}A:A\in \mathcal{A}\right\} .
\end{equation*}%
Since $\sigma $ is doubling we conclude that the collection of $\mathbf{\tau 
}$-grandparents $\mathcal{H}$ also satisfies a $\sigma $-Carleson condition.
In fact, if $H=\pi _{\mathcal{D}}^{\left( \mathbf{\tau }\right) }A\subset
\pi _{\mathcal{D}}^{\left( \mathbf{\tau }\right) }B=K$, then $A\subset K$,
and so if $\mathcal{M}_{K}^{\left( \mathbf{\tau }\right) }$ is the
collection of maximal cubes $A\in \mathcal{A}$ for which $\pi _{\mathcal{D}%
}^{\left( \mathbf{\tau }\right) }A\subset K$, we have 
\begin{eqnarray*}
\sum_{H\in \mathcal{H}:\ H\subset K}\left\vert H\right\vert _{\sigma }
&=&\sum_{A\in \mathcal{A}:\ \pi _{\mathcal{D}}^{\left( \mathbf{\tau }\right)
}A\subset K}\left\vert \pi _{\mathcal{D}}^{\left( \mathbf{\tau }\right)
}A\right\vert _{\sigma }=\sum_{M\in \mathcal{M}_{K}^{\left( \mathbf{\tau }%
\right) }}\sum_{A\in \mathcal{A}:\ A\subset M}\left\vert \pi _{\mathcal{D}%
}^{\left( \mathbf{\tau }\right) }A\right\vert _{\sigma } \\
&\leq &C_{\mathbf{\tau }}\sum_{M\in \mathcal{M}_{K}^{\left( \mathbf{\tau }%
\right) }}\sum_{A\in \mathcal{A}:\ A\subset M}\left\vert A\right\vert
_{\sigma }\leq C_{\mathbf{\tau }}C_{\func{Carleson}}\sum_{M\in \mathcal{M}%
_{K}^{\left( \mathbf{\tau }\right) }}\left\vert M\right\vert _{\sigma }\leq
C_{\mathbf{\tau }}C_{\func{Carleson}}\left\vert M\right\vert _{\sigma }\ .
\end{eqnarray*}%
Moreover, from this $\sigma $-Carleson condition, and the generalized
Carleson Embedding Theorem, we obtain the following quasi-orthogonality
inequality%
\begin{equation}
\sum_{H\in \mathcal{H}}\left\vert H\right\vert _{\sigma }\left(
\sup_{H^{\prime }\in \mathcal{D}:\ H^{\prime }\supset H}\frac{1}{\left\vert
H^{\prime }\right\vert _{\sigma }}\int_{H^{\prime }}\left\vert f\right\vert
d\sigma \right) ^{2}\lesssim \left\Vert f\right\Vert _{L^{2}\left( \sigma
\right) }^{2}\ .  \label{parent quasi}
\end{equation}%
Indeed, this follows from interpolating the trivial estimate $A:L^{\infty
}\left( \sigma \right) \rightarrow \ell ^{\infty }\left( \mathcal{H}\right) $
for the sublinear operator $Af\left( H\right) \equiv \sup_{H^{\prime }\in 
\mathcal{D}:\ H^{\prime }\supset H}E_{H^{\prime }}^{\sigma }\left\vert
f\right\vert $ with the weak type estimate $A:L^{1}\left( \sigma \right)
\rightarrow \ell ^{1,\infty }\left( \mathcal{H}\right) $, which in turn
follows by applying the Carleson condition to the maximal cubes $M$ for
which $Af\left( M\right) >\lambda $, $\lambda >0$. Finally, set 
\begin{equation*}
\alpha _{\mathcal{H};f}\left( H\right) \equiv \sup_{H^{\prime }\in \mathcal{D%
}:\ H^{\prime }\supset H}\frac{1}{\left\vert H^{\prime }\right\vert _{\sigma
}}\int_{H^{\prime }}\left\vert f\right\vert d\sigma ,\ \ \ \ \ H\in \mathcal{%
H},
\end{equation*}%
so that the triple $\left( C_{0},\mathcal{H},\alpha _{\mathcal{H};f}\right) $
constitutes stopping data for the function $f\in L^{2}\left( \sigma \right) $
in the sense of Definition \ref{general stopping data}. Now define an Alpert
projection $\widehat{g}$ so that%
\begin{equation*}
\sum_{F\in \mathcal{F}}\mathsf{P}_{\mathcal{W}_{F}}^{\omega }g=\sum_{F\in 
\mathcal{F}}\mathsf{P}_{\mathcal{C}_{\mathcal{F}}\left( F\right) }^{\omega }%
\widehat{g}.
\end{equation*}%
Then $\left\Vert \widehat{g}\right\Vert _{L^{2}\left( \omega \right) }\leq
\left\Vert g\right\Vert _{L^{2}\left( \omega \right) }$ and the Intertwining
Proposition \ref{strongly adapted} yields%
\begin{equation*}
\left\vert \sum_{H\in \mathcal{H}}\ \sum_{I:\ I\supsetneqq H}\ \left\langle
T_{\sigma }^{\alpha }\bigtriangleup _{I;\kappa _{1}}^{\sigma }f,\mathsf{P}_{%
\mathcal{C}_{\mathcal{H}}^{\mathbf{\tau }-\limfunc{shift}}\left( H\right)
}^{\omega }\widehat{g}\right\rangle _{\omega }\right\vert \lesssim \left( 
\mathcal{V}_{2}^{\alpha ,\kappa _{1}}+\sqrt{A_{2}^{\alpha }}+\mathfrak{T}%
_{T^{\alpha }}^{\left( \kappa _{1}\right) }\right) \ \left\Vert f\right\Vert
_{L^{2}\left( \sigma \right) }\left\Vert g\right\Vert _{L^{2}\left( \omega
\right) }.
\end{equation*}%
Unravelling the definitions shows that this inequality is precisely the
conclusion of the Alternate Intertwining Corollary \ref{alt int prop}.
\end{proof}

\subsection{The Parallel Corona\label{details}}

Armed with the Monotonicity Lemma and the Intertwining Proposition in the
previous two subsections, we can now give the proof of Theorem \ref{pivotal
theorem}, for which it suffices to show that%
\begin{equation}
\left\vert \left\langle T_{\sigma }^{\alpha }f,g\right\rangle _{L^{2}\left(
\omega \right) }\right\vert \lesssim \left( \mathfrak{T}_{T^{\alpha }}+%
\mathfrak{T}_{T^{\alpha }}^{\ast }+\mathcal{BICT}_{T^{\alpha }}+\mathcal{V}%
_{2}^{\alpha ,\kappa _{1}}+\mathcal{V}_{2}^{\alpha ,\kappa _{2},\ast }+\sqrt{%
\mathcal{A}_{2}^{\alpha }}+\sqrt{\mathcal{A}_{2}^{\alpha ,\ast }}\right)
\left\Vert f\right\Vert _{L^{2}\left( \sigma \right) }\left\Vert
g\right\Vert _{L^{2}\left( \omega \right) }\ ,  \label{once we show}
\end{equation}%
since by (\ref{piv control}), 
\begin{equation*}
\mathcal{V}_{2}^{\alpha ,\kappa _{1}}+\mathcal{V}_{2}^{\alpha ,\kappa
_{2},\ast }\leq C_{n,\alpha ,\kappa _{1,}\left( \beta _{1},\gamma
_{1}\right) ,\kappa _{2,}\left( \beta _{2},\gamma _{2}\right) }\sqrt{%
A_{2}^{\alpha }},\ \ \ \ \ \text{for }\kappa _{1}>\theta _{1}+\alpha -n\text{
and }\kappa _{2}>\theta _{2}+\alpha -n\ .
\end{equation*}%
Note that as above we are abbreviating $\mathfrak{T}_{\left( T^{\alpha
}\right) ^{\ast }}^{\left( \kappa \right) }\left( \omega ,\sigma \right) $
with $\mathfrak{T}_{T^{\alpha }}^{\left( \kappa \right) ,\ast }$.

As a first step, we will prove the weaker inequality%
\begin{equation}
\left\vert \left\langle T_{\sigma }^{\alpha }f,g\right\rangle _{L^{2}\left(
\omega \right) }\right\vert \lesssim \left( \mathfrak{T}_{T^{\alpha
}}^{\left( \kappa \right) }+\mathfrak{T}_{T^{\alpha }}^{\left( \kappa
\right) ,\ast }+\mathcal{BICT}_{T^{\alpha }}+\mathcal{V}_{2}^{\alpha ,\kappa
_{1}}+\mathcal{V}_{2}^{\alpha ,\kappa _{2},\ast }+\sqrt{A_{2}^{\alpha }}+%
\mathcal{WBP}_{T^{\alpha }}^{\left( \kappa _{1},\kappa _{2}\right) }\right)
\left\Vert f\right\Vert _{L^{2}\left( \sigma \right) }\left\Vert
g\right\Vert _{L^{2}\left( \omega \right) }\ ,  \label{once we show'}
\end{equation}%
in which we only need the classical Muckenhoupt constant $A_{2}^{\alpha }$,
and then replace $\kappa $-testing with $1$-testing, and\ remove the $\kappa
^{th}$-order weak boundedness constant $\mathcal{WBP}_{T^{\alpha }}^{\left(
\kappa _{1},\kappa _{2}\right) }\left( \sigma ,\omega \right) $ all at the
price of using the one-tailed constants $\mathcal{A}_{2}^{\alpha },\mathcal{A%
}_{2}^{\alpha ,\ast }$ instead of $A_{2}^{\alpha }$.

A crucial result of Nazarov, Treil and Volberg in \cite{NTV1}, \cite{NTV3}
and \cite{NTV4} shows that all of the cubes $I$ and $J$ in the sum 
\begin{equation*}
\left\langle T_{\sigma }^{\alpha }f,g\right\rangle _{L^{2}\left( \omega
\right) }=\dsum\limits_{I,J\in \mathcal{D}}\left\langle T_{\sigma }^{\alpha
}\left( \bigtriangleup _{I;\kappa _{1}}^{\sigma }f\right) ,\bigtriangleup
_{J;\kappa _{2}}^{\omega }g\right\rangle _{L^{2}\left( \omega \right) }
\end{equation*}%
may be assumed $\left( \mathbf{r},\varepsilon \right) $-$\limfunc{good}$.

\subsubsection{The Calder\'{o}n-Zygmund corona construction}

Let $\mu $ be a locally finite positive Borel measure on $\mathbb{R}^{n}$.
Let $\mathcal{F}$ be a collection of Calder\'{o}n-Zygmund stopping cubes for 
$f$, and let $\mathcal{D}=\dbigcup\limits_{F\in \mathcal{F}}\mathcal{C}_{%
\mathcal{F}}\left( F\right) $ be the associated corona decomposition of the
dyadic grid $\mathcal{D}$. Then we have 
\begin{eqnarray*}
E_{F^{\prime }}^{\mu }\left\vert f\right\vert &>&C_{0}E_{F}^{\mu }\left\vert
f\right\vert \text{ whenever }F^{\prime },F\in \mathcal{F}\text{ with }%
F^{\prime }\subsetneqq F, \\
E_{I}^{\mu }\left\vert f\right\vert &\leq &C_{0}E_{F}^{\mu }\left\vert
f\right\vert \text{ for }I\in \mathcal{C}_{\mathcal{F}}\left( F\right) .
\end{eqnarray*}%
For a cube $I\in \mathcal{D}$ let $\pi _{\mathcal{D}}I$ be the $\mathcal{D}$%
-parent of $I$ in the grid $\mathcal{D}$, and let $\pi _{\mathcal{F}}I$ be
the smallest member of $\mathcal{F}$ that contains $I$. For $F,F^{\prime
}\in \mathcal{F}$, we say that $F^{\prime }$ is an $\mathcal{F}$-child of $F$
if $\pi _{\mathcal{F}}\left( \pi _{\mathcal{D}}F^{\prime }\right) =F$ (it
could be that $F=\pi _{\mathcal{D}}F^{\prime }$), and we denote by $%
\mathfrak{C}_{\mathcal{F}}\left( F\right) $ the set of $\mathcal{F}$%
-children of $F$.

For $F\in \mathcal{F}$, define the projection $\mathsf{P}_{\mathcal{C}_{%
\mathcal{F}}\left( F\right) }^{\mu }$ onto the linear span of the Alpert
functions $\left\{ h_{I;\kappa }^{\mu ,a}\right\} _{I\in \mathcal{C}_{F},\
a\in \Gamma _{I,n.\kappa }}$ by%
\begin{equation*}
\mathsf{P}_{\mathcal{C}_{\mathcal{F}}\left( F\right) }^{\mu }f=\sum_{I\in 
\mathcal{C}_{\mathcal{F}}\left( F\right) }\bigtriangleup _{I;\kappa }^{\mu
}f=\sum_{I\in \mathcal{C}_{F},\ a\in \Gamma _{I,n.\kappa }}\left\langle
f,h_{I;\kappa }^{\mu ,a}\right\rangle _{L^{2}\left( \sigma \right)
}h_{I;\kappa }^{\mu ,a}.
\end{equation*}%
The standard properties of these projections are%
\begin{equation*}
f=\sum_{F\in \mathcal{F}}\mathsf{P}_{\mathcal{C}_{\mathcal{F}}\left(
F\right) }^{\mu }f,\ \ \ \ \ \int \left( \mathsf{P}_{\mathcal{C}_{\mathcal{F}%
}\left( F\right) }^{\mu }f\right) d\mu =0,\ \ \ \ \ \left\Vert f\right\Vert
_{L^{2}\left( \mu \right) }^{2}=\sum_{F\in \mathcal{F}}\left\Vert \mathsf{P}%
_{\mathcal{C}_{\mathcal{F}}\left( F\right) }^{\mu }f\right\Vert
_{L^{2}\left( \mu \right) }^{2}.
\end{equation*}%
There is also a $\mu $-Carleson condition satisfied by the stopping cubes,
namely%
\begin{equation*}
\sum_{F^{\prime }\in \mathcal{F}:\ F^{\prime }\subset F}\left\vert F^{\prime
}\right\vert _{\mu }\leq C_{0}\left\vert F\right\vert _{\mu }\text{ for all }%
F\in \mathcal{F}.
\end{equation*}%
Thus with $\alpha _{\mathcal{F}}\equiv $ $E_{F}^{\mu }\left\vert
f\right\vert $, the triple $\left( C_{0},\mathcal{F},\alpha _{\mathcal{F}%
}\right) $ constitutes stopping data for $f$ in the sense of \cite{LaSaShUr3}%
, i.e. Definition \ref{general stopping data} above.

\begin{description}
\item[Important restriction] In the proof of Theorem \ref{pivotal theorem}
we only use the Calder\'{o}n-Zygmund corona decomposition, and in this case,
property (1) can be improved to 
\begin{equation*}
\mathbb{E}_{F}^{\mu }\left\vert f\right\vert \approx \alpha _{\mathcal{F}%
}\left( F\right) \text{ for all }F\in \mathcal{F},
\end{equation*}%
which we assume for the remainder of the proof.
\end{description}

\subsection{Form splittings and decompositions}

Let $\left( C_{0},\mathcal{A},\alpha _{\mathcal{A}}\right) $ constitute
stopping data for $f\in L^{2}\left( \sigma \right) $,\ and let $\left( C_{0},%
\mathcal{B},\alpha _{\mathcal{B}}\right) $ constitute stopping data for $%
g\in L^{2}\left( \omega \right) $ as in the previous subsubsection. We now
organize the bilinear form,%
\begin{eqnarray*}
\left\langle T_{\sigma }^{\alpha }f,g\right\rangle _{\omega }
&=&\left\langle T_{\sigma }^{\alpha }\left( \sum_{I\in \mathcal{D}%
}\bigtriangleup _{I;\kappa _{1}}^{\sigma }f\right) ,\left( \sum_{J\in 
\mathcal{D}}\bigtriangleup _{J;\kappa _{2}}^{\omega }g\right) \right\rangle
_{\omega }=\sum_{I\in \mathcal{D}\ \text{and }J\in \mathcal{D}}\left\langle
T_{\sigma }^{\alpha }\left( \bigtriangleup _{I;\kappa _{1}}^{\sigma
}f\right) ,\left( \bigtriangleup _{J;\kappa _{2}}^{\omega }g\right)
\right\rangle _{\omega } \\
&=&\sum_{\left( A,B\right) \in \mathcal{A}\times \mathcal{B}}\sum_{I\in 
\mathcal{C}_{\mathcal{A}}\left( A\right) \text{ and }J\in \mathcal{C}_{%
\mathcal{B}}\left( B\right) }\left\langle T_{\sigma }^{\alpha }\left(
\bigtriangleup _{I;\kappa _{1}}^{\sigma }f\right) ,\left( \bigtriangleup
_{J;\kappa _{2}}^{\omega }g\right) \right\rangle _{\omega }=\sum_{\left(
A,B\right) \in \mathcal{A}\times \mathcal{B}}\left\langle T_{\sigma
}^{\alpha }\left( \mathsf{P}_{\mathcal{C}_{\mathcal{A}}\left( A\right)
}^{\sigma }f\right) ,\mathsf{P}_{\mathcal{C}_{\mathcal{B}}\left( B\right)
}^{\omega }g\right\rangle _{\omega }\ ,
\end{eqnarray*}%
as a sum over the families of Calder\'{o}n-Zygmund stopping cubes $\mathcal{A%
}$ and $\mathcal{B}$, and then decompose this sum by the \emph{Parallel
Corona decomposition},\emph{\ }in which the `diagonal cut' in the bilinear
form is made according to the relative positions of intersecting coronas,
rather than the traditional way of making the `diagonal cut' according to
relative side lengths of cubes. The parallel corona as used here was
introduced in an unpublished manuscript on the \textit{arXiv} \cite%
{LaSaShUr4} by Lacey, Sawyer, Shen and Uriarte-Tuero that proved the
Indicator/Interval Testing characterization for the Hilbert transform, just
before Michael Lacey's breakthrough in controlling the local form \cite{Lac}%
. This manuscript was referenced in the survey article \cite[see page 21]%
{Lac2}, and subsequently used in at least \cite{Hyt3}, \cite{Tan} and \cite%
{LaSaShUrWi}.

We have%
\begin{eqnarray}
&&\left\langle T_{\sigma }^{\alpha }f,g\right\rangle _{\omega }=\sum_{\left(
A,B\right) \in \mathcal{A}\times \mathcal{B}}\left\langle T_{\sigma
}^{\alpha }\left( \mathsf{P}_{\mathcal{C}_{\mathcal{A}}\left( A\right)
}^{\sigma }f\right) ,\mathsf{P}_{\mathcal{C}_{\mathcal{B}}\left( B\right)
}^{\omega }g\right\rangle _{\omega }  \label{parallel corona decomp'} \\
&=&\left\{ \sum_{\left( A,B\right) \in \func{Near}\left( \mathcal{A}\times 
\mathcal{B}\right) }+\sum_{\left( A,B\right) \in \func{Disjoint}\left( 
\mathcal{A}\times \mathcal{B}\right) }+\sum_{\left( A,B\right) \in \func{Far}%
\left( \mathcal{A}\times \mathcal{B}\right) }\right\} \left\langle T_{\sigma
}^{\alpha }\left( \mathsf{P}_{\mathcal{C}_{\mathcal{A}}\left( A\right)
}^{\sigma }f\right) ,\mathsf{P}_{\mathcal{C}_{\mathcal{B}}\left( B\right)
}^{\omega }g\right\rangle _{\omega }  \notag \\
&\equiv &\func{Near}\left( f,g\right) +\func{Disjoint}\left( f,g\right) +%
\func{Far}\left( f,g\right) .  \notag
\end{eqnarray}%
Here $\func{Near}\left( \mathcal{A}\times \mathcal{B}\right) $ is the set of
pairs $\left( A,B\right) \in \mathcal{A}\times \mathcal{B}$ such that one of 
$A,B$ is contained in the other, and there is no $A_{1}\in \mathcal{A}$ with 
$B\subset A_{1}\subsetneqq A$, nor is there $B_{1}\in \mathcal{B}$ with $%
A\subset B_{1}\subsetneqq B$. The set $\func{Disjoint}\left( \mathcal{A}%
\times \mathcal{B}\right) $ is the set of pairs $\left( A,B\right) \in 
\mathcal{A}\times \mathcal{B}$ such that $A\cap B=\emptyset $. The set $%
\func{Far}\left( \mathcal{A}\times \mathcal{B}\right) $ is the complement of 
$\func{Near}\left( \mathcal{A}\times \mathcal{B}\right) \cup \func{Disjoint}%
\left( \mathcal{A}\times \mathcal{B}\right) $ in $\mathcal{A}\times \mathcal{%
B}$:%
\begin{equation*}
\func{Far}\left( \mathcal{A}\times \mathcal{B}\right) =\left( \mathcal{A}%
\times \mathcal{B}\right) \setminus \left\{ \func{Near}\left( \mathcal{A}%
\times \mathcal{B}\right) \cup \func{Disjoint}\left( \mathcal{A}\times 
\mathcal{B}\right) \right\} .
\end{equation*}%
Note that if $\left( A,B\right) \in \func{Far}\left( \mathcal{A}\times 
\mathcal{B}\right) $, then \textbf{either} $B\subset A^{\prime }$ for some $%
A^{\prime }\in \mathfrak{C}_{\mathcal{A}}\left( A\right) $, \textbf{or} $%
A\subset B^{\prime }$ for some $B^{\prime }\in \mathfrak{C}_{\mathcal{B}%
}\left( B\right) $.

\subsubsection{Disjoint form}

By Lemma \ref{standard delta}, the disjoint form\ $\func{Disjoint}\left(
f,g\right) $ is controlled by the $A_{2}^{\alpha }$ condition, the $\kappa $%
-cube testing conditions (\ref{kappa testing}), and the $\kappa $-weak
boundedness property (\ref{kappa WBP}):%
\begin{equation}
\left\vert \func{Disjoint}\left( f,g\right) \right\vert \lesssim \left( 
\mathfrak{T}_{\alpha }^{\left( \kappa \right) }+\mathfrak{T}_{\alpha
}^{\left( \kappa \right) ,\ast }+\mathcal{WBP}_{T^{\alpha }}^{\left( \kappa
_{1},\kappa _{2}\right) }+\sqrt{A_{2}^{\alpha }}\right) \left\Vert
f\right\Vert _{L^{2}\left( \sigma \right) }\left\Vert g\right\Vert
_{L^{2}\left( \omega \right) }.  \label{disj est}
\end{equation}

\subsubsection{Far form\label{Subsubsection 6.6.1}}

Next we control the far form 
\begin{equation*}
\func{Far}\left( f,g\right) =\sum_{\left( A,B\right) \in \func{Far}\left( 
\mathcal{A}\times \mathcal{B}\right) }\left\langle T_{\sigma }^{\alpha
}\left( \mathsf{P}_{\mathcal{C}_{\mathcal{A}}\left( A\right) }^{\sigma
}f\right) ,\mathsf{P}_{\mathcal{C}_{\mathcal{B}}\left( B\right) }^{\omega
}g\right\rangle _{\omega },
\end{equation*}%
which we first decompose into `far below' and `far above' pieces,%
\begin{eqnarray*}
\func{Far}\left( f,g\right) &=&\sum_{\substack{ \left( A,B\right) \in \func{%
Far}\left( \mathcal{A}\times \mathcal{B}\right)  \\ B\subset A}}\left\langle
T_{\sigma }^{\alpha }\left( \mathsf{P}_{\mathcal{C}_{\mathcal{A}}\left(
A\right) }^{\sigma }f\right) ,\mathsf{P}_{\mathcal{C}_{\mathcal{B}}\left(
B\right) }^{\omega }g\right\rangle _{\omega }+\sum_{\substack{ \left(
A,B\right) \in \func{Far}\left( \mathcal{A}\times \mathcal{B}\right)  \\ %
A\subset B}}\left\langle T_{\sigma }^{\alpha }\left( \mathsf{P}_{\mathcal{C}%
_{\mathcal{A}}\left( A\right) }^{\sigma }f\right) ,\mathsf{P}_{\mathcal{C}_{%
\mathcal{B}}\left( B\right) }^{\omega }g\right\rangle _{\omega } \\
&=&\func{Far}_{\limfunc{below}}\left( f,g\right) +\func{Far}_{\limfunc{above}%
}\left( f,g\right) ,
\end{eqnarray*}%
where as we noted above, if $\left( A,B\right) \in \func{Far}\left( \mathcal{%
A}\times \mathcal{B}\right) $ and $B\subset A$, then $B$ is actually `far
below' the cube $A$ in the sense that $B\subset A^{\prime }$ for some $%
A^{\prime }\in \mathfrak{C}_{\mathcal{A}}\left( A\right) $.

At this point we recall that the Intertwining Proposition \ref{strongly
adapted} was built on the \emph{shifted} corona decomposition,%
\begin{equation*}
\left\langle T_{\sigma }^{\alpha }f,g\right\rangle _{\omega
}=\sum_{A,A^{\prime }\in \mathcal{A}}\left\langle T_{\sigma }^{\alpha
}\left( \mathsf{P}_{\mathcal{C}_{\mathcal{A}}\left( A\right) }^{\sigma
}f\right) ,\mathsf{P}_{\mathcal{C}_{\mathcal{A}}^{\mathbf{\tau }-\func{shift}%
}\left( A^{\prime }\right) }^{\omega }g\right\rangle _{\omega }\ ,
\end{equation*}%
in which the shifted $\mathcal{A}$-coronas $\left\{ \mathcal{C}_{\mathcal{A}%
}^{\mathbf{\tau }-\func{shift}}\left( A^{\prime }\right) \right\}
_{A^{\prime }\in \mathcal{A}}$ are used in place of the parallel $\mathcal{B}
$-coronas $\left\{ \mathcal{C}_{\mathcal{B}}\left( B\right) \right\} _{B\in 
\mathcal{B}}$ in defining a complete set of projections in $L^{2}\left(
\omega \right) $. In fact, using that $\dbigcup\limits_{A^{\prime }\in 
\mathcal{A}:\ A^{\prime }\subsetneqq A}\mathcal{C}_{\mathcal{A}}^{\mathbf{%
\tau }-\func{shift}}\left( A^{\prime }\right) =\left\{ J\in \mathcal{D}%
:J\Subset _{\mathbf{\tau }}A\right\} $, the conclusion of the Intertwining
Proposition \ref{strongly adapted} can be written,%
\begin{eqnarray*}
\left\vert \func{Shift}\left( f,g\right) \right\vert &\lesssim &\left( 
\mathcal{V}_{2}^{\alpha ,\kappa _{1}}+\sqrt{A_{2}^{\alpha }}+\mathfrak{T}%
_{T^{\alpha }}^{\left( \kappa _{1}\right) }\right) \ \left\Vert f\right\Vert
_{L^{2}\left( \sigma \right) }\left\Vert g\right\Vert _{L^{2}\left( \omega
\right) }; \\
\text{where }\func{Shift}\left( f,g\right) &\equiv &\sum_{A\in \mathcal{A}%
}\sum_{I:\ I\supsetneqq A}\sum_{J\in \mathcal{D}:\ J\Subset _{\mathbf{\tau }%
}A}\left\langle T_{\sigma }^{\alpha }\bigtriangleup _{I;\kappa _{1}}^{\sigma
}f,\bigtriangleup _{J;\kappa _{2}}^{\omega }g\right\rangle _{\omega }\ .
\end{eqnarray*}

We now wish to apply this estimate to the far below form $\func{Far}_{%
\limfunc{below}}\left( f,g\right) $ in the parallel corona decomposition,
and for this we write%
\begin{eqnarray*}
\func{Far}_{\limfunc{below}}\left( f,g\right) &=&\sum_{\left( A,B\right) \in 
\func{Far}\left( \mathcal{A}\times \mathcal{B}\right) :\ B\subset
A}\left\langle T_{\sigma }^{\alpha }\left( \mathsf{P}_{\mathcal{C}_{\mathcal{%
A}}\left( A\right) }^{\sigma }f\right) ,\mathsf{P}_{\mathcal{C}_{\mathcal{B}%
}\left( B\right) }^{\omega }g\right\rangle _{\omega } \\
&=&\sum_{A\in \mathcal{A}}\left\langle T_{\sigma }^{\alpha }\left( \mathsf{P}%
_{\mathcal{C}_{\mathcal{A}}\left( A\right) }^{\sigma }f\right) ,\sum_{B\in 
\mathcal{B}:\ \left( A,B\right) \in \func{Far}\left( \mathcal{A}\times 
\mathcal{B}\right) \text{ and }B\subset A}\mathsf{P}_{\mathcal{C}_{\mathcal{B%
}}\left( B\right) }^{\omega }g\right\rangle _{\omega } \\
&=&\sum_{A\in \mathcal{A}}\left\langle T_{\sigma }^{\alpha }\left( \mathsf{P}%
_{\mathcal{C}_{\mathcal{A}}\left( A\right) }^{\sigma }f\right)
,\sum_{A^{\prime }\in \mathcal{A}:\ A^{\prime }\subsetneqq A}\sum_{B\in 
\mathcal{B}\cap \mathcal{C}_{\mathcal{A}}\left( A^{\prime }\right) }\mathsf{P%
}_{\mathcal{C}_{\mathcal{A}}\left( A^{\prime }\right) \cap \mathcal{C}_{%
\mathcal{B}}\left( B\right) }^{\omega }g\right\rangle _{\omega } \\
&=&\sum_{A^{\prime }\in \mathcal{A}}\sum_{I:\ I\supsetneqq A^{\prime
}}\sum_{B\in \mathcal{B}\cap \mathcal{C}_{\mathcal{A}}\left( A^{\prime
}\right) }\left\langle T_{\sigma }^{\alpha }\bigtriangleup _{I;\kappa
_{1}}^{\sigma }f,\mathsf{P}_{\mathcal{C}_{\mathcal{A}}\left( A^{\prime
}\right) \cap \mathcal{C}_{\mathcal{B}}\left( B\right) }^{\omega
}g\right\rangle _{\omega } \\
&=&\sum_{A^{\prime }\in \mathcal{A}}\sum_{I:\ I\supsetneqq A^{\prime
}}\sum_{J\in \mathcal{C}_{\mathcal{A}}\left( A^{\prime }\right) :\ J\subset
B\subset A^{\prime }\text{ for some }B\in \mathcal{B}}\left\langle T_{\sigma
}^{\alpha }\bigtriangleup _{I;\kappa _{1}}^{\sigma }f,\bigtriangleup
_{J;\kappa _{2}}^{\omega }g\right\rangle _{\omega }\ .
\end{eqnarray*}%
If we now replace $A^{\prime }$ with $A$ in the last line, then the
difference between forms is given by%
\begin{eqnarray}
&&\func{Far}_{\limfunc{below}}\left( f,g\right) -\func{Shift}\left(
f,g\right)  \label{Far-Shift} \\
&=&\sum_{A\in \mathcal{A}}\sum_{I:\ I\supsetneqq A}\left\{ \sum_{J\in 
\mathcal{C}_{\mathcal{A}}\left( A\right) :\ J\subset B\subset A\text{ for
some }B\in \mathcal{B}}-\sum_{J\in \mathcal{D}:\ J\Subset _{\mathbf{\tau }%
}A}\right\} \left\langle T_{\sigma }^{\alpha }\bigtriangleup _{I;\kappa
_{1}}^{\sigma }f,\bigtriangleup _{J;\kappa _{2}}^{\omega }g\right\rangle
_{\omega }  \notag \\
&=&\sum_{A\in \mathcal{A}}\sum_{I:\ I\supsetneqq A}\left\{ \sum_{J\in 
\mathcal{W}_{A}}-\sum_{J\in \mathcal{X}_{A}}\right\} \left\langle T_{\sigma
}^{\alpha }\bigtriangleup _{I;\kappa _{1}}^{\sigma }f,\bigtriangleup
_{J;\kappa _{2}}^{\omega }g\right\rangle _{\omega }\equiv S-T\ ,  \notag
\end{eqnarray}%
where%
\begin{equation*}
S=\sum_{A\in \mathcal{A}}\sum_{I:\ I\supsetneqq A}\sum_{J\in \mathcal{W}%
_{A}}\left\langle T_{\sigma }^{\alpha }\bigtriangleup _{I;\kappa
_{1}}^{\sigma }f,\bigtriangleup _{J;\kappa _{2}}^{\omega }g\right\rangle
_{\omega }\text{ and }T=\sum_{A\in \mathcal{A}}\sum_{I:\ I\supsetneqq
A}\sum_{J\in \mathcal{X}_{A}}\left\langle T_{\sigma }^{\alpha
}\bigtriangleup _{I;\kappa _{1}}^{\sigma }f,\bigtriangleup _{J;\kappa
_{2}}^{\omega }g\right\rangle _{\omega }\ ,
\end{equation*}%
and%
\begin{eqnarray*}
\mathcal{W}_{A} &\equiv &\left\{ J\in \mathcal{D}:\ J\in \mathcal{C}_{%
\mathcal{A}}\left( A\right) ,\ \ell \left( J\right) \geq 2^{-\mathbf{\tau }%
}\ell \left( A\right) ,\ \text{and }J\subset B\subset A\text{ for some }B\in 
\mathcal{B}\right\} , \\
\mathcal{X}_{A} &\equiv &\left\{ J\in \mathcal{D}:\ J\in \mathcal{C}_{%
\mathcal{A}}\left( A\right) ,\ \ell \left( J\right) <2^{-\mathbf{\tau }}\ell
\left( A\right) ,\ \text{and there is no }B\in \mathcal{B}\text{ with }%
J\subset B\subset A\right\} .
\end{eqnarray*}

The sum $T$ can be estimated directly by the Intertwining Proposition \ref%
{strongly adapted} using the Alpert projection 
\begin{equation*}
\widehat{g}=\sum_{A\in \mathcal{A}}\sum_{\substack{ J\in \mathcal{C}_{%
\mathcal{A}}\left( A\right) :\ \ell \left( J\right) <2^{-\mathbf{\tau }}\ell
\left( A\right)  \\ \text{and there exists }B\in \mathcal{B}\text{ with }%
J\subset B\subset A}}\bigtriangleup _{J;\kappa _{2}}^{\omega }g.
\end{equation*}%
Indeed, we then have $\sum_{J\in \mathcal{X}_{A}}\bigtriangleup _{J;\kappa
_{2}}^{\omega }g=\sum_{J\in \mathcal{C}_{\mathcal{A}}^{\mathbf{\tau }-\func{%
shift}}\left( A\right) }\bigtriangleup _{J;\kappa _{2}}^{\omega }\widehat{g}$
and so we obtain%
\begin{eqnarray*}
\left\vert T\right\vert &=&\left\vert \sum_{A\in \mathcal{A}}\sum_{I:\
I\supsetneqq A}\sum_{J\in \mathcal{X}_{A}}\left\langle T_{\sigma }^{\alpha
}\bigtriangleup _{I;\kappa _{1}}^{\sigma }f,\bigtriangleup _{J;\kappa
_{2}}^{\omega }g\right\rangle _{\omega }\right\vert =\left\vert \sum_{A\in 
\mathcal{A}}\sum_{I:\ I\supsetneqq A}\sum_{J\in \mathcal{C}_{\mathcal{A}}^{%
\mathbf{\tau }-\func{shift}}\left( A\right) }\left\langle T_{\sigma
}^{\alpha }\bigtriangleup _{I;\kappa _{1}}^{\sigma }f,\bigtriangleup
_{J;\kappa _{2}}^{\omega }\widehat{g}\right\rangle _{\omega }\right\vert \\
&\lesssim &\left( \mathcal{V}_{2}^{\alpha ,\kappa _{1}}+\sqrt{A_{2}^{\alpha }%
}+\mathfrak{T}_{T^{\alpha }}^{\left( \kappa _{1}\right) }\right) \
\left\Vert f\right\Vert _{L^{2}\left( \sigma \right) }\left\Vert \widehat{g}%
\right\Vert _{L^{2}\left( \omega \right) }\leq \left( \mathcal{V}%
_{2}^{\alpha ,\kappa _{1}}+\sqrt{A_{2}^{\alpha }}+\mathfrak{T}_{T^{\alpha
}}^{\left( \kappa _{1}\right) }\right) \ \left\Vert f\right\Vert
_{L^{2}\left( \sigma \right) }\left\Vert g\right\Vert _{L^{2}\left( \omega
\right) }\ .
\end{eqnarray*}

Now we claim that $S$ satisfies%
\begin{equation}
\left\vert S\right\vert \lesssim \left( \mathfrak{T}_{\alpha }^{\left(
\kappa \right) }+\mathfrak{T}_{\alpha }^{\left( \kappa \right) ,\ast }+%
\mathcal{WBP}_{T^{\alpha }}^{\left( \kappa _{1},\kappa _{2}\right) }+\sqrt{%
A_{2}^{\alpha }}\right) \left\Vert f\right\Vert _{L^{2}\left( \sigma \right)
}\left\Vert g\right\Vert _{L^{2}\left( \omega \right) }\ .
\label{diff par shift}
\end{equation}%
To see (\ref{diff par shift}), momentarily fix $A\in \mathcal{A}$ and $J\in 
\mathcal{W}_{A}$ and write 
\begin{equation*}
\sum_{I:\ I\supsetneqq A}\left\langle T_{\sigma }^{\alpha }\bigtriangleup
_{I;\kappa _{1}}^{\sigma }f,\bigtriangleup _{J;\kappa _{2}}^{\omega
}g\right\rangle _{\omega }=\sum_{I:\ A\subsetneqq I\subset \pi _{\mathcal{D}%
}^{\left( \mathbf{\tau }\right) }A}\left\langle T_{\sigma }^{\alpha
}\bigtriangleup _{I;\kappa _{1}}^{\sigma }f,\bigtriangleup _{J;\kappa
_{2}}^{\omega }g\right\rangle _{\omega }+\sum_{I:\ I\supsetneqq \pi _{%
\mathcal{D}}^{\left( \mathbf{\tau }\right) }A}\left\langle T_{\sigma
}^{\alpha }\bigtriangleup _{I;\kappa _{1}}^{\sigma }f,\bigtriangleup
_{J;\kappa _{2}}^{\omega }g\right\rangle _{\omega }\equiv
S_{A,J}^{1}+S_{A,J}^{2}.
\end{equation*}%
We have%
\begin{equation*}
\left\vert \sum_{A\in \mathcal{A}}\sum_{J\in \mathcal{W}_{A}}S_{A,J}^{1}%
\right\vert \lesssim \left( \mathfrak{T}_{\alpha }^{\left( \kappa \right) }+%
\mathfrak{T}_{\alpha }^{\left( \kappa \right) ,\ast }+\mathcal{WBP}%
_{T^{\alpha }}^{\left( \kappa _{1},\kappa _{2}\right) }+\sqrt{A_{2}^{\alpha }%
}\right) \left\Vert f\right\Vert _{L^{2}\left( \sigma \right) }\left\Vert
g\right\Vert _{L^{2}\left( \omega \right) }
\end{equation*}%
by Lemma \ref{standard delta}, since $J\subset I$ and $\frac{\ell \left(
I\right) }{\ell \left( J\right) }=\frac{\ell \left( I\right) }{\ell \left(
A\right) }\frac{\ell \left( A\right) }{\ell \left( J\right) }\leq 2^{\mathbf{%
\tau }}2^{\mathbf{\tau }}<2^{\mathbf{\rho }}$. For the remaining sum,%
\begin{equation*}
\func{Parallel}\left( f,g\right) \equiv \sum_{A\in \mathcal{A}}\sum_{J\in 
\mathcal{W}_{A}}S_{A,J}^{2}=\sum_{A\in \mathcal{A}}\left\langle T_{\sigma
}^{\alpha }\left( \sum_{I:\ I\supsetneqq \pi _{\mathcal{D}}^{\left( \mathbf{%
\tau }\right) }A}\bigtriangleup _{I;\kappa _{1}}^{\sigma }f\right)
,\sum_{J\in \mathcal{W}_{A}}\bigtriangleup _{J;\kappa _{2}}^{\omega
}g\right\rangle _{\omega }\ ,
\end{equation*}%
we apply the Alternate Intertwining Corollary \ref{alt int prop} to obtain%
\begin{equation*}
\left\vert \func{Parallel}\left( f,g\right) \right\vert \lesssim \left( 
\mathcal{V}_{2}^{\alpha ,\kappa _{1}}+\sqrt{A_{2}^{\alpha }}+\mathfrak{T}%
_{T^{\alpha }}^{\left( \kappa _{1}\right) }\right) \ \left\Vert f\right\Vert
_{L^{2}\left( \sigma \right) }\left\Vert g\right\Vert _{L^{2}\left( \omega
\right) }\ .
\end{equation*}%
Altogether then we have%
\begin{equation}
\left\vert \func{Far}\left( f,g\right) \right\vert \lesssim \left( \mathfrak{%
T}_{T^{\alpha }}^{\left( \kappa _{1}\right) }+\mathfrak{T}_{T^{\alpha
}}^{\left( \kappa _{2}\right) ,\ast }+\mathcal{WBP}_{T^{\alpha }}^{\left(
\kappa _{1},\kappa _{2}\right) }+\sqrt{A_{2}^{\alpha }}+\mathcal{V}%
_{2}^{\alpha ,\kappa _{1}}+\mathcal{V}_{2}^{\alpha ,\kappa _{2},\ast
}\right) \left\Vert f\right\Vert _{L^{2}\left( \sigma \right) }\left\Vert
g\right\Vert _{L^{2}\left( \omega \right) }\ .  \label{far est}
\end{equation}

\subsubsection{Near form}

It remains to control the near form $\func{Near}\left( f,g\right) $ either
by the Indicator/Cube Testing conditions and the classical Muckenhoupt
condition $A_{2}^{\alpha }$, or in the case the measures $\sigma $ and $%
\omega $ are comparable, by the $\kappa $-Cube Testing conditions, Bilinear
Indicator/Cube Testing property, and $A_{2}^{\alpha }$. We first further
decompose $\func{Near}\left( f,g\right) $ into%
\begin{eqnarray*}
\func{Near}\left( f,g\right) &=&\left\{ \sum_{\substack{ \left( A,B\right)
\in \func{Near}\left( \mathcal{A}\times \mathcal{B}\right)  \\ B\subset A}}%
+\sum_{\substack{ \left( A,B\right) \in \func{Near}\left( \mathcal{A}\times 
\mathcal{B}\right)  \\ A\subset B}}\right\} \left\langle T_{\sigma }^{\alpha
}\left( \mathsf{P}_{\mathcal{C}_{\mathcal{A}}\left( A\right) }^{\sigma
}f\right) ,\mathsf{P}_{\mathcal{C}_{\mathcal{B}}\left( B\right) }^{\omega
}g\right\rangle _{\omega } \\
&=&\func{Near}_{\limfunc{below}}\left( f,g\right) +\func{Near}_{\limfunc{%
above}}\left( f,g\right) .
\end{eqnarray*}%
To control $\func{Near}_{\limfunc{below}}\left( f,g\right) $ we define
projections%
\begin{equation*}
\mathsf{Q}_{A}^{\omega }g\equiv \sum_{\substack{ B\in \mathcal{B}:\ \left(
A,B\right) \in \func{Near}\left( \mathcal{A}\times \mathcal{B}\right)  \\ %
B\subset A}}\mathsf{P}_{\mathcal{C}_{\mathcal{B}}\left( B\right) }^{\omega
}g,
\end{equation*}%
and observe that, while the Alpert support of $\mathsf{Q}_{A}^{\omega }$
need not be contained in the corona $\mathcal{C}_{\mathcal{A}}\left(
A\right) $, these projections are nevertheless mutually orthogonal in the
index $A\in \mathcal{A}$.

It is now an easy exercise to use the Indicator/Cube Testing condition (\ref%
{def kth order testing}) to control $\func{Near}_{\limfunc{below}}\left(
f,g\right) $,%
\begin{eqnarray}
&&\left\vert \func{Near}_{\limfunc{below}}\left( f,g\right) \right\vert
=\sum_{A\in \mathcal{A}}\left\vert \left\langle T_{\sigma }^{\alpha }\mathsf{%
P}_{\mathcal{C}_{\mathcal{A}}\left( A\right) }^{\sigma }f,\mathsf{Q}%
_{A}^{\omega }g\right\rangle _{\omega }\right\vert  \label{near porism} \\
&\leq &\sum_{A\in \mathcal{A}}\left\Vert T_{\sigma }^{\alpha }\mathsf{P}_{%
\mathcal{C}_{\mathcal{A}}\left( A\right) }^{\sigma }f\right\Vert
_{L^{2}\left( \omega \right) }\left\Vert \mathsf{Q}_{A}^{\omega
}g\right\Vert _{L^{2}\left( \omega \right) }\lesssim \mathfrak{T}_{T^{\alpha
}}^{IC}\sum_{A\in \mathcal{A}}\alpha _{\mathcal{A}}\left( A\right) \sqrt{%
\left\vert A\right\vert _{\sigma }}\left\Vert \mathsf{Q}_{A}^{\omega
}g\right\Vert _{L^{2}\left( \omega \right) }  \notag \\
&\leq &\mathfrak{T}_{T^{\alpha }}^{IC}\left( \sum_{A\in \mathcal{A}}\alpha _{%
\mathcal{A}}\left( A\right) ^{2}\left\vert A\right\vert _{\sigma }\right) ^{%
\frac{1}{2}}\left( \sum_{A\in \mathcal{A}}\left\Vert \mathsf{Q}_{A}^{\omega
}g\right\Vert _{L^{2}\left( \omega \right) }^{2}\right) ^{\frac{1}{2}%
}\lesssim \mathfrak{T}_{T^{\alpha }}^{IC}\left\Vert f\right\Vert
_{L^{2}\left( \sigma \right) }\left\Vert g\right\Vert _{L^{2}\left( \omega
\right) }\ ,  \notag
\end{eqnarray}%
by quasi-orthogonality and the fact that the projections $\mathsf{Q}%
_{A}^{\omega }$ are mutually orthogonal in the index $A\in \mathcal{A}$.
Note that we have \textbf{not} used comparability of measures\ here since
that is only needed for the Bilinear Carleson Embedding Theorem \ref{2 wt
bil CET}. This will give the first inequality (\ref{NIC}) in Theorem \ref%
{pivotal theorem} after we have removed the weak boundedness constant $%
\mathcal{WBP}_{T^{\alpha }}^{\left( \kappa _{1},\kappa _{2}\right) }\left(
\sigma ,\omega \right) $ in the next subsection.

But we must work harder to obtain control by the Bilinear Indicator/Cube
Testing property and $\kappa $-Cube Testing in the presence of the
comparability assumption on $\sigma $ and $\omega $. For this we proceed
instead as follows. For fixed $A\in \mathcal{A}$ write%
\begin{eqnarray*}
&&\sum_{\substack{ B\in \mathcal{B}:\ \left( A,B\right) \in \func{Near}%
\left( \mathcal{A}\times \mathcal{B}\right)  \\ B\subset A}}\left\langle
T_{\sigma }^{\alpha }\mathsf{P}_{\mathcal{C}_{\mathcal{A}}\left( A\right)
}^{\sigma }f,\mathsf{P}_{\mathcal{C}_{\mathcal{B}}\left( B\right) }^{\omega
}g\right\rangle _{\omega }=\sum_{B\in \mathcal{B}\cap \mathcal{C}_{\mathcal{A%
}}\left( A\right) }\left\langle T_{\sigma }^{\alpha }\mathsf{P}_{\mathcal{C}%
_{\mathcal{A}}\left( A\right) }^{\sigma }f,\mathsf{P}_{\mathcal{C}_{\mathcal{%
B}}\left( B\right) }^{\omega }g\right\rangle _{\omega } \\
&=&\sum_{B,B^{\prime }\in \mathcal{B}\cap \mathcal{C}_{\mathcal{A}}\left(
A\right) }\left\langle T_{\sigma }^{\alpha }\left( \mathsf{P}_{\mathcal{C}_{%
\mathcal{B}}\left( B^{\prime }\right) }^{\sigma }\mathsf{P}_{\mathcal{C}_{%
\mathcal{A}}\left( A\right) }^{\sigma }f\right) ,\mathsf{P}_{\mathcal{C}_{%
\mathcal{B}}\left( B\right) }^{\omega }g\right\rangle _{\omega }+\sum_{B\in 
\mathcal{B}\cap \mathcal{C}_{\mathcal{A}}\left( A\right) }\left\langle
T_{\sigma }^{\alpha }\left( \mathsf{P}_{\mathcal{C}_{\mathcal{B}}\left( \pi
_{\mathcal{B}}A\right) }^{\sigma }\mathsf{P}_{\mathcal{C}_{\mathcal{A}%
}\left( A\right) }^{\sigma }f\right) ,\mathsf{P}_{\mathcal{C}_{\mathcal{B}%
}\left( B\right) }^{\omega }g\right\rangle _{\omega } \\
&=&\left\{ \sum_{\substack{ B,B^{\prime }\in \mathcal{B}\cap \mathcal{C}_{%
\mathcal{A}}\left( A\right)  \\ B\cap B^{\prime }=\emptyset }}+\sum 
_{\substack{ B,B^{\prime }\in \mathcal{B}\cap \mathcal{C}_{\mathcal{A}%
}\left( A\right)  \\ B=B^{\prime }}}+\sum_{\substack{ B,B^{\prime }\in 
\mathcal{B}\cap \mathcal{C}_{\mathcal{A}}\left( A\right)  \\ B\varsubsetneqq
B^{\prime }}}+\sum_{\substack{ B,B^{\prime }\in \mathcal{B}\cap \mathcal{C}_{%
\mathcal{A}}\left( A\right)  \\ B^{\prime }\varsubsetneqq B}}\right\}
\left\langle T_{\sigma }^{\alpha }\left( \mathsf{P}_{\mathcal{C}_{\mathcal{B}%
}\left( B^{\prime }\right) }^{\sigma }\mathsf{P}_{\mathcal{C}_{\mathcal{A}%
}\left( A\right) }^{\sigma }f\right) ,\mathsf{P}_{\mathcal{C}_{\mathcal{B}%
}\left( B\right) }^{\omega }g\right\rangle _{\omega } \\
&&+\sum_{B\in \mathcal{B}\cap \mathcal{C}_{\mathcal{A}}\left( A\right)
}\left\langle T_{\sigma }^{\alpha }\left( \mathsf{P}_{\mathcal{C}_{\mathcal{B%
}}\left( \pi _{\mathcal{B}}A\right) }^{\sigma }\mathsf{P}_{\mathcal{C}_{%
\mathcal{A}}\left( A\right) }^{\sigma }f\right) ,\mathsf{P}_{\mathcal{C}_{%
\mathcal{B}}\left( B\right) }^{\omega }g\right\rangle _{\omega } \\
&\equiv &I^{A}+II^{A}+III^{A}+IV^{A}+V^{A},
\end{eqnarray*}%
where $\pi _{\mathcal{B}}A$ denotes the smallest cube $B\in \mathcal{B}$
that contains $A$\footnote{%
We thank Ignacio Uriarte-Tuero for pointing out that term $V^{A}$ was
missing from the argument.}. Then term $I^{A}$ is handled immediately by
Lemma \ref{delta near} to yield%
\begin{equation*}
\sum_{A\in \mathcal{A}}\left\vert I^{A}\right\vert \lesssim \left( \sqrt{%
A_{2}^{\alpha }}+\mathfrak{T}_{T^{\alpha }}^{\left( \kappa \right) }\left(
\sigma ,\omega \right) +\mathfrak{T}_{T^{\alpha ,\ast }}^{\left( \kappa
\right) }\left( \omega ,\sigma \right) +\mathcal{WBP}_{T^{\alpha }}^{\left(
\kappa _{1},\kappa _{2}\right) }\left( \sigma ,\omega \right) \right)
\left\Vert f\right\Vert _{L^{2}\left( \sigma \right) }\left\Vert
g\right\Vert _{L^{2}\left( \omega \right) }\ .
\end{equation*}%
The sum $\sum_{A\in \mathcal{A}}\left\vert II^{A}\right\vert $ of terms $%
II^{A}$ will be handled by the bilinear Carleson Embedding Theorem \ref{2 wt
bil CET}, using the Bilinear Indicator/Cube Testing property $\mathcal{BICT}%
_{T^{\alpha }}\left( \sigma ,\omega \right) <\infty $ as follows.

Note that for $\sigma $ and $\omega $ doubling measures, we have the
following two properties, 
\begin{equation*}
\left\Vert \mathsf{P}_{\mathcal{C}_{\mathcal{B}}\left( B\right) }^{\sigma }%
\mathsf{P}_{\mathcal{C}_{\mathcal{A}}\left( A\right) }^{\sigma }f\right\Vert
_{L^{\infty }\left( \sigma \right) }\lesssim \alpha _{\mathcal{A}}\left(
A\right) \text{ and }\left\Vert \mathsf{P}_{\mathcal{C}_{\mathcal{B}}\left(
B\right) }^{\omega }g\right\Vert _{L^{\infty }\left( \sigma \right)
}\lesssim \alpha _{\mathcal{B}}\left( B\right) ,
\end{equation*}%
since our coronas are Calder\'{o}n-Zygmund, and thus if $A^{\prime }\in 
\mathfrak{C}_{\mathcal{A}}\left( A\right) $, then%
\begin{equation*}
\frac{1}{\left\vert A^{\prime }\right\vert _{\sigma }}\int_{A^{\prime
}}\left\vert f\right\vert d\sigma \lesssim \frac{1}{\left\vert \pi _{%
\mathcal{D}}A^{\prime }\right\vert _{\sigma }}\int_{\pi _{\mathcal{D}%
}A^{\prime }}\left\vert f\right\vert d\sigma \lesssim \alpha _{\mathcal{A}%
}\left( A\right) ,
\end{equation*}%
and so%
\begin{equation*}
\left\vert \mathsf{P}_{\mathcal{C}_{\mathcal{B}}\left( B\right) \cap 
\mathcal{C}_{\mathcal{A}}\left( A\right) }^{\sigma }f\right\vert \lesssim
\sup_{I\in \left[ \mathcal{C}_{\mathcal{B}}\left( B\right) \cap \mathcal{C}_{%
\mathcal{A}}\left( A\right) \right] \cup \mathfrak{C}_{\mathcal{A}}\left(
A\right) }\frac{1}{\left\vert I\right\vert _{\sigma }}\int_{I}\left\vert
f\right\vert d\sigma \lesssim \alpha _{\mathcal{A}}\left( A\right) .
\end{equation*}%
In the first inequality in the above display we have used the telescoping
identities for Alpert wavelets.

We then have, using the Bilinear Indicator/Cube Testing property, that 
\begin{eqnarray*}
\left\vert II^{A}\right\vert &=&\left\vert \sum_{B\in \mathcal{B}\cap 
\mathcal{C}_{A}}\left\langle T_{\sigma }^{\alpha }\left( \mathsf{P}_{%
\mathcal{C}_{\mathcal{B}}\left( B\right) }^{\sigma }\mathsf{P}_{\mathcal{C}_{%
\mathcal{A}}\left( A\right) }^{\sigma }f\right) ,\mathsf{P}_{\mathcal{C}_{%
\mathcal{B}}\left( B\right) }^{\omega }g_{B}\right\rangle _{\omega
}\right\vert \\
&\lesssim &\alpha _{\mathcal{A}}\left( A\right) \sum_{B\in \mathcal{B}\cap 
\mathcal{C}_{\mathcal{A}}\left( A\right) }\alpha _{\mathcal{B}}\left(
B\right) \left\vert \left\langle T_{\sigma }^{\alpha }\left( \frac{\mathsf{P}%
_{\mathcal{C}_{\mathcal{B}}\left( B\right) }^{\sigma }\mathsf{P}_{\mathcal{C}%
_{\mathcal{A}}\left( A\right) }^{\sigma }f}{\left\Vert \mathsf{P}_{\mathcal{C%
}_{\mathcal{B}}\left( B\right) }^{\sigma }\mathsf{P}_{\mathcal{C}_{\mathcal{A%
}}\left( A\right) }^{\sigma }f\right\Vert _{L^{\infty }\left( \sigma \right)
}}\right) ,\frac{\mathsf{P}_{\mathcal{C}_{\mathcal{B}}\left( B\right)
}^{\omega }g_{B}}{\left\Vert \mathsf{P}_{\mathcal{C}_{\mathcal{B}}\left(
B\right) }^{\omega }g_{B}\right\Vert _{L^{\infty }\left( \sigma \right) }}%
\right\rangle _{\omega }\right\vert \\
&\leq &\alpha _{\mathcal{A}}\left( A\right) \sum_{B\in \mathcal{B}\cap 
\mathcal{C}_{\mathcal{A}}\left( A\right) }\mathcal{BICT}_{T^{\alpha }}\left(
\sigma ,\omega \right) \ \alpha _{\mathcal{B}}\left( B\right) \sqrt{%
\left\vert B\right\vert _{\sigma }}\sqrt{\left\vert B\right\vert _{\omega }}%
\ .
\end{eqnarray*}%
Now we use 
\begin{eqnarray*}
\alpha _{\mathcal{A}}\left( A\right) &\lesssim &\frac{1}{\left\vert
A\right\vert _{\sigma }}\int_{A}\left\vert f\right\vert d\sigma \leq
\sup_{K\in \mathcal{D}:\ K\supset I}\frac{1}{\left\vert K\right\vert
_{\sigma }}\int_{K}\left\vert f\right\vert d\sigma ,\ \ \ \ \ \text{for }%
I\in \mathcal{C}_{\mathcal{A}}\left( A\right) , \\
\alpha _{\mathcal{B}}\left( B\right) &\lesssim &\frac{1}{\left\vert
B\right\vert _{\sigma }}\int_{B}\left\vert g\right\vert d\sigma \leq
\sup_{L\in \mathcal{D}:\ L\supset J}\frac{1}{\left\vert L\right\vert
_{\omega }}\int_{L}\left\vert g\right\vert d\omega ,\ \ \ \ \ \text{for }%
J\in \mathcal{C}_{\mathcal{B}}\left( B\right) ,
\end{eqnarray*}%
and apply the bilinear Carleson Embedding Theorem \ref{2 wt bil CET}, with $%
a_{I}\equiv \left\{ 
\begin{array}{ccc}
\sqrt{\left\vert I\right\vert _{\sigma }\left\vert I\right\vert _{\omega }}
& \text{ if } & I\in \mathcal{C}_{\mathcal{A}}\left( A\right) \cap \mathcal{B%
} \\ 
0 & \text{ if } & I\not\in \mathcal{C}_{\mathcal{A}}\left( A\right) \cap 
\mathcal{B}%
\end{array}%
\right. $, to conclude that 
\begin{eqnarray*}
\sum_{A\in \mathcal{A}}\left\vert II^{A}\right\vert &\lesssim &\mathcal{BICT}%
_{T^{\alpha }}\left( \sigma ,\omega \right) \sum_{A\in \mathcal{A}%
}\sum_{B\in \mathcal{B}\cap \mathcal{C}_{\mathcal{A}}\left( A\right) }\sqrt{%
\left\vert B\right\vert _{\sigma }}\sqrt{\left\vert B\right\vert _{\omega }}%
\left( \sup_{K\in \mathcal{D}:\ K\supset B}\frac{1}{\left\vert K\right\vert
_{\sigma }}\int_{K}\left\vert f\right\vert d\sigma \right) \left( \sup_{L\in 
\mathcal{D}:\ L\supset B}\frac{1}{\left\vert L\right\vert _{\omega }}%
\int_{L}\left\vert g\right\vert d\omega \right) \\
&\lesssim &\mathcal{BICT}_{T^{\alpha }}\left( \sigma ,\omega \right)
\left\Vert f\right\Vert _{L^{2}\left( \sigma \right) }\left\Vert
g\right\Vert _{L^{2}\left( \omega \right) }\ .
\end{eqnarray*}%
Finally, observe that the Carleson condition (\ref{bil Car cond}) holds here
for $J\in \mathcal{B}$ since the geometric decay in the $\omega $-Carleson
condition for $\mathcal{B}$ gives%
\begin{eqnarray*}
\sum_{I\in \mathcal{D}:\ I\subset J}a_{I} &=&\sum_{B\in \mathcal{B}:\
B\subset J}\sqrt{\left\vert B\right\vert _{\sigma }\left\vert B\right\vert
_{\omega }}=\sum_{m=0}^{\infty }\sum_{B\in \mathfrak{C}_{\mathcal{B}%
}^{\left( m\right) }\left( J\right) }\sqrt{\left\vert B\right\vert _{\sigma
}\left\vert B\right\vert _{\omega }} \\
&\leq &\sum_{m=0}^{\infty }\sqrt{\sum_{B\in \mathfrak{C}_{\mathcal{B}%
}^{\left( m\right) }\left( J\right) }\left\vert B\right\vert _{\sigma }}%
\sqrt{\sum_{B\in \mathfrak{C}_{\mathcal{B}}^{\left( m\right) }\left(
J\right) }\left\vert B\right\vert _{\omega }}\leq \sum_{m=0}^{\infty }\sqrt{%
\left\vert J\right\vert _{\sigma }}\sqrt{2^{-m\delta }\left\vert
J\right\vert _{\omega }}\leq C^{\prime }\sqrt{\left\vert J\right\vert
_{\sigma }\left\vert J\right\vert _{\omega }},
\end{eqnarray*}%
and now the case of general $J$ follows as usual.

\begin{remark}
This is the only place in the proof where the Bilinear Indicator/Cube
Testing property (\ref{def ind WBP}) is used, and also the only place
requiring the comparability of the measures (through the use of the bilinear
Carleson Embedding Theorem \ref{2 wt bil CET}). It is the Parallel Corona
that permits this relatively simple application of a bilinear testing
property.
\end{remark}

To handle term $III^{A}$ we decompose it into two terms,%
\begin{eqnarray*}
III^{A} &=&\sum_{\substack{ B,B^{\prime }\in \mathcal{B}\cap \mathcal{C}_{%
\mathcal{A}}\left( A\right)  \\ B\varsubsetneqq B^{\prime }\subset A}}%
\left\langle T_{\sigma }^{\alpha }\left( \mathsf{P}_{\mathcal{C}_{\mathcal{B}%
}\left( B^{\prime }\right) }^{\sigma }\mathsf{P}_{\mathcal{C}_{\mathcal{A}%
}\left( A\right) }^{\sigma }f\right) ,\mathsf{P}_{\mathcal{C}_{\mathcal{B}%
}\left( B\right) }^{\omega }g\right\rangle _{\omega }+\sum_{\substack{ B\in 
\mathcal{B}\cap \mathcal{C}_{\mathcal{A}}\left( A\right) ,\ B^{\prime }\in 
\mathcal{B}  \\ B^{\prime }\supsetneqq A}}\left\langle T_{\sigma }^{\alpha
}\left( \mathsf{P}_{\mathcal{C}_{\mathcal{B}}\left( B^{\prime }\right)
}^{\sigma }\mathsf{P}_{\mathcal{C}_{\mathcal{A}}\left( A\right) }^{\sigma
}f\right) ,\mathsf{P}_{\mathcal{C}_{\mathcal{B}}\left( B\right) }^{\omega
}g\right\rangle _{\omega } \\
&\equiv &III_{1}^{A}+III_{2}^{A}\ .
\end{eqnarray*}%
Then we proceed with%
\begin{eqnarray*}
III_{1}^{A} &=&\sum_{\substack{ B,B^{\prime }\in \mathcal{B}\cap \mathcal{C}%
_{\mathcal{A}}\left( A\right)  \\ B\varsubsetneqq B^{\prime }}}\left\langle
T_{\sigma }^{\alpha }\left( \mathsf{P}_{\mathcal{C}_{\mathcal{B}}\left(
B^{\prime }\right) }^{\sigma }\mathsf{P}_{\mathcal{C}_{\mathcal{A}}\left(
A\right) }^{\sigma }f\right) ,\mathsf{P}_{\mathcal{C}_{\mathcal{B}}\left(
B\right) }^{\omega }g\right\rangle _{\omega }=\sum_{\substack{ B,B^{\prime
}\in \mathcal{B}\cap \mathcal{C}_{\mathcal{A}}\left( A\right)  \\ %
B\varsubsetneqq B^{\prime }}}\left\langle \left( \mathsf{P}_{\mathcal{C}_{%
\mathcal{B}}\left( B^{\prime }\right) }^{\sigma }\mathsf{P}_{\mathcal{C}_{%
\mathcal{A}}\left( A\right) }^{\sigma }f\right) ,T_{\omega }^{\alpha ,\ast
}\left( \mathsf{P}_{\mathcal{C}_{\mathcal{B}}\left( B\right) }^{\omega
}g\right) \right\rangle _{\sigma } \\
&=&\sum_{B\in \mathcal{B}\cap \mathcal{C}_{\mathcal{A}}\left( A\right)
}\left\langle \left( \sum_{\substack{ B^{\prime }\in \mathcal{B}\cap 
\mathcal{C}_{\mathcal{A}}\left( A\right)  \\ B\varsubsetneqq B^{\prime }}}%
\mathsf{P}_{\mathcal{C}_{\mathcal{B}}\left( B^{\prime }\right) }^{\sigma }%
\mathsf{P}_{\mathcal{C}_{\mathcal{A}}\left( A\right) }^{\sigma }f\right)
,T_{\omega }^{\alpha ,\ast }\left( \mathsf{P}_{\mathcal{C}_{\mathcal{B}%
}\left( B\right) }^{\omega }g\right) \right\rangle _{\sigma }\ .
\end{eqnarray*}

As in our treatment of the $\func{Far}_{\limfunc{below}}$ form above, we now
apply an argument analogous to that surrounding (\ref{Far-Shift}), in order
to control the sum $\sum_{A\in \mathcal{A}}III_{1}^{A}$ using Lemma \ref%
{standard delta}, and the dual forms of the Intertwining Proposition \ref%
{strongly adapted}, and the Alternate Intertwining Corollary \ref{alt int
prop}. This results in the estimate%
\begin{equation*}
\left\vert \sum_{A\in \mathcal{A}}III_{1}^{A}\right\vert \lesssim \left( 
\sqrt{A_{2}^{\alpha }}+\mathfrak{T}_{T^{\alpha }}^{\left( \kappa \right)
}\left( \sigma ,\omega \right) +\mathfrak{T}_{T^{\alpha ,\ast }}^{\left(
\kappa \right) }\left( \omega ,\sigma \right) +\mathcal{WBP}_{T^{\alpha
}}^{\left( \kappa _{1},\kappa _{2}\right) }\left( \sigma ,\omega \right)
\right) \left\Vert f\right\Vert _{L^{2}\left( \sigma \right) }\left\Vert
g\right\Vert _{L^{2}\left( \omega \right) }\ .
\end{equation*}%
For the sum of terms $III_{2}^{A}$, we also apply an argument analogous to
that surrounding (\ref{Far-Shift}) using Lemma \ref{standard delta} and the
dual forms of Proposition \ref{strongly adapted} and Corollary \ref{alt int
prop}. This also results in the estimate%
\begin{equation*}
\left\vert \sum_{A\in \mathcal{A}}III_{2}^{A}\right\vert \lesssim \left( 
\sqrt{A_{2}^{\alpha }}+\mathfrak{T}_{T^{\alpha }}^{\left( \kappa \right)
}\left( \sigma ,\omega \right) +\mathfrak{T}_{T^{\alpha ,\ast }}^{\left(
\kappa \right) }\left( \omega ,\sigma \right) +\mathcal{WBP}_{T^{\alpha
}}^{\left( \kappa _{1},\kappa _{2}\right) }\left( \sigma ,\omega \right)
\right) \left\Vert f\right\Vert _{L^{2}\left( \sigma \right) }\left\Vert
g\right\Vert _{L^{2}\left( \omega \right) }\ .
\end{equation*}%
In the same way, for the sum of terms $V^{A}$, we first write%
\begin{equation*}
V^{A}=\sum_{B\in \mathcal{B}\cap \mathcal{C}_{\mathcal{A}}\left( A\right)
}\left\langle \left( \mathsf{P}_{\mathcal{C}_{\mathcal{B}}\left( \pi _{%
\mathcal{B}}A\right) }^{\sigma }\mathsf{P}_{\mathcal{C}_{\mathcal{A}}\left(
A\right) }^{\sigma }f\right) ,T_{\omega }^{\alpha ,\ast }\left( \mathsf{P}_{%
\mathcal{C}_{\mathcal{B}}\left( B\right) }^{\omega }g\right) \right\rangle
_{\sigma }\ ,
\end{equation*}%
and then apply once more an argument analogous to that surrounding (\ref%
{Far-Shift}) using Lemma \ref{standard delta} and the dual forms of
Proposition \ref{strongly adapted} and Corollary \ref{alt int prop}, that
results in the estimate%
\begin{equation*}
\left\vert \sum_{A\in \mathcal{A}}V^{A}\right\vert \lesssim \left( \sqrt{%
A_{2}^{\alpha }}+\mathfrak{T}_{T^{\alpha }}^{\left( \kappa \right) }\left(
\sigma ,\omega \right) +\mathfrak{T}_{T^{\alpha ,\ast }}^{\left( \kappa
\right) }\left( \omega ,\sigma \right) +\mathcal{WBP}_{T^{\alpha }}^{\left(
\kappa _{1},\kappa _{2}\right) }\left( \sigma ,\omega \right) \right)
\left\Vert f\right\Vert _{L^{2}\left( \sigma \right) }\left\Vert
g\right\Vert _{L^{2}\left( \omega \right) }\ .
\end{equation*}

The bound for the sum $\sum_{A\in \mathcal{A}}\left\vert
IV_{2}^{A}\right\vert $ is essentially dual to that for $\sum_{A\in \mathcal{%
A}}\left\vert III_{2}^{A}\right\vert $, and so altogether, since $\mathfrak{T%
}_{T^{\alpha }}\leq \mathfrak{T}_{T^{\alpha }}^{\left( \kappa \right) }$, we
have shown that%
\begin{equation*}
\left\vert \func{Near}_{\limfunc{below}}\left( f,g\right) \right\vert
\lesssim \left( \sqrt{A_{2}^{\alpha }}+\mathfrak{T}_{T^{\alpha }}^{\left(
\kappa \right) }\left( \sigma ,\omega \right) +\mathfrak{T}_{T^{\alpha ,\ast
}}^{\left( \kappa \right) }\left( \omega ,\sigma \right) +\mathcal{WBP}%
_{T^{\alpha }}^{\left( \kappa _{1},\kappa _{2}\right) }\left( \sigma ,\omega
\right) +\mathcal{BICT}_{T^{\alpha }}\left( \sigma ,\omega \right) \right)
\left\Vert f\right\Vert _{L^{2}\left( \sigma \right) }\left\Vert
g\right\Vert _{L^{2}\left( \omega \right) }\ .
\end{equation*}%
By symmetry, we also have that the form%
\begin{equation*}
\func{Near}_{\limfunc{above}}\left( f,g\right) =\sum_{\substack{ \left(
A,B\right) \in \func{Near}\left( \mathcal{A}\times \mathcal{B}\right)  \\ %
A\subset B}}\left\langle T_{\sigma }^{\alpha }\left( \mathsf{P}_{\mathcal{C}%
_{\mathcal{A}}\left( A\right) }^{\sigma }f\right) ,\mathsf{P}_{\mathcal{C}_{%
\mathcal{B}}\left( B\right) }^{\omega }g\right\rangle _{\omega }=\sum 
_{\substack{ \left( A,B\right) \in \func{Near}\left( \mathcal{A}\times 
\mathcal{B}\right)  \\ A\subset B}}\left\langle \mathsf{P}_{\mathcal{C}_{%
\mathcal{A}}\left( A\right) }^{\sigma }f,T_{\omega }^{\alpha ,\ast }\left( 
\mathsf{P}_{\mathcal{C}_{\mathcal{B}}\left( B\right) }^{\omega }g\right)
\right\rangle _{\omega }
\end{equation*}%
satisfies%
\begin{equation*}
\left\vert \func{Near}_{\limfunc{above}}\left( f,g\right) \right\vert
\lesssim \left( \sqrt{A_{2}^{\alpha }}+\mathfrak{T}_{T^{\alpha }}^{\left(
\kappa \right) }\left( \sigma ,\omega \right) +\mathfrak{T}_{T^{\alpha ,\ast
}}^{\left( \kappa \right) }\left( \omega ,\sigma \right) +\mathcal{WBP}%
_{T^{\alpha }}^{\left( \kappa _{1},\kappa _{2}\right) }\left( \sigma ,\omega
\right) +\mathcal{BICT}_{T^{\alpha }}\left( \sigma ,\omega \right) \right)
\left\Vert f\right\Vert _{L^{2}\left( \sigma \right) }\left\Vert
g\right\Vert _{L^{2}\left( \omega \right) }.
\end{equation*}%
Combining these estimates completes our control of the near form $\func{Near}%
\left( f,g\right) $:%
\begin{equation}
\left\vert \func{Near}\left( f,g\right) \right\vert \lesssim \left( \sqrt{%
A_{2}^{\alpha }}+\mathfrak{T}_{T^{\alpha }}^{\left( \kappa \right) }\left(
\sigma ,\omega \right) +\mathfrak{T}_{T^{\alpha ,\ast }}^{\left( \kappa
\right) }\left( \omega ,\sigma \right) +\mathcal{WBP}_{T^{\alpha }}^{\left(
\kappa _{1},\kappa _{2}\right) }+\mathcal{BICT}_{T^{\alpha }}\left( \sigma
,\omega \right) \right) \left\Vert f\right\Vert _{L^{2}\left( \sigma \right)
}\left\Vert g\right\Vert _{L^{2}\left( \omega \right) }.  \label{near est}
\end{equation}

The three inequalities (\ref{disj est}), (\ref{far est}) and (\ref{near est}%
) finish the proof of (\ref{once we show'}), thus yielding the inequality%
\begin{equation}
\mathfrak{N}_{T^{\alpha }}\lesssim C_{\kappa _{1},\left( \beta _{1},\gamma
_{1}\right) ,\kappa _{2},\left( \beta _{2},\gamma _{2}\right) }\left( 
\mathfrak{T}_{T^{\alpha }}^{\left( \kappa _{1}\right) }+\mathfrak{T}%
_{T^{\alpha }}^{\left( \kappa _{2}\right) ,\ast }+\mathcal{BICT}_{T^{\alpha
}}+\sqrt{A_{2}^{\alpha }}+\mathcal{V}_{2}^{\alpha ,\kappa _{1}}+\mathcal{V}%
_{2}^{\alpha ,\kappa _{2},\ast }+\mathcal{WBP}_{T^{\alpha }}^{\left( \kappa
_{1},\kappa _{2}\right) }\right) \ ,  \label{gives ineq}
\end{equation}%
after taking the supremum over $f$ and $g$ in their respective unit balls.
It now remains only to remove $\mathcal{WBP}_{T^{\alpha }}^{\left( \kappa
_{1},\kappa _{2}\right) }$ from the right hand side of (\ref{gives ineq}) in
order to finish the proof of Theorem \ref{pivotal theorem}.

\subsection{Eliminating the weak boundedness property by doubling}

Here we show that the weak boundedness constant $\mathcal{WBP}_{T^{\alpha
}}^{\left( \kappa _{1},\kappa _{2}\right) }\left( \sigma ,\omega \right) $
can be easily eliminated from the right hand side of (\ref{once we show'})
or (\ref{gives ineq}) using the doubling properties of the measures. We
first use Corollary \ref{full Tp control} to obtain the inequality%
\begin{eqnarray}
&&\mathcal{WBP}_{T^{\alpha }}^{\left( \kappa _{1},\kappa _{2}\right) }\left(
\sigma ,\omega \right)  \label{unif bound'} \\
&=&\sup_{\mathcal{D}\in \Omega }\sup_{\substack{ Q,Q^{\prime }\in \mathcal{D}
\\ Q\subset 3Q^{\prime }\setminus Q^{\prime }\text{ or }Q^{\prime }\subset
3Q\setminus Q}}\frac{1}{\sqrt{\left\vert Q\right\vert _{\sigma }\left\vert
Q^{\prime }\right\vert _{\omega }}}\sup_{\substack{ f\in \left( \mathcal{P}%
_{Q}^{\kappa _{1}}\right) _{\limfunc{norm}}\left( \sigma \right)  \\ g\in
\left( \mathcal{P}_{Q^{\prime }}^{\kappa _{2}}\right) _{\limfunc{norm}%
}\left( \omega \right) }}\left\vert \int_{Q^{\prime }}T_{\sigma }^{\alpha
}\left( \mathbf{1}_{Q}f\right) \ gd\omega \right\vert <\infty  \notag \\
&\leq &\sup_{\mathcal{D}\in \Omega }\sup_{\substack{ Q,Q^{\prime }\in 
\mathcal{D}  \\ Q\subset 3Q^{\prime }\setminus Q^{\prime }\text{ or }%
Q^{\prime }\subset 3Q\setminus Q}}\frac{1}{\sqrt{\left\vert Q\right\vert
_{\sigma }\left\vert Q^{\prime }\right\vert _{\omega }}}\sup_{f\in \left( 
\mathcal{P}_{Q}^{\kappa _{1}}\right) _{\limfunc{norm}}\left( \sigma \right)
}\int_{Q^{\prime }}\left( T_{\sigma }^{\alpha }\left( \mathbf{1}_{Q}f\right)
\right) ^{2}d\omega  \notag \\
&\leq &\mathfrak{FT}_{T^{\alpha }}^{\left( \kappa \right) }\left( \sigma
,\omega \right) ^{2}\leq C_{\kappa ,\varepsilon }\mathfrak{T}_{T^{\alpha
}}\left( \sigma ,\omega \right) +C_{\kappa ,\varepsilon }\mathcal{A}%
_{2}^{\alpha }\left( \sigma ,\omega \right) +\varepsilon \mathfrak{N}%
_{T^{\alpha }}\left( \sigma ,\omega \right) ,  \notag
\end{eqnarray}%
valid for doubling measures $\sigma $ and $\omega $.

Now we plug (\ref{unif bound'}) into inequality (\ref{gives ineq}) to obtain%
\begin{eqnarray*}
\mathfrak{N}_{T^{\alpha }} &\lesssim &C_{\kappa _{1},\left( \beta
_{1},\gamma _{1}\right) ,\kappa _{2},\left( \beta _{2},\gamma _{2}\right)
}\left\{ \mathfrak{T}_{T^{\alpha }}^{\left( \kappa _{1}\right) }+\mathfrak{T}%
_{T^{\alpha }}^{\left( \kappa _{2}\right) ,\ast }+\mathcal{BICT}_{T^{\alpha
}}+\sqrt{A_{2}^{\alpha }}+\mathcal{V}_{2}^{\alpha ,\kappa _{1}}+\mathcal{V}%
_{2}^{\alpha ,\kappa _{2},\ast }\right\} \\
&&+C_{\kappa _{1},\left( \beta _{1},\gamma _{1}\right) ,\kappa _{2},\left(
\beta _{2},\gamma _{2}\right) }C_{\kappa ,\varepsilon }\left\{ \mathfrak{T}%
_{T^{\alpha }}\left( \sigma ,\omega \right) +\mathcal{A}_{2}^{\alpha }\left(
\sigma ,\omega \right) \right\} +C_{\kappa _{1},\left( \beta _{1},\gamma
_{1}\right) ,\kappa _{2},\left( \beta _{2},\gamma _{2}\right) }\varepsilon 
\mathfrak{N}_{T^{\alpha }}\left( \sigma ,\omega \right) .
\end{eqnarray*}%
If we now choose $\varepsilon >0$ so small that the term $C_{\kappa
_{1},\left( \beta _{1},\gamma _{1}\right) ,\kappa _{2},\left( \beta
_{2},\gamma _{2}\right) }\varepsilon \mathfrak{N}_{T^{\alpha }}\left( \sigma
,\omega \right) $ can be absorbed into the left hand side, we obtain the
desired inequality (\ref{once we show}). This completes the proof of both
inequalities (\ref{NIC}) and (\ref{NBICT}) in Theorem \ref{pivotal theorem}.

\section{Proof of Theorem \protect\ref{restricted weak type} on restricted
weak type}

Here we prove Theorem \ref{restricted weak type}, which we remind the
reader, does not assume comparability of measures. The proof of the theorem
is a standard application of an idea originating four and a half decades
ago, namely the 1973 $\limfunc{good}-\lambda $ inequality of Burkholder\ 
\cite{Bur}, and specifically the 1974 inequality of R. Coifman and C.
Fefferman \cite{CoFe}, and the related 1974 inequality of B. Muckenhoupt and
R. L, Wheeden \cite{MuWh}. We also introduce an $\alpha $-fractional
analogue $A_{\infty }^{\alpha }$ of the $A_{\infty }$ condition, and use it
to improve the inequality in \cite{MuWh} when $\alpha >0$. We begin by
briefly recalling the inequality of Coifman and Fefferman that relates
maximal truncations of a Calder\'{o}n-Zygmund singular integral to the
maximal operator $M$.

\subsection{Good $\protect\lambda $ inequalities}

Given an $\alpha $-fractional Calder\'{o}n-Zygmund operator $T^{\alpha }$,
define the maximal truncation operator $T_{\flat }^{\alpha }$ by%
\begin{equation*}
T_{\flat }^{\alpha }\left( f\sigma \right) \left( x\right) \equiv
\sup_{0<\varepsilon <R<\infty }\left\vert \int_{\left\{ \varepsilon
<\left\vert y\right\vert <R\right\} }K^{\alpha }\left( x,y\right) f\left(
y\right) d\sigma \left( y\right) \right\vert ,\ \ \ \ \ x\in \mathbb{R}^{n},
\end{equation*}%
for any locally finite positive Borel measure $\sigma $ on $\mathbb{R}^{n}$,
and $f\in L^{2}\left( \sigma \right) $. Define the $\alpha $-fractional
Hardy-Littlewood maximal operator $M^{\alpha }$ by%
\begin{equation*}
M^{\alpha }\left( f\sigma \right) \left( x\right) \equiv \sup_{Q\in \mathcal{%
P}^{n}:\ x\in Q}\frac{1}{\left\vert Q\right\vert ^{1-\frac{\alpha }{n}}}%
\int_{Q}\left\vert f\right\vert d\sigma ,\ \ \ \ \ x\in \mathbb{R}^{n},
\end{equation*}%
where here we may take the cubes $Q$ in the supremum to be closed.

Let $\omega $ be an $A_{\infty }$ weight. Suppose first that $\alpha =0$.
Then the Coifman-Fefferman $\limfunc{good}-\lambda $ inequality in \cite[see
inequality (7) on page 245]{CoFe} is%
\begin{equation}
\left\vert \left\{ x\in Q:\ T_{\flat }\left( f\sigma \right) \left( x\right)
>2\lambda \text{ and }M\left( f\sigma \right) \left( x\right) \leq \beta
\lambda \right\} \right\vert _{\omega }\leq C\beta ^{\varepsilon }\left\vert
\left\{ x\in Q:\ T_{\flat }\left( f\sigma \right) \left( x\right) >\lambda
\right\} \right\vert _{\omega },  \label{CF good}
\end{equation}%
for all $\lambda >0$, where $\varepsilon >0$ is the $A_{\infty }$ exponent
in (\ref{reform}). The kernels considered in \cite{CoFe} are convolution
kernels with order $1$ smoothness and bounded Fourier transform. However,
since we are assuming here that $T$ is bounded on unweighted $L^{2}\left( 
\mathbb{R}^{n}\right) $, standard Calder\'{o}n-Zygmund theory \cite[%
Corollary 2 on page 36]{Ste2} implies that $T_{\flat }$ is weak type $\left(
1,1\right) $ on Lebesgue measure. This estimate is the key to the proof in 
\cite[see pages 245-246 where the the weak type $\left( 1,1\right) $
inequality for $T_{\flat }$ is used]{CoFe}, and this proof shows that the
kernel of the operator $T$ may be taken to be a standard kernel in the sense
used here.

In the case $0<\alpha <n$, this $\limfunc{good}-\lambda $ inequality for an $%
A_{\infty }$ weight $\omega $ was extended in \cite{MuWh} (by essentially
the same proof) when $T_{\flat }$ and $M$ are replaced by $I^{\alpha }$ and $%
M^{\alpha }$ respectively:%
\begin{equation}
\left\vert \left\{ x\in Q:\ I^{\alpha }\left( f\sigma \right) \left(
x\right) >\gamma \lambda \text{ and }M^{\alpha }\left( f\sigma \right)
\left( x\right) \leq \beta \lambda \right\} \right\vert _{\omega }\leq \circ
\left( \frac{1}{\gamma }\right) \left\vert \left\{ x\in Q:\ I^{\alpha
}\left( f\sigma \right) \left( x\right) >\lambda \right\} \right\vert
_{\omega },  \label{CF good alpha}
\end{equation}%
for all $\lambda >0$, for some $\beta >0$ chosen sufficiently small. Here
the fractional integral $I^{\alpha }$ is given by $I^{\alpha }\nu \left(
x\right) \equiv \int_{\mathbb{R}^{n}}\left\vert x-y\right\vert ^{\alpha
-n}d\nu \left( y\right) $, and we will use below the obvious fact that $%
\left\vert T_{\flat }^{\alpha }\nu \left( x\right) \right\vert \leq
CI^{\alpha }\nu \left( x\right) $ for $d\nu \geq 0$. ($I_{\alpha }$ is a
positive operator satisfying the weak type $\left( 1,\frac{n}{n-\alpha }%
\right) $\ inequality on Lebesgue measure, and this is why there is no need
to assume any additional unweighted boundedness of $T^{\alpha }$ when $%
\alpha >0$).

However, it is possible to enlarge the collection of weights that satisfy (%
\ref{CF good alpha}) by using a relative $\alpha $-capacity of $E$ in $Q$ in
place of the ratio $\frac{\left\vert E\right\vert }{\left\vert Q\right\vert }
$ appearing in the definition of the $A_{\infty }$ condition. We introduce
this next.

\subsubsection{A fractional good-$\protect\lambda $ inequality}

Let $\mathcal{D}$ be a dyadic grid on $\mathbb{R}^{n}$. Suppose that $\Omega 
$ is an open subset of $\mathbb{R}^{n}$ with compact closure.\ We define the 
\emph{Whitney collection} $\mathcal{W}_{\Omega }$ to be the set $\left\{
Q_{j}\right\} _{j}$ of maximal dyadic cubes $Q_{j}\in \mathcal{D}$ such that 
$3Q_{j}\subset \Omega $. The following three properties are then immediate:%
\begin{equation*}
\left\{ 
\begin{array}{ll}
\text{(disjoint cover)} & \Omega =\bigcup_{j}Q_{j}\text{ and }Q_{j}\cap
Q_{i}=\emptyset \text{ if }i\neq j \\ 
\text{(Whitney condition)} & 3Q_{j}\subset \Omega \text{ and }9Q_{j}\cap
\Omega ^{c}\neq \emptyset \text{ for all }j \\ 
\text{(bounded overlap)} & \sum_{j}\mathbf{1}_{2Q_{j}}\leq C_{n}\mathbf{1}%
_{\Omega }%
\end{array}%
\right. .
\end{equation*}

\begin{definition}
Define the \emph{Whitney decomposition }$\boldsymbol{W}_{I_{\alpha }f}$ of
the fractional integral $I_{\alpha }f$ of a positive measure $f$ to be the
set whose elements are the Whitney collections $\mathcal{W}_{\Omega _{k}}$
for the open sets $\Omega _{k}\equiv \left\{ x\in \mathbb{R}^{n}:I_{\alpha
}f\left( x\right) >2^{k}\right\} $, $k\in \mathbb{Z}$, i.e.%
\begin{equation*}
\boldsymbol{W}_{I_{\alpha }f}\equiv \left\{ \mathcal{W}_{\Omega
_{k}}\right\} _{k\in \mathbb{Z}}\ ,
\end{equation*}%
which we can identify with $\left\{ Q_{j}^{k}\right\} _{k,j}$ if $\mathcal{W}%
_{\Omega _{k}}=\left\{ Q_{j}^{k}\right\} _{j}$.
\end{definition}

The nested property is immediate,%
\begin{equation*}
\begin{array}{ll}
\text{(nested property)} & Q_{j}^{k}\varsubsetneqq Q_{i}^{\ell }\text{
implies }k>\ell%
\end{array}%
,
\end{equation*}%
and the maximum principle is proved in \cite{Saw5}: there is $N$
sufficiently large that%
\begin{equation*}
\begin{array}{cc}
\text{(maximum principle)} & I_{\alpha }\left( \mathbf{1}_{2Q_{i}^{k-N}}f%
\right) \left( x\right) >2^{k-1}\text{ for }x\in \Omega _{k}\cap
Q_{i}^{k-N},\ \ \ \ \ \text{all }k,j.%
\end{array}%
\end{equation*}

Consider now a positive measure $f$ and the Whitney decomposition of $%
I_{\alpha }f$ where 
\begin{equation*}
\Omega _{k}=\left\{ I_{\alpha }f>2^{k}\right\}
=\dbigcup\limits_{j}Q_{j}^{k}\ \ \ \ \ \text{for }k\in \mathbb{Z}.
\end{equation*}%
We claim that if 
\begin{equation*}
E_{i}^{k-N}\left( \beta \right) \equiv Q_{i}^{k-N}\cap \left\{ I_{\alpha
}f>2^{k}\text{ and }M_{\alpha }f\leq \beta 2^{k-N}\right\} ,
\end{equation*}%
then we have%
\begin{equation}
\mathbf{Cap}_{\alpha }\left( E_{i}^{k-N}\left( \beta \right)
;Q_{i}^{k-N}\right) \leq \beta 2^{1-N},\ \ \ \ \ 0<\beta \leq 1.
\label{Cap ineq}
\end{equation}%
To see this we note that the cubes $Q_{j}^{k}$ and $Q_{i}^{k-N}$ above
satisfy%
\begin{eqnarray*}
\left\vert 9Q_{i}^{k-N}\right\vert ^{\frac{\alpha }{n}-1}%
\int_{9Q_{i}^{k-N}}f &\leq &\beta 2^{k-N}, \\
I_{\alpha }\left( \mathbf{1}_{2Q_{i}^{k-N}}f\right) \left( x\right)
&>&2^{k-1}\text{ for }x\in Q_{j}^{k},
\end{eqnarray*}%
where the first inequality follows from the Whitney condition, and the
second inequality from the maximum principle for fractional integrals. This
then shows that the nonnegative function $h\equiv \frac{1}{2^{k-1}}\mathbf{1}%
_{2Q_{i}^{k-N}}f$ satisfies $I_{\alpha }h\geq 1$ on $E$ and 
\begin{equation*}
\left\vert 2Q_{i}^{k-N}\right\vert ^{\frac{\alpha }{n}-1}%
\int_{2Q_{i}^{k-N}}h\leq \left( \frac{2}{9}\right) ^{n-\alpha }\left\vert
9Q_{i}^{k-N}\right\vert ^{\frac{\alpha }{n}-1}\int_{9Q_{i}^{k-N}}\frac{1}{%
2^{k-1}}f\leq \beta 2^{1-N},
\end{equation*}%
which proves (\ref{Cap ineq}).

Using the relative capacity inequality (\ref{Cap ineq}), we can now prove
the good-$\lambda $ inequality for the pair $\left( I_{\alpha },M_{\alpha
}\right) $\ with respect to an $A_{\infty }^{\alpha }$ measure $\omega $.

\begin{lemma}
\label{good Ialpha Malpha}If $\omega \in A_{\infty }^{\alpha }$, or $\omega
\in C_{q}$ for some $q>2$, then there are constants positive constants $%
C,\varepsilon $ such that 
\begin{equation}
\left\vert \left\{ I_{\alpha }f>\gamma \lambda \text{ and }M_{\alpha }f\leq
\beta \lambda \right\} \right\vert _{\omega }\leq C\left( \frac{\beta }{%
\gamma }\right) ^{\varepsilon }\left\vert \left\{ I_{\alpha }f>\lambda
\right\} \right\vert _{\omega }\ .  \label{good lambda alpha}
\end{equation}
\end{lemma}

\begin{proof}
The case where $\omega \in C_{q}$ for some $q>2$ is in \cite[see Remark 6 on
page 13]{CeLiPeRi} with an even smaller constant on the right, so we turn to
the case $\omega \in A_{\infty }^{\alpha }$. It clearly suffices to consider
the special cases where $\lambda =2^{k-N}$, $\gamma =2^{N}$ and $0<\beta
\leq 1$ for all $k\in \mathbb{Z}$ and all sufficiently large $N\in \mathbb{N}
$. Now with $\left\{ I_{\alpha }f>2^{k}\right\}
=\dbigcup\limits_{j}Q_{j}^{k} $ as above we have from the $A_{\infty
}^{\alpha }$ condition, 
\begin{eqnarray*}
\left\vert \left\{ I_{\alpha }f>2^{k}\text{ and }M_{\alpha }f\leq \beta
2^{k-N}\right\} \right\vert _{\omega } &\leq &\sum_{i}\left\vert
\dbigcup\limits_{\substack{ j:\ Q_{j}^{k}\subset Q_{i}^{k-N}  \\ %
Q_{j}^{k}\cap \left\{ M_{\alpha }f\leq \beta 2^{k-N}\right\} \neq \emptyset 
}}Q_{j}^{k}\right\vert _{\omega } \\
&\leq &\sum_{i}\left\vert 2Q_{i}^{k-N}\right\vert _{\omega }C\ \mathbf{Cap}%
_{\alpha }\left( E_{i}^{k-N}\left( \beta \right) ;Q_{i}^{k-N}\right)
^{\varepsilon } \\
&\leq &\sum_{i}C\left( \beta 2^{1-N}\right) ^{\varepsilon }\left\vert
Q_{i}^{k-N}\right\vert _{\omega }=C_{n}C\left( \beta 2^{1-N}\right)
^{\varepsilon }\left\vert \left\{ I_{\alpha }f>2^{k-N}\right\} \right\vert
_{\omega }\ ,
\end{eqnarray*}%
where $C_{n}$ is the bounded overlap constant in the Whitney collection.
This completes the proof of (\ref{good lambda alpha}).
\end{proof}

\subsection{Control of restricted weak type}

From such $\limfunc{good}-\lambda $ inequalities for $A_{\infty }$, $%
A_{\infty }^{\alpha }$ and $C_{q}$ weights $\omega $, standard arguments in 
\cite{CoFe}, \cite{MuWh}, \cite{Saw1} and \cite{CeLiPeRi} show that $%
\left\Vert T_{\flat }^{\alpha }\left( f\sigma \right) \right\Vert
_{L^{2}\left( \omega \right) }\lesssim $ $\left\Vert M^{\alpha }\left(
f\sigma \right) \right\Vert _{L^{2}\left( \omega \right) }$ for $0\leq
\alpha <n$ and $f\in L^{2}\left( \sigma \right) $. We will use a weak type
variant of this latter inequality, together with the equivalence of $%
\mathfrak{N}_{M^{\alpha }}^{\limfunc{weak}}\left( \sigma ,\omega \right) $
and $A_{2}^{\alpha }\left( \sigma ,\omega \right) $, to prove the theorem.

\begin{proof}[Proof of Theorem \protect\ref{restricted weak type}]
Since the restricted weak type inequality is self-dual, we can assume
without loss of generality that $\omega $ is an $A_{\infty }^{\alpha }$ or $%
C_{2+\varepsilon }$ weight. We begin by assuming that $\omega \in A_{\infty
}^{\alpha }$ and showing that the $\limfunc{good}-\lambda $ inequalities for 
$A_{\infty }^{\alpha }$ weights $\omega $ imply \emph{weak type} control,
exercising care in absorbing terms. Indeed, for $t>0$, we obtain from (\ref%
{CF good}) and (\ref{good lambda alpha}) that%
\begin{eqnarray*}
\sup_{0<\lambda \leq t}\lambda ^{2}\left\vert \left\{ T_{\flat }^{\alpha
}\left( f\sigma \right) >\lambda \right\} \right\vert _{\omega }
&=&4\sup_{0<\lambda \leq \frac{t}{2}}\lambda ^{2}\left\vert \left\{ T_{\flat
}^{\alpha }\left( f\sigma \right) >2\lambda \right\} \right\vert _{\omega }
\\
&\leq &4\sup_{0<\lambda \leq \frac{t}{2}}\lambda ^{2}\left\{ \left\vert
\left\{ M^{\alpha }\left( f\sigma \right) >\beta \lambda \right\}
\right\vert _{\omega }+\frac{C}{\beta }\left\vert \left\{ T_{\flat }^{\alpha
}\left( f\sigma \right) >\lambda \right\} \right\vert _{\omega }\right\} \\
&=&\frac{4}{\beta ^{2}}\sup_{0<\lambda \leq \frac{\beta }{2}t}\lambda
^{2}\left\vert \left\{ M^{\alpha }\left( f\sigma \right) >\beta \lambda
\right\} \right\vert _{\omega }+4\sup_{0<\lambda \leq \frac{t}{2}}\frac{C}{%
\beta }\left\vert \left\{ T_{\flat }^{\alpha }\left( f\sigma \right)
>\lambda \right\} \right\vert _{\omega } \\
&\leq &\frac{4}{\beta ^{2}}\left\Vert M^{\alpha }\left( f\sigma \right)
\right\Vert _{L^{2,\infty }\left( \omega \right) }^{2}+\frac{4C}{\beta }%
\sup_{0<\lambda \leq t}\lambda ^{2}\left\vert \left\{ T_{\flat }^{\alpha
}\left( f\sigma \right) >\lambda \right\} \right\vert _{\omega }\ .
\end{eqnarray*}%
Now choose $\beta $ so that $\frac{4C}{\beta }=\frac{1}{2}$. Provided that $%
\sup_{0<\lambda \leq t}\lambda ^{2}\left\vert \left\{ T_{\flat }^{\alpha
}\left( f\sigma \right) >\lambda \right\} \right\vert _{\omega }$ is finite
for each $t>0$, we can absorb the final term on the right into the left hand
side to obtain%
\begin{equation*}
\sup_{0<\lambda \leq t}\lambda ^{2}\left\vert \left\{ T_{\flat }^{\alpha
}\left( f\sigma \right) >\lambda \right\} \right\vert _{\omega }\leq \frac{8%
}{\beta ^{2}}\left\Vert M^{\alpha }\nu \right\Vert _{L^{2,\infty }\left(
\omega \right) }^{2},\ \ \ \ \ t>0,
\end{equation*}%
which gives%
\begin{equation*}
\left\Vert T_{\flat }^{\alpha }\left( f\sigma \right) \right\Vert
_{L^{2,\infty }\left( \omega \right) }^{2}=\sup_{0<\lambda <\infty }\lambda
^{2}\left\vert \left\{ T_{\flat }^{\alpha }\left( f\sigma \right) >\lambda
\right\} \right\vert _{\omega }\leq \frac{8}{\beta ^{2}}\left\Vert M^{\alpha
}\left( f\sigma \right) \right\Vert _{L^{2,\infty }\left( \omega \right)
}^{2}\ .
\end{equation*}

Suppose first that $\alpha =0$. In order to obtain finiteness of the
supremum over $0<\lambda \leq t$, we take $f\in L^{2}\left( \sigma \right) $
wtih $\left\vert f\right\vert \leq 1$ and $\limfunc{supp}f\subset B\left(
0,r\right) $ with $1\leq r<\infty $ and $\left\vert B\left( 0,r\right)
\right\vert _{\sigma }>0$. Then if $x\notin B\left( 0,2r\right) $, we have $%
\left\vert K\left( x,y\right) \right\vert \leq C_{CZ}r^{-n}$ and hence%
\begin{equation*}
T_{\flat }\left( f\sigma \right) \left( x\right) =\sup_{0<\varepsilon
<R<\infty }\left\vert \int_{\left\{ \varepsilon <\left\vert y\right\vert
<R\right\} \cap B\left( 0,r\right) }K\left( x,y\right) f\left( y\right)
d\sigma \left( y\right) \right\vert \leq C_{CZ}\left( \frac{2}{\left\vert
x\right\vert }\right) ^{n}\left\vert B\left( 0,r\right) \right\vert _{\sigma
}\ .
\end{equation*}%
This shows that 
\begin{eqnarray*}
&&\sup_{0<\lambda \leq t}\lambda ^{2}\left\vert \left\{ T_{\flat }\nu
>\lambda \right\} \right\vert _{\omega }\leq t^{2}\left\vert B\left(
0,2r\right) \right\vert _{\omega }+\sup_{0<\lambda <C_{CZ}r^{-n}\left\vert
B\left( 0,r\right) \right\vert _{\sigma }}\lambda ^{2}\left\vert \left\{
T_{\flat }\nu >\lambda \right\} \setminus B\left( 0,2r\right) \right\vert
_{\omega } \\
&\leq &t^{2}\left\vert B\left( 0,2r\right) \right\vert _{\omega
}+\sup_{0<\lambda \leq t}\lambda ^{2}\left\vert \left\{ C_{CZ}\left( \frac{2%
}{\left\vert x\right\vert }\right) ^{n}\left\vert B\left( 0,r\right)
\right\vert >\lambda \right\} \right\vert _{\omega } \\
&=&t^{2}\left\vert B\left( 0,2r\right) \right\vert _{\omega
}+\sup_{0<\lambda \leq t}\lambda ^{2}\left\vert B\left( 0,\gamma _{\lambda
}r\right) \right\vert _{\omega }\ ,
\end{eqnarray*}%
with $\gamma _{\lambda }\equiv 2\sqrt[n]{\frac{C_{CZ}c}{\lambda }}$, since $%
\left\{ C_{CZ}\left( \frac{2}{\left\vert x\right\vert }\right)
^{n}\left\vert B\left( 0,r\right) \right\vert >\lambda \right\} =B\left( 0,2%
\sqrt[n]{\frac{C_{CZ}c}{\lambda }}r\right) $ where $\left\vert B\left(
0,r\right) \right\vert =cr^{n}$.

On the other hand, the $A_{2}$ condition implies that for $\lambda \leq
\lambda _{0}\equiv C_{CZ}c$, we have $\gamma _{\lambda }\geq \gamma
_{\lambda _{0}}=2$ so that 
\begin{equation*}
\left\vert B\left( 0,\gamma _{\lambda }r\right) \right\vert _{\omega
}\lesssim A_{2}\left( \sigma ,\omega \right) \frac{\left\vert B\left(
0,\gamma _{\lambda }r\right) \right\vert ^{2}}{\left\vert B\left( 0,\gamma
_{\lambda }r\right) \right\vert _{\sigma }}\leq A_{2}\left( \sigma ,\omega
\right) \frac{\left( \gamma _{\lambda }r\right) ^{2n}}{\left\vert B\left(
0,2r\right) \right\vert _{\sigma }}A_{2}\left( \sigma ,\omega \right) =\frac{%
4^{n}\left( \frac{C_{CZ}c}{\lambda }\right) ^{2}}{\left\vert B\left(
0,2r\right) \right\vert _{\sigma }},
\end{equation*}%
and hence%
\begin{equation*}
\lambda ^{2}\left\vert B\left( 0,\gamma _{\lambda }r\right) \right\vert
_{\omega }\leq \lambda ^{2}\frac{4^{n}\left( \frac{C_{CZ}c}{\lambda }\right)
^{2}}{\left\vert B\left( 0,2r\right) \right\vert _{\sigma }}=\frac{%
4^{n}C_{CZ}c^{2}}{\left\vert B\left( 0,2r\right) \right\vert _{\sigma }},\ \
\ \ \ \text{for }\lambda \leq \lambda _{0}\ .
\end{equation*}

Finally we have%
\begin{equation*}
\sup_{\lambda _{0}<\lambda \leq t}\lambda ^{2}\left\vert B\left( 0,\gamma
_{\lambda }r\right) \right\vert _{\omega }\leq t^{2}\left\vert B\left(
0,\gamma _{\lambda _{0}}r\right) \right\vert _{\omega }=t^{2}\left\vert
B\left( 0,2r\right) \right\vert _{\omega }\ ,
\end{equation*}%
and altogether then%
\begin{equation*}
\sup_{0<\lambda \leq t}\lambda ^{2}\left\vert \left\{ T_{\flat }\nu >\lambda
\right\} \right\vert _{\omega }\leq t^{2}\left\vert B\left( 0,2r\right)
\right\vert _{\omega }+\frac{4^{n}C_{CZ}c^{2}}{\left\vert B\left(
0,2r\right) \right\vert _{\sigma }}+t^{2}\left\vert B\left( 0,2r\right)
\right\vert _{\omega }
\end{equation*}%
which is finite for $0<t<\infty $.

Thus we conclude that 
\begin{equation*}
\mathfrak{N}_{T}^{\limfunc{restricted}\limfunc{weak}}\left( \sigma ,\omega
\right) \leq \mathfrak{N}_{T}^{\limfunc{weak}}\left( \sigma ,\omega \right)
\leq \mathfrak{N}_{T_{\flat }}^{\limfunc{weak}}\left( \sigma ,\omega \right)
\lesssim \mathfrak{N}_{M}^{\limfunc{weak}}\left( \sigma ,\omega \right)
\approx A_{2}\left( \sigma ,\omega \right) ,
\end{equation*}%
where the final equivalence is well known, and can be obtained by averaging
over dyadic grids $\mathcal{D}$ the inequality $\mathfrak{N}_{M_{\mathcal{D}%
}^{\alpha }}^{\limfunc{weak}}\left( \sigma ,\omega \right) \lesssim
A_{2}\left( \sigma ,\omega \right) $ for dyadic operators 
\begin{equation*}
M_{\mathcal{D}}^{\alpha }f\left( x\right) \equiv \sup_{Q\in \mathcal{P}%
^{n}:\ x\in Q}\frac{1}{\left\vert Q\right\vert ^{1-\frac{\alpha }{n}}}%
\int_{Q}\left\vert f\right\vert d\sigma .
\end{equation*}%
The dyadic inequality is in turn an immediate consequence of the dyadic
covering lemma. Conversely, if $T^{\alpha }$ is elliptic, then $A_{2}\left(
\sigma ,\omega \right) \lesssim \mathfrak{N}_{T}^{\limfunc{restricted}%
\limfunc{weak}}\left( \sigma ,\omega \right) $ (see \cite{LiTr} and \cite%
{SaShUr7}).

The same sort of arguments give the analogous inequality when $0<\alpha <n$, 
\begin{equation*}
\mathfrak{N}_{T^{\alpha }}^{\limfunc{restricted}\limfunc{weak}}\left( \sigma
,\omega \right) \leq \mathfrak{N}_{I^{\alpha }}^{\limfunc{restricted}%
\limfunc{weak}}\left( \sigma ,\omega \right) \leq \mathfrak{N}_{I^{\alpha
}}^{\limfunc{weak}}\left( \sigma ,\omega \right) \lesssim \mathfrak{N}%
_{M^{\alpha }}^{\limfunc{weak}}\left( \sigma ,\omega \right) \approx
A_{2}^{\alpha }\left( \sigma ,\omega \right) ,
\end{equation*}%
and conversely, $A_{2}^{\alpha }\left( \sigma ,\omega \right) \lesssim 
\mathfrak{N}_{T^{\alpha }}^{\limfunc{restricted}\limfunc{weak}}\left( \sigma
,\omega \right) $ if $T^{\alpha }$ is elliptic.

Finally, when $\omega \in C_{2+\varepsilon }$, the proof for strong type
norms of $T_{\flat }f$ and $Mf$ in \cite{Saw1} is easily adapted to weak
type norms, while the proof for strong type norms of $I_{\alpha }f$ and $%
M_{\alpha }f$ in \cite[See Subsubsection 7.2.1 Lemmata on pages 33-35.]%
{CeLiPeRi} - which follows closely the arguments in \cite{Saw1} - is easily
adapted to weak type norms as was just done above. This completes the proof
of Theorem \ref{restricted weak type}.
\end{proof}

\section{Proof of Theorem \protect\ref{A infinity theorem}, a $T1$ theorem}

Inequality (\ref{more'}) in Theorem \ref{A infinity theorem} follows
immediately from Theorems \ref{pivotal theorem} and \ref{restricted weak
type}. On the other hand, if $T^{\alpha }$ is strongly elliptic, then 
\begin{equation*}
\sqrt{\mathcal{A}_{2}^{\alpha }\left( \sigma ,\omega \right) +\mathcal{A}%
_{2}^{\alpha }\left( \omega ,\sigma \right) }\lesssim \mathfrak{N}%
_{T^{\alpha }}\left( \sigma ,\omega \right) ,
\end{equation*}%
by \cite[Lemma 4.1 on page 92.]{SaShUr7}. This completes the proof of
Theorem \ref{A infinity theorem}.

\begin{remark}
If we drop the assumption that one of the weights is $A_{\infty }^{\alpha }$
or $C_{2+\varepsilon }$, then inequality (\ref{more'}) remains true if we
include on the right hand side the Bilinear Indicator Cube Testing constant $%
\mathcal{BICT}_{T^{\alpha }}\left( \sigma ,\omega \right) $ defined in (\ref%
{def ind WBP}) above.
\end{remark}

\section{Proof of Theorem \protect\ref{Stein extension} on optimal
cancellation conditions}

Here we follow very closely the treatment in \cite[Section 3 of Chapter VII]%
{Ste2} to show how Theorem \ref{Stein extension} follows from Theorem \ref{A
infinity theorem}. The argument we follow in \cite[Section 3 of Chapter VII]%
{Ste2} uses balls instead of cubes, and we thus adapt the argument to cubes
by using the distance function%
\begin{equation*}
\left\Vert y\right\Vert \equiv \max_{1\leq k\leq n}\left\vert
y_{k}\right\vert 
\end{equation*}%
instead of the Euclidean distance $\left\vert y\right\vert =\sqrt{%
\sum_{1\leq k\leq n}\left\vert y_{k}\right\vert ^{2}}$, which results in
minimal cosmetic changes. Thus the corresponding balls $B\left( x,r\right)
=\left\{ y\in \mathbb{R}^{n}:\left\Vert x-y\right\Vert <r\right\} $ are
familiar cubes with sides parallel to the axes, centered at $x$ with side
length $2r$. In order to free up superscripts for other uses, we will drop
the fractional superscript $\alpha $ from both the kernel $K^{\alpha }$ and
its associated operator $T^{\alpha }$. Finally, we will need the following
result on truncations, which extends the case $q=2$ of Proposition 1 in
Stein \cite[page 31]{Ste2} to a pair of doubling measures $\sigma $ and $%
\omega $.

\subsection{Boundedness of truncations}

For $\varepsilon >0$, and a smooth $\alpha $-fractional Calder\'{o}n-Zygmund
kernel $K\left( x,y\right) $, define the truncated kernels%
\begin{equation*}
K_{\varepsilon }\left( x,y\right) \equiv \left\{ 
\begin{array}{ccc}
K\left( x,y\right)  & \text{ if } & \varepsilon <\left\Vert x-y\right\Vert 
\\ 
0 & \text{ if } & not%
\end{array}%
\right. ,
\end{equation*}%
and set%
\begin{equation*}
T_{\varepsilon }\left( x\right) \equiv \int K_{\varepsilon }\left(
x,y\right) f\left( y\right) d\sigma \left( y\right) ,\ \ \ \ \ \text{for }%
x\in \mathbb{R}^{n}\text{ and }f\in L^{2}\left( \sigma \right) .
\end{equation*}

\begin{proposition}
Suppose that $\sigma $ and $\omega $ are positive locally finite Borel
measures on $\mathbb{R}^{n}$ satisfying the classical $A_{2}^{\alpha }$
condition, and that $K\left( x,y\right) $ is a smooth $\alpha $-fractional
Calder\'{o}n-Zygmund kernel on $\mathbb{R}^{n}$. Suppose moreover that there
is a bounded operator $T:L^{2}\left( \sigma \right) \rightarrow L^{2}\left(
\omega \right) $, i.e.%
\begin{equation*}
\left\Vert T\left( f\sigma \right) \right\Vert _{L^{2}\left( \omega \right)
}\leq A\left\Vert f\right\Vert _{L^{2}\left( \sigma \right) },\ \ \ \ \ 
\text{for all }f\in L^{2}\left( \sigma \right) ,
\end{equation*}%
associated with the kernel $K\left( x,y\right) $ in the sense that (\ref%
{identify}) holds. Then there is a positive constant $A^{\prime }$ such that
the truncations $T_{\varepsilon }$ satisfy%
\begin{equation}
\left\Vert T_{\varepsilon }\left( f\sigma \right) \right\Vert _{L^{2}\left(
\omega \right) }\leq A^{\prime }\left\Vert f\right\Vert _{L^{2}\left( \sigma
\right) },\ \ \ \ \ \text{for all }f\in L^{2}\left( \sigma \right) \text{
and }\varepsilon >0\text{.}  \label{A'}
\end{equation}%
Moreover, $A^{\prime }\approx A+\sqrt{A_{2}^{\alpha }}$.
\end{proposition}

\begin{proof}
The proof is virtually identical to that of Stein in \cite[page 31]{Ste2}
(which treated a doubling measure $\mu $ in place of Lebesgue measure $dx$)
upon including appropriate use of the classical $A_{2}^{\alpha }\left(
\sigma ,\omega \right) $ condition to handle the extension to two otherwise
arbitrary weights, and we now sketch the details.

For each $x$, the function $K_{\varepsilon }\left( x,\cdot \right) $ is in $%
L^{2}\left( \sigma \right) $ and so $T_{\varepsilon }$ is well-defined on $%
L^{2}\left( \sigma \right) $ by Cauchy-Schwarz. Let $\widetilde{T}%
_{\varepsilon }\equiv T-T_{\varepsilon }$ be the `near' part of $T$. Fix $%
\overline{x}\in \mathbb{R}^{n}$ and $f\in L^{2}\left( \sigma \right) $. All
estimates in what follows are independent of $\varepsilon $, $\overline{x}$
and $f$. The crux of the proof is then to show that there are positive
numbers $C$ and $0<a<\frac{1}{3}$ so that%
\begin{equation}
\left\Vert \mathbf{1}_{B\left( \overline{x},a\varepsilon \right) }\widetilde{%
T}_{\varepsilon }\left( f\sigma \right) \right\Vert _{L^{2}\left( \omega
\right) }\leq \left( A+C\left( 1+\frac{1}{a}\right) ^{n}\sqrt{A_{2}^{\alpha }%
}\right) \left\Vert \mathbf{1}_{B\left( \overline{x},\left( a+1\right)
\varepsilon \right) }f\right\Vert _{L^{2}\left( \sigma \right) }\ ,
\label{crux}
\end{equation}%
where the balls $B\left( \overline{x},r\right) $ are actually cubes in the
new distance function, and we will often refer to them as cubical balls.

Note that $\widetilde{T}_{\varepsilon }\left( f\sigma \right) \left(
x\right) =0$ if $\func{Supp}f\subset B\left( x,\varepsilon \right) ^{c}$ and
that $\widetilde{T}_{\varepsilon }\left( f\sigma \right) \left( x\right)
=T\left( f\sigma \right) \left( x\right) $ if $\func{Supp}f\subset B\left(
x,\varepsilon \right) $, so that%
\begin{equation*}
\mathbf{1}_{B\left( \overline{x},a\varepsilon \right) }\widetilde{T}%
_{\varepsilon }\left( f\sigma \right) =\mathbf{1}_{B\left( \overline{x}%
,a\varepsilon \right) }\widetilde{T}_{\varepsilon }\left( \mathbf{1}%
_{B\left( \overline{x},\left( a+1\right) \varepsilon \right) }f\sigma
\right) .
\end{equation*}%
Next we split the right hand side into two pieces:%
\begin{equation}
\mathbf{1}_{B\left( \overline{x},a\varepsilon \right) }\widetilde{T}%
_{\varepsilon }\left( \mathbf{1}_{B\left( \overline{x},\left( a+1\right)
\varepsilon \right) }f\sigma \right) =\mathbf{1}_{B\left( \overline{x}%
,a\varepsilon \right) }\widetilde{T}_{\varepsilon }\left( \mathbf{1}%
_{B\left( \overline{x},d\varepsilon \right) }f\sigma \right) +\mathbf{1}%
_{B\left( \overline{x},a\varepsilon \right) }\widetilde{T}_{\varepsilon
}\left( \left[ \mathbf{1}_{B\left( \overline{x},\left( a+1\right)
\varepsilon \right) }-\mathbf{1}_{B\left( \overline{x},d\varepsilon \right) }%
\right] f\sigma \right) ,  \label{two pieces}
\end{equation}%
where we choose $2a<d<1-a$. In particular, $B\left( \overline{x}%
,d\varepsilon \right) \subset B\left( x,\varepsilon \right) $ whenever $x\in
B\left( \overline{x},a\varepsilon \right) $. This gives%
\begin{equation*}
\mathbf{1}_{B\left( \overline{x},a\varepsilon \right) }\widetilde{T}%
_{\varepsilon }\left( \mathbf{1}_{B\left( \overline{x},d\varepsilon \right)
}f\sigma \right) =\mathbf{1}_{B\left( \overline{x},a\varepsilon \right)
}T\left( \mathbf{1}_{B\left( \overline{x},d\varepsilon \right) }f\sigma
\right) ,
\end{equation*}%
and%
\begin{eqnarray*}
\left\Vert \mathbf{1}_{B\left( \overline{x},a\varepsilon \right) }\widetilde{%
T}_{\varepsilon }\left( \mathbf{1}_{B\left( \overline{x},d\varepsilon
\right) }f\sigma \right) \right\Vert _{L^{2}\left( \omega \right) }
&=&\left\Vert \mathbf{1}_{B\left( \overline{x},a\varepsilon \right) }T\left( 
\mathbf{1}_{B\left( \overline{x},d\varepsilon \right) }f\sigma \right)
\right\Vert _{L^{2}\left( \omega \right) } \\
&\leq &A\left\Vert \mathbf{1}_{B\left( \overline{x},d\varepsilon \right)
}f\right\Vert _{L^{2}\left( \sigma \right) }\leq A\left\Vert \mathbf{1}%
_{B\left( \overline{x},\left( a+1\right) \varepsilon \right) }f\right\Vert
_{L^{2}\left( \sigma \right) }\ .
\end{eqnarray*}

To estimate the second term on the right hand side of (\ref{two pieces}), we
use $a<d$ and the association of $T$ with $K$ given in (\ref{identify}) to
obtain%
\begin{eqnarray*}
\mathbf{1}_{B\left( \overline{x},a\varepsilon \right) }\left( x\right) 
\widetilde{T}_{\varepsilon }\left( \left[ \mathbf{1}_{B\left( \overline{x}%
,\left( a+1\right) \varepsilon \right) }-\mathbf{1}_{B\left( \overline{x}%
,d\varepsilon \right) }\right] f\sigma \right) \left( x\right) 
&=&\int_{B\left( x,\varepsilon \right) \cap \left\{ B\left( \overline{x}%
,\left( a+1\right) \varepsilon \right) \setminus B\left( \overline{x}%
,d\varepsilon \right) \right\} }K\left( x,y\right) f\left( y\right) d\sigma
\left( y\right) , \\
\text{for }\sigma \text{-a.e. }x &\in &B\left( \overline{x},a\varepsilon
\right) ,
\end{eqnarray*}%
since the cubical annulus $B\left( \overline{x},\left( a+1\right)
\varepsilon \right) \setminus B\left( \overline{x},d\varepsilon \right) $ is
disjoint from the cubical ball $B\left( \overline{x},a\varepsilon \right) $.
For $y$ in the above range of integration, we have%
\begin{equation*}
d\varepsilon <\left\Vert \overline{x}-y\right\Vert \leq \left\Vert \overline{%
x}-x\right\Vert +\left\Vert x-y\right\Vert \leq a\varepsilon +\left\Vert
x-y\right\Vert ,
\end{equation*}%
and using $2a<d$, we conclude that $\left\Vert x-y\right\Vert \geq \left(
d-a\right) \varepsilon \geq a\varepsilon $. Thus $\left\vert K\left(
x,y\right) \right\vert \leq \frac{C}{\left( a\varepsilon \right) ^{n}}%
=C\left( 1+\frac{1}{a}\right) ^{n}\frac{1}{\left\vert B\left( \overline{x}%
,\left( a+1\right) \varepsilon \right) \right\vert }$, and so%
\begin{eqnarray*}
&&\left\Vert \mathbf{1}_{B\left( \overline{x},a\varepsilon \right) }%
\widetilde{T}_{\varepsilon }\left( \left[ \mathbf{1}_{B\left( \overline{x}%
,\left( a+1\right) \varepsilon \right) }-\mathbf{1}_{B\left( \overline{x}%
,d\varepsilon \right) }\right] f\sigma \right) \right\Vert _{L^{2}\left(
\omega \right) } \\
&=&\left\Vert \int_{B\left( x,\varepsilon \right) \cap \left\{ B\left( 
\overline{x},\left( a+1\right) \varepsilon \right) \setminus B\left( 
\overline{x},d\varepsilon \right) \right\} }K\left( x,y\right) f\left(
y\right) d\sigma \left( y\right) \right\Vert _{L^{2}\left( \omega \right) }
\\
&\leq &C\left( 1+\frac{1}{a}\right) ^{n}\frac{1}{\left\vert B\left( 
\overline{x},\left( a+1\right) \varepsilon \right) \right\vert }\sqrt{%
\left\vert B\left( \overline{x},\left( a+1\right) \varepsilon \right)
\right\vert _{\omega }}\int_{B\left( \overline{x},\left( a+1\right)
\varepsilon \right) }\left\vert f\left( y\right) \right\vert d\sigma \left(
y\right)  \\
&\leq &C\left( 1+\frac{1}{a}\right) ^{n}\frac{\sqrt{\left\vert B\left( 
\overline{x},\left( a+1\right) \varepsilon \right) \right\vert _{\omega }}%
\sqrt{\left\vert B\left( \overline{x},\left( a+1\right) \varepsilon \right)
\right\vert _{\sigma }}}{\left\vert B\left( \overline{x},\left( a+1\right)
\varepsilon \right) \right\vert }\left\Vert \mathbf{1}_{B\left( \overline{x}%
,\left( a+1\right) \varepsilon \right) }f\right\Vert _{L^{2}\left( \sigma
\right) } \\
&\leq &C\left( 1+\frac{1}{a}\right) ^{n}\sqrt{A_{2}^{\alpha }\left( \sigma
,\omega \right) }\left\Vert \mathbf{1}_{B\left( \overline{x},\left(
a+1\right) \varepsilon \right) }f\right\Vert _{L^{2}\left( \sigma \right) }.
\end{eqnarray*}%
Plugging our two estimates into (\ref{two pieces}), we obtain (\ref{crux}).

As in \cite[page 31]{Ste2}, we now add up the inequalities in (\ref{crux})
for a suitable collection of cubical balls covering $\mathbb{R}^{n}$ to
obtain (\ref{A'}) with $A^{\prime }=2^{n}\left( 1+\frac{1}{a}\right)
^{n}\left( A+C\left( 1+\frac{1}{a}\right) ^{n}\sqrt{A_{2}^{\alpha }}\right)
^{2}$. Indeed, we have%
\begin{eqnarray*}
\int_{\mathbb{R}^{n}}\left\vert \widetilde{T}_{\varepsilon }\left( f\sigma
\right) \right\vert ^{2}d\omega  &\leq &\sum_{k=1}^{\infty }\int_{B\left( 
\overline{x},a\varepsilon \right) }\left\vert \widetilde{T}_{\varepsilon
}\left( f\sigma \right) \right\vert ^{2}d\omega  \\
&\leq &\left( A+C\sqrt{A_{2}^{\alpha }}\right) ^{2}\sum_{k=1}^{\infty
}\int_{B\left( \overline{x},\left( a+1\right) \varepsilon \right)
}\left\vert f\right\vert ^{2}d\sigma \leq \left( A+C\sqrt{A_{2}^{\alpha }}%
\right) ^{2}N\int_{\mathbb{R}^{n}}\left\vert f\right\vert ^{2}d\sigma 
\end{eqnarray*}%
provided $\dbigcup\limits_{k}B\left( \overline{x}^{k},a\varepsilon \right) =%
\mathbb{R}^{n}$ and $\sum_{k}\mathbf{1}_{B\left( \overline{x}^{k},\left(
a+1\right) \varepsilon \right) }\leq N$. But these two properties are
achieved for any $N>2^{n}\left( 1+\frac{1}{a}\right) ^{n}-1$ by letting $%
\left\{ B\left( \overline{x}^{k},\frac{a}{2}\varepsilon \right) \right\}
_{k=1}^{\infty }$ be a maximal pairwise disjoint collection:

\begin{enumerate}
\item If $z\in \mathbb{R}^{n}\setminus \dbigcup\limits_{k}B\left( \overline{x%
}^{k},a\varepsilon \right) $, then $B\left( z,\frac{a}{2}\varepsilon \right)
\cap \left[ \dbigcup\limits_{k}B\left( \overline{x}^{k},\frac{a}{2}%
\varepsilon \right) \right] =\emptyset $ since if there is $w$ in $B\left( z,%
\frac{a}{2}\varepsilon \right) \cap B\left( \overline{x}^{k},\frac{a}{2}%
\varepsilon \right) $, then $\left\Vert z-\overline{x}^{k}\right\Vert \leq
\left\Vert z-w\right\Vert +\left\Vert w-\overline{x}^{k}\right\Vert
<a\varepsilon $, contradicting pairwise disjointedness of the collection $%
\left\{ B\left( \overline{x}^{k},\frac{a}{\sqrt{n}}\varepsilon \right)
\right\} _{k=1}^{\infty }$. But then $B\left( z,\frac{a}{2}\varepsilon
\right) $ could be included in the collection $\left\{ B\left( \overline{x}%
^{k},\frac{a}{2}\varepsilon \right) \right\} _{k=1}^{\infty }$,
contradicting its maximality.

\item If $z\in \dbigcap\limits_{j=1}^{N+1}B\left( \overline{x}%
^{k_{j}},\left( a+1\right) \varepsilon \right) $, then $B\left( \overline{x}%
^{k_{j}},a\varepsilon \right) \subset B\left( z,2\left( a+1\right)
\varepsilon \right) $ and so%
\begin{equation*}
c\left( 2\left( a+1\right) \varepsilon \right) ^{n}=\left\vert B\left(
z,2a\varepsilon \right) \right\vert \geq \left\vert
\dbigcup\limits_{j=1}^{N+1}B\left( \overline{x}^{k_{j}},a\varepsilon \right)
\right\vert =\sum_{j=1}^{N+1}\left\vert B\left( \overline{x}%
^{k_{j}},a\varepsilon \right) \right\vert =\left( N+1\right) c\left(
a\varepsilon \right) ^{n},
\end{equation*}%
which is a contradiction if $N+1>2^{n}\left( 1+\frac{1}{a}\right) ^{n}$.
\end{enumerate}
\end{proof}

\subsection{The cancellation theorem}

Now we turn to the proof of Theorem \ref{Stein extension}, where we follow
Stein \cite[Section 3 of Chapter VII]{Ste2}, but subtracting a higher order
Taylor polynomial to control estimates for \emph{doubling} measures.

\begin{proof}[Proof of Theorem \protect\ref{Stein extension}]
Recall the cancellation condition (\ref{can cond}), 
\begin{equation}
\int_{\left\Vert x-x_{0}\right\Vert <N}\left\vert \int_{\varepsilon
<\left\Vert x-y\right\Vert <N}K\left( x,y\right) d\sigma \left( y\right)
\right\vert ^{2}d\omega \left( x\right) \leq \mathfrak{A}_{K}\left( \sigma
,\omega \right) \left\vert B\left( x_{0},N\right) \right\vert _{\sigma },
\label{can cond'}
\end{equation}%
for all $\varepsilon ,N,x_{0}$. By the previous proposition, together with
the Independence of Truncations at the end of Subsubsection \ref%
{independence}, the roughly truncated operators $T_{\varepsilon ,N}$, with
kernel $K_{\varepsilon ,N}\left( x,y\right) =K\left( x,y\right) \mathbf{1}%
_{\left\{ \varepsilon <\left\vert x-y\right\vert <N\right\} }$, are bounded
from $L^{2}\left( \sigma \right) $ to $L^{2}\left( \omega \right) $ by a
multiple of $\left\Vert T\right\Vert _{L^{2}\left( \sigma \right)
\rightarrow L^{2}\left( \omega \right) }$ uniformly in $0<\varepsilon
<N<\infty $. Thus we have the following Cube Testing condition for $%
T_{\varepsilon ,N}$ uniformly in $0<\varepsilon <N<\infty $, i.e. 
\begin{equation}
\int_{B\left( x_{0},N\right) }\left\vert \int_{\varepsilon <\left\Vert
x-y\right\Vert <N}K\left( x,y\right) \mathbf{1}_{B\left( x_{0},N\right)
}\left( y\right) d\sigma \left( y\right) \right\vert ^{2}d\omega \left(
x\right) \leq \left\Vert T\right\Vert _{L^{2}\left( \sigma \right)
\rightarrow L^{2}\left( \omega \right) }^{2}\left\vert B\left(
x_{0},N\right) \right\vert _{\sigma }\ ,  \label{IB test}
\end{equation}%
for all cubical balls $B\left( x_{0},N\right) $. However, the inner
integrals with respect to $\sigma $ in (\ref{can cond'}) and (\ref{IB test})
don't match up. On the other hand, their difference is an integral in $%
\sigma $ supported \emph{outside} the cubical ball $B\left( x_{0},N\right) $
where $\omega $ is supported. This fact is exploited in the following
argument of Stein \cite[Section 3 of Chapter VII]{Ste2}.

We begin by proving the necessity of (\ref{can cond}) for the norm
inequality, i.e. 
\begin{equation*}
\mathfrak{A}_{K}\left( \sigma ,\omega \right) \lesssim \left\Vert
T\right\Vert _{L^{2}\left( \sigma \right) \rightarrow L^{2}\left( \omega
\right) }^{2}+A_{2}^{\alpha }\left( \sigma ,\omega \right) .
\end{equation*}%
Set%
\begin{equation*}
I_{\varepsilon ,N}\left( x\right) \equiv \int_{\varepsilon <\left\Vert
x-y\right\Vert <N}K\left( x,y\right) d\sigma \left( y\right) .
\end{equation*}%
First observe that it suffices to show%
\begin{equation}
\int_{\left\Vert x-x_{0}\right\Vert <\frac{N}{2}}\left\vert I_{\varepsilon
,N}\left( x\right) \right\vert ^{2}d\omega \left( x\right) \leq \left\Vert
T\right\Vert _{L^{2}\left( \sigma \right) \rightarrow L^{2}\left( \omega
\right) }^{2}\left\vert B\left( x,N\right) \right\vert _{\sigma },
\label{it suff}
\end{equation}%
since every cubical ball $B\left( x_{0},N\right) $ of radius $N$ can be
covered by a bounded number $J$ of cubical balls of radius $\frac{N}{2}$ ($%
2^{n}$ such cubical balls suffice). Indeed if $B\left( x_{0},N\right)
\subset \overset{\cdot }{\dbigcup }_{j=1}^{J}B\left( x_{j},\frac{N}{2}%
\right) $, then%
\begin{eqnarray*}
&&\int_{\left\Vert x-x_{0}\right\Vert <N}\left\vert \int_{\varepsilon
<\left\Vert x-y\right\Vert <N}K\left( x,y\right) d\sigma \left( y\right)
\right\vert ^{2}d\omega \left( x\right)  \\
&\leq &\sum_{j=1}^{J}\int_{\left\Vert x-x_{j}\right\Vert <\frac{N}{2}%
}\left\vert \int_{\varepsilon <\left\Vert x-y\right\Vert <N}K\left(
x,y\right) d\sigma \left( y\right) \right\vert ^{2}d\omega \left( x\right) 
\\
&\leq &\sum_{j=1}^{J}\left\Vert T\right\Vert _{L^{2}\left( \sigma \right)
\rightarrow L^{2}\left( \omega \right) }^{2}\left\vert B\left( x_{j},\frac{N%
}{2}\right) \right\vert _{\sigma }\lesssim \left\Vert T\right\Vert
_{L^{2}\left( \sigma \right) \rightarrow L^{2}\left( \omega \right)
}^{2}\left\vert B\left( x_{0},N\right) \right\vert _{\sigma },
\end{eqnarray*}%
since $\sigma $ is doubling.

As before, define the truncated kernels%
\begin{equation*}
K_{\varepsilon }\left( x,y\right) \equiv \left\{ 
\begin{array}{ccc}
K\left( x,y\right)  & \text{ if } & \varepsilon <\left\Vert x-y\right\Vert 
\\ 
0 & \text{ if } & not%
\end{array}%
\right. ,
\end{equation*}%
and set%
\begin{equation*}
T_{\varepsilon }\left( x\right) \equiv \int K_{\varepsilon }\left(
x,y\right) f\left( y\right) d\sigma \left( y\right) ,\ \ \ \ \ \text{for }%
x\in \mathbb{R}^{n}\text{ and }f\in L^{2}\left( \sigma \right) .
\end{equation*}%
By the previous proposition, the operators $T_{\varepsilon }^{\alpha }$ are
uniformly bounded from $L^{2}\left( \sigma \right) $ to $L^{2}\left( \omega
\right) $.

Continuing to follow Stein \cite[Section 3 of Chapter VII]{Ste2}, we compare 
$I_{\varepsilon ,N}\left( x\right) $ with $T_{\varepsilon }\left( \mathbf{1}%
_{B\left( x_{0},N\right) }\right) \left( x\right) $. Since 
\begin{equation*}
\left\{ B\left( x,N\right) \setminus B\left( x_{0},N\right) \right\}
\dbigcup \left\{ B\left( x_{0},N\right) \setminus B\left( x,N\right)
\right\} \subset B\left( x,\frac{3N}{2}\right) \setminus B\left( x,\frac{N}{2%
}\right) ,
\end{equation*}%
provided $\left\Vert x-x_{0}\right\Vert <\frac{N}{2}$, and since%
\begin{equation*}
I_{\varepsilon ,N}^{E}\left( x\right) -T_{\varepsilon }\left( \mathbf{1}%
_{B\left( x_{0},N\right) }\sigma \right) \left( x\right) =\int_{B\left(
x,N\right) \setminus B\left( x,\varepsilon \right) }K\left( x,y\right)
d\sigma \left( y\right) -\int_{B\left( x_{0},N\right) \setminus B\left(
x,\varepsilon \right) }K\left( x,y\right) d\sigma \left( y\right) ,
\end{equation*}%
it follows that%
\begin{equation*}
\left\vert I_{\varepsilon ,N}^{E}\left( x\right) -T_{\varepsilon }\left( 
\mathbf{1}_{B\left( x_{0},N\right) }\right) \left( x\right) \right\vert \leq
\int_{B\left( x,\frac{3N}{2}\right) \setminus B\left( x,\frac{N}{2}\right)
}\left\vert K\left( x,y\right) \right\vert d\sigma \left( y\right) \lesssim 
\frac{1}{N^{n}}\left\vert B\left( x,\frac{3N}{2}\right) \right\vert _{\sigma
},
\end{equation*}%
when $\left\Vert x-x_{0}\right\Vert <\frac{N}{2}$. Then%
\begin{eqnarray*}
&&\int_{\left\Vert x-x_{0}\right\Vert <\frac{N}{2}}\left\vert I_{\varepsilon
,N}\left( x\right) \right\vert ^{2}d\omega \left( x\right)  \\
&\lesssim &\int_{B\left( x_{0},\frac{N}{2}\right) }\left\vert T_{\varepsilon
}\left( \mathbf{1}_{B\left( x_{0},N\right) }\sigma \right) \left( x\right)
\right\vert ^{2}d\omega \left( x\right) +\int_{B\left( x_{0},\frac{N}{2}%
\right) }\left\vert I_{\varepsilon ,N}\left( x\right) -T_{\varepsilon
}^{\alpha }\left( \mathbf{1}_{B\left( x_{0},N\right) }\right) \left(
x\right) \right\vert ^{2}d\omega \left( x\right)  \\
&\lesssim &\sup_{\varepsilon >0}\left\Vert T_{\varepsilon }\right\Vert
_{L^{2}\left( \sigma \right) \rightarrow L^{2}\left( \omega \right)
}^{2}\left\vert B\left( x_{0},N\right) \right\vert _{\sigma }+\left\vert
B\left( x_{0},\frac{N}{2}\right) \right\vert _{\omega }\frac{1}{N^{2n}}%
\left\vert B\left( x_{0},\frac{3N}{2}\right) \right\vert _{\sigma }^{2} \\
&\lesssim &\left\{ \left\Vert T\right\Vert _{L^{2}\left( \sigma \right)
\rightarrow L^{2}\left( \omega \right) }^{2}+A_{2}^{\alpha }\left( \sigma
,\omega \right) \right\} \ \left\vert B\left( x_{0},\frac{3N}{2}\right)
\right\vert _{\sigma }\lesssim \left\{ \left\Vert T\right\Vert _{L^{2}\left(
\sigma \right) \rightarrow L^{2}\left( \omega \right) }^{2}+A_{2}^{\alpha
}\left( \sigma ,\omega \right) \right\} \ \left\vert B\left( x_{0},N\right)
\right\vert _{\sigma },
\end{eqnarray*}%
since $\sigma $ is doubling. This proves (\ref{it suff}), and hence the
necessity of (\ref{can cond}) with $\mathfrak{A}_{K}\left( \sigma ,\omega
\right) \lesssim \left\Vert T\right\Vert _{L^{2}\left( \sigma \right)
\rightarrow L^{2}\left( \omega \right) }^{2}+A_{2}^{\alpha }\left( \sigma
,\omega \right) $. The proof of necessity of the dual condition to (\ref{can
cond}) is similar using that $\omega $ is doubling.

Conversely, as in Stein \cite[Section 3 of Chapter VII]{Ste2}, let $%
K^{\varepsilon }\left( x,y\right) $ be a smooth truncation of $K$ given by%
\begin{equation*}
K^{\varepsilon }\left( x,y\right) \equiv \eta \left( \frac{x-y}{\varepsilon }%
\right) K\left( x,y\right) ,
\end{equation*}%
where $\eta \left( x\right) $ is smooth, vanishes if $\left\Vert
x\right\Vert \leq \frac{1}{2}$ and equals $1$ if $\left\Vert x\right\Vert
\geq 1$. Note that the kernels $K^{\varepsilon }\left( x,y\right) $ satisfy (%
\ref{diff ineq}) uniformly in $\varepsilon >0$, and can be used as
truncations in defining the weighted norm inequality as in Subsubsection \ref%
{Subsubsection norm}\ - see Independence of Truncations \ref{independence}.
We will show that the operators $T^{\varepsilon }$ corresponding to $%
K^{\varepsilon }$ satisfy the $\kappa $-Cube Testing conditions, also
uniformly in $\varepsilon >0$. For this we begin by controlling the full $%
\kappa $-Cube Testing condition for $T^{\varepsilon }$ by the following
polynomial variant of (\ref{can cond'}),%
\begin{eqnarray}
&&\int_{\left\Vert x-x_{0}\right\Vert <N}\left\vert \int_{\varepsilon
<\left\Vert x-y\right\Vert <N}K^{\alpha }\left( x,y\right) \frac{p\left(
y\right) }{\left\Vert \mathbf{1}_{B\left( x_{0},N\right) }p\right\Vert
_{\infty }}d\sigma \left( y\right) \right\vert ^{2}d\omega \left( x\right)
\leq \mathfrak{A}_{K^{\alpha }}^{\left( \kappa \right) }\left( \sigma
,\omega \right) \ \int_{\left\Vert x_{0}-y\right\Vert <N}d\sigma \left(
y\right) ,  \label{can cond''} \\
&&\text{for all polynomials }p\text{ of degree less than }\kappa \text{, all 
}0<\varepsilon <N\text{ and }x_{0}\in \mathbb{R}^{n},  \notag
\end{eqnarray}%
where $\mathfrak{A}_{K^{\alpha }}^{\left( \kappa \right) }\left( \sigma
,\omega \right) $ denotes the smallest constant for which (\ref{can cond''})
holds, and where $\kappa \in \mathbb{N}$.

To see this, fix a positive integer $\kappa >n-\alpha $, and define 
\begin{equation*}
I_{\varepsilon ,R}\left( x\right) \equiv \int_{\varepsilon <\left\Vert
x-y\right\Vert <R}K\left( x,y\right) \mathbf{1}_{B\left( x_{0},R\right)
}\left( y\right) d\sigma \left( y\right) .
\end{equation*}%
If $\phi _{\kappa }^{R,x_{0}}$ is a $B\left( x_{0},R\right) $-normalized
polynomial of degree less than $\kappa $ as in Definition \ref{def Q norm},
and $\left\Vert x-x_{0}\right\Vert <2R$, and if we denote by $\limfunc{Tay}%
f\left( x\right) $ the $\left( \kappa -1\right) ^{st}$ degree Taylor
polynomial of $f$ at $x$, then 
\begin{eqnarray*}
T^{\varepsilon }\left( \phi _{\kappa }^{R,x_{0}}\mathbf{1}_{B\left(
x_{0},R\right) }\sigma \right) \left( x\right)  &=&\int K^{\varepsilon
}\left( x,y\right) \phi _{\kappa }^{R,x_{0}}\left( y\right) \mathbf{1}%
_{B\left( x_{0},R\right) }\left( y\right) d\sigma \left( y\right)  \\
&=&\int K^{\varepsilon }\left( x,y\right) \left[ \phi _{\kappa
}^{R,x_{0}}\left( y\right) -\limfunc{Tay}\phi _{\kappa }^{R,x_{0}}\left(
x\right) \right] \mathbf{1}_{B\left( x,3R\right) }\left( y\right) \mathbf{1}%
_{B\left( x_{0},R\right) }\left( y\right) d\sigma \left( y\right)  \\
&&+\phi _{\kappa }^{R,x_{0}}\left( x\right) \int_{B\left( x,3R\right) }%
\limfunc{Tay}K^{\varepsilon }\left( x,y\right) \mathbf{1}_{B\left(
x_{0},R\right) }\left( y\right) d\sigma \left( y\right) .
\end{eqnarray*}%
The first integral is estimated by%
\begin{equation*}
A\int_{B\left( x,3R\right) }\left\vert x-y\right\vert ^{\alpha -n}\left( 
\frac{\left\vert x-y\right\vert }{R}\right) ^{\kappa }\mathbf{1}_{B\left(
x_{0},R\right) }\left( y\right) d\sigma \left( y\right) \lesssim A\frac{1}{%
R^{n}}\left\vert B\left( x,3R\right) \right\vert _{\sigma }\ ,
\end{equation*}%
since we chose $\kappa >n-\alpha $. On the other hand the integral $%
\int_{B\left( x,3R\right) }\limfunc{Tay}K^{\varepsilon }\left( x,y\right) 
\mathbf{1}_{B\left( x_{0},R\right) }\left( y\right) d\sigma \left( y\right) $
differs from $I_{\varepsilon ,R}\left( x\right) $ by%
\begin{equation*}
\int_{B\left( x,3R\right) \setminus B\left( x,\varepsilon \right) }\left\{ 
\limfunc{Tay}K^{\varepsilon }\left( x,y\right) -K\left( x,y\right) \right\} 
\mathbf{1}_{B\left( x_{0},R\right) }\left( y\right) d\sigma \left( y\right) ,
\end{equation*}%
whose modulus is again at most 
\begin{equation*}
\int_{B\left( x,3R\right) \setminus B\left( x,\varepsilon \right)
}\left\vert x-y\right\vert ^{\alpha -n}\left( \frac{\left\vert
x-y\right\vert }{R}\right) ^{\kappa }\mathbf{1}_{B\left( x_{0},R\right)
}\left( y\right) d\sigma \left( y\right) \lesssim A\frac{1}{R^{n}}\left\vert
B\left( x,3R\right) \right\vert _{\sigma }\ .
\end{equation*}%
Thus (\ref{it suff}) implies that%
\begin{eqnarray*}
\int_{B\left( x_{0},2R\right) }\left\vert T^{\varepsilon }\left( \phi
_{\kappa }^{R,x_{0}}\mathbf{1}_{B\left( x_{0},R\right) }\sigma \right)
\right\vert ^{2}d\omega \left( x\right)  &\lesssim &\left\{ A_{2}^{\alpha }+%
\mathfrak{A}_{K}^{\left( \kappa \right) }\left( \sigma ,\omega \right)
\right\} \ \left\vert B\left( x_{0},5R\right) \right\vert _{\sigma } \\
&\lesssim &\left\{ A_{2}^{\alpha }+\mathfrak{A}_{K}^{\left( \kappa \right)
}\left( \sigma ,\omega \right) \right\} \left\vert B\left( x_{0},R\right)
\right\vert _{\sigma }\ ,
\end{eqnarray*}%
since $\sigma $ is doubling. Taking the supremum over cubical balls $B\left(
x_{0},R\right) $ yields%
\begin{equation*}
\mathcal{FT}_{T}^{\left( \kappa \right) }\left( \sigma ,\omega \right)
\lesssim \sqrt{A_{2}^{\alpha }}+\mathfrak{A}_{K}^{\left( \kappa \right)
}\left( \sigma ,\omega \right) .
\end{equation*}%
Similarly we have $\mathcal{FT}_{T^{\ast }}^{\left( \kappa \right) }\left(
\sigma ,\omega \right) \lesssim \sqrt{A_{2}^{\alpha }}+\mathfrak{A}_{K^{\ast
}}^{\left( \kappa \right) }\left( \omega ,\sigma \right) $.

At this point we need the following $\mathfrak{A}_{K}^{\left( \kappa \right)
}$-variant of Corollary \ref{full Tp control}: For every $\kappa \in \mathbb{%
N}$ and $0<\delta <1$, there is a positive constant $C_{\kappa ,\delta }$
such that 
\begin{equation*}
\mathfrak{A}_{K^{\alpha }}^{\left( \kappa \right) }\left( \sigma ,\omega
\right) \leq C_{\kappa ,\delta }\mathfrak{A}_{K^{\alpha }}\left( \sigma
,\omega \right) +\delta \mathfrak{N}_{T}\left( \sigma ,\omega \right) \ ,
\end{equation*}%
and where the constants $C_{\kappa ,\delta }$ depend only on $\kappa $ and $%
\delta $, and not on the operator norm $\mathfrak{N}_{T}\left( \sigma
,\omega \right) $. The proof of this variant is similar to that of Theorem %
\ref{Tp control by T1}, Proposition \ref{full control for doub} and
Corollary \ref{full Tp control}, and is left to the reader. With this
variant in hand, we now have%
\begin{equation*}
\mathcal{FT}_{T}^{\left( \kappa \right) }\left( \sigma ,\omega \right)
\lesssim C_{\kappa ,\delta }\left[ \sqrt{A_{2}^{\alpha }}+\mathfrak{A}%
_{K}\left( \sigma ,\omega \right) \right] +\delta \mathfrak{N}_{T}\left(
\sigma ,\omega \right) ,
\end{equation*}%
for arbitrarily small $\delta >0$.

In view of Theorem \ref{A infinity theorem}, and absorbing the term $\delta 
\mathfrak{N}_{T}\left( \sigma ,\omega \right) $ for $\delta >0$ sufficiently
small, the operator norms of the truncated operators $T^{\varepsilon }$ are
now bounded uniformly in $\varepsilon >0$. Thus there is a sequence $\left\{
\varepsilon _{k}\right\} _{k=1}^{\infty }$ with $\lim_{k\rightarrow \infty
}\varepsilon _{k}=0$ such that the operators $T^{\varepsilon _{k}}$ converge
weakly to a bounded operator $T$ from $L^{2}\left( \sigma \right) $ to $%
L^{2}\left( \omega \right) $. Since the truncated kernels $K^{\varepsilon
_{k}}\left( x,y\right) $ converge pointwise and dominatedly to $K\left(
x,y\right) $, Lebesgue's Dominated Convergence theorem applies to show that
for $x\notin \func{Supp}\left( f\sigma \right) $, and where the doubling
measure $\sigma $ has no atoms and the function $f$ has compact support, we
have%
\begin{equation*}
T\left( f\sigma \right) \left( x\right) =\lim_{k\rightarrow \infty
}T^{\varepsilon _{k}}\left( f\sigma \right) \left( x\right)
=\lim_{k\rightarrow \infty }\int K^{\varepsilon _{k}}\left( x,y\right)
f\left( y\right) d\sigma \left( y\right) =\int K\left( x,y\right) f\left(
y\right) d\sigma \left( y\right) ,
\end{equation*}%
which is the representation (\ref{identify}). This completes the proof of
Theorem \ref{Stein extension}.
\end{proof}

\section{Concluding remarks}

The problem investigated in this paper is that of \emph{fixing a measure pair%
} $\left( \sigma ,\omega \right) $, and then asking for a characterization
of the $\alpha $-fractional Calder\'{o}n-Zygmund operators $T^{\alpha }$
that are bounded from $L^{2}\left( \sigma \right) $ to $L^{2}\left( \omega
\right) $ - the first solution being the one weight case of Lebesgue measure
with $\alpha =0$ in \cite{DaJo}. This problem of fixing a measure pair is in
a sense `orthogonal' to other recent investigations of two weight norm
inequalities, in which one \emph{fixes the elliptic operator }$T^{\alpha }$,
and asks for a characterization of the weight pairs $\left( \sigma ,\omega
\right) $ for which $T^{\alpha }$ is bounded.

This latter investigation for a fixed operator is extraordinarily difficult,
with essentially just \textbf{one} Calder\'{o}n-Zygmund operator $T^{\alpha
} $ known to have a characterization of the weight pairs $\left( \sigma
,\omega \right) $, namely the Hilbert transform on the line, see the two
part paper \cite{LaSaShUr3};\cite{Lac}, and also \cite{SaShUr10} for an
extension to gradient elliptic operators on the line. In particular, matters
appear to be very bleak in higher dimensions due to the example in \cite{Saw}
which shows that the energy side condition, used in virtually all attempted
characterizations, fails to be necessary for even the most basic elliptic
operators - the stronger pivotal condition is however shown in \cite{LaLi}
to be necessary for boundedness of the $g$-function, a \emph{Hilbert space }%
valued Calder\'{o}n-Zygmund operator with a strong gradient positivity
property, and the weight pairs were then characterized in \cite{LaLi} by a
single testing condition\footnote{%
The testing condition (1.3) in \cite{LaLi} implies the weights share no
common point masses, and then an argument in \cite{LaSaUr1} using the
asymmetric $A_{2}$ condition of Stein shows that the $A_{2}$ condition is
implied by the testing condition. Thus (1.3) can be dropped from the
statement of Theorem 1.2.}.

On the other hand, the problem for a fixed measure pair has proved somewhat
more tractable. However, the techniques required for these results are taken
largely from investigations of the problem where the operator is fixed. In
particular, an adaptation of the `pivotal' argument in \cite{NTV4} to the
weighted Alpert wavelets in \cite{RaSaWi}, and a Parallel Corona
decomposition from \cite{LaSaShUr4} are used.

The question of relaxing the side conditions of doubling, comparability of
measures, and $A_{\infty }^{\alpha }$ or $C_{q}$ on the weights remains
open, with the main stumbling blocks being (\textbf{1}) the limitations of
weighted Alpert wavelets which require doubling, and (\textbf{2}) our
bilinear Carleson Embedding Theorem which requires comparability of
measures. There is in fact no known example of a $\alpha $-fractional Calder%
\'{o}n-Zygmund operator for which the $T1$ theorem fails.

For $0<\alpha <n$ there ought to be a larger class $C_{q}^{\alpha }$ of
measures that includes both $A_{\infty }^{\alpha }$ and $C_{q}$, $q>2$, and
for which a weighted norm of the fractional integral $I_{\alpha }$ is
controlled by that of the fractional maximal function $M_{\alpha }$. One
possibility for the definition of such a class $C_{p}^{\alpha }$ of measures
for $0<\alpha <n$ and $1<p<\infty $ is%
\begin{eqnarray*}
&&\frac{\left\vert E\right\vert _{\omega }}{\int_{Q}\left\vert M\mathbf{1}%
_{Q}\right\vert ^{p}\omega }\leq \eta \left( \mathbf{Cap}_{\alpha
}^{Q}\left( E\right) \right) ,\ \ \ \ \ \ \text{for all compact subsets }E%
\text{ of a cube }Q, \\
&&\ \ \ \ \ \ \ \ \ \ \text{for some function }\eta :\left[ 0,1\right]
\rightarrow \left[ 0,1\right] \text{ with }\lim_{t\searrow 0}\eta \left(
t\right) =0.
\end{eqnarray*}

In the case $\alpha =0$, there is the problem analogous to the `$A_{2}$
conjecture' solved in general in \cite{Hyt}, of determining the optimal
dependence of the above estimates on the $A_{2}$ characteristic. In
particular the dependence for the restricted weak type inequality should
follow using the pigeonholing and corona construction introduced in \cite%
{LaPeRe} and used in \cite{Hyt}.

We end by summarizing the drawbacks in the methodology used here. The $T1$
theorem here is proved for general Calder\'{o}n-Zygmund operators, and thus
in the absence of any special positivity properties of the Calder\'{o}%
n-Zygmund kernels $K^{\alpha }$. As a consequence there is no \emph{catalyst}
available to enable control of the difficult `far below' and `stopping'
terms by `goodness' of cubes in the NTV bilinear Haar decomposition (see
e.g. \cite{NTV4}). In the case of the aforementioned Hilbert transform, the
positivity of the derivative of the convolution kernel $\frac{1}{x}$
permitted the derivation of a strong catalyst, namely the energy condition,
from the testing and Muckenhoupt conditions (see e.g. \cite{LaSaShUr3}), and
in the case of Riesz transforms there is a partial reversal energy that
yields the energy condition when the measures are both doubling (see e.g. 
\cite{LaWi} and \cite{SaShUr9}). But the lack of a suitable catalyst for
general Calder\'{o}n-Zygmund operators, see \cite{SaShUr11} and \cite{Saw}
for negative results, limits us to using the weighted Alpert wavelets in 
\cite{RaSaWi}. The weighted Alpert wavelets in turn have two defects%
\footnote{%
Weighted $L^{2}$ projections fail to satisfy $L^{\infty }$ bounds in
general, and the size of an extension of a nonconstant polynomial is
uncontrolled.} that limit their use to doubling measures, and to situations
that avoid the paraproduct / neighbour / stopping form decomposition of NTV
in \cite{NTV4}. This forces us to use the parallel corona, and ultimately to
invoke comparability of measures and the $A_{\infty }^{\alpha }$ or $C_{q}$, 
$q>2$, assumption on one of the measures.

\section{Reference List of conditions}

For the reader's convenience we assemble here a Reference List of the
conditions on weights and weight pairs arising in this paper in roughly the
order of their appearance.

\subsection{Conditions on a single measure}

\begin{enumerate}
\item $\mu $ is doubling if $\int_{2Q}d\mu \lesssim \int_{Q}d\mu ,\ $for$\ $%
all cubes $Q\subset \mathbb{R}^{n}$.

\item $\omega $ is an $A_{\infty }$ weight if $\frac{\left\vert E\right\vert
_{\omega }}{\left\vert Q\right\vert _{\omega }}\leq C\left( \frac{\left\vert
E\right\vert }{\left\vert Q\right\vert }\right) ^{\varepsilon }\ $for all
compact subsets $E$ of a cube $Q$.

\item $\sigma $ is a $C_{p}$ weight if $\frac{\left\vert E\right\vert
_{\sigma }}{\int_{\mathbb{R}^{n}}\left\vert M\mathbf{1}_{Q}\right\vert
^{p}d\sigma }\leq C\left( \frac{\left\vert E\right\vert }{\left\vert
Q\right\vert }\right) ^{\varepsilon }$ whenever $E$ compact $\subset Q$ a
cube.

\item $\omega $ is an $A_{\infty }^{\alpha }$ measure if $\frac{\left\vert
E\right\vert _{\omega }}{\left\vert Q\right\vert _{\omega }}\leq \eta \left( 
\mathbf{Cap}_{\alpha }^{Q}\left( E\right) \right) ,\ $for all compact
subsets $E$ of a cube $Q$.
\end{enumerate}

\subsection{Conditions on a pair of measures}

\begin{enumerate}
\item $\left( \sigma ,\omega \right) $ \emph{comparable} if $\frac{%
\left\vert E\right\vert _{\sigma }}{\left\vert Q\right\vert _{\sigma }}<\eta 
$ whenever $E$ compact $\subset Q$ a cube with $\frac{\left\vert
E\right\vert _{\omega }}{\left\vert Q\right\vert _{\omega }}<\varepsilon $.

\item For $0\leq \alpha <n$, the $\alpha $-fractional Muckenhoupt conditions
for the weight pair $\left( \sigma ,\omega \right) $ are%
\begin{eqnarray*}
A_{2}^{\alpha }\left( \sigma ,\omega \right) &\equiv &\sup_{Q\in \mathcal{P}%
^{n}}\frac{\left\vert Q\right\vert _{\sigma }}{\left\vert Q\right\vert ^{1-%
\frac{\alpha }{n}}}\frac{\left\vert Q\right\vert _{\omega }}{\left\vert
Q\right\vert ^{1-\frac{\alpha }{n}}}<\infty , \\
\mathcal{A}_{2}^{\alpha }\left( \sigma ,\omega \right) &\equiv &\sup_{Q\in 
\mathcal{Q}^{n}}\mathcal{P}^{\alpha }\left( Q,\sigma \right) \frac{%
\left\vert Q\right\vert _{\omega }}{\left\vert Q\right\vert ^{1-\frac{\alpha 
}{n}}}<\infty , \\
\mathcal{A}_{2}^{\alpha ,\ast }\left( \sigma ,\omega \right) &\equiv
&\sup_{Q\in \mathcal{Q}^{n}}\frac{\left\vert Q\right\vert _{\sigma }}{%
\left\vert Q\right\vert ^{1-\frac{\alpha }{n}}}\mathcal{P}^{\alpha }\left(
Q,\omega \right) <\infty , \\
\mathcal{P}^{\alpha }\left( Q,\mu \right) &\equiv &\int_{\mathbb{R}%
^{n}}\left( \frac{\left\vert Q\right\vert ^{\frac{1}{n}}}{\left( \left\vert
Q\right\vert ^{\frac{1}{n}}+\left\vert x-x_{Q}\right\vert \right) ^{2}}%
\right) ^{n-\alpha }d\mu \left( x\right) .
\end{eqnarray*}

\item The $\kappa $\emph{-cube testing conditions} for $T^{\alpha }$ are%
\begin{eqnarray*}
\left( \mathfrak{T}_{T^{\alpha }}^{\left( \kappa \right) }\left( \sigma
,\omega \right) \right) ^{2} &\equiv &\sup_{Q\in \mathcal{P}^{n}}\max_{0\leq
\left\vert \beta \right\vert <\kappa }\frac{1}{\left\vert Q\right\vert
_{\sigma }}\int_{Q}\left\vert T_{\sigma }^{\alpha }\left( \mathbf{1}%
_{Q}m_{Q}^{\beta }\right) \right\vert ^{2}\omega <\infty , \\
\left( \mathfrak{T}_{\left( T^{\alpha }\right) ^{\ast }}^{\left( \kappa
\right) }\left( \omega ,\sigma \right) \right) ^{2} &\equiv &\sup_{Q\in 
\mathcal{P}^{n}}\max_{0\leq \left\vert \beta \right\vert <\kappa }\frac{1}{%
\left\vert Q\right\vert _{\omega }}\int_{Q}\left\vert \left( T_{\sigma
}^{\alpha }\right) ^{\ast }\left( \mathbf{1}_{Q}m_{Q}^{\beta }\right)
\right\vert ^{2}\sigma <\infty ,
\end{eqnarray*}%
with $m_{Q}^{\beta }\left( x\right) \equiv \left( \frac{x-c_{Q}}{\ell \left(
Q\right) }\right) ^{\beta }$ for any cube $Q$ and multiindex $\beta $, where 
$c_{Q}$ is the center of the cube $Q$.

\item The \emph{full }$\kappa $\emph{-cube testing conditions} for $%
T^{\alpha }$ are%
\begin{eqnarray*}
\left( \mathfrak{FT}_{T^{\alpha }}^{\left( \kappa \right) }\left( \sigma
,\omega \right) \right) ^{2} &\equiv &\sup_{Q\in \mathcal{P}^{n}}\max_{0\leq
\left\vert \beta \right\vert <\kappa }\frac{1}{\left\vert Q\right\vert
_{\sigma }}\int_{\mathbb{R}^{n}}\left\vert T_{\sigma }^{\alpha }\left( 
\mathbf{1}_{Q}m_{Q}^{\beta }\right) \right\vert ^{2}\omega <\infty , \\
\left( \mathfrak{FT}_{\left( T^{\alpha }\right) ^{\ast }}^{\left( \kappa
\right) }\left( \omega ,\sigma \right) \right) ^{2} &\equiv &\sup_{Q\in 
\mathcal{P}^{n}}\max_{0\leq \left\vert \beta \right\vert <\kappa }\frac{1}{%
\left\vert Q\right\vert _{\omega }}\int_{\mathbb{R}^{n}}\left\vert \left(
T^{\alpha }\right) _{\omega }^{\ast }\left( \mathbf{1}_{Q}m_{Q}^{\beta
}\right) \right\vert ^{2}\sigma <\infty .
\end{eqnarray*}

\item The weak boundedness constant is%
\begin{equation*}
\mathcal{WBP}_{T^{\alpha }}^{\left( \kappa _{1},\kappa _{2}\right) }\left(
\sigma ,\omega \right) =\sup_{\mathcal{D}\in \Omega }\sup_{\substack{ %
Q,Q^{\prime }\in \mathcal{D}  \\ Q\subset 3Q^{\prime }\setminus Q^{\prime }%
\text{ or }Q^{\prime }\subset 3Q\setminus Q}}\frac{1}{\sqrt{\left\vert
Q\right\vert _{\sigma }\left\vert Q^{\prime }\right\vert _{\omega }}}\sup 
_{\substack{ f\in \left( \mathcal{P}_{Q}^{\kappa _{1}}\right) _{\limfunc{norm%
}}  \\ g\in \left( \mathcal{P}_{Q}^{\kappa _{2}}\right) _{\limfunc{norm}}}}%
\left\vert \int_{Q^{\prime }}T_{\sigma }^{\alpha }\left( \mathbf{1}%
_{Q}f\right) \ gd\omega \right\vert <\infty ,
\end{equation*}%
where $\left( \mathcal{P}_{\kappa }^{Q}\right) _{\limfunc{norm}}$ is the
space of $Q$-normalized polynomials of degree less than $\kappa $
(Definition \ref{def Q norm}).

\item The Bilinear Indicator/Cube Testing property is%
\begin{equation*}
\mathcal{BICT}_{T^{\alpha }}\left( \sigma ,\omega \right) \equiv \sup_{Q\in 
\mathcal{P}^{n}}\sup_{E,F\subset Q}\frac{1}{\sqrt{\left\vert Q\right\vert
_{\sigma }\left\vert Q\right\vert _{\omega }}}\left\vert \int_{F}T_{\sigma
}^{\alpha }\left( \mathbf{1}_{E}\right) \omega \right\vert <\infty ,
\end{equation*}%
where the second supremum is taken over all compact sets $E$ and $F$
contained in a cube $Q$.

\item The $\kappa ^{th}$-order fractional Pivotal Conditions $\mathcal{V}%
_{2}^{\alpha ,\kappa },\mathcal{V}_{2}^{\alpha ,\kappa ,\ast }<\infty $, $%
\kappa \geq 1$, are%
\begin{eqnarray*}
\left( \mathcal{V}_{2}^{\alpha ,\kappa }\right) ^{2} &=&\sup_{Q\supset \dot{%
\cup}Q_{r}}\frac{1}{\left\vert Q\right\vert _{\sigma }}\sum_{r=1}^{\infty }%
\mathrm{P}_{\kappa }^{\alpha }\left( Q_{r},\mathbf{1}_{Q}\sigma \right)
^{2}\left\vert Q_{r}\right\vert _{\omega }\ , \\
\left( \mathcal{V}_{2}^{\alpha ,\kappa ,\ast }\right) ^{2} &=&\sup_{Q\supset 
\dot{\cup}Q_{r}}\frac{1}{\left\vert Q\right\vert _{\omega }}%
\sum_{r=1}^{\infty }\mathrm{P}_{\kappa }^{\alpha }\left( Q_{r},\mathbf{1}%
_{Q}\omega \right) ^{2}\left\vert Q_{r}\right\vert _{\sigma }\ , \\
\mathrm{P}_{\kappa }^{\alpha }\left( Q,\mu \right) &=&\int_{\mathbb{R}^{n}}%
\frac{\ell \left( Q\right) ^{\kappa }}{\left( \ell \left( Q\right)
+\left\vert y-c_{Q}\right\vert \right) ^{n+\kappa -\alpha }}d\mu \left(
y\right) ,\ \ \ \ \ \kappa \geq 1,
\end{eqnarray*}%
where the supremum is taken over all subdecompositions of a cube $Q\in 
\mathcal{P}^{n}$ into pairwise disjoint subcubes $Q_{r}$.
\end{enumerate}

\end{document}